\documentclass[final,notitlepage,12pt,reqno]{amsart}
\usepackage{graphicx}
\usepackage{CJK}
\usepackage{calc}
\newlength{\depthofsumsign}
\newcommand{\nsum}[1][1.0]{
    \mathop{%
        \raisebox
            {-#1\depthofsumsign+1\depthofsumsign}
            {\scalebox
                {#1}
                {$\displaystyle\sum$}%
            }
    }
}
\newcommand{\nprod}[1][1.0]{
    \mathop{%
        \raisebox
            {-#1\depthofsumsign+1\depthofsumsign}
            {\scalebox
                {#1}
                {$\displaystyle\prod$}%
            }
    }
}

\makeatletter
\let\I\@undefined
\makeatother

\usepackage{geometry}
\usepackage{amsthm}

\usepackage{enumitem}[2011/09/28]
\setenumerate{align=left, leftmargin=0pt,labelsep=.5em, labelindent=0\parindent,listparindent=\parindent,itemindent=*}
\usepackage{empheq}
\usepackage{natbib}
\usepackage[all]{xy}
\usepackage{url}

\usepackage{caption}[2012/02/19]

\usepackage{fouriernc}
\usepackage{fourier}
\usepackage{amssymb,bm,amsmath}
\usepackage{amsfonts}
\usepackage{lscape}
\usepackage[OT2,T1,T2A]{fontenc}

\usepackage[english,french,german,russian]{babel}
\usepackage{appendix}

\DeclareMathOperator{\lcm}{lcm}

\DeclareMathOperator{\Gal}{Gal}

\DeclareMathOperator{\sn}{sn}
\DeclareMathOperator{\cn}{cn}
\DeclareMathOperator{\dn}{dn}

\DeclareMathOperator{\JZ}{Z}

\DeclareMathOperator{\D}{d}
\DeclareMathOperator{\I}{Im}
\DeclareMathOperator{\R}{Re}

\def\Xint#1{\mathchoice
   {\XXint\displaystyle\textstyle{#1}}%
   {\XXint\textstyle\scriptstyle{#1}}%
   {\XXint\scriptstyle\scriptscriptstyle{#1}}%
   {\XXint\scriptscriptstyle\scriptscriptstyle{#1}}%
   \!\int}
\def\XXint#1#2#3{{\setbox0=\hbox{$#1{#2#3}{\int}$}
     \vcenter{\hbox{$#2#3$}}\kern-.5\wd0}}

\def\RotSymbol#1#2#3{\rotatebox[origin=c]{#1}{$#2#3$}}
\def\rotcirclearrowleft{\mathpalette{\RotSymbol{-14}}\circlearrowleft}

\def\counterint{\Xint\rotcirclearrowleft}

\bibpunct{[}{]}{,}{n}{}{;}

\def\eor{\hfill$ \square$}

\theoremstyle{plain}
\newtheorem{theorem}{Theorem}[subsection]
\newtheorem{proposition}[theorem]{Proposition}
\newtheorem{lemma}[theorem]{Lemma}

\theoremstyle{definition}

\theoremstyle{remark}

\newtheorem{remark}{Remark}[theorem]

\numberwithin{equation}{subsection}

\usepackage{color}

\begin{document}

\selectlanguage{english}
\pagenumbering{roman}
\title[Kontsevich--Zagier Integrals for  Automorphic Green's Functions. I]{Kontsevich--Zagier Integrals for  Automorphic Green's Functions.\\ I
}
\author{Yajun Zhou}
\address{Program in Applied and Computational Mathematics, Princeton University, Princeton, NJ 08544, USA}
\email{yajunz@math.princeton.edu}



\maketitle
\begin{abstract}


     In the framework of Kontsevich--Zagier periods, we derive  integral representations for weight-$k$ automorphic Green's functions  invariant under modular transformations in  $\varGamma_0(N)$ ($N\in\mathbb Z_{\geq1} $), provided that there are no cusp forms on the respective Hecke congruence groups  with  an even integer weight $k\geq4$. These Kontsevich--Zagier integral representations for automorphic Green's functions give explicit formulae for certain Eichler--Shimura maps connecting   Eichler cohomology  to  Maa{\ss } cusp forms. We construct integral representations for weight-4 Gross--Zagier renormalized Green's functions (automorphic self-energy) from limit scenarios of the respective Kontsevich--Zagier integrals. We reduce the weight-4 automorphic self-energy  on $ X_0(4)(\mathbb C)=\varGamma_0(4)\backslash\mathfrak H^*$  to an explicit form, which  supports  an algebraicity conjecture of Gross and Zagier. \textit{Subject Classification (AMS 2010)}:          11F03, 11F37,          11F67, 11G07, 11M06,          11Y70, 33B15, 33C05, 33C20,         33C75,          33C80, 33E05, 33E20, 33E30  \end{abstract}

\tableofcontents
\clearpage

\pagenumbering{arabic}
\section{Introduction}\subsection{Background and Motivations\label{subsec:background}}
Let $ \varGamma$ be a congruence subgroup of the modular group $ SL(2,\mathbb Z):=\left\{ \left.\left(\begin{smallmatrix}a&b\\c&d\end{smallmatrix}\right)\right| a,b,c,d\in\mathbb Z;ad-bc=1\right\}$   with projective counterpart $ \overline{\varGamma}=\varGamma/Z(\varGamma)$,  where $ Z(\varGamma):=\{\hat\zeta\in\varGamma|\hat \zeta\hat\gamma=\hat\gamma\hat\zeta,\forall\hat\gamma\in\varGamma\}$ is the center of the group $ \varGamma$. For an even number $k>2$,  the automorphic Green's function of weight $k$ on the projective   congruence subgroup $ \overline{\varGamma}$ is explicitly given as
\begin{align}G^{\mathfrak H/\overline {\varGamma}}_{k/2}(z_1,z_2):=-\frac{2}{[\varGamma:\overline{\varGamma}]}\sum_{\hat  \gamma\in\varGamma}Q_{\frac{k}{2}-1}
\left( 1+\frac{\vert z_{1} -\hat  \gamma z_2\vert ^{2}}{2\I z_1\I(\hat\gamma z_2)} \right)=-2\sum_{\hat  \gamma\in\overline{\varGamma}}Q_{\frac{k}{2}-1}
\left( 1+\frac{\vert z_{1} -\hat  \gamma z_2\vert ^{2}}{2\I z_1\I(\hat\gamma z_2)} \right),\label{eq:auto_Green_defn_Q_nu}\end{align}where $ \hat \gamma z:=\frac{az+b}{cz+d}$ for any transformation $\hat  \gamma={\left(\begin{smallmatrix}a&b\\c& d\end{smallmatrix}\right)}$, and  $Q_\nu $ is the Legendre function of the second kind  defined by  the Laplace integral\begin{align}\label{eq:Q_nu_Laplace_int}Q_{\nu}(t):=\int_0^\infty\frac{\D u}{(t+\sqrt{t^2-1}\cosh u)^{\nu+1}},\quad t>1,\nu>-1.\end{align}The function $ G^{\mathfrak H/\overline {\varGamma}}_{k/2}(z_1,z_2)$ is bi-$ \varGamma$-invariant: $ G^{\mathfrak H/\overline {\varGamma}}_{k/2}(z_1,z_2)=G^{\mathfrak H/\overline {\varGamma}}_{k/2}(\hat \gamma z_1,\hat \gamma' z_2),\forall\hat \gamma,\hat\gamma'\in\varGamma$ (thus referred to as ``automorphic''), and serves as the propagator of a Maa{\ss} wave equation on the Riemann surface $ (\mathfrak H\cup\mathbb Q\cup\{i\infty\})\backslash\varGamma$ (hence the name  ``Green's function'').  Here in Eq.~\ref{eq:auto_Green_defn_Q_nu}, we have adopted the normalization in~\cite[][p.~207]{GrossZagier1985},~\cite[][pp.~238--239]{GrossZagierI} and~\cite[][p.~544]{GrossZagierII}; the same function has also been defined in~\cite[][\S6.6]{Hejhal1983} and~\cite[][\S5]{IwaniecGSM53},  up to different normalizing constants. In the automorphic Green's function, the arguments  $ z_1$ and $ z_2$ both reside in the upper half-plane $ \mathfrak H=\{z\in\mathbb C|\I z>0\}$. When the automorphic Green's function on $ \overline{\varGamma}$ assumes a finite value,  its arguments are not equivalent per any transformation in the respective symmetry group: $ z_1\notin \varGamma z_2$. In view of the $ \varGamma$-invariance,  the automorphic Green's function $G^{\mathfrak H/\overline {\varGamma}}_{k/2}(z_1,z_2)$ is effectively defined on the ``off-diagonal'' points in the Cartesian product of two orbit spaces $(\varGamma\backslash\mathfrak H)\times(\varGamma\backslash\mathfrak H) $. The notation  $G^{\mathfrak H/\overline {\varGamma}}_{k/2}$  with  superscript $\mathfrak H/\overline {\varGamma} $ (as opposed to $\overline \varGamma\backslash\mathfrak H $) is a matter of historical convention.

The automorphic Green's function $ G^{\mathfrak H/\overline {\varGamma}}_{k/2}(z_1,z_2),z_1\notin \varGamma z_2$ has intriguing behavior when its arguments are CM points,  \textit{i.e.}~when $ \R z_1$, $\R z_2$, $(\I z_1)^2$, $(\I z_2)^2$ are all rational numbers.
 In the absence of holomorphic cusp forms of even weights $k\geq 4$ on Hecke congruence groups $ \varGamma=\varGamma_0(N):=\left\{ \left.\left(\begin{smallmatrix}a&b\\c&d\end{smallmatrix}\right)\right| a,b,c,d\in\mathbb Z;ad-bc=1;c\equiv0\pmod N\right\}$ of levels $N\in\mathbb Z_{\geq1}$  (\textit{i.e.}~$\dim\mathcal S_k(\varGamma)=\dim\mathcal S_k(\varGamma_0(N))=0$), it has been  postulated that the values of the automorphic Green's function at CM points are expressible in terms of the logarithms of certain algebraic numbers   (see~\cite[][p.~317]{GrossZagierI} and~\cite[][p.~556]{GrossZagierII}):\begin{empheq}[box=\fbox]{align*}{}&\exp \left[(\I z\I z')^{(k-2)/2}G^{\mathfrak H/\overline {\varGamma}}_{k/2}(z,z')\right]\in\overline{\mathbb Q}\\ \text{if } az^2+bz+c={}&a'z'^2+b'z'+c'=0\text{ for some integers } a,b,c,a',b',c'.\end{empheq} This is known as (the ``cusp-form-free version'' of) the Gross--Kohnen--Zagier algebraicity conjecture. It has also been speculated that the minimal polynomial for the (conjecturally) algebraic number $ \exp \left[(\I z\I z')^{(k-2)/2}G^{\mathfrak H/\overline {\varGamma}}_{k/2}(z,z')\right]$  always has  solvable Galois group. For situations where there are cusp forms of even weights $k\geq4$ (\textit{i.e.}~$\dim\mathcal S_k(\varGamma)=\dim\mathcal S_k(\varGamma_0(N))>0$), a more technical algebraicity conjecture has been formulated for certain linear combinations of automorphic Green's functions that are modified by Hecke operators~\cite{GrossZagierI,GrossZagierII}.
In \cite[][p.~317]{GrossZagierI}, Gross and Zagier originally formulated an algebraicity conjecture  for $ \dim\mathcal S_k(\varGamma_0(N))\geq0$ with the additional constraint that the CM points  $ z$ and $z'$ share the same discriminant, that is, $ b^2-4ac=b'^2-4a'c'$ for the coefficients of their respective minimal polynomials. Such a restricted version of the Gross--Kohnen--Zagier algebraicity conjecture~\cite[][p.~556]{GrossZagierII} is also termed as the Gross--Zagier algebraicity conjecture.

The  Gross--Kohnen--Zagier algebraicity conjecture has triggered a series of critical developments in arithmetic algebraic geometry and automorphic function theory during the past few decades, among which the Chow group method
and the Borcherds lift approach have produced  proofs of the conjecture in various particular scenarios. In the cases where the  discriminants for the two CM arguments coincide~\cite[][p.~317, Conjecture (4.4)]{GrossZagierI}, S.-W. Zhang has interpreted the automorphic Green's function on  $\varGamma_0(N)$ in terms of  CM-Chow-cycles and has demonstrated the truth of the Gross--Zagier algebraicity conjecture for non-degenerate height pairings on CM-cycles~\cite{SWZhang1997}. In his PhD thesis~\cite{MellitThesis},
A. Mellit has proposed an extension of Zhang's method~\cite{SWZhang1997}, leading to a confirmation of the statement $\exp( G_2^{\mathfrak H/PSL(2,\mathbb Z)}(z,i)\I z)\in\overline {\mathbb Q}$ for all the CM points $z$. Building upon the work of R.~E.~Borcherds \cite{Borcherds,Borcherds1999},  J.~H.~Bruiner and T.~H.~Yang have furnished explicit formulae for certain special CM values of automorphic Green's functions associated with Hilbert modular groups~\cite{BruinierYang2006}. M.~S.~Via\-zov\-ska has adopted the Borcherds lift~\cite{Borcherds} to verify the  algebraicity conjecture on the full modular group $ PSL(2,\mathbb Z)$ for a pair of points $ z_1,z_2\in\mathbb Q(i\sqrt{|D|})$ belonging to the same imaginary quadratic field~\cite{Viazovska2011,Viazovska2012}, along with explicit factorization formulae \cite{Viazovska2012} for the algebraic number in question, up to valuation at ramified primes in the ring of integers for $\mathbb Q(i\sqrt{|D|}) $.
These theoretical developments accommodate to generic combinations of weights $k$ and  groups $ \varGamma$ (allowing the presence of cusp forms $ \dim\mathcal S_k(\varGamma)\geq0$), but the algebraicity results for automorphic Green's functions are established  only for restrictive types of CM points.

In this series of works, we shall access the values of automorphic Green's functions at arbitrary CM points via Kontsevich--Zagier periods \cite{KontsevichZagier} and elliptic function theory. In Part I, we first construct integral representations for automorphic Green's functions fulfilling the cusp-form-free condition $ \dim\mathcal S_k(\varGamma_0(N))=0$, in the spirit of Kontsevich and Zagier; we then  explicitly compute the Kontsevich--Zagier integrals for automorphic Green's functions in several reducible scenarios.
\subsection{Notations  and Statement of Results}Before stating the main results of Part I, we fix the notations and terminologies for certain arithmetic functions and give a brief overview of their analytic and algebraic properties.

For $ z\in\mathfrak H$, we write  \begin{align}\left\{\begin{array}{r@{\;:=\;}l}E_4(z)&\dfrac{45}{\pi^{4}}\nsum\limits_{\substack{m,n\in \mathbb Z\\m^2+n^2\neq0}}\dfrac{1}{(m+nz)^{4}}=1+240\nsum\limits_{n=1}^\infty\dfrac{n^3e^{2\pi inz}}{1-e^{2\pi inz}}\\E_6(z)&\dfrac{945}{2\pi^{6}}\nsum\limits_{\substack{m,n\in \mathbb Z\\m^2+n^2\neq0}}\dfrac{1}{(m+nz)^{6}}=1-504\nsum\limits_{n=1}^\infty\dfrac{n^5e^{2\pi inz}}{1-e^{2\pi inz}}\end{array}\right.\quad \label{eq:E4_E6_defn}\end{align}for the Eisenstein series, and \begin{align}\Delta(z):=\frac{[E_4(z)]^{3}-[E_6(z)]^2}{1728}=e^{2\pi iz}\left[\prod_{n=1}^\infty(1-e^{2\pi inz})\right]^{24}\label{eq:Delta_defn}\end{align}for the modular discriminant of Weierstra{\ss}.
(Some authors may define $ E_4$, $ E_6$ and $ \Delta$ with other normalizing constants, making each of them differ from our conventions in Eqs.~\ref{eq:E4_E6_defn} and \ref{eq:Delta_defn} by a rational number times a certain integer power of $ \pi$.) Writing  $ \eta(z)=e^{\pi iz/12}\prod_{n=1}^\infty(1-e^{2\pi inz}),z\in\mathfrak H$ for  the Dedekind eta function, we may recast Eq.~\ref{eq:Delta_defn} into $ \Delta(z)=[\eta(z)]^{24}$.

 In this work, we set $ \Delta'(z):=\partial \Delta(z)/\partial z$, and use the \textit{ad hoc} notation for the ``weight-2 Eisenstein series'' as follows:\begin{align}E_2^{\vphantom{*}}(z)=\frac{1}{2\pi i}\left[ \frac{\Delta'(z)}{\Delta(z)} -\frac{6i}{\I z}\right],\quad E_2^*(z)=\frac{1}{2\pi i} \frac{\Delta'(z)}{\Delta(z)}=1-24\sum_{n=1}^\infty\frac{ne^{2\pi inz}}{1-e^{2\pi inz}}.\label{eq:E2_defn}\end{align}Accordingly, one has the transformation laws under $ SL(2,\mathbb Z) $~\cite[][p.~68]{Schoeneberg}:  \begin{align} E_2^*(\hat\gamma z)=(cz+d)^2E^*_2(z)-\frac{6ic}{\pi}(cz+d),\quad \forall\hat\gamma=\begin{pmatrix}a & b \\
c & d \\
\end{pmatrix}\in SL(2,\mathbb Z)\end{align}  and under $ \varGamma_0(N)$~\cite[][p.~484]{RN3}: \begin{align} NE^*_2(N(\hat\gamma z))-E_2^*(\hat\gamma z)=(cz+d)^2[NE^*_2(N z)- E_2^*( z)],\quad \forall\hat\gamma=\begin{pmatrix}a & b \\
c & d \\
\end{pmatrix}\in \varGamma_0(N),N\in\mathbb Z_{>0}.\label{eq:E2_star_HeckeN_transf}\end{align}  We reserve the notation $ E_2$ (\textit{without} an asterisk) for a non-holomorphic function (see \cite[][Chap.~11]{RLN2}) so that  the functions $ E_2$, $ E_4$ and $ E_6$ follow similar transformation laws~\cite[][p.~67]{Schoeneberg}:\begin{align}
E_{2}(\hat \gamma z)=(cz+d)^2 E_2(z),\quad E_{4}(\hat \gamma z)=(cz+d)^4 E_4(z),\quad E_{6}(\hat \gamma z)=(cz+d)^6 E_6(z)\label{eq:E2E4E6_mod_transf}
\end{align}for $ z\in\mathfrak H$ and $\hat\gamma={\left(\begin{smallmatrix}a&b\\c& d\end{smallmatrix}\right)}\in SL(2,\mathbb Z) $. (Some authors may switch the nomenclature for $ E_2^{\vphantom{*}}$ and $ E_2^*$, placing priority on holomorphy.)
The following Eisenstein series are descendants of $ E_4$ and $ E_6$:\begin{align}
\left\{\begin{array}{r@{\;=\;}l}
E_{8}(z) & [E_4(z)]^2=1+480\nsum\limits_{n=1}^{\infty}\dfrac{n^{7}e^{2\pi inz}}{1-e^{2\pi inz}},\\[12pt]
E_{10}(z) & E_4(z)E_6(z)=1-264\nsum\limits_{n=1}^\infty\dfrac{n^{9}e^{2\pi inz}}{1-e^{2\pi inz}}, \\[12pt]
E_{14}(z) & [E_4(z)]^{2}E_6(z)=1-24\nsum\limits_{n=1}^\infty\dfrac{n^{13}e^{2\pi inz}}{1-e^{2\pi inz}}. \\
\end{array}\right.
\label{eq:E8E10E14_defn}\end{align}

We normalize  Klein's $j$-invariant as\begin{align} j(z):=\frac{1728[E_4(z)]^{3}}{[E_4(z)]^{3}-[E_6(z)]^2}=\frac{[E_4(z)]^{3}}{\Delta(z)},\end{align}so that  $ j(i)=1728$ and $ j(e^{\pi i/3})=0$. If $N$ is a positive integer and $z\in\mathfrak H$ is a CM point (such that $ [\mathbb Q(z):\mathbb Q]=2$), then each of the following expressions\begin{align}
j(z),\quad \frac{E_2(Nz)}{E_2(z)},\quad \frac{E_4(Nz)}{E_4(z)},\quad \frac{E_6(Nz)}{E_6(z)},\quad \frac{\Delta(Nz)}{\Delta(z)},\quad \frac{[E_2(z)]^{2}}{E_4(z)},\quad \frac{[E_2(z)]^{3}}{E_6 (z)}\label{eq:fns_alg_val_at_CM_pts}
\end{align} either is infinity or represents an algebraic number solvable in radical form~\cite[][p.~86, Proposition~27]{Zagier2008Mod123}. Furthermore, these algebraic numbers generate abelian extensions of the imaginary quadratic field $ \mathbb Q(z)$.
The rationale behind the aforementioned assertions on algebraicity and field extensions is sketched in Appendix~\ref{app:algebraicity}.

The main results of this article are summarized in the following two theorems.
In what follows,
the real and imaginary parts of a point $ z$ (sometimes marked with an additional  superscript or subscript) in the upper half-plane $ \mathfrak H$ are  abbreviated as $ x\equiv \R z$ and $ y\equiv \I z$ (sometimes marked with a corresponding superscript or subscript, as is applicable to $ z$).
\begin{theorem}[Integral Representations for $ G_{k/2}^{\mathfrak H/\overline{\varGamma}_0(N)}(z,z')$ where  $\dim\mathcal S_k(\varGamma_0(N))=0 $]\label{thm:KZ_int_repns}\begin{enumerate}[label=\emph{(\alph*)}, ref=(\alph*), widest=a] \item \label{itm:KZ_b} For any pair of non-equivalent points $ z$ and $z'$ (such that $ z\notin\varGamma_0(N)z'$) with $ N\in\{2,3,4\}$, the following integral identity holds\begin{align}
G_2^{\mathfrak H/\overline{\varGamma}_0(N)}(z,z')={}&\frac{4\pi^{2}}{y'}\R\int_{z'}^{i\infty} \alpha_N(\zeta)[1-\alpha_N(\zeta)][NE_{2}(N\zeta)-E_2(\zeta)]^{2}\varrho^{\mathfrak H/\overline\varGamma_0(N)}_2(\zeta,z)\frac{(\zeta-z')(\zeta-\overline{z'})\D \zeta}{i(N-1)^{2}}\notag\\&-\frac{4\pi^{2}}{y'}\R\int_{0}^{i\infty} \alpha_N(\zeta)[1-\alpha_N(\zeta)][NE_{2}(N\zeta)-E_2(\zeta)]^{2}\varrho^{\mathfrak H/\overline\varGamma_0(N)}_2(\zeta,z)\frac{\zeta^{2}\D \zeta}{i(N-1)^{2}},\label{eq:G2HeckeN_unified}
\end{align}where \begin{align}\alpha_N(z):=
\left\{1+\frac{1}{N^{6/(N-1)}}\left[ \frac{\eta(z)}{\eta(Nz)} \right]^{24/(N-1)}\right\}^{-1}
\end{align}is  $ \varGamma_0(N)$-invariant for $ N\in\{2,3,4\}$, and \begin{align}
\varrho^{\mathfrak H/\overline\varGamma_0(N)}_2(\zeta,z)={}&-(N-1)\frac{y}{2\pi }\frac{\partial}{\partial y}\left\{ \frac{1}{y} \frac{1}{\alpha_{N}(\zeta)-\alpha_{N}(z)}\frac{1}{NE_{2}(Nz)-E_{2}(z)}\right\}.
\end{align} \item \label{itm:KZ_c} Suppose that  $z\notin\varGamma_0(2)z'$, then we have the following integral formula:\begin{align}
G_3^{\mathfrak H/\overline{\varGamma}_0(2)}(z,z')={}&\frac{4\pi^{2}}{(y')^{2}}\R\int_{z'}^{i\infty} \alpha_2(\zeta)[1-\alpha_2(\zeta)][2E_{2}(2\zeta)-E_2(\zeta)]^{3}\varrho^{\mathfrak H/\overline\varGamma_0(2)}_3(\zeta,z)\frac{(\zeta-z')^{2}(\zeta-\overline{z'})^{2}\D \zeta}{i}\notag\\&-\frac{4\pi^{2}}{(y')^{2}}\R\int_{0}^{i\infty} \alpha_2(\zeta)[1-\alpha_2(\zeta)][2E_{2}(2\zeta)-E_2(\zeta)]^{3}\varrho^{\mathfrak H/\overline\varGamma_0(2)}_3(\zeta,z)\frac{\zeta^{4}\D \zeta}{i},\label{eq:G3Hecke2_KZ_int}
\end{align}where\begin{align}
\varrho^{\mathfrak H/\overline\varGamma_0(2)}_3(\zeta,z)=\frac{ y^2}{8\pi}\left( \frac{\partial}{\partial y}\frac{1}{y} \right)^2\left\{ \frac{1}{\alpha_{2}(\zeta)-\alpha_2(z)}\frac{1}{[2E_{2}(2z)-E_2(z)]^2} \right\}.\label{eq:rho_3_Hecke2}
\end{align}\item \label{itm:KZ_a}  The following integral representations hold for weights $k\in\{4,6,8,10,14\}$  and  $ j(z)\neq j(z')$:\begin{align}
G_{k/2}^{\mathfrak H/PSL(2,\mathbb Z)}(z,z')={}&\frac{1728\pi^{2}}{ (y')^{(k-2)/2}}\R\int_{z'}^{i\infty} \frac{E_{k}(\zeta)}{j(\zeta)}\rho^{\mathfrak H/PSL(2,\mathbb Z)}_{k/2}(\zeta,z)\frac{(\zeta-z')^{(k-2)/2}(\zeta-\overline{z'})^{(k-2)/2}\D \zeta}{i}\notag\\&-\frac{1728\pi^{2}}{ (y')^{(k-2)/2}}\R\int_{0}^{i\infty} \frac{E_{k}(\zeta)}{j(\zeta)}\rho^{\mathfrak H/PSL(2,\mathbb Z)}_{k/2}(\zeta,z)\frac{\zeta^{k-2}\D \zeta}{i}\label{eqn:GkPSL2Z}
\end{align} where\begin{align}\label{eqn:rho_k_PSL2Z}
\rho^{\mathfrak H/PSL(2,\mathbb Z)}_{k/2}(\zeta,z)={}&\frac{(-1)^{k/2}}{2^{(k-4)/2}}\frac{1}{\left( \frac{k-2}{2} \right)!}\frac{ y^{(k-2)/2}}{864 \pi}\left( \frac{\partial}{\partial y}\frac{1}{y} \right)^{(k-2)/2}\left[  \frac{j(\zeta)j(z)}{j(\zeta)-j(z)}\frac{{E_{6}(z)}}{E_{4}(z)E_{k}(z)}\right].
\end{align}Here in Eq.~\ref{eqn:rho_k_PSL2Z},  the Eisenstein series follow the definitions in Eqs.~\ref{eq:E4_E6_defn} and \ref{eq:E8E10E14_defn};  it is also understood that $ j(z)E_6(z)/[E_4(z)E_k(z)]=[E_4(z)]^2 E_6(z)/[\Delta(z)E_k(z)]$ extends to  a smooth function in $ z\in\mathfrak H$, so that   $ \rho^{\mathfrak H/PSL(2,\mathbb Z)}_{k/2}(\zeta,z)$ is well-behaved whenever $ j(\zeta)\neq j(z)$.\end{enumerate}\end{theorem}\begin{theorem}[Evaluations of Certain Weight-4 Automorphic Self-Energies]\begin{enumerate}[label=\emph{(\alph*)}, ref=(\alph*), widest=a]\label{thm:GZ_rn} \item\label{itm:GZ_rn_a}  We have the following special values of weight-4 Gross--Zagier renormalized Green's functions (automorphic self-energies):{\allowdisplaybreaks\begin{align}
G_2^{\mathfrak H/PSL(2,\mathbb Z)}(i):={}&\lim_{z\to i}\left[ G_2^{\mathfrak H/PSL(2,\mathbb Z)}(z,i) -4\log(2\pi|z-i|)\right]-\frac{2\log|\Delta(i)|}{3}=-4(\log2+\log3);\\G_2^{\mathfrak H/\overline\varGamma_{0}(2)}\left( \frac{i}{\sqrt{2}} \right):={}&\lim_{z\to i/\sqrt{2}}\left[ G_2^{\mathfrak H/\overline\varGamma_{0}(2)}\left(z,\frac{i}{\sqrt{2}}\right) -2\log\left(2\pi\left\vert z-\frac{i}{\sqrt{2}}\right\vert\right)\right]-\frac{\log|\Delta(i/\sqrt{2})|}{3}=-3\log2;\\G_2^{\mathfrak H/\overline\varGamma_{0}(3)}\left( \frac{i}{\sqrt{3}} \right):={}&\lim_{z\to i/\sqrt{3}}\left[ G_2^{\mathfrak H/\overline\varGamma_{0}(3)}\left(z,\frac{i}{\sqrt{3}}\right) -2\log\left(2\pi\left\vert z-\frac{i}{\sqrt{3}}\right\vert\right)\right]-\frac{\log|\Delta(i/\sqrt{3})|}{3}=-2\log\frac{3}{\sqrt[3]{4}};\\G_2^{\mathfrak H/\overline\varGamma_{0}(4)}\left( \frac{i}{\sqrt{4}} \right):={}&\lim_{z\to i/\sqrt{4}}\left[ G_2^{\mathfrak H/\overline\varGamma_{0}(4)}\left(z,\frac{i}{\sqrt{4}}\right) -2\log\left(2\pi\left\vert z-\frac{i}{\sqrt{4}}\right\vert\right)\right]-\frac{\log|\Delta(i/\sqrt{4})|}{3}=-\log2;\end{align}\begin{align}G_2^{\mathfrak H/PSL(2,\mathbb Z)}\left(\frac{1+i\sqrt{3}}{2}\right):={}&\lim_{z\to\frac{1+ i\sqrt{3}}{2}}\left[ G_2^{\mathfrak H/PSL(2,\mathbb Z)}\left(z,\frac{1+i\sqrt{3}}{2}\right) -6\log\left(2\pi\left\vert z-\frac{1+i\sqrt{3}}{2}\right\vert\right)\right]-\log\left\vert \Delta\left( \frac{1+i\sqrt{3}}{2} \right)\right\vert\notag\\={}&-3(2\log2+\log3);\\G_2^{\mathfrak H/\overline\varGamma_{0}(2)}\left(\frac{i-1}{2}\right):={}&\lim_{z\to \frac{i-1}{2}}\left[ G_2^{\mathfrak H/\overline\varGamma_{0}(2)}\left(z,\frac{i-1}{2}\right) -4\log\left(2\pi\left\vert z-\frac{i-1}{2}\right\vert\right)\right]-\frac{2}{3}\log\left\vert \Delta\left( \frac{i-1}{2} \right)\right\vert\notag\\={}&-4\log2;\\G_2^{\mathfrak H/\overline\varGamma_{0}(3)}\left(\frac{3+i\sqrt{3}}{6}\right):={}&\lim_{z\to \frac{3+i\sqrt{3}}{6}}\left[ G_2^{\mathfrak H/\overline\varGamma_{0}(3)}\left(z,\frac{3+i\sqrt{3}}{6}\right) -6\log\left(2\pi\left\vert z-\frac{3+i\sqrt{3}}{6}\right\vert\right)\right]-\log\left\vert \Delta\left( \frac{3+i\sqrt{3}}{6} \right)\right\vert\notag\\={}&-3\log3.
\end{align}} \item\label{itm:GZ_rn_b} The weight-4, level-4 automorphic self-energy can be evaluated in closed form:\begin{align}
G_2^{\mathfrak H/\overline{\varGamma}_0(4)}(z):=\lim_{z'\to z}\left[G_2^{\mathfrak H/\overline{\varGamma}_0(4)}(z,z')-2\log|2\pi(z-z')|\right]-\frac{\log|\Delta(z)|}{3}=-\frac{1}{3}\log\left\vert\frac{\Delta(z)}{\Delta(2z)}\right\vert,\quad \forall z\in\mathfrak H.\label{eq:G2_Hecke4_KZ_rn_statement}
\end{align} In particular, when $z$ is a CM point, Eq.~\ref{eq:G2_Hecke4_KZ_rn_statement} represents the logarithm of an algebraic number, solvable in radical form (see~Eq.~\ref{eq:fns_alg_val_at_CM_pts}), and the Galois group $ \Gal(\mathbb Q(z,e^{-6G_2^{\mathfrak H/\overline{\varGamma}_0(4)}(z)})/\mathbb Q(z))$  is abelian.\end{enumerate}\end{theorem}\subsection{Plan of the Proof}
We devote the entire  \S\ref{sec:int_repn_AGF} to the proof of Theorem~\ref{thm:KZ_int_repns}, namely, the analytic derivation of    integral representations for all the automorphic Green's functions satisfying the cusp-form-free condition $ \dim\mathcal S_k(\varGamma_0(N))=0$. (See Appendix~\ref{app:dimSk_vanish} for the reason why the claimed formulae in Theorem~\ref{thm:KZ_int_repns} have exhausted all the cusp-form-free scenarios.) Our analysis in   \S\ref{sec:int_repn_AGF}  is inspired by and amplified from the following laconic expression for  $ G_2^{\mathfrak H/PSL(2,\mathbb Z)}(i,i\sqrt{2})$ mentioned in the thought-provoking survey  of Kontsevich and   Zagier~\cite[][\S3.4]{KontsevichZagier}:\begin{align}-\frac{G_2^{\mathfrak H/PSL(2,\mathbb Z)}(i,i\sqrt{2})}{\sqrt{2}}\equiv\frac{20G}{\pi}+1728\pi^2\int_{\sqrt{2}}^\infty\frac{E_4(iy)\Delta(iy)}{[E_6(iy)]^{2}}(y^2-2)\D y.\label{eq:KZ_laconic}\end{align}Here, $ G=\sum_{\ell=0}^{\infty}{(-1)^{\ell}}{(2\ell+1)^{-2}}$ is Catalan's constant.
In    \S\ref{sec:int_repn_AGF}, the guiding principle for our confirmation of the integral representations stated in  Theorem~\ref{thm:KZ_int_repns}
will be a uniqueness theorem (Lemma~\ref{lm:spec_AGF}) that characterizes automorphic Green's functions.

  The analytic techniques we employ in     \S\ref{sec:int_repn_AGF} may appear familiar to experts in the Eichler--Shimura theory~\cite{Eichler1957,Shimura1959,Manin1973}. As seen from the statement of Theorem~\ref{thm:KZ_int_repns}, all these integral representations  for automorphic Green's functions (\textit{d'apr\`es} Kontsevich--Zagier) can be written in terms of $ \int_{z'}^{i\infty}F_{k}(\zeta,z)(\zeta-z')^{(k-2)/2}(\zeta-\overline{z'})^{(k-2)/2}\D \zeta$, where $ F_{k}(\zeta,z)$ is a weight-$k$ meromorphic cusp form in the variable $\zeta$. If one replaces $   F_{k}(\zeta,z)$ by a holomorphic cusp form $ F_k(\zeta)$, then the Eichler integrals $ \int_{z}^{i\infty}F_{k}(\zeta)(z-\zeta)^{k-2}\D \zeta$ are known to be connected to the harmonic Maa{\ss}  forms~\cite{BringmannGuerzhoyKentOno2013,MuehlenbruchRaji}, via the Eichler--Shimura map. Our analysis in      \S\ref{sec:int_repn_AGF} is tailored for meromorphic versions of Eichler integrals, and effectively presents some explicit formulae for Eichler--Shimura maps that send meromorphic versions of Eichler cohomology (integral representations of automorphic Green's functions in Theorem~\ref{thm:KZ_int_repns}) to  Maa{\ss}  cusp forms (definitions of automorphic Green's functions as infinite series in Eq.~\ref{eq:auto_Green_defn_Q_nu}). As we have to cope with the meromorphic versions of Eichler integrals $ \int_{z'}^{i\infty}F_{k}(\zeta,z)(\zeta-z')^{(k-2)/2}(\zeta-\overline{z'})^{(k-2)/2}\D \zeta$ in our proof of Theorem~\ref{thm:KZ_int_repns}, some non-trivial residue analysis is required to establish the path independence and modular invariance of the proposed integral representations of automorphic Green's functions. In particular, the proof of parts \ref{itm:KZ_c} and \ref{itm:KZ_a} in Theorem~\ref{thm:KZ_int_repns} for weights $ k>4$ (see \S\ref{subsec:high_weight_KZ})  draws heavily on the knowledge of elliptic integrals and Ramanujan's elliptic function theory to alternative bases (see \S\ref{subsec:add_form_Legendre_Ramanujan}).

The major purpose of Theorem~\ref{thm:KZ_int_repns}  is to construct some identities involving \textit{integrals} and \textit{series}, which are relevant to the Gross--Kohnen--Zagier algebraicity conjecture in the cusp-form-free scenarios (boxed equation in \S\ref{subsec:background}). This provides some concrete examples for an abstract theoretical framework of Eichler--Shimura isomorphisms between the Eichler cohomology groups (``\textit{integrals}'') and the space of Maa{\ss} cusp forms (``\textit{series}''), which  has been  laid out by Mellit in~\cite[][Chap.~1]{MellitThesis}.  In the current and subsequent instalments for this series of works, we shall treat the integral representations in  Theorem~\ref{thm:KZ_int_repns} with explicit computations.
This computational approach is motivated by the following observation:  the Gross--Kohnen--Zagier algebraicity conjecture for automorphic Green's functions at CM points amounts to the analytic verifications of certain proposed identities involving only Kontsevich--Zagier periods (absolutely convergent integrals of algebraic functions over algebraic domains)~\cite[][\S1.1]{KontsevichZagier}. A ``motivic Hodge
conjecture'' formulated by Kontsevich and Zagier~\cite[][\S1.2 and \S4.1]{KontsevichZagier} suggests the feasibility to prove any ``identity involving periods'' using only finitely many steps of permissible algebraic manipulations on integrals.

In \S\ref{sec:analysis_KZ_int_Part1} of this article, we prove Theorem~\ref{thm:GZ_rn}, thereby accomplishing a modest goal of evaluating Kontsevich--Zagier integrals for certain weight-4 automorphic self-energies. (For a formulation of the Gross--Zagier renormalization procedure that leads to automorphic self-energies, see~\cite[][Chap.~II, \S5]{GrossZagierI} or \S\ref{subsec:int_repn_G2_GZ_rn}  of this article.)
Upon appropriate  variable substitutions, the integrands we encounter in \S\ref{sec:analysis_KZ_int_Part1}  are certain products of complete elliptic integrals and elementary functions. Such integrals over elliptic integrals are sometimes referred to as \textit{multiple elliptic integrals}, on which various developments  have been reported recently \cite{Bailey2008,BBGW,BNSW2011,Zucker2011,Wan2012,Zhou2013Pnu,Zhou2013Spheres,RogersWanZucker2013preprint,Zhou2013Int3Pnu}.

 To achieve our goal stated in  Theorem~\ref{thm:GZ_rn}, we develop a variety of analytic tools (with combinatorial, modular or geometric flavor) to simplify  some Kontsevich--Zagier integrals into closed-form expressions.  In \S\ref{subsec:hypergeo_Pnumu}, we present a glimpse into the kaleidoscope of hypergeometric transformations
on  Kontsevich--Zagier integrals, which lead to the proof of Theorem~\ref{thm:GZ_rn}\ref{itm:GZ_rn_a} in  \S\ref{subsec:int_repn_G2_GZ_rn}.
We prepare some multiple integral identities in \S\ref{subsec:Ramanujan_Jacobi}, which are inspired by spherical geometry and modular transformations of elliptic functions. These efforts
culminate in the closed-form evaluation of the weight-4, level-4 automorphic self-energy  in \S\ref{subsec:G2_Hecke4_GZ_rn}, which proves Theorem~\ref{thm:GZ_rn}\ref{itm:GZ_rn_b}.
 We note that the analytic expression in Eq.~\ref{eq:G2_Hecke4_KZ_rn_statement}   lends  evidence to a special case of the extended Gross--Zagier algebraicity conjecture \cite[][p.~317, Conjecture (4.4)]{GrossZagierI}.
\subsection*{Acknowledgements} This work was partly supported by the Applied Mathematics Program within the Department
of Energy (DOE) Office of Advanced Scientific Computing Research (ASCR) as part
of the Collaboratory on Mathematics for Mesoscopic Modeling of Materials (CM4). The manuscript was completed during the author's visit to Prof.\ Weinan E at Princeton University in 2013. The author thanks
Dr.\ Stefan T. Patrikis (Harvard/MIT) for his inspiring lectures on ``$L$-Functions and Arithmetic'' at Harvard University in 2007, which formed an excellent introduction to the wonderful world of modular forms and elliptic curves. The author is indebted to  Prof.\ Xiao\-wei Zhuang (Harvard) for her support of research related to elliptic functions in early 2008. The author is grateful to Prof.\ Bruce  C. Berndt (UIUC), Prof.\ Michael P. Brenner (Harvard),  Prof.\ Benedict H. Gross (Harvard), Prof.\ M. Lawrence Glasser (Clarkson) and   Prof.\ Shou-Wu Zhang (Princeton) for their encouragements. The author acknowledges  Prof.\ Bruce C. Berndt (UIUC), Dr.\ Qingtao Chen (ICTP, Trieste), Prof.\ M. Lawrence Glasser (Clarkson University),  Prof.\ Fei Han (National University of Singapore), and an anonymous referee for their valuable help on improving the readability of this manuscript.
\section{\label{sec:int_repn_AGF}Integral Representations of Certain Automorphic Green's Functions }
\setcounter{equation}{0}\setcounter{theorem}{0}

In this section, we shall derive  integral representations of  automorphic Green's functions on certain Hecke congruence groups, for  even weights $ k\geq4$. The main focus of \S\S\ref{subsec:KZ_integrals} and \ref{subsec:add_form_Legendre_Ramanujan} will be the weight-4 cases (covering Theorem~\ref{thm:KZ_int_repns}\ref{itm:KZ_b} and the $ k=4$ part of Theorem~\ref{thm:KZ_int_repns}\ref{itm:KZ_a}), and the higher weight scenarios will be treated in \S\ref{subsec:high_weight_KZ} (proving Theorem~\ref{thm:KZ_int_repns}\ref{itm:KZ_c} and  the $ k>4$ part of Theorem~\ref{thm:KZ_int_repns}\ref{itm:KZ_a}).

As we may recall, the Hecke congruence group of level $ N$ is defined as~\cite[][Eq.~1.6.5]{Shimura1994} \begin{align}\varGamma_0(N):=\left\{ \left.\begin{pmatrix}a & b \\
Nc & d \\
\end{pmatrix}\right|a,b,c,d\in\mathbb Z;ad-Nbc=1 \right\},\end{align}and its projective counterpart is given by $\overline {\varGamma}_0(N)=\varGamma_0(N)/\{\hat I,-\hat I\} $. The Hecke congruence group $ \varGamma_0(N)$ ($ N>1$) is normalized by the Fricke involution   $ \widehat w_N:z\mapsto -1/(Nz)$, as the matrix multiplication\begin{align} \begin{pmatrix}0&-1/\sqrt{N}\\\sqrt{N}&0 \\
\end{pmatrix}\begin{pmatrix}a&b\\cN&d \\
\end{pmatrix} =\begin{pmatrix}d&-c\\-b N&a \\
\end{pmatrix}\begin{pmatrix}0&-1/\sqrt{N}\\\sqrt{N}&0 \\
\end{pmatrix}\label{eq:Fricke_inv_norm_Hecke}\end{align} entails the conjugation relation  $\widehat w_N\varGamma_0(N)=\varGamma_0(N)\widehat w_N$. Now that  the  resolvent kernel of the upper half-plane
has  $ SL(2,\mathbb R)$-homogeneity at any pair of distinct points $ z_1,z_2\in\mathfrak H$: \begin{align}
G^{\mathfrak H}_{s}(z_1,z_2):=-2Q_{s-1}
\left( 1+\frac{\vert z_{1} - z_2\vert ^{2}}{2y_1y_2} \right)=G^{\mathfrak H}_{s}(\hat\gamma z_1,\hat \gamma z_2),\quad \forall\hat \gamma\in SL(2,\mathbb R),\label{eq:SL2_homogeneity}
\end{align}one can deduce from Eqs.~\ref{eq:Fricke_inv_norm_Hecke}--\ref{eq:SL2_homogeneity} the  following identity~\cite[][Chap.~II, \S2, Eq.~2.25]{GrossZagierI}:\begin{align}
G_s^{\mathfrak H/\overline{\varGamma}_0(N)}(z,z')=G_s^{\mathfrak H/\overline{\varGamma}_0(N)}\left( -\frac{1}{Nz} ,-\frac{1}{Nz'}\right),\quad \R s>1,
\label{eq:Fricke_inv_HeckeN_GZ}\end{align}whenever both sides of Eq.~\ref{eq:Fricke_inv_HeckeN_GZ} are finite.

We will be mainly interested in weight-$k$ automorphic Green's functions on $ \varGamma_0(N)$ when there are no cusp forms on the same Hecke congruence group with the same weight: $ \dim\mathcal S_k(\varGamma_0(N))=0$.   For even weights $ k\geq4$, the solutions to the equation  $ \dim\mathcal S_k(\varGamma_0(1))=\dim\mathcal S_k(SL(2,\mathbb Z))=0$ are exhausted by  $ k=4,6,8,10,14$~\cite[][Theorem~3.5.2]{DiamondShurman}. The following dimension formulae are  standard: $\dim\mathcal S_4(\varGamma_0(2))=\dim\mathcal S_4(\varGamma_0(3))=\dim\mathcal S_4(\varGamma_0(4))=\dim\mathcal S_6(\varGamma_0(2))=0$~\cite[][Table~A]{Miyake1989J1}. In fact, there are no other solutions to $\dim\mathcal S_k(\varGamma_0(N))=0 $ for even weights $ k\geq4$ and integers $ N>1$, as one can prove using lower bounds on the dimensions of the spaces of cusp forms~\cite{Martin2005} (see ~Appendix~\ref{app:dimSk_vanish}). \begin{figure}[ht]\begin{center}\includegraphics[width=15cm]{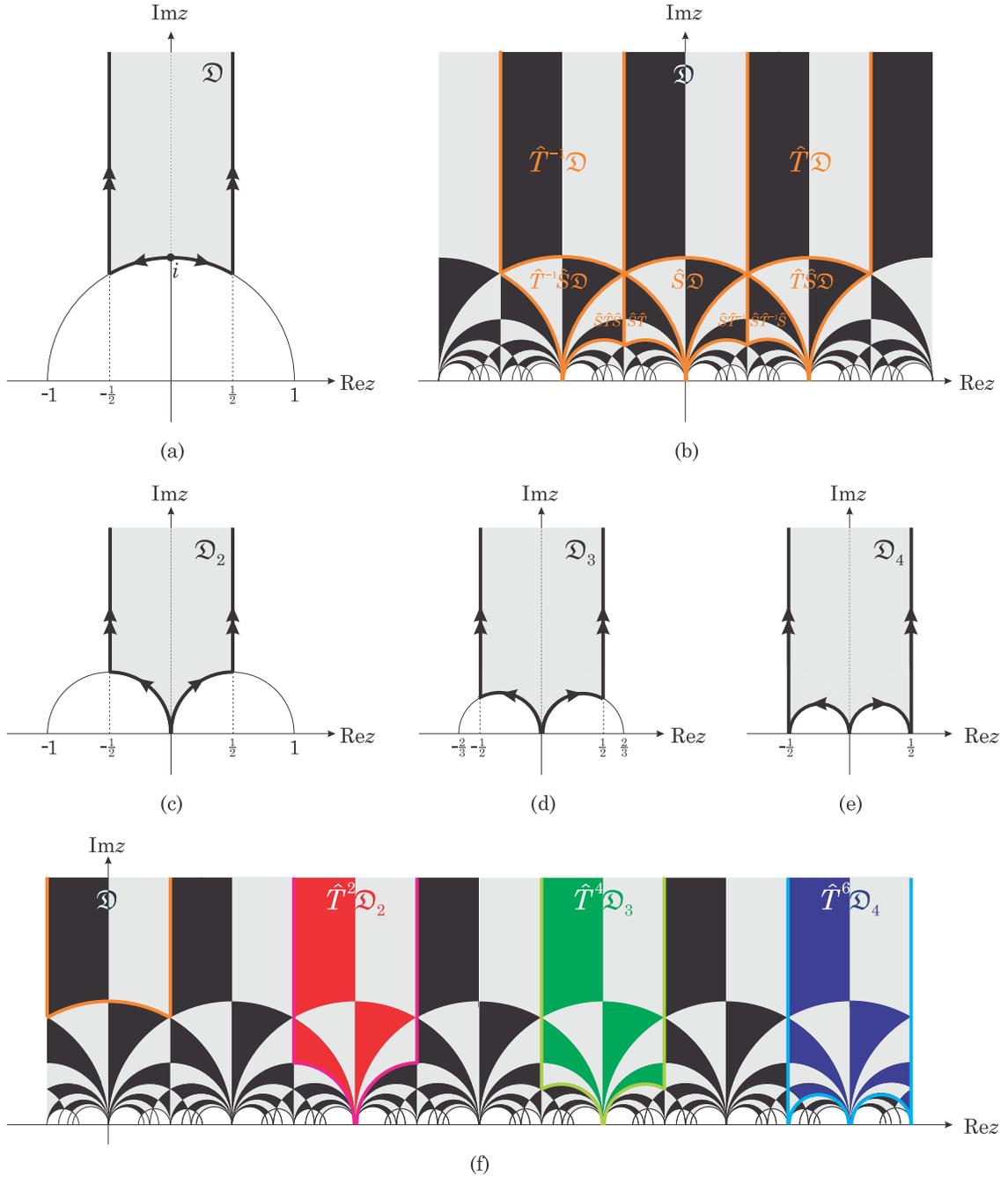}\end{center}
\caption[Fundamental domain of $ PSL(2,\mathbb Z)$ and its tessellations]{(Adapted from Fig.~61 in \cite{FrickeKlein}.)
(\textit{a})  Fundamental domain $ \protect\mathfrak D$ of the full modular group $ PSL(2,\protect\mathbb Z)$. The modular elliptic curve $ X_0(1)(\mathbb C)=SL(2,\mathbb Z)\backslash\mathfrak H^*$ compactifies $ Y_0(1)(\mathbb C)=SL(2,\mathbb Z)\backslash \mathfrak H$, a quotient space that identifies the corresponding sides of the boundary   $\partial \protect\mathfrak D$  along  the  arrows.     (\textit{b}) Tessellation of the upper half-plane $ \protect\mathfrak H$ by   successive translations [generator $\hat T=\protect\left(\protect\begin{smallmatrix}1&1\\0&1\protect\end{smallmatrix}\protect\right)$]   and inversions  [generator $\hat S=\protect\left(\protect\begin{smallmatrix}0&-1\\1&0\protect\end{smallmatrix}\protect\right)$] of the fundamental domain $\protect \mathfrak D$. Each tile  is  then subdivided and colored in black or grey   according as the pre-image satisfies  $ \R z<0$ or $ \R z>0$ in the fundamental domain $\protect \mathfrak D$. (\textit{c})--(\textit{e}) Fundamental domains $ \mathfrak D_N$ of the projective Hecke congruence subgroups $ \overline{\varGamma}_0(N)$ where $ N\in\{2,3,4\}$. (\textit{f}) Overlay of some horizontal translates of the fundamental domains in panels \textit{c}--\textit{e} on top of panel \textit{b}. It is graphically evident that $ [\varGamma_0(1):\varGamma_0(2)]=3$,  $ [\varGamma_0(1):\varGamma_0(3)]=4$ and  $ [\varGamma_0(1):\varGamma_0(4)]=6$.
}
\label{fig:fundomain}
\end{figure}
Instead of converting the infinite sum in  Eq.~\ref{eq:auto_Green_defn_Q_nu} directly into integrals,   we shall resort to an alternative characterization of automorphic Green's functions  (see~\cite{GrossZagier1985,SWZhang1997,MellitThesis,Viazovska2011}; see also~\cite[][\S6.6]{Hejhal1983} and~\cite[][\S5.1]{IwaniecGSM53} for the same function with different normalizing constants) recapitulated in the next lemma. \begin{lemma}[Spectral Characterization of Automorphic Green's Functions]\label{lm:spec_AGF}The function $ G^{\mathfrak H/\overline {\varGamma}}_{k/2}(z_1,z_2)$ (where $\R k>2 $) can be prescribed as the unique solution to the following set of requirements:\begin{enumerate}[leftmargin=*,  label=\emph{(AGF\arabic*)},ref=(AGF\arabic*),
widest=a, align=left]\item \label{itm:AGF1}\emph{(\textbf{Symmetry})}
For any modular transformation $\hat  \gamma\in\varGamma$ and two points $ z_1,z_2\in\mathfrak H$ in the upper half-plane, one has $ G^{\mathfrak H/\overline {\varGamma}}_{k/2}(z_1,z_2)=G^{\mathfrak H/\overline {\varGamma}}_{k/2}(z_2,z_1)=G^{\mathfrak H/\overline {\varGamma}}_{k/2}(\hat  \gamma z_1,z_2)=G^{\mathfrak H/\overline {\varGamma}}_{k/2}( z_1,\hat  \gamma z_2)\in\mathbb C\cup\{-\infty\}$. Moreover, if $z_1\notin \varGamma z_2 :=\{\hat\gamma z_2|\hat \gamma\in\varGamma\}$, then $ G^{\mathfrak H/\overline {\varGamma}}_{k/2}(z_1,z_2)\in\mathbb C$.\item\label{itm:AGF2} \emph{(\textbf{Differential Equation})} The function $ G_{k/2}^{\mathfrak H/\overline {\varGamma}}(z_1,z_2)$ is smooth on the set $\mathfrak H\times\mathfrak H\smallsetminus\{(z,\hat\gamma z)|z\in\mathfrak H,\hat \gamma\in\varGamma\} $,\footnote{In the current work, we use the notation $A\smallsetminus B $ with a ``lowercase'' backslash ($ \smallsetminus$) to indicate set minus operation, while writing the ``uppercase'' backslash~($\backslash $) for orbit spaces such as $ SL(2,\mathbb Z)\backslash\mathfrak H$.}  and we have the following differential equation\[y_{m}^2\left( \frac{\partial^2}{\partial x_{m}^2}+ \frac{\partial^2}{\partial y_{m}^2}\right)G_{k/2}^{\mathfrak H/\overline {\varGamma}}(z_1,z_2)=\frac{k}{2}\left(\frac{k}{2}-1\right)G_{k/2}^{\mathfrak H/\overline {\varGamma}}(z_1,z_2),\quad m=1,2\]when  $z_1\notin \varGamma z_2 $.\item\label{itm:AGF3}\emph{(\textbf{Asymptotic Behavior})} In the limit of  $ z_1\to z_2$, one has  $ G_{k/2}^{\mathfrak H/\overline {\varGamma}}(z_1,z_2)=2m^{\overline {\varGamma}}_{z_2}\log|z_1-z_2|+O(1)$, where $ m^{\overline {\varGamma}}_{z_2}:=\#\{\hat \gamma\in\overline{\varGamma}|\hat \gamma z_2=z_2\}$ counts the number of inequivalent transformations that leave the point $ z_2$ intact. As $ z_1$ tends to a cusp of the orbit space  $  \varGamma\backslash\mathfrak H^*\equiv\varGamma\backslash(\mathfrak H\cup\mathbb Q\cup\{i\infty\})$, the function $ G_{k/2}^{\mathfrak H/\overline {\varGamma}}(z_1,z_2)$ tends to zero.
\end{enumerate}\end{lemma}\begin{proof}One can verify directly that Eq.~\ref{eq:auto_Green_defn_Q_nu} satisfies all the requirements in \ref{itm:AGF1}--\ref{itm:AGF3}. Now pick the point $ z_2=z'$, and suppose that we have a function $ F(z),z\in\mathfrak H$ satisfying the following properties:\begin{enumerate}[leftmargin=*,  label={(AGF\arabic*$^*$)},
widest=a, align=left]\item \label{itm:AGF1'}{(\textbf{Symmetry})}
For any  $\hat  \gamma\in\varGamma$,  one has $ F(z)=F(\hat  \gamma z)\in\mathbb C\cup\{-\infty\}$. Moreover, if $z\notin \varGamma z' $, then $ F(z)\in\mathbb C$.\item\label{itm:AGF2'} {(\textbf{Differential Equation})} The smooth function $ F(z),z\notin \varGamma z'$ satisifies the following differential equation\[\Delta_z^{\mathfrak H}F(z)\equiv y^2\left( \frac{\partial^2}{\partial x^2}+ \frac{\partial^2}{\partial y^2}\right)F(z)=\frac{k}{2}\left(\frac{k}{2}-1\right)F(z).\]\item\label{itm:AGF3'}{(\textbf{Asymptotic Behavior})} As  $ z\to z'$, one has  $ F(z)=2m^{\overline {\varGamma}}_{z'}\log|z-z'|+O(1)$.  As $ z$ tends to a cusp of   $  \varGamma\backslash\mathfrak H^*$, the function $ F(z)$ tends to zero.
\end{enumerate}Then, the expression $ h(z)=F(z)-G_{k/2}^{\mathfrak H/\overline {\varGamma}}(z,z')$ extends to  a smooth and bounded automorphic  function on $ \varGamma\backslash\mathfrak H$, and so does $ \Delta_z^{\mathfrak H}h(z)$.  In other words, we have $ h\in\mathcal D(\varGamma\backslash\mathfrak H)=\{f:\mathfrak H\longrightarrow \mathbb C|f\in C^\infty(\mathfrak H);f(\hat\gamma z)=f(z),\forall\hat\gamma\in\varGamma;\sup_{z\in\mathfrak H}|f(z)|<+\infty;\sup_{z\in\mathfrak H}|\Delta_z^{\mathfrak H}f(z)|<+\infty\}$~\cite[][\S4.1]{IwaniecGSM53}. Had $h(z)$ been not identically zero, it would serve as an eigenfunction of the (negative semi-definite) operator $ \Delta_z^{\mathfrak H}:\mathcal D(\varGamma\backslash\mathfrak H)\longrightarrow\mathcal D(\varGamma\backslash\mathfrak H)$, with eigenvalue $ k(k-2)/4\in\mathbb C\smallsetminus(-\infty,0]$ --- a contradiction.   \end{proof}
\begin{remark}
Here, we recall that the totality of cusps for   $  \varGamma\backslash\mathfrak H^*$ is represented by the orbit space $ \varGamma\backslash\mathbb P^1(\mathbb Q)=\varGamma\backslash(\mathbb Q\cup\{i\infty\})$~\cite[][p.~58]{DiamondShurman}. In particular, the orbit space  $ X_{0}(1)(\mathbb C):=SL(2,\mathbb Z)\backslash\mathfrak H^*$ has only one cusp at infinity. For the orbit spaces $ X_0(N)(\mathbb C):=\varGamma_0(N)\backslash\mathfrak H^*$ related to the Hecke congruence groups of levels $N=2$, $3$ and $4$, the following facts are also familiar: $ \varGamma_0(2)\backslash(\mathbb Q\cup\{i\infty\})=\varGamma_0(3)\backslash(\mathbb Q\cup\{i\infty\})=\{0,i\infty\}$, $\varGamma_0(4)\backslash(\mathbb Q\cup\{i\infty\})=\{0,\frac{1}{2},i\infty\} $. The geometric constructions for the modular  curves $ Y_0(N)(\mathbb C)=\varGamma_0(N)\backslash\mathfrak H$ ($ N\in\{1,2,3,4\}$) are illustrated in Fig.~\ref{fig:fundomain}. For a technical description  of  $ X_0(N)(\mathbb C)=\varGamma_0(N)\backslash\mathfrak H^*$ as a compact Riemann surface, see \cite[][Chap.~1]{Shimura1994} or \cite[][Chap.~2]{DiamondShurman}. \eor\end{remark}\begin{remark}\label{rmk:Green_recip}One can use Eq.~\ref{eq:auto_Green_defn_Q_nu} to demonstrate ``Green's reciprocity''  $ G_{k/2}^{\mathfrak H/\overline {\varGamma}}(z,z')=G_{k/2}^{\mathfrak H/\overline {\varGamma}}(z',z)$ by transitivity of the group actions. According to the uniqueness argument in the proof of Lemma \ref{lm:spec_AGF}, we see that Green's reciprocity will automatically be honored once criteria~\ref{itm:AGF1'}--\ref{itm:AGF3'} are met. Therefore, it is \textit{not} a requirement that an integral representation for an automorphic Green's function  should display the reciprocal symmetry explicitly.\eor\end{remark}

\subsection{Kontsevich--Zagier Integrals for Weight-4 Automorphic Green's Functions on  $\overline \varGamma_0(N)$    ($ N\in\{1,2,3,4\}$)\label{subsec:KZ_integrals}}
We shall start constructing integral representations for automorphic Green's functions following a suggestion of  Kontsevich and Zagier~\cite{KontsevichZagier}, after reading between the lines around their statement of the integral formula for $ G_2^{\mathfrak H/PSL(2,\mathbb Z)}(i,i\sqrt{2})$ (Eq.~\ref{eq:KZ_laconic}). We shall refer to such formulae for automorphic Green's functions as ``Kontsevich--Zagier integrals'', both to acknowledge the source of our inspirations and to reckon the fact that, up to a scaling factor that is an integer power of $ \pi$, such integral formulae turn out to be ``Kontsevich--Zagier periods'' (integrals of algebraic functions over algebraic domains) when the arguments are CM points.

\begin{proposition}[Kontsevich--Zagier Integral Representations for $ G_2^{\mathfrak H/\overline{\varGamma}_0(N)}(z,z')$ where $N\in\{2,3,4\}$]\label{prop:G2HeckeNunified} We have the following integral representations for weight-4 automorphic Green's functions evaluated at a pair of non-equivalent points $ z$ and $z'$ (such that $ z\notin\varGamma_0(N)z'$) for $ N\in\{2,3,4\}$:\begin{align}
G_2^{\mathfrak H/\overline{\varGamma}_0(N)}(z,z')={}&\frac{4\pi^{2}}{y'}\R\int_{z'}^{i\infty} \alpha_N(\zeta)[1-\alpha_N(\zeta)][NE_{2}(N\zeta)-E_2(\zeta)]^{2}\varrho^{\mathfrak H/\overline\varGamma_0(N)}_2(\zeta,z)\frac{(\zeta-z')(\zeta-\overline{z'})\D \zeta}{i(N-1)^{2}}\notag\\&-\frac{4\pi^{2}}{y'}\R\int_{0}^{i\infty} \alpha_N(\zeta)[1-\alpha_N(\zeta)][NE_{2}(N\zeta)-E_2(\zeta)]^{2}\varrho^{\mathfrak H/\overline\varGamma_0(N)}_2(\zeta,z)\frac{\zeta^{2}\D \zeta}{i(N-1)^{2}},\label{eq:G2HeckeN_unified}
\end{align}where \begin{align}
\varrho^{\mathfrak H/\overline\varGamma_0(N)}_2(\zeta,z)={}&\frac{\alpha_N(z)[1-\alpha_N(z)]}{[\alpha_N(\zeta)-\alpha_N(z)]^{2}}-\frac{N-1}{12}\frac{N^{2}[E_{2}(N z)]^2-N^{2}E_{4}(Nz)-[E_{2}(z)]^2+E_{4}(z)}{[\alpha_N(\zeta)-\alpha_N(z)][NE_{2}(N z)-E_{2}(z)]^2},\label{eq:rho2HeckeN_unified}\\ \text{with }\frac{1}{\alpha_N(z)}={}&1+\frac{1}{N^{6/(N-1)}}\left[ \frac{\eta(z)}{\eta(Nz)} \right]^{24/(N-1)}=\frac{1}{1- \alpha_N(-1/(Nz))},\quad \lim_{z\to i\infty}\alpha_N(z)=0.\label{eq:alpha_HeckeN_unified}
\end{align} Here in Eq.~\ref{eq:G2HeckeN_unified}, the paths of  integration  can be chosen as arbitrary curves joining the end points in the complex $ \zeta$-plane, so long as the singularities of the integrands are circumvented.\end{proposition}\begin{proof}We shall establish Eq.~\ref{eq:G2HeckeN_unified} by going through criteria~\ref{itm:AGF1'}--\ref{itm:AGF3'} with respect to the variable $z\in\mathfrak H$.

It is well known that the functions $ \alpha_N(z),N\in\{2,3,4\}$ (Eq.~\ref{eq:alpha_HeckeN_unified}) are modular invariants that induce bijective mappings $ \alpha_N:\varGamma_0(N)\backslash\mathfrak H^*\longrightarrow\mathbb C\cup\{\infty\}$ on the respective Riemann surfaces (see~\cite[][Theorem~4.9]{ApostolVol2} and~\cite[][Table~2]{Maier2009}). Thus, whenever the function $ \varrho^{\mathfrak H/\overline\varGamma_0(N)}_2(\zeta,z)$ defined in Eq.~\ref{eq:rho2HeckeN_unified} is regular, one can directly check the modular invariance $ \varrho^{\mathfrak H/\overline\varGamma_0(N)}_2(\zeta,\hat \gamma z) =\varrho^{\mathfrak H/\overline\varGamma_0(N)}_2(\zeta,z),\forall\hat \gamma\in\varGamma_0(N)$ where $ N\in\{2,3,4\}$. However, the regularity of an integral representation with respect to the variable $z$ cannot be taken for granted: one needs two additional steps  to verify the symmetry criterion~\ref{itm:AGF1'}, as we explain in the next couple of paragraphs.

First, we use residue calculus to demonstrate the path independence of the integral  formula in Eq.~\ref{eq:G2HeckeN_unified}. Simple computations reveal that\footnote{Hereafter, the abbreviation ``a.e.'' that precedes a complex variable refers to  ``almost every'' complex number in question. An equality that holds for  ``almost every''  complex number in a specified domain is allowed to fail in a subset of measure zero.  } \begin{align}
\alpha_N'(w):=\frac{\partial\alpha_{N}(w)}{\partial w}={}&\frac{2\pi i}{N-1}\alpha_{N}(w)[1-\alpha_{N}(w)][NE_{2}(Nw)-E_{2}(w)],&& \text{a.e. }w\in\mathfrak H;
\label{eq:alpha_N_deriv_E2}\\
\varrho^{\mathfrak H/\overline\varGamma_0(N)}_2(\zeta,z)={}&\frac{y}{i }\frac{\partial}{\partial y}\left\{ \frac{1}{y} \frac{\alpha_{N}(z)[1-\alpha_N(z)]}{[\alpha_{N}(\zeta)-\alpha_{N}(z)]\alpha_{N}'(z)}\right\},&& \text{a.e. }\zeta,z\in\mathfrak H,\label{eq:rho2HeckeN_unified_deriv_form}
\end{align}so the right-hand side of   Eq.~\ref{eq:G2HeckeN_unified} may be rewritten as\begin{align}\mathcal I_N(z,z')\equiv{}&
\R\int_{z'}^{i\infty} \frac{[\alpha_N'(\zeta)]^{2}}{\alpha_N(\zeta)[1-\alpha_N(\zeta)]}y\frac{\partial}{\partial y}\left\{ \frac{1}{y} \frac{\alpha_{N}(z)[1-\alpha_N(z)]}{[\alpha_{N}(\zeta)-\alpha_{N}(z)]\alpha_{N}'(z)}\right\}\frac{(\zeta-z')(\zeta-\overline{z'})\D \zeta}{y'}\notag\\&-\R\int_{0}^{i\infty} \frac{[\alpha_N'(\zeta)]^{2}}{\alpha_N(\zeta)[1-\alpha_N(\zeta)]}y\frac{\partial}{\partial y}\left\{ \frac{1}{y} \frac{\alpha_{N}(z)[1-\alpha_N(z)]}{[\alpha_{N}(\zeta)-\alpha_{N}(z)]\alpha_{N}'(z)}\right\}\frac{ \zeta^2\D \zeta}{y'}.\label{eq:deriv_form_G2Hecke234}\tag{\ref{eq:G2HeckeN_unified}$'$}
\end{align}  Here as usual, a complex-analytic   function $ f(z)=f(x+iy)$ is regarded as a bivariate function of $ (x,y)$ during the computations of partial derivatives with respect to $ y=\I z$. For any two well-behaved analytic functions $ f$ and $g$ that are $ \varGamma$-invariant (\textit{i.e.~}$f(\hat\gamma z)=f(z)$, $ g(\hat \gamma z)=g( z)$, $ \forall\hat\gamma\in\varGamma\leq SL(2,\mathbb Z)$), the residues of the following three functions of $ \zeta\in\mathfrak H$: \begin{align}\frac{[f'(\zeta)]^{2}\zeta^n}{g(\zeta)}y\frac{\partial}{\partial y}\left\{ \frac{1}{y} \frac{g(z)}{f(\zeta)-f(z)}\frac{1}{f'(z)}\right\} ,\quad n\in\{0,1,2\}\label{eq:G2_three_fns}\end{align} at the point $ \zeta=\hat \gamma z$ are equal to \begin{align}
y\frac{\partial}{\partial y}\left\{ \frac{( \hat \gamma z)^{n}}{y} \frac{1}{\partial( \hat \gamma z)/\partial z}\right\} ,\quad n\in\{0,1,2\}.
\end{align}These three residues are all real numbers:\begin{align*}-\frac{1}{\I( \hat \gamma z)},\quad(n=0);\qquad-\frac{\R ( \hat \gamma z)}{\I ( \hat \gamma z)},\quad (n=1);\qquad -\frac{|( \hat \gamma z)|^2}{\I ( \hat \gamma z)},\quad (n=2).\end{align*}
Therefore, the real parts of the contour integrals in Eq.~\ref{eq:deriv_form_G2Hecke234} are always path independent.

Second, we  check that the expression for  $ \varrho^{\mathfrak H/\overline\varGamma_0(N)}_2(\zeta,z)$ given in Eq.~\ref{eq:rho2HeckeN_unified_deriv_form} is indeed regular whenever $ z\in\mathfrak H\smallsetminus(\varGamma_0(N)\zeta)$. With the formulae:\footnote{The  identities listed in Eq.~\ref{eq:alpha_N_ratio_ids} are classical. They can be traced back to Ramanujan's alternative-base representations for the Weierstra{\ss} discriminant $ \Delta(z)=[\eta(z)]^{24}$ \cite[][Chap.~33, Corollary~3.4, Theorem~9.10 and Theorem~11.6]{RN5} and weight-2 Eisenstein series~\cite[][Chap.~33, Theorem~9.11]{RN5},~\cite[][Chap.~33, Corollary~2.11]{RN5},~\cite[][Chap.~17, Entries 13(viii)--(ix)]{RN3}. For an arithmetic perspective on such identities and their generalizations, see~\cite[][]{Maier2009}.}\begin{align}
\left.\begin{array}{r@{\;=\;}l}
\left\{\dfrac{\alpha_{2}(z)[1-\alpha_2(z)]}{\alpha_{2 }'(z)}\right\} ^6& -\dfrac{\alpha_{2}(z)[1-\alpha_2(z)]^{2}}{2^{12} \pi ^6 [\eta (z)]^{24}} \\[10pt]
\left\{\dfrac{\alpha_{3}(z)[1-\alpha_3(z)]}{\alpha_{3 }'(z)}\right\}^6 & -\dfrac{\alpha_{3}(z)[1-\alpha_3(z)]^{3}}{2^6 3^3 \pi ^6[\eta (z)]^{24}} \\[10pt]\left\{\dfrac{\alpha_{4}(z)[1-\alpha_4(z)]}{\alpha_{4 }'(z)}\right\}^6 & -\dfrac{\alpha_{4}(z)[1-\alpha_4(z)]^{4}}{2^{10} \pi ^6[\eta (z)]^{24}} \end{array}\right\}\label{eq:alpha_N_ratio_ids}
\end{align}and the non-vanishing property of the Dedekind eta function $ \eta(z)\neq0,\forall z\in\mathfrak H$, one can confirm that the function  $ \varrho^{\mathfrak H/\overline\varGamma_0(N)}_2(\zeta,z)$ is smooth  wherever $ \alpha_N(z)\in\mathbb C\smallsetminus\{\alpha_N(\zeta)\}$.  It remains to show that   $ \varrho^{\mathfrak H/\overline\varGamma_0(N)}_2(\zeta,z)$ is regular when $ \alpha_N(z)\to\infty$ for certain points $ z\in\mathfrak H$. By direct computation on the expressions in Eq.~\ref{eq:rho2HeckeN_unified}, one can show that (see Remark~\ref{rmk:classical_invariants} below) \begin{align}\lim_{a_2(z)\to\infty}
\varrho^{\mathfrak H/\overline\varGamma_0(2)}_2(\zeta,z)=-\frac{1}{2},\quad \alpha_2(\zeta)\in\mathbb  C;\qquad  \lim_{a_3(z)\to\infty}
\varrho^{\mathfrak H/\overline\varGamma_0(3)}_2(\zeta,z)=-\frac{1}{3}, \quad \alpha_3(\zeta)\in\mathbb  C,\label{eq:alpha_inf_rho_lim}
\end{align}    and $ \alpha_4(z) $ never diverges for $z\in\mathfrak H$. (Here, we note that  $ \alpha_2(z)=\infty$ corresponds to the period-2 elliptic point $ m_{z}^{\overline\varGamma_0(2)}=2$, $ z\in\varGamma_0(2)\frac{i-1}{2}$; while  $ \alpha_3(z)=\infty$ corresponds to the period-3 elliptic point $ m_{z}^{\overline\varGamma_0(3)}=3$,  $z\in\varGamma_0(3)\frac{3+i\sqrt{3}}{6}$. The orbit space $ X_0(4)(\mathbb C)=\varGamma_0(4)\backslash\mathfrak H^*$ contains no elliptic points: $ m_z^{\overline\varGamma_0(4)}\equiv1,\forall z\in\mathfrak H$.) This completes the verification of criterion~\ref{itm:AGF1'}.

It is relatively straightforward to show that the integral representation in Eq.~\ref{eq:deriv_form_G2Hecke234} satisfies the correct differential equation in the variable $z$, as dictated by criterion~\ref{itm:AGF2'}. This is because\begin{align}\left[ y^2\left( \frac{\partial^{2}}{\partial x^{2}} +\frac{\partial^{2}}{\partial y^{2}}\right)-2\right]\left[y\frac{\partial}{\partial y}\frac{f(z)}{y}\right]=0\label{eq:weight_4_diff_eqn_test}\end{align}holds for any complex-analytic function $ f(z)=f(x+iy)$.

The verification of criterion~\ref{itm:AGF3'} naturally breaks down into two subtasks: the logarithmic asymptotics at the diagonal and the vanishing behavior at the cusps.

For the logarithmic asymptotics, we may momentarily assume that $ z'$ is not an elliptic point, so that $ m_{z'}^{\overline\varGamma_0(N)}=1$ and $ \alpha_N(z') $ is finite. As $ z$ approaches $z'$, the first term in Eq.~\ref{eq:rho2HeckeN_unified} dominates the contribution to   $ \varrho^{\mathfrak H/\overline\varGamma_0(N)}_2(\zeta,z)$, so the integral representation in Eq.~\ref{eq:G2HeckeN_unified}  goes asymptotically as \begin{align}
\sim {}&\frac{4\pi^{2}}{y'}\R\int_{z'}^{i\infty} \alpha_N(\zeta)[1-\alpha_N(\zeta)][NE_{2}(N\zeta)-E_2(\zeta)]^{2}\frac{\alpha_N(z)[1-\alpha_N(z)]}{[\alpha_N(\zeta)-\alpha_N(z)]^{2}}\frac{(\zeta-z')(\zeta-\overline{z'})\D \zeta}{i(N-1)^{2}}\notag\\\sim{}&-\R\int_{z'}^{2z'} \frac{[\alpha_N'(\zeta)]^{2}}{\alpha_N(\zeta)[1-\alpha_N(\zeta)]}\frac{\alpha_N(z)[1-\alpha_N(z)]}{[\alpha_N'(\zeta)]^{2}(\zeta-z)^2}\frac{(\zeta-z')(\zeta-\overline{z'})\D \zeta}{iy'}=-\R\int_{z'}^{2z'} \frac{1}{\zeta-z}\frac{(z'-\overline{z'})\D \zeta}{iy'}+O(1)\notag\\={}&2\log|z'-z|+O(1).
\end{align}

Before analyzing the  behavior of Eq.~\ref{eq:G2HeckeN_unified} near the infinite cusp $ z\to i\infty$,  we note that Eqs.~\ref{eq:alpha_N_deriv_E2} and \ref{eq:rho2HeckeN_unified_deriv_form} allow us to compute the following  limit:\begin{align}
\varrho^{\mathfrak H/\overline\varGamma_0(N)}_2(\zeta,z)={}&(1-N)\frac{y}{2\pi }\frac{\partial}{\partial y}\left\{ \frac{1}{y} \frac{1}{[\alpha_{N}(\zeta)-\alpha_{N}(z)][NE_{2}(Nz)-E_{2}(z)]}\right\}\to0,\quad\text{as }z\to i\infty\label{eq:rho2_cusp_limit}
\end{align}for every fixed $ \zeta\in\mathfrak H$.

We   then exploit the path independence of the  integral representation  in  Eq.~\ref{eq:deriv_form_G2Hecke234} and recast it into\begin{align}
\mathcal I_N(z,z')={}&-\R\int_{0}^{z'} \frac{i[\alpha_N'(\zeta)]^{2}\varrho^{\mathfrak H/\overline\varGamma_0(N)}_2(\zeta,z)}{\alpha_N(\zeta)[1-\alpha_N(\zeta)]}\frac{ \zeta^2\D \zeta}{y'}-\R\int_{z'}^{i\infty} \frac{i[\alpha_N'(\zeta)]^{2}\varrho^{\mathfrak H/\overline\varGamma_0(N)}_2(\zeta,z)}{\alpha_N(\zeta)[1-\alpha_N(\zeta)]}{\frac{ 2\zeta x'\D \zeta}{y'}}\notag\\&+\R\int_{z'}^{i\infty} \frac{i[\alpha_N'(\zeta)]^{2}\varrho^{\mathfrak H/\overline\varGamma_0(N)}_2(\zeta,z)}{\alpha_N(\zeta)[1-\alpha_N(\zeta)]}{\frac{ |z'|^{2}\D \zeta}{y'}}.\label{eq:G2HeckeN_unified''}\tag{\ref{eq:G2HeckeN_unified}$''$}
\end{align}We need to treat the three integrals in  Eq.~\ref{eq:G2HeckeN_unified''}  separately in order to conclude that $ \lim_{z\to i\infty} \mathcal I_N(z,z')=0$. Here, by Eq.~\ref{eq:rho2_cusp_limit}, we can conclude that the first integral in Eq.~\ref{eq:G2HeckeN_unified''} vanishes as $ z$ tends to the infinite cusp:\begin{align}
\lim_{z\to i\infty}\R\int_{0}^{z'} \frac{i[\alpha_N'(\zeta)]^{2}\varrho^{\mathfrak H/\overline\varGamma_0(N)}_2(\zeta,z)}{\alpha_N(\zeta)[1-\alpha_N(\zeta)]}\frac{ \zeta^2\D \zeta}{y'}=0.\label{eq:cusp_lim0a}
\end{align} In   Eq.~\ref{eq:cusp_lim0a}, it is easy to bound $ |\varrho^{\mathfrak H/\overline\varGamma_0(N)}_2(\zeta,z)|$ by a finite constant $ C\in(0,+\infty)$ for $ z\to i\infty$ and bounded $ |\zeta|$, and we have absolute integrability:\begin{align}C\int_{0}^{z'} \left\vert\frac{i[\alpha_N'(\zeta)]^{2}}{\alpha_N(\zeta)[1-\alpha_N(\zeta)]}\right\vert\frac{ |\zeta|^2|\D \zeta|}{y'}<+\infty.\end{align}(Bearing in mind that $ 1-\alpha_N(\zeta)=\alpha_N(-1/(N\zeta))=O(e^{-2\pi i/(N\zeta)})$ and $ \alpha'_N(\zeta)=O(|\zeta|^{-2}e^{-2\pi i/(N\zeta)})$ when $\I\zeta\to0^+$, one sees that the integrand above is bounded, hence absolutely integrable.) So, one can interchange the limit and the integral sign in   Eq.~\ref{eq:cusp_lim0a}, according to the  dominated convergence theorem. Meanwhile, one can also apply the dominated convergence theorem to the second integral in Eq.~\ref{eq:G2HeckeN_unified''} after transforming it into an integral from $ 0$ to $z'$:\begin{align}
-\R\int_{z'}^{i\infty} \frac{i[\alpha_N'(\zeta)]^{2}\varrho^{\mathfrak H/\overline\varGamma_0(N)}_2(\zeta,z)}{\alpha_N(\zeta)[1-\alpha_N(\zeta)]}{\frac{ 2\zeta x'\D \zeta}{y'}}=\R\int_{0}^{z'} \frac{i[\alpha_N'(\zeta)]^{2}\varrho^{\mathfrak H/\overline\varGamma_0(N)}_2(\zeta,z)}{\alpha_N(\zeta)[1-\alpha_N(\zeta)]}{\frac{ 2\zeta x'\D \zeta}{y'}}\to0,\quad\text{as }z\to i\infty.\label{eq:cusp_lim0b}
\end{align}Here, the first equality in Eq.~\ref{eq:cusp_lim0b} owes to the following vanishing identity:\begin{align}
\R\int_{0}^{i\infty} \frac{i[\alpha_N'(\zeta)]^{2}\varrho^{\mathfrak H/\overline\varGamma_0(N)}_2(\zeta,z)\zeta\D \zeta}{\alpha_N(\zeta)[1-\alpha_N(\zeta)]}=0,\quad \forall z\in\mathfrak H.\label{eq:rho2_int0}
\end{align}One can easily verify Eq.~\ref{eq:rho2_int0} for boundary points $ z\in\mathfrak H\cap\partial\mathfrak D_N$ (Fig.~\ref{fig:fundomain}\textit{c}--\textit{e}), that is, for all the  points $ z\in\mathfrak H$ satisfying\begin{align}
|\R z|=\frac{1}{2},\I z>\frac{1}{2}\sqrt{\frac{4-N}{N}}\quad\text{or}\quad-\frac{1}{2}<\R z<0 ,\left|z+\frac{1}{N}\right |=\frac{1}{N}\quad \text{or}\quad 0<\R z<\frac{1}{2} ,\left|z-\frac{1}{N}\right |=\frac{1}{N},\label{eq:HeckeN_fun_dom_bd}
\end{align} by choosing the contour of integration along the imaginary axis $ \zeta/i\in(0,+\infty)$. This is because in such scenarios,  $ \alpha_N(z) $, $ \alpha_N(\zeta)$, $ i\alpha'_N(\zeta)$ and $ \varrho^{\mathfrak H/\overline\varGamma_0(N)}_2(\zeta,z)$ are all real numbers. As the integral in Eq.~\ref{eq:rho2_int0} is annihilated by the differential operator \begin{align*}-2+ y^2 \left( \frac{\partial^2}{\partial x^2}+\frac{\partial^2}{\partial y^2}\right)=-2-(z-\overline{z})^{2}\frac{\partial}{\partial z}\frac{\partial}{\partial\overline{ z}},\end{align*}and satisfies homogeneous Dirichlet boundary condition on the domain $ \mathfrak D_N$ bounded by the curves specified in Eq.~\ref{eq:HeckeN_fun_dom_bd}, we see that  Eq.~\ref{eq:rho2_int0}  must hold within  $ \mathfrak D_N$. Such a vanishing identity then extends to every  $z\in\mathfrak H$, because the closure of the aforementioned domain  $ \mathfrak H\cap(\mathfrak D_N\cup\partial\mathfrak D_N)$ tessellates the upper half-plane  under the actions of $ \varGamma_0(N)$ ($ N\in\{2,3,4\}$). After the confirmation of the vanishing asymptotics for the first two integrals in   Eq.~\ref{eq:G2HeckeN_unified''}, we need to handle the last one with integration by parts:{\allowdisplaybreaks\begin{align}&
\R\int_{z'}^{i\infty} \frac{i[\alpha_N'(\zeta)]^{2}\varrho^{\mathfrak H/\overline\varGamma_0(N)}_2(\zeta,z)}{\alpha_N(\zeta)[1-\alpha_N(\zeta)]}{\D \zeta}\notag\\={}&(N-1)\frac{y}{2\pi }\frac{\partial}{\partial y}\R\left\{ \frac{1}{y} \frac{1}{NE_{2}(Nz)-E_{2}(z)}\int_{z'}^{i\infty} \frac{[\alpha_N'(\zeta)]^{2}}{i\alpha_N(\zeta)[1-\alpha_N(\zeta)]}{\frac{\D \zeta}{\alpha_{N}(\zeta)-\alpha_{N}(z)}}\right\}\notag\\={}&y\frac{\partial}{\partial y}\R\left\{ \frac{1}{y} \frac{1}{NE_{2}(Nz)-E_{2}(z)}\int_{z'}^{i\infty} [NE_{2}(N\zeta)-E_{2}(\zeta)]\D\log[\alpha_{N}(\zeta)-\alpha_{N}(z)]\right\}\notag\\={}&-y\frac{\partial}{\partial y}\R\left\{ \frac{1}{y} \frac{1}{NE_{2}(Nz)-E_{2}(z)}\int_{z'}^{i\infty} \log[\alpha_{N}(\zeta)-\alpha_{N}(z)]\frac{\partial[NE_{2}(N\zeta)-E_{2}(\zeta)-N+1]}{\partial\zeta}\D\zeta\right\}\notag\\&-y\frac{\partial}{\partial y}\R\left\{ \frac{1}{y} \frac{NE_{2}(Nz')-E_{2}(z')}{NE_{2}(Nz)-E_{2}(z)}\log[\alpha_{N}(z')-\alpha_{N}(z)]\right\}.\label{eq:cusp_lim0c}
\end{align}}As $ y\to+\infty$, one has $ \alpha_N(z)=O(e^{2\pi iz})$, $ NE_{2}(Nz)-E_{2}(z)=N-1+O(e^{2\pi iz})$, so accordingly,\begin{align*}&\int_{z'}^{i\infty} \log[\alpha_{N}(\zeta)-\alpha_{N}(z)]\frac{\partial[NE_{2}(N\zeta)-E_{2}(\zeta)-N+1]}{\partial\zeta}\D\zeta\notag\\\to{}&\int_{z'}^{i\infty} \log[\alpha_{N}(\zeta)]\frac{\partial[NE_{2}(N\zeta)-E_{2}(\zeta)-N+1]}{\partial\zeta}\D\zeta\in\mathbb C\end{align*}and the contributions to\footnote{Exploiting the  symmetry of the integral in question, one may assume that the points  $ z$ and  $z'$ are both in the closure of the fundamental domain $ \mathfrak D_N\cup\partial\mathfrak D_N$. For sufficiently large $ y=\I z$, one can then  designate   the integration paths  in the $ \zeta$-plane as follows: the path joining $ z'$ to $ z-\frac i3$ is a straight line segment; the path connecting $ z-\frac i3$ to $ z+\frac i3$ is a semi-circle with radius $ 1/3$; the path extending from $ z+\frac i3$ to $ i\infty$ runs parallel to the $ \I \zeta$-axis. It is ready to check  that $ \alpha_N(z)\neq \alpha_N(\zeta)$ along these integration  paths. }\begin{align*}&\frac{1}{NE_{2}(Nz)-E_{2}(z)}\frac{\partial}{\partial z}\int_{z'}^{i\infty} \log[\alpha_{N}(\zeta)-\alpha_{N}(z)]\frac{\partial[NE_{2}(N\zeta)-E_{2}(\zeta)-N+1]}{\partial\zeta}\D\zeta\notag\\={}&\frac{2\pi i[1-\alpha_{N}(z)]}{N-1}\left(\int_{z'}^{z-\frac{i}{3}}+\int_{z-\frac{i}{3}}^{z+\frac{i}{3}}+\int_{z+\frac{i}{3}}^{i\infty}\right) \frac{\alpha_{N}(z)}{\alpha_{N}(z)-\alpha_N(\zeta)}\frac{\partial[NE_{2}(N\zeta)-E_{2}(\zeta)-N+1]}{\partial\zeta}\D\zeta\end{align*} can be decomposed into three parts:{\allowdisplaybreaks\begin{align*}\left|\int_{z'}^{z-\frac{i}{3}}\frac{\alpha_{N}(z)}{\alpha_{N}(z)-\alpha_N(\zeta)}\frac{\partial[NE_{2}(N\zeta)-E_{2}(\zeta)-N+1]}{\partial\zeta}\D\zeta\right|\leq{}&C_1\left\vert\int_{z'}^{z-\frac{i}{3}}|\alpha_{N}(z)|\D\zeta\right\vert=O(|z\alpha_{N}(z)|);\notag\\\left|\int_{z-\frac{i}{3}}^{z+\frac{i}{3}}\frac{\alpha_N(z)}{\alpha_{N}(z)-\alpha_N(\zeta)}\frac{\partial[NE_{2}(N\zeta)-E_{2}(\zeta)-N+1]}{\partial\zeta}\D\zeta\right|\leq{}& C_2\int_{0}^\pi\frac{|e^{2\pi iz+\frac{2\pi i}{3} e^{i\phi}}|}{|1-e^{\frac{2\pi i}{3} e^{i\phi}}|}\D\phi=O(e^{-\pi y});\notag\\\left|\int_{z+\frac{i}{3}}^{i\infty}\frac{\alpha_{N}(z)}{\alpha_{N}(z)-\alpha_N(\zeta)}\frac{\partial[NE_{2}(N\zeta)-E_{2}(\zeta)-N+1]}{\partial\zeta}\D\zeta\right|\leq{}& C_3\left|\int_{z+\frac{i}{3}}^{i\infty}\frac{ \alpha_{N}(z)e^{2\pi i\zeta}}{\alpha_{N}(z)-\alpha_N(\zeta)}\D\zeta\right|=O(|z\alpha_{N}(z)|),\end{align*}}where $ C_1$, $ C_2$ and $ C_3$ are  constant bounds, deducible from the mean value theorems for integration. As a consequence, one can verify that the expression in Eq.~\ref{eq:cusp_lim0c} has order $ O(1/y)$ as $ z\to i\infty$.

From the result  $ \lim_{z\to i\infty} \mathcal I_N(z,z')=0$ proved in the last paragraph, we can also deduce the vanishing behavior at another cusp    $ \lim_{z\to0} \mathcal I_N(z,z')=0$. This becomes possible once we confirm that $\mathcal I_N(z,z')$ respects the Fricke involution in the same way as $ G_2^{\mathfrak H/\overline{\varGamma}_0(N)}(z,z')$ (Eq.~\ref{eq:Fricke_inv_HeckeN_GZ}). Indeed, by   the symmetry $ \varrho^{\mathfrak H/\overline\varGamma_0(N)}_2(\zeta,z)=\varrho^{\mathfrak H/\overline\varGamma_0(N)}_2(-1/(N\zeta),-1/(Nz))$ along with  Eqs.~\ref{eq:G2HeckeN_unified''} and \ref{eq:rho2_int0}, we can verify that \begin{align}\mathcal
I_N(z,z')-\mathcal
I_N\left( -\frac{1}{Nz} ,-\frac{1}{Nz'}\right)=-\R\int_{0}^{i\infty} \frac{i[\alpha_N'(\zeta)]^{2}\varrho^{\mathfrak H/\overline\varGamma_0(N)}_2(\zeta,z)}{\alpha_N(\zeta)[1-\alpha_N(\zeta)]}{\frac{ 2\zeta x'\D \zeta}{y'}}\equiv0.\label{eq:I_N_Fricke_inv_check}
\end{align}

The orbit space $ \varGamma_0(4)\backslash\mathfrak H^{*}$ has a third cusp at $ z=1/2$. Akin to  Eq.~\ref{eq:alpha_inf_rho_lim}, one can compute\begin{align}
\lim_{z\to1/2}
\varrho^{\mathfrak H/\overline\varGamma_0(4)}_2(\zeta,z)=\lim_{a_4(z)\to\infty}
\varrho^{\mathfrak H/\overline\varGamma_0(4)}_2(\zeta,z)=0, \quad \alpha_4(\zeta)\in\mathbb  C.\label{eq:rho2Hecke4cusp_half}
\end{align}As we can always choose the paths of integration for $ \mathcal I_4(z,z')$  (Eq.~\ref{eq:G2HeckeN_unified''}) away from the cusp $ z=1/2$, the vanishing limit $ \lim_{z\to1/2}\mathcal I_4(z,z')=0$ follows directly from  Eq.~\ref{eq:rho2Hecke4cusp_half} and  an application of the dominated convergence theorem.

To summarize, we have verified that the integral representation $ \mathcal I_N(z,z')$ ($ N\in\{2,3,4\}$) satisfies all the three criteria \ref{itm:AGF1'}--\ref{itm:AGF3'}, provided that $z'$ is not an elliptic point: $ m_{z'}^{\overline\varGamma_0(N)}=1$. Hence, in these scenarios,  $ \mathcal I_N(z,z')$ must be identical to the automorphic Green's function $G_2^{\mathfrak H/\overline\varGamma_0(N)}(z,z') $, according to Lemma~\ref{lm:spec_AGF}. For the isolated elliptic points where  $ m_{z'}^{\overline\varGamma_0(N)}>1$, $ N\in\{2,3\}$, one can justify the integral representation in Eq.~\ref{eq:G2HeckeN_unified} by continuity in the variable $z'$.
\end{proof}\begin{remark}\label{rmk:classical_invariants}The function $ \alpha_3(z)$ is just the Ramanujan cubic invariant (see \cite[][Chap.~33, \S\S2--8]{RN5} and \cite[][Chap.~9]{RLN2}):\begin{align}
\alpha_{3}(z):=\left( \frac{[\eta(z/3)]^3}{3[\eta(3z)]^3}+1 \right)^{-3}=\left( \frac{[\eta(z)]^{12}}{27[\eta(3z)]^{12}}+1 \right)^{-1},\quad z\in\mathfrak H\smallsetminus\left({\varGamma}_0(3)\frac{3+i\sqrt{3}}{6}\right)\label{eq:alpha3_defn}
\end{align}where the equivalence of the two definitions for  $ \alpha_{3}(z)$ in Eq.~\ref{eq:alpha3_defn}  follows from the cubic modular equation for the Dedekind eta function~\cite[][p.~345, Eq.~(iv)]{RN3}, which is also provable by the cubic identity of Borwein--Borwein--Garvan~\cite{BorweinBorweinGarvan}. The function $ 1/\alpha_3(z)$ extends to  a holomorphic function on $ \mathfrak H$, with a third-order zero at $ z=(3+i\sqrt{3})/6$, which corresponds to the period-3 elliptic point on $ X_0(3)(\mathbb C)=\varGamma_0(3)\backslash\mathfrak H^*$.
Meanwhile,
the modular lambda function  $ \lambda(z):=16[\eta (z/2)]^8 [\eta (2 z)]^{16}/[\eta (z)]^{24},z\in\mathfrak H $ is related to the modular invariants $\alpha_2 $ and $ \alpha_4$ as follows:\begin{align}
\alpha_{2}(z)=1-\frac{1}{[1-2\lambda(2z+1)]^2};\quad \alpha_4(z)=\lambda(2z).\label{eq:alpha2_alpha4}
\end{align}One has $ \lambda(i)=1/2$, so $ \alpha_2(z)$ diverges at  $z= (i-1)/2$, which corresponds to the period-2 elliptic point on  the orbit space $ X_0(2)(\mathbb C)=\varGamma_0(2)\backslash\mathfrak H^*$.
 \eor\end{remark}\begin{remark}We note that the viability of the integral representations in Eq.~\ref{eq:G2HeckeN_unified} draws heavily on the special geometric properties (Fig.~\ref{fig:fundomain}\textit{c}--\textit{e}) of the compact Riemann surfaces $ \varGamma_0(N)\backslash\mathfrak H^*$ ($ N\in\{2,3,4\}$). While the mapping $ \alpha_N:\varGamma_0(N)\backslash\mathfrak H^*\longrightarrow \mathbb C\cup\{\infty\}$ remains bijective so long as $ (N-1)$ divides $24 $ \cite[][Table~2]{Maier2009}, the regularity of $ \alpha_N(z)[1-\alpha_N(z)]/{\alpha'_N}(z)$ (see Eq.~\ref{eq:alpha_N_ratio_ids}) will be lost for $ N>4$. For example, the reader may check that $ \alpha_5'(z)$ vanishes at $ z=(\pm 3+i)/5$, which are the two non-equivalent period-2 elliptic points on $ \varGamma_0(5)\backslash\mathfrak H^*$ (see \cite[][Table~2, Theorem~5.5, and Corollary~5.5.2]{Maier2009}). Therefore, for $ N=5$,  the integral in Eq.~\ref{eq:G2HeckeN_unified} is not regular  at $ z=(\pm 3+i)/5$, thus failing criterion~\ref{itm:AGF1'}. In fact, Eq.~\ref{eq:G2HeckeN_unified} ceases to be a valid integral representation for automorphic Green's functions on the Hecke congruence group $ \overline\varGamma_0(N)$ for all the higher levels $N\geq5$. Nonetheless, integral representations like Eq.~\ref{eq:G2HeckeN_unified}  do adapt to certain automorphic Green's functions on higher level Hecke congruence groups with   Nebentypus characters.  As an example, for a special  Dirichlet character \[ \chi_5(d)\equiv\left( \frac{d}{5} \right)=\left( \frac{5}{d} \right)=\begin{cases}1, & d\equiv\pm1\pmod5 \\
-1, & d\equiv\pm2\pmod5 \\
0, & d\equiv0\pmod5 \\
\end{cases}\] we have \begin{align}
G_2^{\mathfrak H/(\overline \varGamma_0(5),\chi_5)}(z,z'):={}&-2\sum_{\hat  \gamma=\left(\begin{smallmatrix}a&b\\c&d\end{smallmatrix}\right)\in\overline{\varGamma}_0(5)}\chi _{5}(d)Q_{1}
\left( 1+\frac{\vert z -\hat  \gamma z'\vert ^{2}}{2\I z\I(\hat\gamma z')} \right)\notag\\={}&\frac{\pi^{2}}{y'}\R\int_{z'}^{i\infty} \alpha_5(\zeta)\frac{[\eta(\zeta)]^{5}}{\eta(5\zeta)}[5E_{2}(5\zeta)-E_2(\zeta)]\varrho^{\mathfrak H/(\overline \varGamma_0(5),\chi_5)}_2(\zeta,z)\frac{(\zeta-z')(\zeta-\overline{z'})\D \zeta}{i}\notag\\&-\frac{\pi^{2}}{y'}\R\int_{0}^{i\infty} \alpha_5(\zeta)\frac{[\eta(\zeta)]^{5}}{\eta(5\zeta)}[5E_{2}(5\zeta)-E_2(\zeta)]\varrho^{\mathfrak H/(\overline \varGamma_0(5),\chi_5)}_2(\zeta,z)\frac{\zeta^{2}\D \zeta}{i},
\end{align} where the function\begin{align}
\varrho^{\mathfrak H/(\overline \varGamma_0(5),\chi_5)}_2(\zeta,z):=-\frac{y}{2\pi\ }\frac{\partial}{\partial y}\left\{ \frac{1}{y} \frac{1-\alpha_{5}(z)}{\alpha_{5}(\zeta)-\alpha_{5}(z)}\frac{\eta(5z)}{[\eta(z)]^{5}}\right\}
\end{align}  is regular for $ z\in\mathfrak H\smallsetminus(\varGamma_0(5)\zeta)$.
\eor\end{remark}

The projective counterpart for  $ SL(2,\mathbb Z)\equiv\varGamma_0(1)$ is   the full modular group $ PSL(2,\mathbb Z)=SL(2,\mathbb Z)/\{\hat I,-\hat I\}$. The fundamental domain (Fig.~\ref{fig:fundomain}\textit{a})\begin{align}
\mathfrak D:=\left\{ z\in\mathfrak H\left| |z|>1,-\frac12<\R z\leq\frac{1}{2} \right. \right\}\cup \left\{ z\in\mathfrak H\left| |z|=1,0\leq \R z\leq\frac{1}{2} \right. \right\}\label{eq:fun_domain_SL2Z}
\end{align}tessellates the upper half plane $ \mathfrak H$ under actions of the  full modular group $ PSL(2,\mathbb Z)$, generated by the translation $ \hat{T}:z\mapsto z+1$ and the inversion $ \hat{S}:z\mapsto-1/z$ (Fig.~\ref{fig:fundomain}\textit{b}). It is possible to assign two univalent branches to the square root  expression
$ \sqrt{\smash[b]{(j(z)-1728)/j(z)}}$ so that it induces a smooth bijection between $ (\mathfrak D\cup \hat{S}\mathfrak D)\smallsetminus\{e^{\pi i/3},e^{2\pi i/3}\}$ and $ \mathbb C$. In particular, we shall require that  $ \sqrt{\smash[b]{(j(z)-1728)/j(z)}}>0$ for $ z/i>1$, and    $ \sqrt{\smash[b]{(j(z)-1728)/j(z)}}<0$ for $ 0<z/i<1$. In this manner, we have $ j(z)=j(-1/z),\forall z\in\mathfrak H$, but  $\sqrt{\smash[b]{(j(z)-1728)/j(z)}}=-\sqrt{\smash[b]{(j(-1/z)-1728)/j(-1/z)}},\forall z\in\mathfrak D\smallsetminus\{e^{\pi i/3}\}$.

Define
\begin{align}
\alpha_1(z)=\frac{1-\sqrt{(j(z)-1728)/j(z)}}{2},\quad \forall z\in(\mathfrak D\cup \hat{S}\mathfrak D)\smallsetminus\{e^{\pi i/3},e^{2\pi i/3}\}\label{eq:alpha1_defn}
\end{align}with the aforementioned convention about square roots, then we have \begin{align}
\alpha_1(z)=1-\alpha_1\left(-\frac1z\right),\quad\lim_{z\to i\infty}\alpha_1(z)=0,
\end{align}in parallel to Eq.~\ref{eq:alpha_HeckeN_unified}. The definition in Eq.~\ref{eq:alpha1_defn} can be smoothly extended to all $ z\in\mathfrak H\smallsetminus(SL(2,\mathbb Z)e^{\pi i/3})$, by actions of an index-2 normal subgroup of $ SL(2,\mathbb Z)$ \cite[][p.~86]{Schoeneberg}.
The derivative of $\alpha_1(z)$ can be computed from the differentiation
 formula for the $ j$-invariant:\begin{align}i
\frac{\partial j(z)}{\partial z}=2\pi  j(z)\frac{E_{6}(z)}{E_4(z)},\quad \text{a.e. }z\in\mathfrak H.\label{eq:j'z_E4E6}
\end{align}

With the preparations in the last two paragraphs, we can state and prove the Kontsevich--Zagier integral representations for $ G_2^{\mathfrak H/PSL(2,\mathbb Z)}(z,z')$ in the next proposition.

\begin{proposition}[Kontsevich--Zagier Integral Representations for $ G_2^{\mathfrak H/PSL(2,\mathbb Z)}(z,z')$]\label{prop:G2PSL2Z}For a pair of non-equivalent points $ z,z'\in\mathfrak H$ such that $ j(z)\neq j(z')$, one has \begin{align}G_2^{\mathfrak H/PSL(2,\mathbb Z)}(z,z')={}&\frac{1728\pi^{2}}{y'}\R\int_{z'}^{i\infty} \frac{E_{4}(\zeta)}{j(\zeta)}\rho^{\mathfrak H/PSL(2,\mathbb Z)}_2(\zeta,z)\frac{(\zeta-z')(\zeta-\overline{z'})\D \zeta}{i}\notag\\&-\frac{1728\pi^{2}}{y'}\R\int_{0}^{i\infty} \frac{E_{4}(\zeta)}{j(\zeta)}\rho^{\mathfrak H/PSL(2,\mathbb Z)}_2(\zeta,z)\frac{\zeta^{2}\D \zeta}{i},
\label{eq:G2PSL2Z_Eisenstein_form}
\end{align}where \begin{align}
\rho^{\mathfrak H/PSL(2,\mathbb Z)}_2(\zeta,z)=\frac{[j(\zeta)]^{2}j(z)}{432[j(\zeta)-j(z)]^{2}}-\frac{2j(\zeta)[j(\zeta)+j(z)]}{[j(\zeta)-j(z)]^{2}}+\frac{j(\zeta)j(z)}{2592[j(\zeta)-j(z)]}{\frac{E_6(z)}{[E_4(z)]^2}}\left[ E_{2}(z)-\frac{E_6(z)}{E_4(z)} \right].\label{eq:rho2PSL2Z_Eisenstein}
\end{align}Here in Eq.~\ref{eq:G2PSL2Z_Eisenstein_form}, the integration paths can be arbitrary, provided that  $ j(\zeta)\neq j(z)$ along the way.\end{proposition}\begin{proof}As in Proposition~\ref{prop:G2HeckeNunified},  we shall check criteria~\ref{itm:AGF1'}--\ref{itm:AGF3'} for the variable $z\in\mathfrak H$.

The symmetry property in \ref{itm:AGF1'} follows from two simple observations. First, we point out that the function $ \rho^{\mathfrak H/PSL(2,\mathbb Z)}_2(\zeta,z)$ defined in Eq.~\ref{eq:rho2PSL2Z_Eisenstein} is regular so long as $ j(\zeta)\neq j(z)$, according to the following computation:\begin{align}\lim_{j(z)\to0}
\rho^{\mathfrak H/PSL(2,\mathbb Z)}_2(\zeta,z)=\lim_{E_{4}(z)\to0}
\rho^{\mathfrak H/PSL(2,\mathbb Z)}_2(\zeta,z)=-\frac{4}{3}.\label{eq:alpha_inf_rho_lim_PSL2Z}
\end{align}Second, using Eq.~\ref{eq:j'z_E4E6} and Ramanujan's differential equation for the Eisenstein series \cite[][Eq.~30]{Ramanujan1916}, one can rewrite Eq.~\ref{eq:rho2PSL2Z_Eisenstein} as\begin{align}\rho^{\mathfrak H/PSL(2,\mathbb Z)}_2(\zeta,z)=\frac{y}{864 \pi}\frac{\partial}{\partial y}\left\{ \frac{1}{y} \frac{j(\zeta)j(z)}{j(\zeta)-j(z)}\frac{{E_{6}(z)}}{[E_{4}(z)]^2}\right\},
\label{eq:rho2PSL2Z_Eisenstein'}\tag{\ref{eq:rho2PSL2Z_Eisenstein}$'$}
\end{align} so that one may   confirm path independence by a variant of  Eq.~\ref{eq:G2_three_fns}.
 Namely, given that  $ f(\hat \gamma z)=f(z)$ and $ M(\hat \gamma z)=(cz+d)^4 M(z)$ for $ \forall\hat\gamma=\left( \begin{smallmatrix}a&b\\c&d\end{smallmatrix}\right)\in SL(2,\mathbb Z)$, the residues of the following three functions of $ \zeta\in\mathfrak H$:
\begin{align}&M(\zeta)\zeta^ny\frac{\partial}{\partial y}\left\{ \frac{1}{y} \frac{f'(z)}{f(\zeta)-f(z)}\frac{1}{M(z)}\right\}\notag\\ ={}&M(\zeta)\zeta^ny\frac{\partial}{\partial y}\left\{ \frac{1}{y} \frac{f'(z)}{f(\hat\gamma^{-1}\zeta)-f(z)}\frac{1}{M(z)}\right\} ,\quad n\in\{0,1,2\}\label{eq:G2_three_fns'}\tag{\ref{eq:G2_three_fns}$'$}\end{align} at the point $ \zeta=\hat \gamma z$ are all real numbers:\begin{align*}-\frac{1}{\I( \hat \gamma z)},\quad(n=0);\qquad-\frac{\R ( \hat \gamma z)}{\I ( \hat \gamma z)},\quad (n=1);\qquad -\frac{|( \hat \gamma z)|^2}{\I ( \hat \gamma z)},\quad (n=2).\end{align*}

To verify the differential equation in criterion~\ref{itm:AGF2'}, it would suffice to combine Eq.~\ref{eq:rho2PSL2Z_Eisenstein'} with Eq.~\ref{eq:weight_4_diff_eqn_test}.

For criterion~\ref{itm:AGF3'}, the procedures for logarithmic asymptotics are essentially similar to that of  Proposition~\ref{prop:G2HeckeNunified}. For the verification of cusp behavior, we need a  decomposition $ \rho^{\mathfrak H/PSL(2,\mathbb Z)}_2(\zeta,z)=\varrho(\zeta,z)+\varrho(\zeta,-1/z)$ where\begin{align}
\varrho(\zeta,z)=\frac{y}{i }\frac{\partial}{\partial y}\left\{ \frac{1}{y} \frac{\alpha_{1}(z)[1-\alpha_1(z)]}{[\alpha_{1}(\zeta)-\alpha_{1}(z)]\alpha_{1}'(z)}\right\}=\varrho\left( -\frac{1}{\zeta} ,-\frac{1}{z}\right).
\end{align}If we replace $ \rho^{\mathfrak H/PSL(2,\mathbb Z)}_2(\zeta,z)$ by $ \varrho(\zeta,z)$ in  the integrands of Eq.~\ref{eq:G2PSL2Z_Eisenstein_form} and call the result $ \mathcal I_1(z,z')$, then we can justify the limits $ \lim_{z\to0}\mathcal I_1(z,z')=\lim_{z\to i\infty}\mathcal I_1(z,z')=0$ in the same fashion as what we did for $ \mathcal I_N(z,z'),N\in\{2,3,4\}$ in Proposition~\ref{prop:G2HeckeNunified}. The right-hand side of Eq.~\ref{eq:G2PSL2Z_Eisenstein_form}, being equal to  $\mathcal  I_1(z,z')+\mathcal  I_1(-1/z,z')$, should thus vanish as $ z$ goes to the infinite cusp $ i\infty$.

This completes the justification of the integral representation in Eq.~\ref{eq:G2PSL2Z_Eisenstein_form}.       \end{proof}\begin{remark}Recalling the  asymptotic behavior of automorphic Green's functions~\cite[][\S5]{GrossZagier1985}:\begin{align}G_{s}^{\mathfrak H/PSL(2,\mathbb Z)}(z_1,z_2)\sim{}&\frac{2\pi}{1-2s}\frac{y_1^{1-s}y_{2}^{s}}{\zeta(2s)}\sum_{\substack{m,n\in\mathbb Z\\m^2+n^2\neq0}}\frac{1}{|mz_{2}+n|^{2s}},\quad y_1\to+\infty,\R s>1\label{eq:Gs_asympt}\end{align} we now have\begin{align}
\label{eq:Epstein_s2_int_star}
\sum_{\substack{m,n\in\mathbb Z\\m^2+n^2\neq0}}\frac{ y^2}{|mz+n|^{4}}={}&-\frac{\pi^3}{60}\lim_{y'\to+\infty}G_2^{\mathfrak H/PSL(2,\mathbb Z)}(z,z')y'\notag\\={}&\frac{144\pi^{5}}{5}\R\int_{0}^{i\infty} \frac{E_{4}(\zeta)}{j(\zeta)}\rho^{\mathfrak H/PSL(2,\mathbb Z)}_2(\zeta,z)\frac{\zeta^{2}\D \zeta}{i},
\end{align}where each one of the three terms is $ SL(2,\mathbb Z)$-invariant with respect to $z$.

For $ z=i$, one has the asymptotic behavior\begin{align}
G_s^{\mathfrak H/PSL(2,\mathbb Z)}(z,i)\sim\frac{4\pi}{1-2s}\frac{2\zeta(s)L(s,\chi_{-4})}{\zeta(2s) y^{s-1}}:=\frac{4\pi}{1-2s}\frac{2\zeta(s)}{\zeta(2s) y^{s-1}}\sum^\infty_{n=0}\frac{(-1)^n}{(2n+1)^s},\quad y\to+\infty,
\label{eq:Gs_zi_asympt}\end{align}according to Eq.~\ref{eq:Gs_asympt} and an integration of the two-squares theorem~\cite[][p.~298]{SteinII}\begin{align}\label{eq:two_squares}\sum_{\substack{m,n\in\mathbb Z\\m^2+n^2\neq0}}y^{s-1}e^{-(m^{2}+n^2)y}=2\sum_{\ell=1}^\infty\frac{y^{s-1}}{\cosh\ell y}\end{align}over $ y\in(0,+\infty)$. In Eq.~\ref{eq:Gs_zi_asympt}, we have used the following notation for $ L$-functions:
\begin{align}
L(s,\chi_D):=\sum_{n=1}^\infty\left( \frac{D}{n} \right)\frac{1}{n^s}\label{eq:L_s_chi_D_defn}
\end{align}where $ \left(\frac{\cdot}{ n}\right)$ is the Jacobi--Kronecker symbol.

In particular, for $ s=2$, the right-hand side of Eq.~\ref{eq:Gs_zi_asympt} involves Catalan's constant $ G:=\sum_{n=0}^\infty(-1)^n(2n+1)^{-2}=\frac{1}{2}\int_0^\infty y/(\cosh y)\D y$. Noting that Eq.~\ref{eq:rho2PSL2Z_Eisenstein} entails $  \rho^{\mathfrak H/PSL(2,\mathbb Z)}_2(\zeta,i)=2j(\zeta)/[j(\zeta)-1728]=2j(\zeta)\Delta(\zeta)/[E_{6}(\zeta)]^{2}$, one may combine Eqs.~\ref{eq:G2PSL2Z_Eisenstein_form} and \ref{eq:Epstein_s2_int_star} into the following  identity:\begin{align}&
G_2^{\mathfrak H/PSL(2,\mathbb Z)}(i,z)=-\frac{40 G}{\pi y}-3456\pi^{2}\R\int_{z/i}^\infty\frac{E_4(it)\Delta(it)}{[E_6(it)]^{2}}\frac{(t+ix)^2- y^2}{y}\D t,\quad j(z)\neq 1728,\label{eq:G2_PSL2Z_z_i}
\end{align}which incorporates the integral representation of  $ G_2^{\mathfrak H/PSL(2,\mathbb Z)}(i,i\sqrt{2})$ (Eq.~\ref{eq:KZ_laconic}) mentioned in the survey of Kontsevich--Zagier \cite{KontsevichZagier} as a special case.
\eor\end{remark}\begin{remark}So far, for an arbitrary pair of non-equivalent points $ z,z'\in\mathfrak H$, we have furnished integral representations for weight-4 automorphic Green's functions on the  congruence subgroups  $ PSL(2,\mathbb Z)=\overline \varGamma_0(1),$ $\overline \varGamma_0(2),$ $\overline \varGamma_0(3)$ and $\overline \varGamma_0(4)$, in the spirit of Kontsevich and Zagier. This completes the proof of Theorem~\ref{thm:KZ_int_repns}\ref{itm:KZ_b}, as well as the $ k=4$ case of Theorem~\ref{thm:KZ_int_repns}\ref{itm:KZ_a}. The Gross--Kohnen--Zagier algebraicity conjecture (boxed equation in \S\ref{subsec:background}) can be accessed by the analysis of the corresponding integral representations of automorphic Green's functions at all the CM points.
One may  wish to carry on Mellit's algebro-geometric proof  of the claim $ \exp(G_2^{\mathfrak H/PSL(2,\mathbb Z)}(z,i)\I z)\in\overline{\mathbb Q}$ at CM values of $ z$ \cite{MellitThesis} to the analysis of other integral representations in  this subsection (as well as the rest of Theorem \ref{thm:KZ_int_repns}), in an attempt to verify the  Gross--Kohnen--Zagier algebraicity conjecture for all the cusp-form-free scenarios. Whilst it is plausible to pursue such an abstract approach to the Gross--Kohnen--Zagier algebraicity of automorphic Green's functions in the language of Arakelov intersections and Chow groups for elliptic curves~\cite{SWZhang1997,MellitThesis} as well as   Mellit's criterion of ``geometric representability''~\cite{MellitThesis}, it is beyond the scope of this article. In the rest of this paper (see, in particular,  \S\S\ref{subsec:add_form_Legendre_Ramanujan}, \ref{subsec:int_repn_G2_GZ_rn} and \ref{subsec:G2_Hecke4_GZ_rn}), we will only evaluate a special subset of weight-4 automorphic Green's functions, using explicit and constructive methods.\eor
\end{remark}
\subsection{Addition Formulae and  Legendre--Ramanujan Representations for Automorphic Green's Functions\label{subsec:add_form_Legendre_Ramanujan}}
The automorphic Green's functions on different congruence subgroups are related to each other by some ``addition formulae''. Apart from the  congruence subgroups $ \varGamma_0(N)$ ($ N\in\{1,2,3,4\}$) studied in \S\ref{subsec:KZ_integrals}, it is sometimes useful to consider the $ \varLambda$-group and $ \vartheta$-group, so as to facilitate the statements for certain types of  ``addition formulae''.

Here, the  $ \varLambda$-group $\varLambda\equiv \varGamma(2)$ is identical to the principal congruence subgroup of level 2, where~\cite[][Eq.~1.6.1]{Shimura1994}: \[ SL(2,\mathbb Z)\supseteq \varGamma_0(N)\supseteq\varGamma(N):=\left\{\left.\begin{pmatrix}Na+1&Nb\\Nc&Nd+1\end{pmatrix}\right|a,b,c,d\in\mathbb Z,(Na+1)(Nd+1)-N^{2}bc=1\right\}.\]The   $ \varLambda$-group characterizes the symmetry of the modular lambda function   $ \lambda(z):=\frac{16[\eta (z/2)]^8 [\eta (2 z)]^{16}}{[\eta (z)]^{24}},$ $ z\in\mathfrak H $, as $ \lambda(\hat \gamma z)=\lambda(z)$ for any $ \hat \gamma\in\varLambda$. The canonical isomorphism $ \varGamma_0(4)\longrightarrow \varGamma(2)$ is induced by  the duplication map $ \left( \begin{smallmatrix}2&0\\0&1\end{smallmatrix} \right):z\mapsto 2z$~\cite[][p.~28]{Zagier2008Mod123}.
Therefore, one has (see Eq.~\ref{eq:alpha2_alpha4}) \begin{align}
G_s^{\mathfrak H/\overline{\varGamma}_{0}(4)}(z_1,z_2)=G_{s}^{\mathfrak H/\overline{\varGamma}(2)}(2z_{1},2z_{2}),\quad z_1\in\mathfrak H\smallsetminus(\varGamma_0(4)z_{2}),\R s>1.\label{eq:Hecke4_add}
\end{align}

The function $ 1-\alpha_2((z-1)/2)=1/[1-2\lambda(z)]^{2}$ (see Eq.~\ref{eq:alpha2_alpha4}) is invariant under the transformations $\hat T^2:z\mapsto z+2$ and $\hat S:z\mapsto -1/z$, which are generators of the projective $ \vartheta$-group $ \overline{\varGamma}_{\vartheta}$~\cite[][p.~85]{Schoeneberg}.
The $ \vartheta$-group $ \varGamma_{\vartheta}$ is  related to the Hecke congruence group $ \varGamma_0(2)$ by a conjugation~\cite[][p.~86]{Schoeneberg}:\begin{align}\begin{pmatrix}1 & -1 \\
1 & 0 \\
\end{pmatrix}\varGamma_0(2)=\varGamma_{\vartheta}\begin{pmatrix}1 & -1 \\
1 & 0 \\
\end{pmatrix}
.\label{eq:theta_Hecke2_conj}
\end{align}Naturally, one can show that\begin{align}
G_{s}^{\mathfrak H/\overline{\varGamma}_\vartheta}(z,z')=G_s^{\mathfrak H/\overline{\varGamma}_0(2)}\left( -\frac{1}{z-1},-\frac{1}{z'-1} \right)=G_s^{\mathfrak H/\overline{\varGamma}_0(2)}\left( \frac{z-1}{2} ,\frac{z'-1}{2}\right),\label{eq:GsTheta_GsHecke2_reln}
\end{align}by  invoking the conjugation relation $ \varGamma_0(2)\cong \varGamma_\vartheta$ (Eq.~\ref{eq:theta_Hecke2_conj}) as well as the Fricke involution on $ \varGamma_0(2)\backslash\mathfrak H^*$ (Eq.~\ref{eq:Fricke_inv_HeckeN_GZ}).
\begin{proposition}[Some Addition Formulae for Automorphic Green's Functions]\label{prop:add_form_auto_Green}We have the following algebraic identities whenever all the summands are finite:\begin{align}G_s^{\mathfrak H/PSL(2,\mathbb Z)}(z,z')={}&G_s^{\mathfrak H/\overline{\varGamma}_\vartheta}(z,z')+G_s^{\mathfrak H/\overline{\varGamma}_\vartheta}\left(-\frac{z+1}{z},z'\right)+G_s^{\mathfrak H/\overline{\varGamma}_\vartheta}(z+1,z'),
\label{eq:SL2Z_Theta_add}\\
G_s^{\mathfrak H/PSL(2,\mathbb Z)}(z,z')={}&G_s^{\mathfrak H/\overline{\varGamma}_{0}(3)}(z,z')+G_s^{\mathfrak H/\overline{\varGamma}_{0}(3)}\left( -\frac{1}{z} ,z'\right)+G_s^{\mathfrak H/\overline{\varGamma}_{0}(3)}\left( -\frac{1}{z+1} ,z'\right)+G_s^{\mathfrak H/\overline{\varGamma}_{0}(3)}\left( -\frac{1}{z-1} ,z'\right),\label{eq:Gs_Hecke3_add_G2_PSL}\\G_s^{\mathfrak H/PSL(2,\mathbb Z)}(z,z')={}&G_s^{\mathfrak H/\overline{\varGamma}(2)}(z,z')+G_s^{\mathfrak H/\overline{\varGamma}(2)}\left(-\frac{1}{z},z'\right)+G_s^{\mathfrak H/\overline{\varGamma}(2)}\left(-\frac{z+1}{z},z'\right)+G_s^{\mathfrak H/\overline{\varGamma}(2)}\left(\frac{z\vphantom{1}}{z+1},z'\right)\notag\\& +G_s^{\mathfrak H/\overline{\varGamma}(2)}(z+1,z')+G_s^{\mathfrak H/\overline{\varGamma}(2)}\left(-\frac{1}{z+1},z'\right),\label{eq:SL2Z_Gamma2_add}
\\G_s^{\mathfrak H/\overline{\varGamma}_{0}(2)}(z,z')={}&G_s^{\mathfrak H/\overline{\varGamma}(2)}(z,z')+G_s^{\mathfrak H/\overline{\varGamma}(2)}(z+1,z'),\label{eq:Hecke2_Gamma2_add}\\G_s^{\mathfrak H/\overline{\varGamma}_\vartheta}(z,z')={}&G_s^{\mathfrak H/\overline{\varGamma}(2)}(z,z')+G_s^{\mathfrak H/\overline{\varGamma}(2)}\left(-\frac1z,z'\right).\label{eq:G_s_Theta_add_form}\end{align}\end{proposition}\begin{proof}One simply goes through criteria \ref{itm:AGF1'}--\ref{itm:AGF3'} for the right-hand side of each proposed formula,  with respect to the variable $z $. The details of these routine procedures are left to the readers.
  \end{proof}

\begin{remark}The numbers of terms on both sides of Eqs.~\ref{eq:SL2Z_Theta_add}--\ref{eq:SL2Z_Gamma2_add} are consistent with the indices of subgroups $ [\varGamma_0(1):\varGamma_0(2)]=3$,  $ [\varGamma_0(1):\varGamma_0(3)]=4$ and  $ [\varGamma_0(1):\varGamma_0(4)]=6$ (see Fig.~\ref{fig:fundomain}\textit{f}), while the situations in Eqs.~\ref{eq:Hecke2_Gamma2_add} and \ref{eq:G_s_Theta_add_form} are compatible with $[\varGamma_0(2):\varGamma_0(4)] =[\varGamma_0(1):\varGamma_0(4)]/[\varGamma_0(1):\varGamma_0(2)]=2$. If our scope is not restricted to  the groups  isomorphic to $ \varGamma_0(N)$ ($N\in\{1,2,3,4\}$), then it is also possible to construct other addition formulae in a similar spirit. \eor\end{remark}

The integral representations for $ G_2^{\mathfrak H/\overline\varGamma_0(N)}(z,z')$ ($ N\in\{1,2,3,4\}$) in Eqs.~\ref{eq:G2HeckeN_unified} and \ref{eq:G2PSL2Z_Eisenstein_form} involve the Eisenstein series. One can recast such integral representations in terms of fractional degree Legendre functions $P_{-1/2} $, $ P_{-1/3}$, $ P_{-1/4}$ and $ P_{-1/6}$. The case of $P_{-1/2} $ is directly related  to the complete elliptic integral of the first kind    $\mathbf K(\sqrt{\lambda})=\frac{\pi}{2}P_{-1/2}(1-2\lambda)=\int_0^{\pi/2}(1-\lambda\cos^2\theta)^{-1/2}\D \theta,\lambda\in\mathbb C\smallsetminus[1,+\infty)$.  By convention, for $ \lambda>1$, one defines $ \mathbf K(\sqrt{\lambda}):=\mathbf K(\sqrt{\lambda-i0^+})=\overline{\mathbf K(\sqrt{\lambda+i0^+})}$. The other three Legendre functions  $ P_{-1/3}$, $ P_{-1/4}$ and $ P_{-1/6}$ feature prominently in Ramanujan's elliptic function theory to alternative bases \cite[][Chap.~33]{RN5}.

For any complex degree $\nu$, the Legendre function of the first kind $P_\nu $ is defined via the Mehler--Dirichlet integral
\begin{align}P_\nu(1)=1;\quad P_\nu(\cos\theta):=\frac{2}{\pi}\int_0^\theta\frac{\cos\frac{(2\nu+1)\beta}{2}}{\sqrt{\smash[b]{2(\cos\beta-\cos\theta)}}}\D\beta,\quad \theta\in(0,\pi),\nu\in\mathbb C, \label{eq:Pnu_defn_MD}\end{align}and admits an analytic continuation as a well-defined function on the slit plane: $ P_{\nu }(\xi),\xi\in\mathbb C\smallsetminus(-\infty,-1]$.
In particular, for $ -1<\nu<0$, one can  use the Euler integral representation of hypergeometric functions:\footnote{In our notation, the upright $ \Gamma$ is reserved for the Euler integral of the second kind $ \Gamma(\xi):=\int_0^\infty t^{\xi-1}e^{-t}\D t,\xi>0$ and its analytic continuation, in contrast to the  congruence subgroup   $ \varGamma$ set in slanted typeface.}\begin{align}
_2F_1\left( \left.\begin{array}{c}
a,b \\
c \\
\end{array}\right|w \right)=\frac{\Gamma(c)}{\Gamma(b)\Gamma(c-b)}\int_0^1t^{b-1}(1-t)^{c-b-1}(1-tw)^{-a}\D t,\quad \R c>\R b>0,-\pi<\arg(1-w)<\pi\label{eq:Euler_int}
\end{align}to define  $ P_{\nu}(\xi)={_2}F_1\left( \left.\begin{smallmatrix}-\nu,\nu+1\\1\end{smallmatrix}\right| \smash{\frac{1-\xi}{2}}\right),\xi\in\mathbb C\smallsetminus(-\infty,-1] $.
Hereinafter, unless explicitly specified otherwise (such as the sign convention for $ \sqrt{\smash[b]{(j(z)-1728)/j(z)}}$), fractional powers of non-zero complex numbers are taken as $ w^\nu:=e^{\nu\log w}$ where $ \log w:=\log|w|+i\arg w$, $ \log|w|\in\mathbb R$ and $ -\pi<\arg w\leq \pi$.

\begin{proposition}[Integral Representations for $ G_2^{\mathfrak H/\overline{\varGamma}_0(N)}(z,z')$, $N\in\{1,2,3,4\}$ Using Legendre Functions]\label{prop:G2_HeckeN_Pnu}For degrees $ \nu\in\{-1/2,-1/3,-1/4,-1/6\}$, the function \begin{align}
R_{\nu}( \xi):={}&\frac{1-\xi^{2}}{P_{\nu}(\xi)}\frac{\D P_{\nu}(\xi)}{\D \xi}+\frac{1-\xi^{2}}{P_{\nu}(-\xi)}\frac{\D P_{\nu}(-\xi)}{\D \xi}-\frac{\sin(\nu\pi)}{\pi}\left\{\frac{1}{ [P_{\nu}(\xi)]^{2}\I \frac{iP_{\nu}(-\xi)}{P_{\nu}(\xi)}}-\frac{1}{[P_{\nu}(-\xi)]^{2}\I \frac{iP_{\nu}(\xi)}{P_{\nu}(-\xi)}}\right\},\label{eq:R_nu_defn}
\end{align}is well-defined for $ \xi\in\mathbb C\smallsetminus((-\infty,-1]\cup[1,+\infty))$, and  extends continuously to all $ \xi\in\mathbb C$,  smoothly across the branch cut of $ P_\nu(\xi),\xi\in\mathbb C\smallsetminus(-\infty,-1]$:\begin{align}
R_\nu(1)=-R_\nu(-1):=\lim_{\xi\to1}R_\nu(\xi)=0;\quad R_\nu(\xi)=-R_\nu(-\xi):= R_\nu(\xi+i0^+)\equiv R_\nu(\xi-i0^+),\quad \forall \xi>1.\label{eq:R_nu_defn_ext}
\end{align}For degrees $ \nu\in\{-1/2,-1/3,-1/4,-1/6\}$, and the corresponding levels $ N=4\sin^2(\nu\pi)\in\{1,2,3,4\}$, define \begin{align}\varrho_{2,\nu}(\xi|z):=\frac{4\alpha_N(z)[1-\alpha_N(z)]}{[\xi-1+2\alpha_N(z)]^{2}}-\frac{R_{\nu}(1-2\alpha_N(z))}{\xi-1+2\alpha_N(z)},\quad \emph{a.e. }\xi\in\mathbb C, z\in\mathfrak H,\label{eq:rho_2_nu_xi_z_defn}\end{align}and $ \rho_{2,-1/6}(\xi|z)=\varrho_{2,-1/6}(\xi|z)+\varrho_{2,-1/6}(\xi|-1/z)$, then we have an integral representation
{\allowdisplaybreaks\begin{align}
G_2^{\mathfrak H/PSL(2,\mathbb Z)}(z,z')={}&\frac{\pi}{\I \frac{iP_{-1/6}(-\sqrt{(j(z')-1728)/j(z')})}{P_{-1/6}(\sqrt{(j(z')-1728)/j(z')})}}\R\int_{\sqrt{(j(z')-1728)/{j(z')}}}^1[P_{-1/6}(\xi)]^2\rho_{2,-1/6}(\xi|z)\times\notag\\&\times\left[\frac{iP_{-1/6}(-\xi)}{P_{-1/6}(\xi)}-\vphantom{\overline{\frac{\sqrt{j}}{}}}\frac{iP_{-1/6}(-\sqrt{(j(z')-1728)/j(z')})}{P_{-1/6}(\sqrt{(j(z')-1728)/j(z')})}\right]\times\notag\\&\times\left[\frac{iP_{-1/6}(-\xi)}{P_{-1/6}(\xi)}-\overline{\left(\frac{iP_{-1/6}(-\sqrt{(j(z')-1728)/j(z')})}{P_{-1/6}(\sqrt{(j(z')-1728)/j(z')})}\right)}\right]\D \xi\notag\\&+\frac{\pi}{\I \frac{iP_{-1/6}(-\sqrt{(j(z')-1728)/j(z')})}{P_{-1/6}(\sqrt{(j(z')-1728)/j(z')})}
}\R\int_{-1}^1[P_{-1/6}(\xi)]^2\rho_{2,-1/6}(\xi|z)\D \xi,\quad \emph{a.e. }z,z'\in\mathfrak H\label{eq:G2_z_z'_arb}
\intertext{and the following integral representations for $ \nu\in\{-1/2,-1/3,-1/4\}$ and  $ N=4\sin^2(\nu\pi)\in\{2,3,4\}$:}G_2^{\mathfrak H/\overline{\varGamma}_0(N)}(z,z')={}&\frac{\pi}{\sqrt{N}\I\frac{iP_\nu(2\alpha_{N}(z')-1)}{P_\nu(1-2\alpha_{N}(z'))}}\R\int_{1-2\alpha_{N}(z')}^1[P_\nu(\xi)]^2\varrho_{2,\nu}(\xi|z)\times\notag\\&\times\left[\frac{iP_\nu(-\xi)}{P_\nu(\xi)}-\vphantom{\overline{\frac{1}{}}}\frac{iP_\nu(2\alpha_{N}(z')-1)}{P_\nu(1-2\alpha_{N}(z'))}\right]\left[\frac{iP_\nu(-\xi)}{P_\nu(\xi)}-\overline{\left(\frac{iP_\nu(2\alpha_{N}(z')-1)}{P_\nu(1-2\alpha_{N}(z'))}\right)}\right]\D \xi\notag\\&+\frac{\pi}{\sqrt{N}\I \frac{iP_{\nu}(2\alpha_{N}(z')-1)}{P_\nu(1-2\alpha_{N}(z'))}
}\R\int_{-1}^1[P_{\nu}(-\xi)]^2\varrho_{2,\nu}(\xi|z)\D \xi,\quad \emph{a.e. }z,z'\in\mathfrak H.\label{eq:G2Hecke234_Pnu}\end{align}}Here in Eqs.~\ref{eq:G2_z_z'_arb} and \ref{eq:G2Hecke234_Pnu}, the paths of integration can be taken as any curves in the double slit plane $ \xi\in \mathbb C\smallsetminus((-\infty,-1]\cup[1,+\infty))$ that miss the singularities of the integrands.
\end{proposition}\begin{proof}From Ramanujan's elliptic function theory to alternative bases, one may read off the following identities~(see \cite[][Chap.~33,  Theorem~11.6, Theorem~9.11, Corollary~2.11]{RN5} and~\cite[][Chap.~17, Entries 13(viii)--(ix)]{RN3}):\begin{align}
\left[ P_{-1/6}\left( \sqrt{\frac{j(z)-1728}{j(z)}} \right) \right]^{12}={}&\Delta(z)j(z)=[E_{4}(z)]^{3},&&\I z>0,|\R z|<\frac{1}{2},|z+1|>1,|z-1| >1;\label{eq:P_sixth_eta}\\{[}P_\nu(1-2\alpha_N(z))]^2={}&\frac{NE_{2}(Nz)-E_2(z)}{N-1},&&\I z>0,|\R z|<\frac{1}{2},\left\vert z+\frac{1}{N} \right\vert>\frac{1}{N},\left\vert z-\frac{1}{N} \right\vert>\frac{1}{N}.\label{eq:P_nu_sqr_E2_diff}
\end{align}Here, in Eq.~\ref{eq:P_nu_sqr_E2_diff}, the degree $ \nu\in\{-1/2,-1/3,-1/4\}$ corresponds to level  $ N=4\sin^2(\nu\pi)\in\{2,3,4\}$. Referring back to Eqs.~\ref{eq:alpha_N_deriv_E2} and \ref{eq:alpha_N_ratio_ids}, we see that with $ \nu\in\{-1/2,-1/3,-1/4,-1/6\}$, the relation $ P_\nu(\xi)\neq0$ holds for $ \xi\in\mathbb C\smallsetminus((-\infty,-1]\cup[1,+\infty))$. Thus, Eq.~\ref{eq:R_nu_defn} is well defined.

By taking ratios of a formula (any one among Eq.~\ref{eq:P_sixth_eta} and the three forms of Eq.~\ref{eq:P_nu_sqr_E2_diff}) at a pair of points $ z$ and $ -1/(Nz)$, where $ N\in\{1,2,3,4\}$, one can verify the following formulae:\begin{align}z=\frac{i P_\nu(2\alpha_N(z)-1)}{\sqrt{N}P_\nu(1-2\alpha_N(z))},&&
\I z>0,|\R z|<\frac{1}{2},\left\vert z+\frac{1}{N} \right\vert>\frac{1}{N},\left\vert z-\frac{1}{N} \right\vert>\frac{1}{N}\label{eq:z_Pnu_ratios}
\end{align} after extracting appropriate roots, and recalling that $E_2(-1/(Nw))=N^2 w^2 E_2(Nw) $, $ \Delta(-1/z)=z^{12}\Delta(z)$.
With the Legendre differential equation and special values of the Legendre function, one can verify the following differentiation formula that appeared  in Ramanujan's notebooks~\cite[][p.~88]{RN2}: \begin{align}\frac{\D }{\D \xi}\frac{P_{\nu }(-\xi)}{P_{\nu }(\xi)}=-\frac{2}{\pi}\frac{\sin(\nu\pi)}{(1-\xi^{2})[P_{\nu }(\xi)]^2},\quad \xi\in\mathbb C\smallsetminus((-\infty,-1]\cup[1,+\infty)).\label{eq:P_nu_ratio_deriv}\end{align}
Using the information above, together with  Ramanujan's  differential equation for the Eisenstein series \cite[][Eq.~30]{Ramanujan1916}, one can verify the following identities:\begin{align}
R_{-1/6}\left( \sqrt{\frac{j(z)-1728}{j(z)}}
\right)\sqrt{\frac{j(z)-1728}{j(z)}}={}&-\frac{E_{6}(z)}{3[E_{4}(z)]^{2}}\left[ E_2(z)-\frac{ E_6(z)}{E_4(z)} \right],\label{eq:R_sixth_Eisenstein}\\ R_{\nu}(1-2\alpha_N(z))={}&-\frac{N-1}{6}\frac{N^{2}[E_{2}(Nz)]^2-N^{2}E_{4}(N z)-[E_{2}(z)]^2+E_{4}(z)}{[NE_{2}(Nz)-E_{2}(z)]^{2}},\label{eq:Rnu_Eisenstein}
\end{align} where $\nu\in\{-1/2,-1/3,-1/4\}$ corresponds to  $ N=4\sin^2(\nu\pi)\in\{2,3,4\}$. This  proves the continuous extensibility of $R_{\nu}$ for $\nu\in\{-1/2,-1/3,-1/4,-1/6\}$ (Eq.~\ref{eq:R_nu_defn_ext}). Upon variable substitutions according to Eq.~\ref{eq:z_Pnu_ratios}, one can  show that $ \rho_{2,-1/6}(1-2\alpha_1(\zeta)|z)=\rho_2^{\mathfrak H/PSL(2,\mathbb Z)}(\zeta,z)$  as well as $ \varrho_{2,\nu}(1-2\alpha_N(\zeta)|z)=\varrho_2^{\mathfrak H/\overline\varGamma_0(N)}(\zeta,z)$ for $ \nu\in\{-1/2,-1/3,-1/4\}$ and $N=4\sin^2(\nu\pi)\in\{2,3,4\} $. This finally reveals the analytic connections between Eqs.~\ref{eq:G2_z_z'_arb} and \ref{eq:G2PSL2Z_Eisenstein_form} (resp.~\ref{eq:G2Hecke234_Pnu} and \ref{eq:G2HeckeN_unified}) .
\end{proof}\begin{remark}If $z$ is a CM point, then one can use the explicit formulae in Eqs.~\ref{eq:R_sixth_Eisenstein} and \ref{eq:Rnu_Eisenstein} to verify that the respective values of $ R_\nu$ are solvable algebraic numbers (see Eq.~\ref{eq:fns_alg_val_at_CM_pts} and Appendix~\ref{app:algebraicity}). Consequently, the related expression $ \varrho_\nu(\xi|z)$ will be an algebraic function of $ \xi$ with algebraic coefficients. Recalling the Euler integral representation for hypergeometric functions from Eq.~\ref{eq:Euler_int}, we see that  Eqs.~\ref{eq:G2_z_z'_arb} and \ref{eq:G2Hecke234_Pnu}  (up to an overall factor that is an integer power of $ \pi$) are integrals of algebraic functions over algebraic domains, if both $ z$ and $z'$ are CM points. This qualifies them as Kontsevich--Zagier periods \cite[][\S1.1]{KontsevichZagier}. As pointed out in the survey of Kontsevich and Zagier~\cite[][\S3.4]{KontsevichZagier}, the conjectural relation $ \sqrt{y_1y_2}G_2^{\mathfrak H/\overline{\varGamma}_0(N)}(z_1,z_2)\in\log\overline{\mathbb Q}$ (where $ N\in\{1,2,3,4\}$, $ [\mathbb Q(z_1):\mathbb Q]=[\mathbb Q(z_2):\mathbb Q]=2$, $ z_1\notin\varGamma_0(N)z_2$) is a type of ``period identity''. \eor\end{remark}\begin{remark}\label{rmk:spec_val_G2}Now we showcase a few examples where the ``period identity at CM points''  $ \sqrt{y_1y_2}$ $G_2^{\mathfrak H/\overline{\varGamma}_0(N)}(z_1,z_2)\in\log\overline{\mathbb Q}$ can be proved by elementary manipulations of integrals.

For $ \nu\in\{-1/6,-1/4,-1/3\}$, the integral representations in Eqs.~\ref{eq:G2_z_z'_arb} and \ref{eq:G2Hecke234_Pnu} immediately bring us some special values of automorphic Green's functions:\begin{align}
G_{2}^{\mathfrak H/PSL(2,\mathbb Z)}\left(\frac{1+i\sqrt{3}}{2},\frac{i}{\sqrt{1}}\right)={}&-\frac{8\pi}{3}\int_{0}^{1}[P_{-1/6}(\xi)]^2\D \xi=-\frac{12}{\sqrt{3}}\log(2+\sqrt{3}),\label{eq:G2_ell3_i_spec_val}\\G_2^{\mathfrak H/\overline{\varGamma}_{0}(2)}\left( \frac{i-1}{2},\frac{i}{\sqrt{2}} \right)={}&-\frac{\pi}{\sqrt{2}}\int_{0}^{1}[P_{-1/4}(\xi)]^2\D \xi=-\frac{4}{\sqrt{2}}\log(1+\sqrt{2}),\label{eq:G2_Hecke2_spec_val}\\G_2^{\mathfrak H/\overline{\varGamma}_{0}(3)}\left( \frac{3+i\sqrt{3}}{6},\frac{i}{\sqrt{3}} \right)={}&-\frac{2\pi}{3\sqrt{3}}\int_{0}^{1}[P_{-1/3}(\xi)]^2\D \xi=-2\log2.\label{eq:G2_Hecke3_spec_val}
\end{align}In all these three cases of $ G_2^{\mathfrak H/\overline\varGamma_0(N)}(z,z')$, we have $ \alpha_N(z)=\infty$ (see Eqs.~\ref{eq:alpha_inf_rho_lim} and \ref{eq:alpha_inf_rho_lim_PSL2Z}) and $ \alpha_N(z')=1/2$, so it is clear that the values of the respective automorphic Green's functions are certain algebraic multiples of \begin{align} \pi\int_0^1[P_\nu(\xi)]^2\D \xi=\frac{\pi}{2 \nu +1}\left\{1+\frac{\sin (\nu\pi   )}{\pi} \left[ \psi ^{(0)}\left(\frac{\nu+2 }{2}\right)-\psi ^{(0)}\left(\frac{\nu+1 }{2}\right )\right]\right\},\quad \nu\in\mathbb C\smallsetminus\{-1/2\}.\label{eq:Pnu_sqr_0to1}\end{align}Here, the logarithmic derivatives of Euler's gamma function are the polygamma functions $ \psi^{(m)}(w):=\D^{m+1}\log\Gamma(w)/\D w^{m+1}$ for $ m\in\mathbb Z_{\geq0}$. The integral formula in Eq.~\ref{eq:Pnu_sqr_0to1} (see \cite[][item~7.113.1]{GradshteynRyzhik}) can be proved by simple applications of the Legendre differential equations. For $ w\in\mathbb Q\cap(0,1)$, one  can evaluate  $ \psi^{(0)}(w)$ using the explicit formula provided by the digamma theorem of Gau{\ss}~\cite[][\S1.7.3]{HTF1}. This leads to the logarithmic expressions in Eqs.~\ref{eq:G2_ell3_i_spec_val}--\ref{eq:G2_Hecke3_spec_val}.

Exploiting the addition formulae in Proposition~\ref{prop:add_form_auto_Green}, one can deduce a few  special values of  $ G_2^{\mathfrak H/\overline\varGamma_0(4)}(z/2,i/\sqrt{4})=G_2^{\mathfrak H/\overline\varGamma(2)}(z,i)=G_2^{\mathfrak H/\overline\varGamma_{\vartheta}}(z,i)/2=G_2^{\mathfrak H/\overline\varGamma_{0}(2)}((z-1)/2,(i-1)/2)/2$ (see Eqs.~\ref{eq:Hecke4_add}, \ref{eq:G_s_Theta_add_form} and \ref{eq:GsTheta_GsHecke2_reln}) that are related to Eqs.~\ref{eq:G2_ell3_i_spec_val} and \ref{eq:G2_Hecke2_spec_val}:{\allowdisplaybreaks\begin{align}-\frac{2}{\sqrt{3}}\log(2+\sqrt{3})={}&\frac{G_2^{\mathfrak H/\overline{\varGamma}_0(2)}(e^{\pi i/3},i)}{2}\notag\\={}&G_2^{\mathfrak H/\overline{\varGamma}(2)}(e^{\pi i/3},i)=G_2^{\mathfrak H/\overline{\varGamma}(2)}(i\sqrt{3},i)=G_2^{\mathfrak H/\overline{\varGamma}(2)}(e^{2\pi i/3},i)=G_2^{\mathfrak H/\overline{\varGamma}(2)}(i/\sqrt{3},i),\label{eq:log2_plus_sqrt3_more}\\
-\frac{2}{\sqrt{2}}\log(1+\sqrt{2})={}&G_2^{\mathfrak H/\overline\varGamma(2)}(-1+i\sqrt{2},i)=G_2^{\mathfrak H/\overline\varGamma(2)}\left(\frac{-1+i\sqrt{2}}{3},i\right)\notag\\={}&G_2^{\mathfrak H/\overline\varGamma(2)}\left(\frac{1+i\sqrt{2}}{3},i\right)=G_2^{\mathfrak H/\overline\varGamma(2)}(1+i\sqrt{2},i).\label{eq:G2Gamma2spec_val}
\end{align}}The main idea here is the following four-fold symmetry:\begin{align}
G_s^{\mathfrak H/\overline\varGamma(2)}(z,i)=G_s^{\mathfrak H/\overline\varGamma(2)}\left( \frac{1+z}{1-z} ,i\right)=G_s^{\mathfrak H/\overline\varGamma(2)}\left( -\frac{1}{z} ,i\right)= G_s^{\mathfrak H/\overline\varGamma(2)}\left( \frac{z-1}{z+1} ,i\right),\label{eq:GsGamma2i_D4}
\end{align}which descends from the identity $G_s^{\mathfrak H/\overline\varGamma(2)}(z,i) =G_s^{\mathfrak H/\overline\varGamma_0(4)}(z/2,i/2)=G_s^{\mathfrak H/\overline\varGamma_0(4)}(-1/(2z),i/2)$ $=G_s^{\mathfrak H/\overline\varGamma(2)}(-1/z,i) $ (see Eq.~\ref{eq:Fricke_inv_HeckeN_GZ}) along with Eqs.~\ref{eq:GsTheta_GsHecke2_reln} and \ref{eq:G_s_Theta_add_form}.

One can derive a special addition formula\begin{align}
G_s^{\mathfrak H/PSL(2,\mathbb Z)}(z,i)=2[G_s^{\mathfrak H/\overline{\varGamma}(2)}(z,i)+G_s^{\mathfrak H/\overline{\varGamma}(2)}(z+1,i)+G_s^{\mathfrak H/\overline{\varGamma}(2)}(2z+1,i)]\label{eq:GsPSL2Zi_GsGamma2i_add_form}
\end{align}from Eqs.~\ref{eq:SL2Z_Gamma2_add} and \ref{eq:GsGamma2i_D4}. \eor\end{remark}Recalling the definition of the $L$-function $ L(s,\chi_D)$ from Eq.~\ref{eq:L_s_chi_D_defn}, and that\begin{align}G\equiv L(2,\chi_{-4})={}&\sum_{\ell=0}^\infty\frac{(-1)^\ell}{(2\ell+1)^2}=\frac{\psi ^{(1)}\left(\frac{1}{4}\right)-\psi ^{(1)}\left(\frac{3}{4}\right)}{16},\label{eq:L2chi_4_defn_G}\\
L(2,\chi_{-3})={}&\sum_{\ell=0}^\infty\left[ \frac{1}{(3\ell+1)^{2}} -\frac{1}{(3\ell+2)^{2}}\right]=\frac{\psi ^{(1)}\left(\frac{1}{3}\right)-\psi ^{(1)}\left(\frac{2}{3}\right)}{9},\label{eq:L2chi_3_defn}\\L(2,\chi_{-8})=L(2,\chi_{-2})={}&\sum _{n=0}^{\infty } \frac{(-1)^{ n (n-1)/2}}{(2 n+1)^2}=\frac{\psi ^{(1)}\left(\frac{1}{8}\right)+\psi ^{(1)}\left(\frac{3}{8}\right)-\psi ^{(1)}\left(\frac{5}{8}\right)-\psi ^{(1)}\left(\frac{7}{8}\right)}{64},
\end{align}  we may state and prove some integral identities involving fractional degree Legendre functions and these special $L$-values.\begin{proposition}[Integral Representations of Some Special $L$-Values]\label{prop:Pnu_sec_L_weight4}We have the following identities:{\allowdisplaybreaks\begin{align}
 L(2,\chi_{-4})={}&-\frac{\pi}{40}\lim_{y'\to+\infty}G_2^{\mathfrak H/PSL(2,\mathbb Z)}(i,z')y'=-\frac{\pi^{2}}{20}\R\int_{-1}^1\frac{[P_{-1/6}(\xi)]^2}{\xi^{2}}\D \xi\notag\\={}&-\frac{\pi}{16}\lim_{y'\to+\infty}G_2^{\mathfrak H/\overline\varGamma_{\vartheta}}(i,z')y'=\frac{\pi^{2}}{32}\int_{-1}^1[P_{-1/4}(\xi)]^2\D \xi\notag\\={}&-\frac{\pi}{8}\lim_{y'\to+\infty}G_2^{\mathfrak H/\overline\varGamma(2)}(i,z')y'=-\frac{\pi^{2}}{16}\R\int_{-1}^1\frac{[P_{-1/2}(\xi)]^2}{\xi^2}\D \xi;\label{eq:L2chi_4_Pnu}\\ L(2,\chi_{-3})={}&-\frac{\pi}{45}\lim_{y'\to+\infty}G_{2}^{\mathfrak H/PSL(2,\mathbb Z)}\left( \frac{1+i\sqrt{3}}{2},z'\right)y'=\frac{4\pi^{2}}{135}\int_{-1}^1[P_{-1/6}(\xi)]^2\D \xi\notag\\={}&-\frac{2\pi}{9}\lim_{y'\to+\infty}G_{2}^{\mathfrak H/\overline{\varGamma}_0(3)}\left( \frac{3+i\sqrt{3}}{6},z'\right)y'=\frac{2\pi^2}{81}\int_{-1}^1[P_{-1/3}(\xi)]^2\D \xi\notag\\={}&-\frac{2\pi}{135}\lim_{y'\to+\infty}G_2^{\mathfrak H/PSL(2,\mathbb Z)}(i\sqrt{3},z')y'\notag\\={}&-\frac{56\pi}{135}\lim_{y'\to+\infty}G_2^{\mathfrak H/\overline{\varGamma}_{0}(3)}\left(  \frac{i}{\sqrt{3}},z'\right)y'=-\frac{4\pi^{2}}{81}\R\int_{-1}^1\frac{[P_{-1/3}(\xi)]^2}{\xi^2}\D \xi;\label{eq:L2chi3}\\ L(2,\chi_{-2})={}&-\frac{\pi}{40}\lim_{y'\to+\infty}G_2^{\mathfrak H/PSL(2,\mathbb Z)}(i\sqrt{2},z')y'=-\frac{\pi}{16}\lim_{y'\to+\infty}G_2^{\mathfrak H/\overline\varGamma_{\vartheta}}(1+i\sqrt{2},z')y'\notag\\={}&-\frac{\pi}{16}\lim_{y'\to+\infty}G_2^{\mathfrak H/\overline\varGamma_{0}(2)}\left( \frac{i}{\sqrt{2}},\frac{z'-1}{2} \right)y'=-\frac{\pi^2}{16}\R\int_{-1}^1\frac{[P_{-1/4}(\xi)]^2}{\xi^2}\D\xi
.\label{eq:L2chi_2_Pnu}\end{align}}Here, all the integration paths are taken as curves  in the slit plane $ \mathbb C\smallsetminus(-\infty,-1]$ that circumvent the singularities of the integrands.\end{proposition}\begin{proof}It is clear that all these integrals involving Legendre functions can be associated with the claimed asymptotic behavior of automorphic Green's functions, by Eqs.~\ref{eq:G2_z_z'_arb} and \ref{eq:G2Hecke234_Pnu}. What remains to be shown are their connections to the special $ L$-values.  We shall explain such connections for Eq.~\ref{eq:L2chi3} in detail and sketch the proof for the rest.

The first equality in  Eq.~\ref{eq:L2chi3}  follows from  Eq.~\ref{eq:Gs_asympt} and an integration of Ramanujan's formula (see~\cite{BorweinBorweinCubic},~\cite{Berndt1992} or~\cite[][Theorem~3.7.10]{NTRamanujan})\begin{align*}\sum_{\substack{m,n\in\mathbb Z\\m^2+n^2\neq0}}ye^{-(m^{2}+mn+n^2)y}=6\sum_{\ell=0}^\infty\left[\frac{y}{e^{(3\ell+1)y}-1} -\frac{y}{e^{(3\ell+2)y}-1}\right]\end{align*}over $ y\in(0,+\infty)$. The first two lines in  Eq.~\ref{eq:L2chi3} are related to each other by an addition formula for  $ z\in\mathfrak H\smallsetminus(SL(2,\mathbb Z)\frac{1+i\sqrt{3}}{2})$:\begin{align}&
G_2^{\mathfrak H/PSL(2,\mathbb Z)}\left(z,\frac{1+i\sqrt{3}}{2}\right)\notag\\={}&G_2^{\mathfrak H/\overline{\varGamma}_{0}(3)}\left( z ,\frac{3+i\sqrt{3}}{6}\right)+G_2^{\mathfrak H/\overline{\varGamma}_{0}(3)}\left( \frac{z}{3} ,\frac{3+i\sqrt{3}}{6}\right)+G_2^{\mathfrak H/\overline{\varGamma}_{0}(3)}\left( \frac{z+1}{3} ,\frac{3+i\sqrt{3}}{6}\right)+G_2^{\mathfrak H/\overline{\varGamma}_{0}(3)}\left( \frac{z-1}{3} ,\frac{3+i\sqrt{3}}{6}\right),\label{eq:G2_Hecke3_add_G2_PSL}
\end{align} which can be proved by applying the Fricke involution (Eq.~\ref{eq:Fricke_inv_HeckeN_GZ}) to the last three terms on the right-hand side of Eq.~\ref{eq:Gs_Hecke3_add_G2_PSL}.  In particular, Eq.~\ref{eq:G2_Hecke3_add_G2_PSL} entails the ratio between the asymptotic behavior of two automorphic Green's functions:  \begin{align}\lim_{y\to+\infty}G_2^{\mathfrak H/PSL(2,\mathbb Z)}\left(z,\frac{1+i\sqrt{3}}{2}\right)y=(1+3+3+3)\lim_{y\to+\infty}G_2^{\mathfrak H/\overline{\varGamma}_{0}(3)}\left( z ,\frac{3+i\sqrt{3}}{6}\right)y.\label{eq:PSL_Hecke3_asympt_conv}\end{align}     The third line in   Eq.~\ref{eq:L2chi3} is related to $ L(2,\chi_{-3})$ by Eq.~\ref{eq:Gs_asympt} and  another generalization of the two-squares theorem (Eq.~\ref{eq:two_squares}) due to Ramanujan~\cite[][p.~75, Eq.~3.7.8]{NTRamanujan}:\begin{align}
\sum_{\substack{m,n\in\mathbb Z\\m^2+n^2\neq0}}ye^{-(3m^{2}+n^2)y}=2\sum_{\ell=0}^\infty y\left( \frac{1}{e^{(3\ell+1)y}-1}- \frac{1}{e^{(3\ell+2)y}-1}\right)+4\sum_{\ell=0}^\infty y\left( \frac{1}{e^{4(3\ell+1)y}-1}- \frac{1}{e^{4(3\ell+2)y}-1}\right).
\end{align}   The last line in    Eq.~\ref{eq:L2chi3}  follows from  an addition formula\begin{align}&
G_2^{\mathfrak H/PSL(2,\mathbb Z)}(z,i\sqrt{3})=G_2^{\mathfrak H/\overline{\varGamma}_{0}(3)}\left( z ,\frac{i}{\sqrt{3}}\right)+G_2^{\mathfrak H/\overline{\varGamma}_{0}(3)}\left( \frac{z}{3} ,\frac{i}{\sqrt{3}}\right)+G_2^{\mathfrak H/\overline{\varGamma}_{0}(3)}\left( \frac{z+1}{3} ,\frac{i}{\sqrt{3}}\right)+G_2^{\mathfrak H/\overline{\varGamma}_{0}(3)}\left( \frac{z-1}{3} ,\frac{i}{\sqrt{3}}\right),\label{eq:G2PSL2Zsqrt3}\end{align}which can be verified in a similar vein as Eq.~\ref{eq:G2_Hecke3_add_G2_PSL}.

The readers may fill in the details for Eqs.~\ref{eq:L2chi_4_Pnu} and \ref{eq:L2chi_2_Pnu} by referring back to the addition formulae in Eqs.~\ref{eq:SL2Z_Theta_add}, \ref{eq:G_s_Theta_add_form} and \ref{eq:GsPSL2Zi_GsGamma2i_add_form}, as well as Ramanujan's generalization of the two-squares theorem to the sum $|i\sqrt{2}m+n|^2=2m^2+n^2 $~\cite[][Theorems~3.7.2 and 3.7.3]{NTRamanujan}.
\end{proof}\begin{remark}The following integral formula is a standard result \cite[][item~7.112.3]{GradshteynRyzhik}:\begin{align}
\label{eq:P_nu_sqr}\int_{-1}^1[P_\nu(\xi)]^2\D \xi={}&\frac{2}{2\nu+1}\left[1-\frac{2\sin^2(\nu\pi)}{\pi^{2}}\psi^{(1)}(\nu+1)\right],\quad \nu\in\mathbb C\smallsetminus(\mathbb Z_{<0}\cup\{-1/2\}),
\end{align}and can be proved by Legendre differential equations. The first two lines in  Eq.~\ref{eq:L2chi3}  imply the following identity\begin{align}5\int_{-1}^1[P_{-1/3}(\xi)]^2\D \xi=6\int_{-1}^1[P_{-1/6}(\xi)]^2\D \xi.\label{eq:5third_6sixth}\end{align}This can also be verified by spelling out both sides of Eq.~\ref{eq:5third_6sixth} with the help of Eq.~\ref{eq:P_nu_sqr}, before invoking the duplication and reflection formulae of $ \psi^{(1)}(w)=\D^2\log\Gamma(w)/\D w^2$: \begin{align}4\psi^{(1)}(2w)={}&\psi^{(1)}(w)+\psi^{(1)}\left(w+\frac{1}{2}\right)&&\Longrightarrow&&4\psi^{(1)}\left(\frac{1}{3}\right)=\psi^{(1)}\left(\frac{1}{6}\right)+\psi^{(1)}\left(\frac{2}{3}\right),\label{eq:psi1_id1}\\\frac{\pi^{2}}{\sin^2(\pi w)}={}&\psi^{(1)}(w)+\psi^{(1)}(1-w)&&\Longrightarrow&& 4\pi^{2}=\psi^{(1)}\left(\frac{1}{6}\right)+\psi^{(1)}\left(\frac{5}{6}\right)=3\left[ \psi^{(1)}\left(\frac{1}{3}\right) +\psi^{(1)}\left(\frac{2}{3}\right)\right],\label{eq:psi1_id2}\end{align}which descend from the respective properties of Euler's gamma function. We will encounter some higher-weight analogs of Eq.~\ref{eq:5third_6sixth} later in Proposition~\ref{prop:high_w_L_values} of \S\ref{subsec:high_weight_KZ}.

 A generic  integral formula  for {\allowdisplaybreaks\begin{align*}
\R\int_{-1}^1\frac{[P_\nu(\xi)]^2}{\xi^2}\D \xi,\quad -1<\nu<0
\end{align*}}will be given in Eq.~\ref{eq:PnuPnu_xx_int_unit_interval_polygamma}. The readers are invited to check  Eq.~\ref{eq:PnuPnu_xx_int_unit_interval_polygamma} against the special cases $ \nu\in\{-1/2,-1/3,-1/4,-1/6\}$ evaluated in the proposition above.  \eor\end{remark}\subsection{Kontsevich--Zagier Integrals for Higher Weight Automorphic Green's Functions\label{subsec:high_weight_KZ}}In this subsection, we will treat the five remaining cases where there are no cusp forms on $ \varGamma_0(N)$, namely, $\dim\mathcal S_6(\varGamma_0(2))=\dim\mathcal S_6(\varGamma_0(1))=\dim\mathcal S_8(\varGamma_0(1))=\dim\mathcal S_{10}(\varGamma_0(1))=\dim\mathcal S_{14}(\varGamma_0(1))=0 $.\begin{proposition}[Kontsevich--Zagier Integral Representation for   $ G_3^{\mathfrak H/\overline{\varGamma}_0(2)}(z,z')$]\label{prop:G3Hecke2}For  $z\notin\varGamma_0(2)z'$, we have the following integral representation:\begin{align}
G_3^{\mathfrak H/\overline{\varGamma}_0(2)}(z,z')={}&\frac{4\pi^{2}}{(y')^{2}}\R\int_{z'}^{i\infty} \alpha_2(\zeta)[1-\alpha_2(\zeta)][2E_{2}(2\zeta)-E_2(\zeta)]^{3}\varrho^{\mathfrak H/\overline\varGamma_0(2)}_3(\zeta,z)\frac{(\zeta-z')^{2}(\zeta-\overline{z'})^{2}\D \zeta}{i}\notag\\&-\frac{4\pi^{2}}{(y')^{2}}\R\int_{0}^{i\infty} \alpha_2(\zeta)[1-\alpha_2(\zeta)][2E_{2}(2\zeta)-E_2(\zeta)]^{3}\varrho^{\mathfrak H/\overline\varGamma_0(2)}_3(\zeta,z)\frac{\zeta^{4}\D \zeta}{i},\label{eq:G3Hecke2_KZ_int}
\end{align}where\begin{align}
\varrho^{\mathfrak H/\overline\varGamma_0(2)}_3(\zeta,z)=\frac{ y^2}{8\pi}\left( \frac{\partial}{\partial y}\frac{1}{y} \right)^2\left\{ \frac{1}{\alpha_{2}(\zeta)-\alpha_2(z)}\frac{1}{[2E_{2}(2z)-E_2(z)]^2} \right\}.\label{eq:rho_3_Hecke2}
\end{align}Here, the invariant $ \alpha_2(z)$ is the same as the one prescribed in Eq.~\ref{eq:alpha_HeckeN_unified} or \ref{eq:alpha2_alpha4}, and the  paths of integration should avoid the singularities of the integrands.\end{proposition}\begin{proof}One can go through  criteria~\ref{itm:AGF1'}--\ref{itm:AGF3'} in a similar fashion as what we did in Proposition~\ref{prop:G2HeckeNunified}. The only non-trivial step is to demonstrate that the integral representation in Eq.~\ref{eq:G3Hecke2_KZ_int} respects the Fricke involution (Eq.~\ref{eq:Fricke_inv_HeckeN_GZ}). While it is easy to verify $ \varrho^{\mathfrak H/\overline\varGamma_0(2)}_3(\zeta,z)=-\varrho^{\mathfrak H/\overline\varGamma_0(2)}_3(-1/(2\zeta),-1/(2z))$, we still need to show that (see Eq.~\ref{eq:I_N_Fricke_inv_check})\begin{align}
\R\int_{0}^{i\infty} \alpha_2(\zeta)[1-\alpha_2(\zeta)][2E_{2}(2\zeta)-E_2(\zeta)]^{3}\varrho^{\mathfrak H/\overline\varGamma_0(2)}_3(\zeta,z)\frac{\zeta^{n}\D \zeta}{i}=0,\quad n\in\{1,2,3\}.\label{eq:rho3Hecke2_int0}
\end{align} For $ n\in\{1,3\}$, the trick used in proving Eq.~\ref{eq:rho2_int0} is  applicable to  Eq.~\ref{eq:rho3Hecke2_int0}. It remains to verify Eq.~\ref{eq:rho3Hecke2_int0} for the scenario $n=2$. This requires some new techniques, as elaborated in the rest of the proof.

As in the verification of  Eq.~\ref{eq:rho2_int0}, now it would suffice to prove\begin{align}
\R\int_{0}^{i\infty} \alpha_2(\zeta)[1-\alpha_2(\zeta)][2E_{2}(2\zeta)-E_2(\zeta)]^{3}\varrho^{\mathfrak H/\overline\varGamma_0(2)}_3(\zeta,z)\frac{\zeta^{2}\D \zeta}{i}=0,\label{eq:G3Hecke2_mid_int0}
\end{align}for all the points $z\in\mathfrak H\cap \partial\mathfrak D_2$ (Fig.~\ref{fig:fundomain}\textit{c}) satisfying \[|\R z|=\frac{1}{2},\I z>\frac{1}{2}\quad \text{or}\quad -\frac{1}{2}<\R z<0,\left\vert z+\frac{1}{2} \right\vert=\frac{1}{2}\quad\text{or}\quad 0<\R z<\frac{1}{2},\left\vert z-\frac{1}{2} \right\vert=\frac{1}{2},\]and choosing the contour of integration along the $ \I\zeta$-axis. To fulfill this task, we shall produce an explicit formula for \begin{align}
\int_{0}^{i\infty} \frac{\alpha_2(\zeta)[1-\alpha_2(\zeta)][2E_{2}(2\zeta)-E_2(\zeta)]^{3}}{\alpha_{2}(\zeta)-\alpha_2(z)}\frac{\zeta^{2}\D \zeta}{i}\label{eq:alpha2_weight6_int0}
\end{align}by manipulating elliptic integrals.

 The following (degree-1) modular transformations for the complete elliptic integrals of the first kind are well known. \begin{enumerate}[label={(\arabic*)}, ref=(\arabic*), widest=a]\item Imaginary modulus transformation:\begin{align}
\mathbf K(\sqrt{\lambda})=\frac{1}{\sqrt{1-\lambda}}\mathbf K\left( \sqrt{\frac{\lambda}{\lambda-1}} \right),\quad \lambda\in\mathbb C\smallsetminus[1,+\infty).\label{eq:im_mod}
\end{align}\item Inverse modulus transformation:\begin{align}
\mathbf K\left( \frac{1}{\sqrt{\lambda}} \right)=\begin{cases}\sqrt{\lambda}[\mathbf K(\sqrt{\lambda})-i\mathbf K(\sqrt{1-\lambda})], & \I\lambda>0\text{ or }\lambda<1;\\[6pt]
\sqrt{\lambda}[\mathbf K(\sqrt{\lambda})+i\mathbf K(\sqrt{1-\lambda})], & \I\lambda<0\text{ or }\lambda>1. \\
\end{cases}\label{eq:inv_mod}
\end{align}
\end{enumerate}
The transformation laws of the complete elliptic integral $ \mathbf K$ (Eqs.~\ref{eq:im_mod}--\ref{eq:inv_mod}) and the $ N=4$ case of Eq.~\ref{eq:z_Pnu_ratios} together imply the following functional equations:\begin{align}\begin{array}{r@{\;=\;}lr@{\;=\;}lr@{\;=\;}l}
\lambda(z+1)&\dfrac{\lambda(z)}{\lambda(z)-1},\quad & \lambda\left(-\dfrac{1}{z}\right)&1-\lambda(z),\quad &\lambda(z+2)&\lambda(z),\\[8pt]\lambda\left( -\dfrac{1}{z-1} \right)&\dfrac{1}{1-\lambda(z)},& \lambda\left( \dfrac{z\mathstrut}{z+1} \right)&\dfrac{1}{\lambda(z)},& \lambda\left( 1-\dfrac{1}{z} \right)&1-\dfrac{1}{\lambda(z)}.
\end{array}\label{eq:lambda_transf}
\end{align}

From a special case of  Eq.~\ref{eq:P_nu_sqr_E2_diff}, we know that \begin{align}
{[}P_{-1/4}(1-2\alpha_2(w))]^2={}&2E_{2}(2w)-E_2(w),&&\I w>0,|\R w|<\frac{1}{2},\left\vert w+\frac{1}{2} \right\vert>\frac{1}{2},\left|w-\frac{1}{2} \right\vert>\frac{1}{2}.\label{eq:P_quarter_sqr_E2_diff}
\end{align}According to Ramanujan's base-4  theory, the Legendre function of degree $-1/4 $ satisfies~\cite[][Chap.~33, Eqs.~9.1 and 9.2]{RN5}: \begin{align}
P_{-1/4}(\cos\theta)={}&\frac{2}{\pi}\frac{1}{\sqrt{\smash[b]{1+\sin(\theta/2)}}}\mathbf K\left( \sqrt{\frac{2\sin(\theta/2)}{1+\sin(\theta/2)}} \right)=\frac{2\sqrt{2}}{\pi}\frac{1}{\sqrt{\smash[b]{1+\cos(\theta/2)}}}\mathbf K\left( \sqrt{\frac{1-\cos(\theta/2)}{1+\cos(\theta/2)}} \right),\quad 0\leq\theta<\pi.\label{eq:P_quarter_sin4}
\end{align}

Now, it is clear  that the integral in Eq.~\ref{eq:G3Hecke2_mid_int0} is equal to \begin{align}\frac1{4\pi}\int_{-1}^1\frac{[P_{-1/4}(\xi)P_{-1/4}(-\xi)]^2}{\xi-1+2\alpha_{2}(z)}\D \xi=\frac{8}{\pi^{5}}\int_0^1\frac{[\mathbf K(\sqrt{t})\mathbf K(\sqrt{1-t})]^{2}}{\alpha_{2}(z)-\left(\frac{t}{2-t}\right)^2}\frac{t\D t}{2-t}=\frac{8}{\pi^{5}}\int_0^1\frac{[\mathbf K(\sqrt{t})\mathbf K(\sqrt{1-t})]^{2}}{\alpha_{2}(z)-\left(\frac{1-t}{1+t}\right)^2}\frac{(1-t)\D t}{1+t}.
\label{eq:G3Hecke2_mid_int0'}\tag{\ref{eq:G3Hecke2_mid_int0}$'$}
\end{align}On one hand,  we have the following identity by residue calculus:\footnote{Hereafter, an integration ``$ \int_{a+i0^+}^{b+i0^+}$'' (resp.~``$ \int_{a-i0^+}^{b-i0^+}$'') for $ a,b\in\mathbb R\cup\{-\infty,+\infty\}$ is carried along a path parallel to the real axis, whose imaginary part is a positive (resp.~negative) infinitesimal.}\begin{align}&
\int_{-\infty+i0^+}^{+\infty+i0^+}\frac{[\mathbf K(\sqrt{t})\mathbf K(\sqrt{1-t})]^{2}}{\alpha_{2}(z)-\left(\frac{1-t}{1+t}\right)^2}\frac{(1-t)\D t}{1+t}\notag\\={}&\frac{\pi i}{2}\left[ \frac{1-2\lambda(2z+1)}{1-\lambda(2z+1)} \right]^2\left[ \mathbf K\left( \sqrt{\frac{1}{1-\lambda(2z+1)}+i0^{+}} \right) \mathbf K\left( \sqrt{\frac{\lambda(2z+1)}{\lambda(2z+1)-1}-i0^{+}} \right)\right]^2\notag\\={}
&\frac{\pi ^{5}i}{16}[P_{-1/4}( 1-2\alpha_{2}(z)+i0^{+})P_{-1/4}( 2\alpha_{2}(z)-1-i0^{+} )]^2.\label{eq:P_quarter_4_int_Re}\end{align}Here in the last step, we have used the identities\begin{subequations}\begin{align}P_{-1/4}\left( 1-\frac{2}{(1-2\lambda)^2} \right)={}&\frac{1}{\pi}\sqrt{\frac{2(2\lambda-1)}{\lambda-1}}\mathbf K\left( \sqrt{\frac{1}{1-\lambda}} \right)=\frac{1}{\pi}\sqrt{\frac{2(2\lambda-1)}{\lambda}}\mathbf K\left( \sqrt{\frac{1}{\lambda}} \right),\quad\text{a.e. }\lambda\in\mathbb C;\label{eq:P_quarter_K_a}\\P_{-1/4}\left( \frac{2}{(1-2\lambda)^2} -1\right)={}&\begin{cases}\dfrac{2\sqrt{1-2\lambda}}{\pi}\mathbf K(\sqrt{\lambda}), & \R\lambda<1/2; \\[8pt]
\dfrac{2\sqrt{2\lambda-1}}{\pi}\mathbf K(\sqrt{1-\lambda}), & \R\lambda>1/2,
\end{cases}\label{eq:P_quarter_K_b}\end{align}which follow from the transformation laws of the complete elliptic integrals (Eqs.~\ref{eq:im_mod}--\ref{eq:inv_mod}) and analytic continuations of Eq.~\ref{eq:P_quarter_sin4}. \end{subequations} On the other hand, by some transformations of the complete elliptic integrals of the first kind, we can establish an identity \begin{align}&
\R\int_{-\infty+i0^+}^{+\infty+i0^+}\frac{[\mathbf K(\sqrt{t})\mathbf K(\sqrt{1-t})]^{2}}{\alpha_{2}(z)-\left(\frac{1-t}{1+t}\right)^2}\frac{(1-t)\D t}{1+t}\notag\\={}&\int_0^1\frac{6[\mathbf K(\sqrt{t})\mathbf K(\sqrt{1-t})]^{2}}{\alpha_{2}(z)-\left(\frac{t}{2-t}\right)^2}\frac{t\D t}{2-t}-\R\left\{\frac{\pi ^{5}i}{32}[P_{-1/4}( 2\alpha_{2}(z)-1-i0^{+} )]^4\right\}\label{eq:P_quarter_4_int_Re'}\tag{\ref{eq:P_quarter_4_int_Re}$'$}
\end{align}for $ |\R z|=1/2,\I z>1/2$. Concretely speaking,
 we have \begin{align}
\int_{1+i0^+}^{+\infty+i0^+}\frac{[\mathbf K(\sqrt{t})\mathbf K(\sqrt{1-t})]^{2}}{\alpha_{2}(z)-\left(\frac{1-t}{1+t}\right)^2}\frac{(1-t)\D t}{1+t}\overset{t=\frac{1}{1-s}}{=\!\!=\!\!=\!\!=\!\!=\!\!=}{}&-\int_{0+i0^+}^{1+i0^+}\frac{\left[\mathbf K\left(\sqrt{\frac{1}{1-s}}\right)\mathbf K\left(\sqrt{\frac{s}{s-1}}\right)\right]^{2}}{\alpha_{2}(z)-\left(\frac{s}{2-s}\right)^2}\frac{s\D s}{(2-s)(1-s)^{2}}\notag\\\underset{\text{Eq.~\ref{eq:inv_mod}}}{\overset{\text{Eq.~\ref{eq:im_mod}}}{=\!\!=\!\!=\!\!=\!\!=\!\!=}}{}&-\int_0^1\frac{[\mathbf K(\sqrt{1-s})+i\mathbf K(\sqrt{s})]^{2}[\mathbf K(\sqrt{s})]^2}{\alpha_{2}(z)-\left(\frac{s}{2-s}\right)^2}\frac{s\D s}{2-s};\label{eq:K4_int_T1}\\\int_{-\infty+i0^+}^{0+i0^+}\frac{[\mathbf K(\sqrt{t})\mathbf K(\sqrt{1-t})]^{2}}{\alpha_{2}(z)-\left(\frac{1-t}{1+t}\right)^2}\frac{(1-t)\D t}{1+t}\overset{t=\frac{s}{s-1}}{=\!\!=\!\!=\!\!=\!\!=\!\!=}{}&\int_{0-i0^+}^{1-i0^+}\frac{\left[\mathbf K\left(\sqrt{\frac{1}{1-s}}\right)\mathbf K\left(\sqrt{\frac{s}{s-1}}\right)\right]^{2}}{\alpha_{2}(z)-\left(\frac{1}{1-2s}\right)^2}\frac{\D s}{(1-2s)(1-s)^{2}}\notag\\\underset{\text{Eq.~\ref{eq:inv_mod}}}{\overset{\text{Eq.~\ref{eq:im_mod}}}{=\!\!=\!\!=\!\!=\!\!=\!\!=}}{}&\int_0^1\frac{[\mathbf K(\sqrt{1-s})-i\mathbf K(\sqrt{s})]^{2}[\mathbf K(\sqrt{s})]^2}{\alpha_{2}(z)-\left(\frac{1}{1-2s}\right)^2}\frac{\D s}{1-2s},\label{eq:K4_int_T2}
\end{align}which add up to\begin{align}&
\R\int_{-\infty+i0^+}^{0+i0^+}\frac{[\mathbf K(\sqrt{t})\mathbf K(\sqrt{1-t})]^{2}}{\alpha_{2}(z)-\left(\frac{1-t}{1+t}\right)^2}\frac{(1-t)\D t}{1+t}+\R\int_{1+i0^+}^{+\infty+i0^+}\frac{[\mathbf K(\sqrt{t})\mathbf K(\sqrt{1-t})]^{2}}{\alpha_{2}(z)-\left(\frac{1-t}{1+t}\right)^2}\frac{(1-t)\D t}{1+t}\notag\\={}&-\int_0^1\frac{[\mathbf K(\sqrt{s})]^4}{\alpha_{2}(z)-\left(\frac{1}{1-2s}\right)^2}\frac{\D s}{1-2s}-\int_0^1\frac{\{[\mathbf K(\sqrt{1-t})]^{2}-[\mathbf K(\sqrt{t})]^{2}\}[\mathbf K(\sqrt{t})]^2}{\alpha_{2}(z)-\left(\frac{t}{2-t}\right)^2}\frac{t\D t}{2-t}\notag\\={}&-\frac{1}{2}\int_0^1\frac{[\mathbf K(\sqrt{s})]^4-[\mathbf K(\sqrt{1-s})]^4}{\alpha_{2}(z)-\left(\frac{1}{1-2s}\right)^2}\frac{\D s}{1-2s}-\int_0^1\frac{\{[\mathbf K(\sqrt{1-t})]^{2}-[\mathbf K(\sqrt{t})]^{2}\}[\mathbf K(\sqrt{t})]^2}{\alpha_{2}(z)-\left(\frac{t}{2-t}\right)^2}\frac{t\D t}{2-t}.\label{eq:K4_over_1_2s_int}
\end{align}We can treat the remaining integral over $s$ in the last line of  Eq.~\ref{eq:K4_over_1_2s_int} in two ways. By residue calculus, we have \begin{align}&
\int_{-\infty+i0^+}^{+\infty+i0^+}\frac{[\mathbf K(\sqrt{s})]^4-[\mathbf K(\sqrt{1-s})]^4}{\alpha_{2}(z)-\left(\frac{1}{1-2s}\right)^2}\frac{\D s}{1-2s}\notag\\={}&\frac{\pi i}{2}\left[ 1-2\lambda\left( -\frac{1}{z}+1 \right) \right]^2\left\{ \left[\mathbf K\left( \sqrt{1-\lambda\left( -\frac{1}{z}+1 \right)} \right) \right]^4-\left[\mathbf K\left( \sqrt{\lambda\left( -\frac{1}{z}+1 \right)} \right) \right]^4\right\}\notag\\={}&\R\left\{\frac{\pi ^{5}i}{16}[P_{-1/4}( 2\alpha_{2}(z)-1-i0^{+} )]^4\right\}
\end{align}for $ |\R z|=1/2,\I z>1/2$. Meanwhile,
transformations in the spirit of Eqs.~\ref{eq:K4_int_T1} and \ref{eq:K4_int_T2} would bring us\begin{align}&
\left(
\int_{-\infty+i0^+}^{0+i0^+}+\int_{1+i0^+}^{+\infty+i0^+}\right)\frac{[\mathbf K(\sqrt{s})]^4-[\mathbf K(\sqrt{1-s})]^4}{\alpha_{2}(z)-\left(\frac{1}{1-2s}\right)^2}\frac{\D s}{1-2s}\notag\\={}&\int_0^1\frac{\{2[\mathbf K(\sqrt{t})]^{4}-[\mathbf K(\sqrt{1-t})-i\mathbf K(\sqrt{t})]^{4}-[\mathbf K(\sqrt{1-t})+i\mathbf K(\sqrt{t})]^{4}\}}{\alpha_{2}(z)-\left(\frac{1-t}{1+t}\right)^2}\frac{(1-t)\D t}{1+t}\notag\\={}&\int_0^1\frac{\{12[\mathbf K(\sqrt{1-t})]^{2}-2[\mathbf K(\sqrt{t})]^{2}\}[\mathbf K(\sqrt{t})]^2}{\alpha_{2}(z)-\left(\frac{t}{2-t}\right)^2}\frac{t\D t}{2-t}.\label{eq:K4_diff_comb_int}
\end{align}Thus, we see that the identity claimed in Eq.~\ref{eq:P_quarter_4_int_Re'} is a result of   Eqs.~\ref{eq:K4_over_1_2s_int}--\ref{eq:K4_diff_comb_int}.

To wrap up, we combine Eqs.~\ref{eq:P_quarter_4_int_Re} and \ref{eq:P_quarter_4_int_Re'} into \begin{align}
\int_{0}^{i\infty} \frac{\alpha_2(\zeta)[1-\alpha_2(\zeta)][2E_{2}(2\zeta)-E_2(\zeta)]^{3}}{\alpha_{2}(\zeta)-\alpha_2(z)}\frac{\zeta^{2}\D \zeta}{i}=-\frac{ \I z+4 (\I z)^3}{12}[2E_2(2z)-E_2(z)]^2\label{eq:alpha2_weight6_int0'}\tag{\ref{eq:alpha2_weight6_int0}$'$}
\end{align}where the integration is carried out along the $ \I\zeta$-axis, and  $ |\R z|=1/2,\I z>1/2$.
By differentiation, we have justified Eq.~\ref{eq:G3Hecke2_mid_int0} for  $ |\R z|=1/2,\I z>1/2$. The rest of the verification for  Eq.~\ref{eq:G3Hecke2_mid_int0}  follows readily from the reflection formula  $ \varrho^{\mathfrak H/\overline\varGamma_0(2)}_3(\zeta,z)=-\varrho^{\mathfrak H/\overline\varGamma_0(2)}_3(-1/(2\zeta),-1/(2z))$.   \end{proof}\begin{remark}\label{rmk:spec_high_w_auto_G_int}Due to the subtle restriction in Eq.~\ref{eq:rho3Hecke2_int0}, the integral representation for $ G_3^{\mathfrak H/\overline{\varGamma}_0(2)}(z,z')$ does not immediately generalize to weight-6 automorphic Green's functions with levels $ N>2$, or to automorphic Green's functions with weights higher than 6. Nevertheless, when one of the arguments in the automorphic Green's functions is a special CM point, say $ \alpha_N(z)=\infty$, it is possible to establish certain  integral representations for automorphic Green's functions by going through criteria~\ref{itm:AGF1'}--\ref{itm:AGF3'} with respect to the other variable $ z'$. Some examples of weights 6 and 10 are given below:{\allowdisplaybreaks\begin{align}&
G_3^{\mathfrak H/PSL(2,\mathbb Z)}\left(\frac{1+i\sqrt{3}}{2},z'\right)\notag\\={}&\notag\frac{4\pi^2}{9\left[ \I \frac{iP_{-1/6}(-\sqrt{(j(z')-1728)/j(z')})}{P_{-1/6}(\sqrt{(j(z')-1728)/j(z')})} \right]^2}\R\int^1_{\sqrt{(j(z')-1728)/j(z')}}\xi[P_{-1/6}(\xi)]^4\times\notag\\&\times\left[\tfrac{iP_{-1/6}(-\xi)}{P_{-1/6}(\xi)}-\vphantom{\overline{\tfrac{\sqrt{j}}{}}}\tfrac{iP_{-1/6}(-\sqrt{(j(z')-1728)/j(z')})}{P_{-1/6}(\sqrt{(j(z')-1728)/j(z')})}\right]^2\left[\tfrac{iP_{-1/6}(-\xi)}{P_{-1/6}(\xi)}-\overline{\left(\tfrac{iP_{-1/6}(-\sqrt{(j(z')-1728)/j(z')})}{P_{-1/6}(\sqrt{(j(z')-1728)/j(z')})}\right)}\right]^2\D \xi\notag\\&+\frac{4\pi^2}{9\left[ \I \frac{iP_{-1/6}(-\sqrt{(j(z')-1728)/j(z')})}{P_{-1/6}(\sqrt{(j(z')-1728)/j(z')})} \right]^2}\int_{-1}^{1}\xi[P_{-1/6}(\xi)]^4\D \xi,\quad  \text{a.e. }z'\in\mathfrak H;\label{eq:KZ_G3_ell3_star}\\
&\frac{G_3^{\mathfrak H/\overline{\varGamma}_0(2)}\left(\frac{i-1}{2},z'\right)}{2}=G_3^{\mathfrak H/\overline{\varGamma}(2)}(i,2z'+1)\notag\\={}&\frac{\pi^{2}}{32\left[\I\frac{iP_{-1/4}(2\alpha_{2}(z')-1)}{P_{-1/4}(1-2\alpha_{2}(z'))}\right]^2}\R\int^1_{1-2\alpha_{2}(z')}\xi[P_{-1/4}(\xi)]^4\times \notag\\&\times\left[\frac{iP_{-1/4}(-\xi)}{P_{-1/4}(\xi)}-\vphantom{\overline{\tfrac{\sqrt{j}}{}}}\frac{iP_{-1/4}(2\alpha_{2}(z')-1)}{P_{-1/4}(1-2\alpha_{2}(z'))}\right]^2\left[\frac{iP_{-1/4}(-\xi)}{P_{-1/4}(\xi)}-\overline{\left(\frac{iP_{-1/4}(2\alpha_{2}(z')-1)}{P_{-1/4}(1-2\alpha_{2}(z'))}\right)}\right]^2\D \xi\notag\\{}&+\frac{\pi^{2}}{32\left[\I\frac{iP_{-1/4}(2\alpha_{2}(z')-1)}{P_{-1/4}(1-2\alpha_{2}(z'))}\right]^2}\int_{-1}^1\xi[P_{-1/4}(\xi)]^4\D \xi, \quad\text{a.e. }z'\in\mathfrak H;\label{eq:G3_Gamma2}\\&G_3^{\mathfrak H/\overline{\varGamma}_{0}(3)}\left( \frac{3+i\sqrt{3}}{6},z'\right)\notag\\={}&\frac{\pi^{2}}{54\left[\I\frac{iP_{-1/3}(2\alpha_{3}(z')-1)}{P_{-1/3}(1-2\alpha_{3}(z'))}\right]^2}\R\int^1_{1-2\alpha_{3}(z')}\xi[P_{-1/3}(\xi)]^4\times \notag\\&\times\left[\frac{iP_{-1/3}(-\xi)}{P_{-1/3}(\xi)}-\vphantom{\overline{\tfrac{\sqrt{j}}{}}}\frac{iP_{-1/3}(2\alpha_{3}(z')-1)}{P_{-1/3}(1-2\alpha_{3}(z'))}\right]^2\left[\frac{iP_{-1/3}(-\xi)}{P_{-1/3}(\xi)}-\overline{\left(\frac{iP_{-1/3}(2\alpha_{3}(z')-1)}{P_{-1/3}(1-2\alpha_{3}(z'))}\right)}\right]^2\D \xi\notag\\&+\frac{ \pi ^2}{54\left[\I\frac{iP_{-1/3}(2\alpha_{3}(z')-1)}{P_{-1/3}(1-2\alpha_{3}(z'))}\right]^2}\int_{-1}^1\xi[P_{-1/3}(\xi)]^4\D \xi,\quad\text{a.e. }z'\in\mathfrak H;\label{eq:G3Hecke3_spec}\\&\frac{G_5^{\mathfrak H/\overline{\varGamma}_0(2)}\left(\frac{i-1}{2},z'\right)}{2}=G_5^{\mathfrak H/\overline{\varGamma}(2)}(i,2z'+1)\notag\\={}&-\frac{\pi^{4}}{6144\left[\I\frac{iP_{-1/4}(2\alpha_{2}(z')-1)}{P_{-1/4}(1-2\alpha_{2}(z'))}\right]^4}\int_{1-2\alpha_{2}(z')}^1\xi(5-6\xi^2)[P_{-1/4}(\xi)]^8\times\notag\\&\times\left[\frac{iP_{-1/4}(-\xi)}{P_{-1/4}(\xi)}-\vphantom{\overline{\tfrac{\sqrt{j}}{}}}\frac{iP_{-1/4}(2\alpha_{2}(z')-1)}{P_{-1/4}(1-2\alpha_{2}(z'))}\right]^4\left[\frac{iP_{-1/4}(-\xi)}{P_{-1/4}(\xi)}-\overline{\left(\frac{iP_{-1/4}(2\alpha_{2}(z')-1)}{P_{-1/4}(1-2\alpha_{2}(z'))}\right)}\right]^4\D \xi\notag\\&-\frac{\pi^{4}}{6144\left[\I\frac{iP_{-1/4}(2\alpha_{2}(z')-1)}{P_{-1/4}(1-2\alpha_{2}(z'))}\right]^4}\int_{-1}^1\xi(5-6\xi^2)[P_{-1/4}(\xi)]^8\D \xi,\quad\text{a.e. }z'\in\mathfrak H.\label{eq:G5Gamma2_spec}
\end{align}}Here, Eq.~\ref{eq:G3_Gamma2} reflects the fact that\begin{align}\lim_{z\to(i-1)/2}
\varrho^{\mathfrak H/\overline\varGamma_0(2)}_3(\zeta,z)=\frac{\pi}{8}[1-2\alpha_2(\zeta)].
\end{align}The validity of Eq.~\ref{eq:G5Gamma2_spec} hinges on a non-trivial vanishing identity \begin{align}
\int _{-1}^1\xi(5-6\xi^2)[P_{-1/4}(-\xi)]^2[P_{-1/4}(\xi)]^6\D \xi=0,
\end{align} which can be proved in a similar manner as Eq.~\ref{eq:G3Hecke2_mid_int0}.\eor\end{remark}\begin{proposition}[Kontsevich--Zagier Integral Representations for   $ G_{k/2}^{\mathfrak H/PSL(2,\mathbb Z)}(z,z')$ where $k\in\{6,8,10,14\}$]\label{prop:G3457_PSL2Z}With the notations for the Eisenstein series as given in Eqs.~\ref{eq:E4_E6_defn} and \ref{eq:E8E10E14_defn}, we have the following integral representations for weights $k\in\{6,8,10,14\}$  and  $ j(z)\neq j(z')$:\begin{align}
G_{k/2}^{\mathfrak H/PSL(2,\mathbb Z)}(z,z')={}&\frac{1728\pi^{2}}{(y')^{(k-2)/2}}\R\int_{z'}^{i\infty} \frac{E_{k}(\zeta)}{j(\zeta)}\rho^{\mathfrak H/PSL(2,\mathbb Z)}_{k/2}(\zeta,z)\frac{(\zeta-z')^{(k-2)/2}(\zeta-\overline{z'})^{(k-2)/2}\D \zeta}{i}\notag\\&-\frac{1728\pi^{2}}{(y')^{(k-2)/2}}\R\int_{0}^{i\infty} \frac{E_{k}(\zeta)}{j(\zeta)}\rho^{\mathfrak H/PSL(2,\mathbb Z)}_{k/2}(\zeta,z)\frac{\zeta^{k-2}\D \zeta}{i}\label{eq:GkPSL2Z}
\end{align} where\begin{align}\label{eq:rho_k_PSL2Z}
\rho^{\mathfrak H/PSL(2,\mathbb Z)}_{k/2}(\zeta,z)={}&\frac{(-1)^{k/2}}{2^{(k-4)/2}}\frac{1}{\left( \frac{k-2}{2} \right)!}\frac{ y^{(k-2)/2}}{864 \pi}\left( \frac{\partial}{\partial y}\frac{1}{y} \right)^{(k-2)/2}\left[  \frac{j(\zeta)j(z)}{j(\zeta)-j(z)}\frac{{E_{6}(z)}}{E_{4}(z)E_{k}(z)}\right].
\end{align}Here in Eq.~\ref{eq:rho_k_PSL2Z},  it is understood that $ j(z)E_6(z)/[E_4(z)E_k(z)]=[E_4(z)]^2 E_6(z)/[\Delta(z)E_k(z)]$ defines a smooth function for all $ z\in\mathfrak H$, so that  $ \rho^{\mathfrak H/PSL(2,\mathbb Z)}_{k/2}(\zeta,z)$ is well-behaved when $ j(\zeta)\neq j(z)$.\end{proposition}\begin{proof}Contrary to the practice in Proposition~\ref{prop:G2PSL2Z}, we shall go over criteria~\ref{itm:AGF1'}--\ref{itm:AGF3'} with respect to the variable $z'$. The major advantage of this approach is the simplification of the proofs for~criteria~\ref{itm:AGF2'} and \ref{itm:AGF3'}. Especially, it is  easy to check the compatibility with the cusp behavior: $ \lim_{z'\to i\infty}G_{k/2}^{\mathfrak H/PSL(2,\mathbb Z)}(z,z')=0$.

The remaining challenge resides in the symmetry criterion~\ref{itm:AGF1'}. In particular, we need to expend some effort to justify that the proposed integral representations in Eq.~\ref{eq:GkPSL2Z} are invariant under the inversion $ z'\mapsto-1/z'$. This amounts to the verification of the following vanishing identities (see Eq.~\ref{eq:rho3Hecke2_int0}):\begin{align}
\R\int_{0}^{i\infty} \frac{E_{k}(\zeta)}{j(\zeta)}\rho^{\mathfrak H/PSL(2,\mathbb Z)}_{k/2}(\zeta,z)\frac{\zeta^{n}\D \zeta}{i}=0,\quad n\in[2,k-4]\cap(2\mathbb Z),|\R z|=\frac{1}{2},\I z>\frac{\sqrt{3}}{2},\label{eq:rho_k_PSL_int0}
\end{align} where the paths of integration are along
the $ \I\zeta$-axis.
Once  Eq.~\ref{eq:rho_k_PSL_int0} is confirmed, the vanishing identity remains valid on (see Eq.~\ref{eq:fun_domain_SL2Z})\begin{align}\mathfrak H\cap\partial(\mathfrak D\cup\hat S\mathfrak D)=\left\{ z\in\mathfrak H\left| |\R z|=\frac{1}{2},\I z>\frac{\sqrt{3}}{2} \right. \right\}\cup \left\{ z\in\mathfrak H\left| |\R z|\leq\frac{1}{2},|z+1|=1,|z-1|=1 \right. \right\}\label{eq:fun_domain_SL2Z_S_bd}\end{align}
by the inversion $ \hat S:z\mapsto-1/z$.
Consequently,   Eq.~\ref{eq:rho_k_PSL_int0} is applicable to the whole fundamental domain $ \mathfrak D$ (by a homogeneous Dirichlet boundary value problem) and to the entire upper half-plane $ \mathfrak H$ (by tessellation under the actions of $ SL(2,\mathbb Z)$).

As in Proposition~\ref{prop:G3Hecke2}, we can prove Eq.~\ref{eq:rho_k_PSL_int0} by supplying explicit integral formulae in the forms of\begin{align}
\int_{0}^{i\infty} \frac{E_{k}(\zeta)j(z)}{j(\zeta)-j(z)}\frac{{E_{6}(z)}}{E_{4}(z)E_{k}(z)}\frac{\zeta^{n}\D \zeta}{ i}=p_{k,n} (\I z),\quad n\in[2,k-4]\cap(2\mathbb Z),|\R z|=\frac{1}{2},\I z>\frac{\sqrt{3}}{2},\label{eq:p_k_n_defn}
\end{align} where the polynomials $p_{k,n} $ are given below:\begin{align}
\left\{\begin{array}{r@{\;=\;}l}
p_{6,2} (y)& 0, \\[2pt]
p_{8,2}(y)=-p_{8,4}(y) & \dfrac{9y+40 y^3+16 y^5}{80} ,\\[6pt]
p_{10,2}(y)=-p_{10,6}(y) & \dfrac{27 y+84 y^3-112 y^5-64 y^7}{448} ,\\[6pt] p_{10,4}(y) &0,\\[6pt]
p_{14,2}(y)=-p_{14,10}(y) & \dfrac{261 y-5500 y^3-32736 y^5-29568 y^7-14080 y^9-3072 y^{11}}{33792}, \\[6pt]p_{14,4}(y)=-p_{14,8}(y)&\dfrac{9 y + 112 y^3 + 480 y^5 + 768 y^7 + 256 y^9}{2304},\\[6pt] p_{14,6}(y)&0.
\end{array}\right.\label{eq:p_k_n_list}
\end{align}It is clear that Eqs.~\ref{eq:p_k_n_defn} and \ref{eq:p_k_n_list} lead to a verification of Eq.~\ref{eq:rho_k_PSL_int0}.

We now illustrate  Eq.~\ref{eq:p_k_n_list} with a detailed computation for $ p_{8,2}(y)$.

We start by pointing out that the integral representation for $ p_{8,2} y$ can be rewritten using elliptic integrals:\begin{align}
\int_{0}^{i\infty} \frac{E_{8}(\zeta)j(z)}{j(\zeta)-j(z)}\frac{{E_{6}(z)}}{E_{4}(z)E_{8}(z)}\frac{\zeta^{2}\D \zeta}{ i}=-\frac{64}{\pi^{7}}\frac{j(z)E_{6}(z)}{E_{4}(z)E_{8}(z)}\int_0^1\frac{[\mathbf K(\sqrt{t})]^4[\mathbf K(\sqrt{1-t})]^2}{\frac{256(1-t+t^2)^3}{t^{2}(1-t)^2}-j(z)}\frac{(1-t+t^2)^2\D t}{t(1-t)}\label{eq:p_8_2_ell_int_form}
\end{align}for $ |\R z|=1/2,\I z>{\sqrt{3}}/{2}$. To show Eq.~\ref{eq:p_8_2_ell_int_form}, we recall the expression of  the $ j$-invariant as a rational function of  the modular lambda function:\begin{align}&&
j(z)={}&\frac{256\{1-\lambda(z)+[\lambda(z)]^2\}^3}{[\lambda(z)]^2[1-\lambda(z)]^2},&&\forall z\in\mathfrak H,&&
\end{align} along with  Ramanujan's work~\cite{Ramanujan1916} on the relation between the Eisenstein series and the complete elliptic integrals of the first kind $ \mathbf K(\sqrt{\lambda(z)})$: \begin{align}&&E_4(z)={}&\left[ \frac{2\mathbf K(\sqrt{\lambda(z)})}{\pi} \right]^{4}\{1-\lambda(z)+[\lambda(z)]^{2}\},&&\I z>0,|\R z|<1,\left|z+\frac12\right|>\frac12,\left|z-\frac12\right|>\frac{1}{2},&&\label{eq:E4_Ell_Ramanujan}\\&&E_6(z)={}&\left[ \frac{2\mathbf K(\sqrt{\lambda(z)})}{\pi} \right]^{6}\frac{[\lambda(z)+1][\lambda(z)-2][2\lambda(z)-1]}{2},&&\I z>0,|\R z|<1,\left|z+\frac12\right|>\frac12,\left|z-\frac12\right|>\frac{1}{2}.&&\label{eq:E6_Ell_Ramanujan}\end{align}Furthermore, the variable transformation from $ \zeta$ to $ t=\lambda(\zeta)$ is mediated by the identities:\begin{align}z={}&\frac{i\mathbf K(\sqrt{1-\lambda(z)})}{\mathbf K(\sqrt{\lambda(z)})},&& \I z>0,|\R z|<1,\left|z+\frac12\right|>\frac12,\left|z-\frac12\right|>\frac{1}{2},\label{eq:lambda_K_ratio_z}\\
\frac{\D}{\D t}\frac{\mathbf K(\sqrt{1-t})}{\mathbf K(\sqrt{t})}={}&-\frac{\pi}{4t(1-t)[\mathbf K(\sqrt{t})]^2},&&t\in\mathbb C\smallsetminus((-\infty,0]\cup[1,+\infty)),
\label{eq:ratio_deriv}
\end{align}which are special cases of Eqs.~\ref{eq:z_Pnu_ratios} and \ref{eq:P_nu_ratio_deriv}, respectively.

Then, we consider the following integral related to the right-hand side of Eq.~\ref{eq:p_8_2_ell_int_form}:\begin{align}
-\frac{64}{\pi^{7}}\frac{j(z)E_{6}(z)}{E_{4}(z)E_{8}(z)}\int_{-\infty+i0^+}^{+\infty+i0^+}\frac{[\mathbf K(\sqrt{t})]^4[\mathbf K(\sqrt{1-t})]^2}{\frac{256(1-t+t^2)^3}{t^{2}(1-t)^2}-j(z)}\frac{(1-t+t^2)^2\D t}{t(1-t)}
\end{align}and compute it in two ways. In the first approach, we can close the contour  in the upper half $t$-plane and collect residues at the three simple poles therein. Without loss of generality, we may assume that $ \R z=1/2$, and these three poles are\begin{align}t=\lambda(z),\quad t=\lambda\left( \frac{z-1}{z} \right)=1-\frac{1}{\lambda(z)},\quad t=\lambda\left( -\frac{1}{z-1} \right)=\frac{1}{1-\lambda(z)}.\label{eq:3poles_t}\end{align}By the differentiation formula for the $ j$-invariant (Eq.~\ref{eq:j'z_E4E6}), the transformation laws of the Eisenstein series (Eq.~\ref{eq:E2E4E6_mod_transf}) and residue calculus, we have\begin{align}&
-\frac{64}{\pi^{7}}\frac{j(z)E_{6}(z)}{E_{4}(z)E_{8}(z)}\int_{-\infty+i0^+}^{+\infty+i0^+}\frac{[\mathbf K(\sqrt{t})]^4[\mathbf K(\sqrt{1-t})]^2}{\frac{256(1-t+t^2)^3}{t^{2}(1-t)^2}-j(z)}\frac{(1-t+t^2)^2\D t}{t(1-t)}\notag\\={}&\frac{1}{i}\left[ z^2+\left(\frac{z-1}{z}\right)^2 z^6+\left(-\frac{1}{z-1}\right)^2 (z-1)^6 \right]\notag\\={}&\frac{9 y + 40 y^3 + 16 y^5}{16}+\frac{21 - 156 y^2 + 48 y^4 - 64 y^6}{64i},\quad\text{for }\I z=\frac{z-\frac{1}{2}}{i}=y.\label{eq:p_8_2_comp1}
\end{align}In the second approach, we employ the transformation laws of $ \mathbf K$ (Eqs.~\ref{eq:im_mod} and \ref{eq:inv_mod}) to deduce  (see Eqs.~\ref{eq:K4_int_T1}, \ref{eq:K4_int_T2} and \ref{eq:K4_diff_comb_int})\begin{align}
&-\frac{64}{\pi^{7}}\frac{j(z)E_{6}(z)}{E_{4}(z)E_{8}(z)}\int_{-\infty+i0^+}^{+\infty+i0^+}\frac{[\mathbf K(\sqrt{t})]^4[\mathbf K(\sqrt{1-t})]^2}{\frac{256(1-t+t^2)^3}{t^{2}(1-t)^2}-j(z)}\frac{(1-t+t^2)^2\D t}{t(1-t)}\notag\\{}&+\frac{64}{\pi^{7}}\frac{j(z)E_{6}(z)}{E_{4}(z)E_{8}(z)}\int_0^1\frac{[\mathbf K(\sqrt{t})]^4[\mathbf K(\sqrt{1-t})]^2}{\frac{256(1-t+t^2)^3}{t^{2}(1-t)^2}-j(z)}\frac{(1-t+t^2)^2\D t}{t(1-t)}\notag\\={}&\frac{64}{\pi^{7}}\frac{j(z)E_{6}(z)}{E_{4}(z)E_{8}(z)}\int_{0}^{1}\frac{[\mathbf K(\sqrt{t})]^4[\mathbf K(\sqrt{1-t})-i\mathbf K(\sqrt{t})]^2+[\mathbf K(\sqrt{1-t})-i\mathbf K(\sqrt{t})]^4[\mathbf K(\sqrt{t})]^2}{\frac{256(1-t+t^2)^3}{t^{2}(1-t)^2}-j(z)}\frac{(1-t+t^2)^2\D t}{t(1-t)}\notag\\={}&\frac{64}{\pi^{7}}\frac{j(z)E_{6}(z)}{E_{4}(z)E_{8}(z)}\int_0^1\frac{4i[\mathbf K(\sqrt{1-t})\mathbf K(\sqrt{t})]^3-6i\mathbf K(\sqrt{1-t})[\mathbf K(\sqrt{t})]^{5}-4[\mathbf K(\sqrt{t})]^4[\mathbf K(\sqrt{1-t})]^2}{\frac{256(1-t+t^2)^3}{t^{2}(1-t)^2}-j(z)}\frac{(1-t+t^2)^2\D t}{t(1-t)}.\label{eq:p_8_2_comp2}
\end{align}Comparing Eqs.~\ref{eq:p_8_2_comp1} and \ref{eq:p_8_2_comp2}, we have completed the evaluation of $ p_{8,2}(y)$.

Except that one needs to solve  $ p_{14,2}(y)$ and $ p_{14,4}(y)$ from two simultaneous equations, the derivations of all the other polynomials in   Eq.~\ref{eq:p_k_n_list} will follow from similar procedures as the computation for  $ p_{8,2}(y)$.
\end{proof}

At this point, we have completed the verification of all the integral representations of automorphic Green's functions proposed in Theorem~\ref{thm:KZ_int_repns}. As direct applications of the results in Propositions~\ref{prop:G3Hecke2} and \ref{prop:G3457_PSL2Z}, we list some integral representations of special $ L$-values in the next proposition. \begin{proposition}[Some Definite Integrals over Products of Legendre Functions]\label{prop:high_w_L_values}We have the following integral formulae for special $ L$-values:{\allowdisplaybreaks\begin{align}
\zeta(3)={}&-\frac{\pi^{2}}{42}\lim_{y'\to+\infty}G_{3}^{\mathfrak H/PSL(2,\mathbb Z)}\left( \frac{1+i\sqrt{3}}{2} ,z'\right)(y')^{2}=-\frac{2\pi^4}{189}\int_{-1}^1\xi[P_{-1/6}(\xi)]^4\D \xi\notag\\={}&-\frac{2\pi^{2}}{21}\lim_{y'\to+\infty}G_3^{\mathfrak H/\overline{\varGamma}(2)}(i,z')(y')^2=-\frac{\pi^{4}}{168}\int_{-1}^1 \xi[P_{-1/4}(\xi)]^4\D \xi\notag\\={}&-\frac{2\pi^{2}}{3}\lim_{y'\to+\infty}G_{3}^{\mathfrak H/\overline{\varGamma}_0(3)}\left( \frac{3+i\sqrt{3}}{6},z'\right)(y')^{2}=-\frac{\pi^4}{243}\int_{-1}^1\xi[P_{-1/3}(\xi)]^4\D \xi\notag\\={}&-\frac{4\pi^{2}}{189}\lim_{y'\to+\infty}G_3^{\mathfrak H/PSL(2,\mathbb Z)}(i,z')(y')^2=\frac{4\pi^{4}}{189}\R\int_{-1}^1\left( 1-\frac{5\xi^2}{18} \right)\frac{[P_{-1/6}(\xi)]^{4}}{\xi^{3}}\D \xi\notag\\={}&-\frac{8\pi^{2}}{567}\lim_{y'\to+\infty}G_3^{\mathfrak H/PSL(2,\mathbb Z)}(i\sqrt{2},z')(y')^{2}=-\frac{2\pi^{2}}{63}\lim_{y'\to+\infty}G_3^{\mathfrak H/\overline\varGamma_{\vartheta}}(1+i\sqrt{2},z')(y')^{2}\notag\\={}&-\frac{2\pi^{2}}{63}\lim_{y'\to+\infty}G_3^{\mathfrak H/\overline\varGamma_{0}(2)}\left( \frac{i}{\sqrt{2}},\frac{z'-1}{2} \right)(y')^{2}=\frac{2\pi^{4}}{63}\R\int_{-1}^1\frac{(8-3\xi^2)[P_{-1/4}(\xi)]^{4}}{16\xi^{3}}\D \xi;\label{eq:zeta3_P_sixth_P_third}\\L(4,\chi_{-4})={}&-\frac{\pi^{3}}{120}\lim_{y'\to+\infty}G_4^{\mathfrak H/PSL(2,\mathbb Z)}(i,z')(y')^3=-\frac{\pi^6}{120}\R\int_{-1}^1\left( 1-\frac{19\xi^2}{27} \right)\frac{[P_{-1/6}(\xi)]^{6}}{2\xi^{4}}\D \xi;\\L(4,\chi_{-3})={}&-\frac{4\pi^{3}}{405}\lim_{y'\to+\infty}G_{4}^{\mathfrak H/PSL(2,\mathbb Z)}\left( \frac{1+i\sqrt{3}}{2},z'\right)(y')^{3}\notag\\={}&-\frac{\pi^6}{10935}\int_{-1}^1(7-16\xi^2)[P_{-1/6}(\xi)]^{6}\D \xi;\\\zeta(5)={}&-\frac{2\pi^{4}}{385}\lim_{y'\to+\infty}G_{5}^{\mathfrak H/PSL(2,\mathbb Z)}\left( \frac{1+i\sqrt{3}}{2},z'\right)(y')^{4}=\frac{\pi^{8}}{249480}\int_{-1}^1(51-64\xi^2)\xi[P_{-1/6}(\xi)]^{8}\D \xi\notag\\={}&-\frac{8\pi^{4}}{525}\lim_{y'\to+\infty}G_5^{\mathfrak H/\overline\varGamma(2)}(i,z')(y')^4=\frac{\pi^{8}}{100800}\int_{-1}^1\xi(5-6\xi^2)[P_{-1/4}(\xi)]^8\D \xi\notag\\={}&-\frac{64\pi^{4}}{17325}\lim_{y'\to+\infty}G_5^{\mathfrak H/PSL(2,\mathbb Z)}(i,z')(y')^4\notag\\={}&\frac{64\pi^{8}}{17325}\R\int_{-1}^1\left( 1-\frac{61\xi^2}{54}\, \, +\frac{5\xi^4}{24} \right)\frac{[P_{-1/6}(\xi)]^{8}}{4\xi^{5}}\D \xi;\label{eq:zeta5_int_P_sixth}\\\zeta(7)={}&-\frac{8\pi^{6}}{8085}\lim_{y'\to+\infty}G_{7}^{\mathfrak H/PSL(2,\mathbb Z)}\left( \frac{1+i\sqrt{3}}{2},z'\right)(y')^{6}\notag\\={}&-\frac{\pi^{12}}{58939650}\int_{-1}^1(168-485\xi^2+320\xi^4)\xi[P_{-1/6}(\xi)]^{12}\D \xi\notag\\={}&-\frac{1024\pi^{6}}{1902285}\lim_{y'\to+\infty}G_7^{\mathfrak H/PSL(2,\mathbb Z)}(i,z')(y')^6\notag\\={}&\frac{1024\pi^{12}}{1902285}\R\int_{-1}^1\left( 1-\frac{107 \xi^2}{54}+\frac{33871 \xi^4}{29160}-\frac{2077 \xi^6}{11664} \right)\frac{[P_{-1/6}(\xi)]^{12}}{16\xi^{7}}\D \xi.\label{eq:zeta7_int_P_sixth}
\end{align}}\end{proposition}\begin{proof}These are routine computations based on the asymptotic behavior  (Eq.~\ref{eq:Gs_asympt}), the addition formulae (Proposition~\ref{prop:add_form_auto_Green}), and the integral representations (Proposition~\ref{prop:G3Hecke2}, Remark~\ref{rmk:spec_high_w_auto_G_int}, Proposition~\ref{prop:G3457_PSL2Z}) for higher weight automorphic Green's functions.

For example, via explicit computations of  $ \rho^{\mathfrak H/PSL(2,\mathbb Z)}_{k/2}(\zeta,e^{i\pi/3})$ (Eq.~\ref{eq:rho_k_PSL2Z}) for $ k\in\{6,10,14\}$, one may spell out {\allowdisplaybreaks\begin{align}&G_3^{\mathfrak H/PSL(2,\mathbb Z)}(e^{\pi i/3},z)-\frac{4\pi^2}{9 y^2}\int_{-1}^{1}\xi[P_{-1/6}(\xi)]^4\D \xi\notag\\={}&\frac{4\pi^2}{9 y^2}\R\int^1_{\sqrt{\frac{j(z)-1728}{j(z)}}}\xi[P_{-1/6}(\xi)]^4\left[\frac{iP_{-1/6}(-\xi)}{P_{-1/6}(\xi)}- z\right]^{2}\left[\frac{iP_{-1/6}(-\xi)}{P_{-1/6}(\xi)}- \overline{z}\right]^{2}\D \xi,\label{eq:GT3_KZ}\\
&G_{5}^{\mathfrak H/PSL(2,\mathbb Z)}(e^{\pi i/3},z)+\frac{\pi^{4}}{1296 y^{4}}\int_{-1}^1(51-64\xi^2)\xi[P_{-1/6}(\xi)]^{8}\D \xi\notag\\={}&-\frac{\pi^{4}}{1296 y^{4}}\R\int^1_{\sqrt{\frac{j(z)-1728}{j(z)}}}(51-64\xi^{2})\xi[P_{-1/6}(\xi)]^{8}\left[\frac{iP_{-1/6}(-\xi)}{P_{-1/6}(\xi)}- z\right]^{4}\left[\frac{iP_{-1/6}(-\xi)}{P_{-1/6}(\xi)}- \overline{z}\right]^{4}\D \xi;\label{eq:GT5_KZ}\\{}&G_{7}^{\mathfrak H/PSL(2,\mathbb Z)}(e^{\pi i/3},z)-\frac{\pi^{6}}{58320 y^{6}}\int_{-1}^1(168-485\xi^2+320\xi^4)\xi[P_{-1/6}(\xi)]^{12}\D \xi\notag\\={}&\frac{\pi^{6}}{58320 y^{6}}\R\int^1_{\sqrt{\frac{j(z)-1728}{j(z)}}}(168-485\xi^2+320\xi^4)\xi[P_{-1/6}(\xi)]^{12}\times\notag\\{}&\times\left[\frac{iP_{-1/6}(-\xi)}{P_{-1/6}(\xi)}- z\right]^{6}\left[\frac{iP_{-1/6}(-\xi)}{P_{-1/6}(\xi)}- \overline{z}\right]^{6}\D \xi,\label{eq:GT7_KZ}
\end{align}}for $ |z|\geq1,-\frac12<\R z\leq\frac{1}{2},z\neq \frac{1}{2}+i\frac{\sqrt{3}}{2}$, where all the integrals are taken over straight line segments.
Asymptotic analysis of Eqs.~\ref{eq:GT3_KZ}--\ref{eq:GT7_KZ} then reveals certain integral representations for $ \zeta(3)$, $ \zeta(5)$ and $ \zeta(7)$ in Eqs.~\ref{eq:zeta3_P_sixth_P_third}, \ref{eq:zeta5_int_P_sixth} and \ref{eq:zeta7_int_P_sixth}.

We note that those conversion ratios between the asymptotic behavior of certain automorphic Green's functions can be recovered from their addition formulae, such as \begin{align}&\lim_{y\to+\infty}G_3^{\mathfrak H/PSL(2,\mathbb Z)}\left(z,\frac{1+i\sqrt{3}}{2}\right) y^{2}\notag\\={}&(1+3^{2}+3^{2}+3^{2})\lim_{y\to+\infty}G_3^{\mathfrak H/\overline{\varGamma}_{0}(3)}\left( z ,\frac{3+i\sqrt{3}}{6}\right) y^{2},\label{eq:PSL_Hecke3_asympt_conv'}\tag{\ref{eq:PSL_Hecke3_asympt_conv}$'$}\end{align}which is a weight-6 analog of Eq.~\ref{eq:PSL_Hecke3_asympt_conv}. \end{proof}\begin{remark}It is worth noting that the deep connections between Eichler integrals and special $ L$-values have  been discovered and expounded by Shimura~\cite{Shimura1959} and Manin~\cite{Manin1973}. The recent work of Bringmann, Fricke and Kent~\cite{BringmannFrickeKent2014} discussed special $ L$-values associated with Eichler integral representations of harmonic Maa{\ss} forms.    \eor\end{remark}
\begin{table}[t]\scriptsize\caption{Some special values of  $j$-invariants and automorphic Green's functions}\label{tab:Zagier_Q_examples}

\begin{align*}\begin{array}{r@{\;\in\;}l|l|l|l|l}\hline\hline  \vphantom{\dfrac{\frac12}{\frac12}}z&\mathfrak Z_{D}& j(z)&e^{-\frac{2 y^2}{3}G_3^{\mathfrak H/PSL(2,\mathbb Z)}(e^{\pi i/3},z)}\!\!\!\!\!&e^{-24 y^4G_5^{\mathfrak H/PSL(2,\mathbb Z)}(e^{\pi i/3},z)}\!\!\!\!\!&e^{-96 y^6G_7^{\mathfrak H/PSL(2,\mathbb Z)}(e^{\pi i/3},z)}\!\!\!\!\! \\\hline  \vphantom{\dfrac{\frac\int2}{\frac12}}i &\mathfrak Z_{-4}&2^63^3&(3)^{2}&\left(\dfrac{2^{56}}{3^{27}}\right)^2&\left(\dfrac{3^{45}}{2^{64}}\right)^4 \\[9pt] \dfrac{1+i\sqrt{7}}{2}&\mathfrak Z_{-7}&-3^{3}5^3&\dfrac{5^{3}}{3}&\dfrac{3^{189}}{5^{91}}&\dfrac{5^{133}}{3^{63}}\\[6pt]i\sqrt{2}&\mathfrak Z_{-8}&2^65^3&(5)^{2}&\left( \dfrac{5^{37}}{2^{56}} \right)^2&\left( \dfrac{2^{832}}{5^{347}} \right)^4\\[8pt] \dfrac{1+i\sqrt{11}}{2}&\mathfrak Z_{-11}&-2^{15}&2^{5}&2^{75}&2^{165}\\[5pt]i\sqrt{3}&\mathfrak Z_{-12}&2^43^35^3&\left(\dfrac{3^{3}}{5}\right)^2&\left( \dfrac{5^{277}}{2^{224}3^{243}} \right)^2&\left( \dfrac{3^{1215}5^{67}}{2^{2048}} \right)^4\\[8pt]i\sqrt{4}&\mathfrak Z_{-16}&2^33^311^3&\left(\dfrac{11^{3}}{3^{5}}\right)^2&\left( \dfrac{2^{448}}{3^{27}11^{107}} \right)^2&\left( \dfrac{3^{3447}}{2^{2048}11^{977}} \right)^4\\[8pt]\dfrac{1+i\sqrt{19}}{2}&\mathfrak Z_{-19}&-2^{15}3^3&\dfrac{3^{5}}{2^{3}}&\dfrac{3^{351}}{2^{485}}&\dfrac{2^{28493}}{3^{17883}}\\[8pt]\dfrac{1+i3\sqrt{3}}{2}&\mathfrak Z_{-27}&-2^{15}3^15^3&\dfrac{5^{13}}{2^{11}3^9}&\dfrac{2^{8555}}{3^{2187}5^{2161}}&\dfrac{5^{41533}}{2^{2699}3^{59049}}\\[6pt]i\sqrt{7}& \mathfrak Z_{-28}&3^{3}5^{3}17^3&\left( \dfrac{3^{7}17^{3}}{5^{9}} \right)^2&\left( \dfrac{17^{637}}{3^{1323}5^{203}} \right)^2&\left( \dfrac{3^{15435}5^{17563}}{17^{15953}} \right)^4\\[9pt]\dfrac{1+i\sqrt{43}}{2}&\mathfrak Z_{-43}&-2^{18}3^35^{3}&\dfrac{2^{42}}{3^{19}5^3}&\dfrac{3^{4239}5^{6959}}{2^{22806}}&\dfrac{3^{271773}}{2^{213078}5^{93683}}\\[8pt]\dfrac{1+i\sqrt{67}}{2}&\mathfrak Z_{-67}&-2^{15}3^35^{3}11^3&\dfrac{3^{29}11^{21}}{2^{51}5^{27}}&\dfrac{5^{11999}11^{959}}{2^{17285}3^{8721}}&\dfrac{3^{153837}5^{939013}}{2^{298819}11^{614315}}\\[8pt]\dfrac{1+i\sqrt{163}}{2}&\mathfrak Z_{-163}&-2^{18}3^35^{3}23^329^3&\dfrac{23^{21}29^{69}}{2^{198}3^{139}5^3}&\dfrac{5^{95279}23^{68159}}{2^{94326}3^{131841}29^{46561}}&\dfrac{2^{81641562}29^{1033765}}{3^{7144587}5^{7794803}23^{12653867}}\\[8pt]\hline\hline\end{array}
\end{align*}\end{table}
\begin{remark}\label{rmk:GKZ_class_number_one}For some special points  $z,z'\in\mathfrak H$ and $ \varGamma=SL(2,\mathbb Z)\equiv\varGamma_0(1)$, the Gross--Kohnen--Zagier algebraicity conjecture (boxed equation in \S\ref{subsec:background}) can be  directly verified by the methods developed in the joint works of Gross--Zagier~\cite{GrossZagierI} and Gross--Kohnen--Zagier~\cite{GrossZagierII}. For example, if one defines the totality of CM points whose minimal polynomial has discriminant $D$ by $ \mathfrak Z_D:=\{z\in\mathfrak H|\exists a,b,c \in\mathbb Z,a>0,\gcd(a,b,c)=1,b^2-4ac=D,az^2+bz+c=0\}$, then one can show that~\cite[][p.~50, Theorem~II.2]{ZagierKyushuJ1}
\begin{align}\exp\left[\frac{(D_1D_2)^{(k-2)/4}}{2}G_{k/2}^{\mathfrak H/PSL(2,\mathbb Z)}(z_1,z_2)\right]\in\mathbb Q,\quad k=6,10,14\label{eq:Zagier_Q}\end{align}holds for $ z_1\in\mathfrak Z_{D_1},z_2\in\mathfrak Z_{D_2}$, where the discriminants $ D_1$ and $D_2$ are two distinct members of the finite set $ \{-3$, $-$4, $-$7, $-$8, $-$11, $-$12, $-$16, $-$19, $-$27, $-$28, $-$43, $-$67, $-163\}$. Fixing the point $ z_1=e^{\pi i/3}\in\mathfrak Z_{-3} $, one can tabulate the special values of $G_3^{\mathfrak H/PSL(2,\mathbb Z)}(e^{\pi i/3},z) $, $G_5^{\mathfrak H/PSL(2,\mathbb Z)}(e^{\pi i/3},z) $ and $G_7^{\mathfrak H/PSL(2,\mathbb Z)}(e^{\pi i/3},z) $ that support the statement in Eq.~\ref{eq:Zagier_Q} (Table~\ref{tab:Zagier_Q_examples}). The exact factorizations of all the rational numbers in Table~\ref{tab:Zagier_Q_examples} can be justified by the Gross--Kohnen--Zagier theory (\textit{e.g.}~\cite[][p.~76, Example~II.6]{ZagierKyushuJ1}). One may also wish to check these tabulated values numerically\footnote{We have implemented in \textit{Mathematica} all the Kontsevich--Zagier integral representations of automorphic Green's functions declared in Theorem~\ref{thm:KZ_int_repns}. The source code (NumSuppAGF1.nb) for the numerical implementation can be downloaded from \url{http://arxiv.org/format/1312.6352}.} against the integral formulae in Eqs.~\ref{eq:GT3_KZ}--\ref{eq:GT7_KZ}.
\eor\end{remark}\section{Some Analytic Tools for Kontsevich--Zagier Integrals and Gross--Zagier Renormalization\label{sec:analysis_KZ_int_Part1}}
\setcounter{equation}{0}
In \S\S\ref{subsec:KZ_integrals} and \ref{subsec:high_weight_KZ}, we have constructed Kontsevich--Zagier integral representations for automorphic Green's functions $ G_{k/2}^{\mathfrak H/\overline\varGamma_0(N)}(z,z')$ with even weights $ k\geq4$ satisfying the cusp-form-free condition  $ \dim\mathcal S_k(\varGamma_0(N))=0$. We have also computed a few special values of weight-4 automorphic Green's functions
from their respective Kontsevich--Zagier integrals (see Remark~\ref{rmk:spec_val_G2}).

In this section, we prepare a few analytic results that enable us to compute various Kontsevich--Zagier integrals for automorphic Green's functions, and work out closed-form results for some special scenarios, as announced in Theorem~\ref{thm:GZ_rn}. The generic cases will be treated systematically in subsequent work.

In \S\ref{subsec:hypergeo_Pnumu}, we present some hypergeometric techniques for evaluating certain integrals involving  products of Legendre functions. These integrals over Legendre functions will lead to special values of  weight-4 ``Gross--Zagier renormalized'' automorphic Green's functions in \S\ref{subsec:int_repn_G2_GZ_rn}. Here, the Gross--Zagier renormalization~\cite[][Chap.~II, \S5]{GrossZagierI} is a procedure to subtract logarithmic divergence of automorphic Green's functions on the ``diagonal points''.
 In \S\ref{subsec:Ramanujan_Jacobi}, we fulfill two tasks. We first
recall some geometric transformations inspired by Ramanujan's Notebooks, which facilitate the reduction of certain multiple elliptic integrals;
 we then
combine Ramanujan's transformations with the  Jacobi elliptic functions to evaluate some non-trivial multiple elliptic integrals.
The integral identities derived in  \S\ref{subsec:Ramanujan_Jacobi} will assist in the quantitative analysis of weight-4 ``Gross--Zagier renormalized'' automorphic Green's functions on $ \overline\varGamma_0(4)$  in  \S\ref{subsec:G2_Hecke4_GZ_rn}.

In \S\S\ref{subsec:Ramanujan_Jacobi}--\ref{subsec:G2_Hecke4_GZ_rn}, we will draw extensively on  modular transformations.
In addition to the degree-1 modular transformations (Eqs.~\ref{eq:im_mod}--\ref{eq:lambda_transf}), we  need Landen's transformations (sometimes   named after  Landen and Gau{\ss} \cite[][items~163.02 and 164.02]{ByrdFriedman}, depending on context) for the complete elliptic integrals of the first kind:\begin{align}\mathbf K(\sqrt{1-\lambda})={}&\frac{2}{1+\sqrt{\lambda}}\mathbf K\left( \frac{1-\sqrt{\lambda}}{1+\sqrt{\lambda}} \right),\quad \lambda\in\mathbb C\smallsetminus(-\infty,0];\label{eq:Landen_1}\\\mathbf K(\sqrt{\lambda})={}&\frac{1}{1+\sqrt{\lambda}}\mathbf K\left( \frac{2\sqrt[4]{\lambda}}{1+\sqrt{\lambda}} \right),\quad |\lambda|<1.\label{eq:Landen_2}\end{align}
With the ``$ \lambda$-$ \mathbf K$ relation'' (Eq.~\ref{eq:lambda_K_ratio_z}), one can derive from  Landen's transformations  the  ``duplication formula'' and the ``dimidiation formula'' \cite[][\S135]{WeberVol3} for the modular lambda function $ \lambda(z)=16[\eta (z/2)]^8 [\eta (2 z)]^{16}/[\eta (z)]^{24}$:\begin{align}\label{eq:double_half_lambda}\lambda(2z)=\left[ \frac{1-\sqrt{1-\lambda(z)}}{1+\sqrt{1-\lambda(z)}} \right]^2,\quad \lambda\left( \frac{\vphantom{1}z}{2} \right)=\frac{4\sqrt{\lambda(z)}}{[1+\sqrt{\lambda(z)}]^2},\quad -1<\R z<1,\left|z+\frac12\right|>\frac12,\left|z-\frac12\right|>\frac{1}{2}.\end{align}
We shall refer to Eqs.~\ref{eq:Landen_1}--\ref{eq:double_half_lambda} (and descendants thereof) collectively as degree-2 modular transformations.

Apart from hypergeometric, modular and geometric transformations, the calculations in this section are essentially  applications of residue calculus to elliptic integrals, the prototype of which is presented in the next paragraph.

Let $ f(\lambda),\lambda\in\mathbb C\smallsetminus[1,+\infty)$ be a complex analytic function in the slit plane with well-defined one-sided limits $ f(\lambda\pm i0^+)$ for $ \lambda>1$, satisfying the bounds $ |f(\lambda)|=O(\log^\nu|1-\lambda|)$  for $\nu\in\mathbb R\smallsetminus\{0\},\lambda\to1$ and $ |f(\lambda)|=O(|\lambda|^{-\nu'})$ for $0<\nu'<1,|\lambda|\to+\infty$. One might  use Cauchy's integral formula to verify that \begin{align}
f(\lambda)=\frac{1}{2\pi i}\int_0^1\frac{f\big(\frac1\mu-i0^+\big)-f\big(\frac{1}{\mu}+i0^+\big)}{1-\lambda\mu}\frac{\D\mu}{\mu},\quad \lambda\in\mathbb C\smallsetminus[1,+\infty), \label{eq:f_slit_plane_Cauchy_int}
\end{align}where the integration is carried along the open unit interval $ (0,1)$ in the complex $ \mu$-plane. Setting $ f(\lambda)=[\mathbf K(\sqrt{\lambda})]^n,n\in\mathbb Z_{>0}$ in Eq.~\ref{eq:f_slit_plane_Cauchy_int} and  referring to the inverse modulus transformation (Eq.~\ref{eq:inv_mod}), one obtains the integral identity
\begin{align}
{[}\mathbf K(\sqrt{\lambda})]^n={}&\frac{1}{\pi}\I\int_0^1\frac{\mathbf [\mathbf K(\sqrt{\smash[b]{\vphantom{1}\mu}})+i\mathbf K(\sqrt{\smash[b]{1-\mu}})]^n}{1-\lambda\mu}\mu^{(n-2)/2}\D\mu,\quad 0<\lambda<1,n\in\mathbb Z_{>0}.\label{eq:Kn}\tag{\ref{eq:f_slit_plane_Cauchy_int}-K$^{n}$}
\end{align}
In recent literature~\cite{Wan2012,Zhou2013Pnu}, one may find discussions on the particular cases where $ n=1$ and $2$:\begin{align}&&\mathbf K(\sqrt{\lambda})={}&\frac{1}{\pi}\int_0^1\frac{\mathbf K(\sqrt{\smash[b]{1-\mu}})}{1-\lambda\mu}\frac{\D\mu}{\sqrt{\smash[b]{\mu}}},&& \lambda\in\mathbb C\smallsetminus[1,+\infty),&&
\label{eq:K1}\tag{\ref{eq:f_slit_plane_Cauchy_int}-K$^{1}$}\\&&{[}\mathbf K(\sqrt{\lambda})]^2={}&\frac{2}{\pi}\int_0^1\frac{\mathbf K(\sqrt{\smash[b]{\vphantom{1}\mu}})\mathbf K(\sqrt{\smash[b]{1-\mu}})}{1-\lambda\mu}\D\mu,&& \lambda\in\mathbb C\smallsetminus[1,+\infty),&&\label{eq:K2}\tag{\ref{eq:f_slit_plane_Cauchy_int}-K$^{2}$}
\end{align}which were proved by either combinatorial algorithms~\cite[][Eqs.~26--27]{Wan2012} or geometric transformations~\cite[][Eqs.~31 and 40$^*$]{Zhou2013Pnu}.

\subsection{Hypergeometric Evaluations of Some Definite Integrals Involving Associated Legendre Functions\label{subsec:hypergeo_Pnumu}}For $ -1<\nu<0$, we have already defined  two real-valued functions $ P_\nu(t),t>1$ and $ Q_\nu(t),t>1$ using Eqs.~\ref{eq:Euler_int} and \ref{eq:Q_nu_Laplace_int}, respectively.
In this subsection, we also need the associated Legendre functions $ P^\mu_\nu(t),t>1$ and $ Q^\mu_\nu(t),t>1$ where $ \mu>-1/2$. A convenient way to introduce these associated Legendre functions is to use Hobson's integral representations~\cite[][p.~270, Eq.~139 and p.~276, Eq.~150]{Hobson1931}:
\begin{align}
P^{\mu}_\nu(\cosh\psi):={}&\frac{2^{\mu+1}(\sinh\psi)^\mu}{\sqrt{\pi}\Gamma(\frac{1}{2}-\mu)}\int_0^\psi\frac{\cosh\frac{(2\nu+1)u}{2}\D u}{(2\cosh\psi-2\cosh u)^{\mu+\frac{1}{2}}},&& \mu>-\frac{1}{2},\psi>0;\label{eq:Pmunu_defn}\\Q^{\mu}_\nu(\cosh\psi):={}& e^{i\mu\pi}2^\mu\frac{\sqrt{\pi}(\sinh\psi)^\mu}{\Gamma(\frac{1}{2}-\mu)}\int_\psi^\infty\frac{e^{-(\nu+\frac{1}{2})u}\D u}{(2\cosh u-2\cosh \psi)^{\mu+\frac{1}{2}}},&&\mu>-\frac{1}{2},\mu+\nu+1>0,\psi>0,\label{eq:Qmunu_defn}
\end{align}where the integrals are taken along the real axis, so that one always raises a real-valued base (with zero phase) to a positive power of $ \mu+\frac{1}{2}$. The definitions of associated Legendre functions in
Eqs.~\ref{eq:Pmunu_defn}--\ref{eq:Qmunu_defn} are compatible with the previously defined Legendre functions, in the sense that  $ P_\nu^{\vphantom0}(t)=P_\nu^{0}(t),t>1$ and $ Q_\nu^{\vphantom0}(t)=Q_{\nu}^{0}(t),t>1$. Furthermore, these associated Legendre functions satisfy Whipple's relation~\cite[][p.~245, Eq.~92]{Hobson1931}:\begin{align}
Q^{\mu}_\nu(\cosh\psi)={}&e^{i\mu\pi}\sqrt{\frac{\pi}{2}}\frac{\Gamma(\mu+\nu+1)}{\sqrt{\smash[b]{\sinh\psi}}}P^{-\nu-1/2}_{-\mu-1/2}(\coth\psi),\label{eq:Whipple1}
\end{align}whenever $ \psi>0$ and the associated Legendre functions on both sides are defined in  the domains specified by
Eqs.~\ref{eq:Pmunu_defn}--\ref{eq:Qmunu_defn}.
Using Whipple's relation (Eq.~\ref{eq:Whipple1}), one can readily verify the following formula: \begin{align}
\int_1^\infty\frac{[P_\nu^{\mu}(\xi)]^2}{\xi^2}\D \xi={}&-\frac{2e^{2i\nu\pi}}{\pi[\Gamma(-\mu-\nu)]^2}\int_1^\infty\frac{\big[Q^{-\nu-1/2}_{-\mu-1/2}(\xi)\big]^2}{\xi^{2}}\D\xi,\label{eq:PnumuPnumu_int_Whipple}\end{align}provided that both integrals are well-defined. We refer to such an integral identity as ``Legendre--Whipple duality''.

In the following proposition, we will present some integral formulae involving associated Legendre functions  $ P^\mu_\nu$ and $ Q^\mu_\nu$, with the (usually) tacit assumptions that the order $ \mu$ and the degree $\nu$ are appropriately chosen to meet various requirements.
\begin{proposition}[Some Hypergeometric Reduction Formulae Involving $ P^\mu_\nu$ and $ Q^\mu_\nu$]\begin{enumerate}[label=\emph{(\alph*)}, ref=(\alph*), widest=a] \item We have the following integral identities:\begin{align}
\int_1^\infty\frac{[Q_\nu^{\mu}(\xi)]^2}{\xi^2}\D \xi={}&\frac{e^{2i\mu\pi}}{2(1+\nu ) (1+\mu +\nu )}\frac{\mu\pi}{\sin(\mu\pi)} \frac{\Gamma (1+\mu +\nu )}{\Gamma (1-\mu +\nu )}{_3}F_2\left(\left.\begin{array}{c}1,1+\mu,1+\mu+\nu\\2+\nu  ,2+\mu+\nu \end{array}\right|1\right)\notag\\={}&\frac{e^{2i\mu\pi}}{2(1+\nu ) (1-\mu +\nu )}\frac{\mu\pi}{\sin(\mu\pi)} \frac{\Gamma (1+\mu +\nu )}{\Gamma (1-\mu +\nu )}{_3}F_2\left(\left.\begin{array}{c}1,1-\mu,1-\mu+\nu\\2+\nu  ,2-\mu+\nu \end{array}\right|1\right)\notag\\={}&\frac{e^{2i\mu\pi}}{2}\frac{\mu\pi}{\sin(\mu\pi)} \left[\frac{\Gamma (1+\mu +\nu )}{\Gamma (2+\nu )}\right]^2{_3}F_2\left(\left.\begin{array}{c}1+\mu,1-\mu,1+\nu\\2+\nu  ,2+\nu \end{array}\right|1\right).\label{eq:QnumuQnumu_xx_int_3F2}
\end{align}In particular, we have\begin{align}
\int_1^\infty\frac{[Q_\nu(\xi)]^2}{\xi^2}\D\xi={}&\frac{1}{2(\nu+1)^{2}}{_3}F_2\left(\left.\begin{array}{c}1,1,\nu+1\\\nu +2 ,\nu +2\end{array}\right|1\right)\label{eq:QnuQnu_xx_int_3F2}
\end{align}and \begin{align}
\int_1^\infty\frac{[P_\nu(\xi)]^2}{\xi^2}\D\xi={}&\frac{(2\nu+1)\tan(\nu\pi)}{\nu\pi}{_3}F_2\left(\left.\begin{array}{c}1,\frac{1}{2}-\nu,-\nu\\[4pt]\frac{3}{2},1-\nu \end{array}\right|1\right)\notag\\={}&-\frac{(2\nu+1)\tan(\nu\pi)}{(\nu+1)\pi}{_3}F_2\left(\left.\begin{array}{c}1,\frac{3}{2}+\nu,1+\nu\\[4pt]\frac{3}{2},2+\nu \end{array}\right|1\right)\notag\\={}&\frac{2(2\nu+1)}{\pi\cos(\nu\pi)}{_3}F_2\left(\left.\begin{array}{c}\frac{1}{2},\frac{3}{2}+\nu,\frac{1}{2}-\nu\\[4pt]\frac{3}{2},\frac{3}{2} \end{array}\right|1\right).\label{eq:PnuPnu_xx_int_3F2}
\end{align}Furthermore, we have the following closed-form evaluation for $\nu\in(-1,-1/2)\cup(-1/2,0) $:\begin{align}
\int_1^\infty\frac{[P_\nu(\xi)]^2}{\xi^2}\D \xi={}&\frac{2}{\pi}\int_0^{\pi/2}\frac{\sin(2\nu+1)\theta}{\sin\theta\cos(\nu\pi)}\D\theta=\frac{1}{\cos(\nu\pi)}+\frac{\tan(\nu\pi)}{\pi}\left[ \psi ^{(0)}\left(\frac{\nu+2 }{2}\right)-\psi ^{(0)}\left(\frac{\nu+1 }{2}\right )\right].
\label{eq:PnuPnu_xx_int_closed_form}\end{align} \item  The following integral formula holds:\begin{align}&
\int_1^\infty\frac{[P_\nu^{\mu}(\xi)]^2}{\xi^2}\D \xi\notag\\={}&\frac{\psi ^{(0)}(\frac{1}{2}-\nu )-\psi ^{(0)}(-\mu -\nu )-\psi ^{(0)}(1-\mu +\nu )-\gamma_{0}+\dfrac{2(\mu ^{2}-\nu^{2} )}{1-2\nu}{_4} F_3\left(\left.\begin{array}{c}1,1,1-\mu -\nu ,1+\mu -\nu \\[4pt]2,2,\frac{3}{2}-\nu \end{array}\right|1\right)}{\cos(\mu\pi)\Gamma (1-\mu +\nu ) \Gamma (-\mu -\nu )},\label{eq:PnumuPnumu_xx_int_4F3}
\end{align} where $ \gamma_0=-\psi^{(0)}(1)$ is the Euler--Mascheroni constant. In particular, one has\begin{align}&
\int_1^\infty\frac{[P_\nu(\xi)]^2}{\xi^2}\D \xi\notag\\={}&\frac{\sin(\nu\pi)}{\pi}\left[ \frac{2}{\nu }+\pi  \cot (\nu\pi)-\psi ^{(0)}\left(\frac{1}{2}-\nu \right)+2 \psi ^{(0)}(\nu )+\gamma _{0}+\frac{2\nu ^2}{1-2\nu}{ _4}F_3\left(\left.\begin{array}{c}1,1,1-\nu ,1-\nu \\[4pt]2,2,\frac{3}{2}-\nu \end{array}\right|1\right)\right]\label{eq:PnuPnu_xx_int_4F3}
\end{align} and\begin{align}
\int_1^\infty\frac{[Q_\nu(\xi)]^2}{\xi^2}\D \xi={}&\frac{\pi}{\sin(\nu\pi)}\left[ \psi ^{(0)}(\nu +1)+\gamma_{0}-\frac{\nu(\nu+1)}{2} {_4}F_3\left(\left.\begin{array}{c}1,1,1-\nu ,2+\nu \\[4pt]2,2,2 \end{array}\right|1\right)\right].\label{eq:QnuQnu_xx_int_4F3}
\end{align}Furthermore, the following formula is true for $ \nu\in(-1,0)$:\begin{align}
\int_1^\infty\frac{[Q_\nu(\xi)]^2}{\xi^2}\D \xi={}&\frac{\pi[\psi ^{(0)}(\nu +1)-\log 2]}{\sin(\nu\pi)}+\frac{\left[\psi ^{(0)}\left(\frac{\nu+2 }{2}\right)+\psi ^{(0)}\left(\frac{1-\nu}{2}\right)\right]^2-\left[\psi ^{(0)}\left(\frac{\nu +1}{2}\right)+\psi ^{(0)}\left(-\frac{\nu }{2}\right)\right]^2}{8}\notag\\&+\frac{\psi ^{(1)}\left(\frac{\nu +1}{2}\right)+\psi ^{(1)}\left(-\frac{\nu }{2}\right)-\psi ^{(1)}\left(\frac{\nu+2 }{2}\right)-\psi ^{(1)}\left(\frac{1-\nu}{2}\right)}{8}.\label{eq:QnuQnu_xx_int_closed_form}
\end{align} \item The following identity holds for $ -1<\nu<0$:\begin{align}
\int_1^\infty\frac{P_\nu(\xi)Q_{\nu}(\xi)}{\xi^2}\D\xi={}&\frac{1}{2}\left[ \frac{\Gamma (\nu +1)}{\Gamma (\nu +\frac{3}{2})} \right]^2{_3}F_2\left(\left.\begin{array}{c}\frac{1}{2},\frac{1}{2},\nu+\frac{1}{2}\\[4pt]\nu +\frac{3}{2} ,\nu +\frac{3}{2}\end{array}\right|1\right)=\frac{1}{2}\left[ \psi ^{(0)}\left(\frac{\nu+2 }{2}\right)-\psi ^{(0)}\left(\frac{\nu+1 }{2}\right )\right],\label{eq:PnuQnu_xx_int_closed_form}
\end{align}and the integral formula \begin{align}\int_{-1}^1\frac{P_\nu(\xi)P_\nu(-\xi)-[P_\nu(0)]^{2}}{\xi^2}\D \xi=2\{[P_\nu(0)]^{2}-1\},\quad \forall\nu\in\mathbb C\label{eq:minus2_id'}\end{align}is true so long as the contour of integration lies in the double-slit plane $ \xi\in\mathbb C\smallsetminus((-\infty,-1]\cup[1,+\infty))$.

Additionally, we have an integral formula for $ -1<\nu<0$: {\allowdisplaybreaks\begin{align}
\R\int_{-1}^1\frac{[P_\nu(\xi)]^2}{\xi^2}\D \xi={}&-2\cos^2(\nu\pi)\int_1^\infty\frac{[P_\nu(\xi)]^2}{\xi^2}\D \xi+\frac{4\sin(\nu\pi)\cos(\nu\pi)}{\pi}\int_1^\infty\frac{P_\nu(\xi)Q_{\nu}(\xi)}{\xi^2}\D \xi\notag\\{}&-\frac{4\sin^2(\nu\pi)}{\pi^2}\int_1^\infty\frac{[Q_\nu(\xi)]^2}{\xi^2}\D \xi\notag\\={}&-2\cos(\nu\pi)-\frac{4 \sin (  \nu\pi ) [ \psi ^{(0)}(\nu +1)-\log2]}{\pi }\notag\\&-\frac{\sin ^2(  \nu\pi ) }{2 \pi ^2}\left\{\left[\psi ^{(0)}\left(\frac{\nu+2 }{2}\right)+\psi ^{(0)}\left(\frac{1-\nu}{2}\right)\right]^2-\left[\psi ^{(0)}\left(\frac{\nu +1}{2}\right)+\psi ^{(0)}\left(-\frac{\nu }{2}\right)\right]^2\right\}\notag\\&-\frac{\sin ^2(  \nu\pi )}{2 \pi ^2} \left[\psi ^{(1)}\left(\frac{\nu +1}{2}\right)+\psi ^{(1)}\left(-\frac{\nu }{2}\right)-\psi ^{(1)}\left(\frac{\nu+2 }{2}\right)-\psi ^{(1)}\left(\frac{1-\nu}{2}\right)\right],\label{eq:PnuPnu_xx_int_unit_interval_polygamma}
\end{align}}which generalizes the evaluations in Proposition~\ref{prop:Pnu_sec_L_weight4}. \item For $ -1<\nu<0$, we have\begin{align}&
\int_{0}^1\left\{\frac{[P_{\nu}(\xi)]^2-[P_{\nu}(-\xi)]^2}{\xi^{2}}-\frac{4\sin(\nu\pi)}{\pi \xi}\right\}\D \xi+\R\int_{-1}^1\frac{[P_\nu(\xi)]^2}{\xi^2}\D \xi\notag\\={}&-2\cos^2(\nu\pi)\int_1^\infty\frac{[P_\nu(\xi)]^2}{\xi^2}\D \xi-\frac{4}{\pi}\int_{0}^1\frac{P_{\nu}(\xi)P_{\nu}(-\xi)-[P_\nu(0)]^{2}}{\xi^{2}}\log \xi\D \xi+\frac{4[P_\nu(0)]^2\sin(\nu\pi)}{\pi}\notag\\={}&\frac{4\sin(\nu\pi) [1-\gamma_{0} - \psi ^{(0)}(\nu +1)-\log 2]}{\pi }-2\cos(\nu\pi).
\label{eq:Pnu_sqr_diff_int_log_over_sqr}\end{align}For $ -1/2<\nu<0$, we have\begin{align}&
\int_1^\infty\left[ P_\nu(\xi)Q_\nu(\xi)-\frac{1}{(2\nu+1)\xi} \right]\D \xi\notag\\={}&\frac{\tan(\nu\pi)}{\pi}\int_1^\infty[Q_\nu(\xi)]^2\D \xi-\frac{\pi}{4}\int_0^1\frac{[P_\nu(\xi)]^2-[P_\nu(-\xi)]^2}{\sin(\nu\pi)\cos(\nu\pi)}\D \xi-\int_{0}^1\frac{P_{\nu}(\xi)P_{\nu}(-\xi)}{\cos(\nu\pi)}\log \xi\D \xi\notag\\={}&\frac{\pi  \tan (  \nu\pi )+\psi ^{(0)}\left(\frac{1}{2}-\nu \right)-2 \psi ^{(0)}(\nu )+\psi ^{(0)}\left(\nu +\frac{3}{2}\right)+2 \log 2}{2 (2 \nu +1)}-\frac{1}{\nu  (2 \nu +1)}.\label{eq:PQ-1_over_x_int}
\end{align}\end{enumerate}\end{proposition}\begin{proof}\begin{enumerate}[label=(\alph*),widest=a]\item  We may compute \begin{align}&\int_1^\infty\frac{[Q_\nu^{\mu}(\xi)]^2}{\xi^2}\D \xi=\int_0^\infty\frac{[Q_\nu^{\mu}(\sqrt{1+t})]^2}{2(1+t)^{3/2}}\D t\notag\\={}&\frac{e^{2i\mu\pi}}{2\pi i}\frac{\Gamma(1+\mu+\nu)}{\Gamma(1-\mu+\nu)}\int_{c-i\infty}^{c+i\infty}\frac{\Gamma (s)\Gamma (s+\mu)\Gamma(s-\mu)\  \Gamma (\nu +1-s)\Gamma(1-s) }{ 2\Gamma (\nu +1+s)}\D s\label{eq:QmunuQmunu_xx_int_comp}\end{align}from the  following Mellin inversion formula~\cite[][\S10.11, Eq. 54(1)]{Marichev1983}  \begin{align}
{[}Q_\nu^{\mu}(\sqrt{1+t})]^2={}&\frac{\sqrt{\pi}}{2}\frac{e^{2i\mu\pi}}{2\pi i}\frac{\Gamma(1+\mu+\nu)}{\Gamma(1-\mu+\nu)}\int_{c-i\infty}^{c+i\infty}\frac{\Gamma (s)\Gamma (s+\mu)\Gamma(s-\mu)\  \Gamma (\nu +1-s) }{ \Gamma (\nu +1+s)\Gamma (s+\frac{1}{2})}\frac{\D s}{t^s},&&t>0,|\mu|<c<\nu+1
\end{align}and an elementary identity  $ \int_0^\infty t^{-s}(1+t)^{-3/2}\D t=2\Gamma(1-s)\Gamma(s+\frac{1}{2})/\sqrt{\pi}$ for $ -1/2<\R s<1$. Here, by convention, the  integrations $\int_{c-i\infty}^{c+i\infty}(\cdots)\D s:= \lim_{T\to +\infty}\int_{c-iT}^{c+iT}(\cdots)\D s
$ are  carried out along a
vertical line $\R s = c$.

With the help of Barnes's second lemma~\cite[][Eq.~4.2.2.1]{Slater}, one can reduce the contour integral  in Eq.~\ref{eq:QmunuQmunu_xx_int_comp} into particular values of generalized hypergeometric series $ _3F_2$, as claimed in Eq.~\ref{eq:QnumuQnumu_xx_int_3F2}.  Specializing  Eq.~\ref{eq:QnumuQnumu_xx_int_3F2} to $ Q_\nu^{\vphantom0}=Q_\nu^0$, one obtains Eq.~\ref{eq:QnuQnu_xx_int_3F2}. Applying  the Legendre--Whipple duality in Eq.~\ref{eq:PnumuPnumu_int_Whipple} to  Eq.~\ref{eq:QnumuQnumu_xx_int_3F2}, one can verify Eq.~\ref{eq:PnuPnu_xx_int_3F2}.

In the last line of  Eq.~\ref{eq:PnuPnu_xx_int_3F2}, one can represent the generalized hypergeometric series  $ _3F_2$ as a definite integral involving the hypergeometric series  $ _2F_1$ \cite[][item~7.512.5]{GradshteynRyzhik}:\begin{align}&
{_3}F_2\left(\left.\begin{array}{c}\frac{1}{2},\frac{3}{2}+\nu,\frac{1}{2}-\nu\\[4pt]\frac{3}{2},\frac{3}{2} \end{array}\right|1\right)=\frac{1}{2}\int_0^1{_2}F_1\left(\left.\begin{array}{c}\frac{3}{2}+\nu,\frac{1}{2}-\nu\\[4pt]\frac{3}{2} \end{array}\right|u\right)\frac{\D u}{\sqrt{u}}=\frac{1}{2}\int_0^1\frac{\sin \left((2 \nu+1 )\arcsin\sqrt{u} \right)}{(2\nu+1)\sqrt{u(1-u)}}\frac{\D u}{\sqrt{u}},
\end{align}where we have spelt out the hypergeometric series $_2F_1 $ in the integrand using elementary functions \cite[][item~9.121.30]{GradshteynRyzhik}. After a trigonometric substitution $ u=\sin^2\theta$, we arrive at the first identity in   Eq.~\ref{eq:PnuPnu_xx_int_closed_form}.

We can compute the integral over $ \theta$ via integration by parts and a familiar Fourier expansion~\cite[][item~1.422.2]{GradshteynRyzhik}:
\begin{align}
\int_0^{\pi/2}\frac{\sin(2\nu+1)\theta}{\sin\theta}\D\theta={}&(2\nu+1)\int_0^{\pi/2}\cos(2\nu+1)\theta\log\cot\frac{\theta}{2}\D\theta\notag\\={}&2(2\nu+1)\int_0^{\pi/2}\cos(2\nu+1)\theta\sum_{\ell=0}^\infty\frac{\cos(2\ell+1)\theta}{2\ell+1}\D\theta\notag\\={}&\frac{\sin ( \nu\pi)}{2}\sum_{\ell=0}^{\infty}(-1)^\ell\left(\frac{1 }{ \ell+\nu +1}-\frac{1 }{\ell-\nu}\right),\label{eq:angle_int_psi}\end{align}before invoking the standard partial fraction expansion for the digamma function $ \psi^{(0)}(z)$.
Clearly, the $ \nu=-1/2$ scenario can be recovered in  a limit procedure:\begin{align}\int_1^\infty\frac{[P_{-1/2}(\xi)]^2}{\xi^2}\D \xi=\lim_{\nu\to-1/2}\int_1^\infty\frac{[P_\nu(\xi)]^2}{\xi^2}\D \xi={}&\frac{4}{\pi^{2}}\int_0^{\pi/2}\frac{\theta\D\theta}{\sin\theta}=\frac{8G}{\pi^{2}}.\end{align}
\item With the following Mellin inversion formula~\cite[][\S10.11, Eq.~28(1)]{Marichev1983}
\begin{align}
&
[P_{\nu}^{\mu}(\sqrt{1+t})]^2\notag\\= {}&\frac{1}{\sqrt{\pi}\Gamma(1-\mu+\nu)\Gamma(-\mu-\nu)}\frac{1}{2\pi i}\int_{c-i\infty}^{c+i\infty}\frac{\Gamma (s- \mu) \Gamma (\frac{1}{2}-s) \Gamma (\nu +1-s) \Gamma (-\nu-s )}{\Gamma (1-s)\Gamma(1-\mu-s)}\frac{\D s}{t^s},&&\mu<c<\min\{-\nu,\nu+1,1/2\},\label{eq:PnuPnu_Mellin}
\end{align}one can compute \begin{align}&
\int_1^\infty\frac{[P_\nu^{\mu}(\xi)]^2}{\xi^2}\D \xi=\int_0^\infty\frac{[P_\nu^{\mu}(\sqrt{1+t})]^2}{2(1+t)^{3/2}}\D t\notag\\={}&\frac{1}{{\pi}\Gamma(1-\mu+\nu)\Gamma(-\mu-\nu)}\frac{1}{2\pi i}\int_{c-i\infty}^{c+i\infty}\frac{ \Gamma (\frac{1}{2}-s) \Gamma (\nu +1-s) \Gamma (-\nu-s )\Gamma (s- \mu)\Gamma (\frac{1}{2}+s)}{\Gamma(1-\mu-s)}\D s\notag\\={}&\frac{1}{{\pi}\Gamma(1-\mu+\nu)\Gamma(-\mu-\nu)}G_{3,3}^{2,3}\left(1\left|
\begin{array}{c}
 \frac{1}{2},-\nu ,\nu +1 \\
 -\mu ,\frac{1}{2},\mu  \\
\end{array}
\right.\hspace{-.5em}\right).
\end{align}By taking the limit $ \xi\to1$ in the following reduction of the Meijer $G$-function $ G^{2,3}_{3,3}$:\begin{align}
G_{3,3}^{2,3}\left(\xi\left|
\begin{array}{c}
 \frac{1}{2},-\nu ,\nu +1 \\
 -\mu ,\frac{1}{2},\mu  \\
\end{array}
\right.\hspace{-.5em}\right)={}&\frac{2 \pi  (2 \nu +1) \sqrt{\xi}  }{(2 \mu -1) (2 \mu +1)\cos(\nu\pi)}{_3}F_2\left(\left.\begin{array}{c}1,\frac{1}{2}-\nu ,\nu +\frac{3}{2}\\\frac{3}{2}-\mu ,\mu +\frac{3}{2}\end{array}\right|\xi\right)\notag\\&+\frac{\pi   \Gamma (1-\mu +\nu ) \Gamma (-\mu -\nu ) }{\Gamma (1-2 \mu )\xi^{\mu}\cos(\mu\pi)} {_2}F_1\left(\left.\begin{array}{c}-\mu -\nu ,1-\mu +\nu \\1-2 \mu\end{array}\right|\xi\right),
\end{align}one arrives at Eq.~\ref{eq:PnumuPnumu_xx_int_4F3}, after somewhat laborious computations. One can deduce Eq.~\ref{eq:PnuPnu_xx_int_4F3} as a special case of  Eq.~\ref{eq:PnumuPnumu_xx_int_4F3} with $ P_\nu^{\vphantom0}=P^0_\nu$. By the Legendre--Whipple duality (Eq.~\ref{eq:PnumuPnumu_int_Whipple}), we obtain  Eq.~\ref{eq:QnuQnu_xx_int_4F3} as a particular situation of  Eq.~\ref{eq:PnumuPnumu_xx_int_4F3} for $ P_{-1/2}^{-\nu-1/2}$.

One may also derive   Eq.~\ref{eq:QnuQnu_xx_int_4F3} directly from   Eq.~\ref{eq:QnuQnu_xx_int_3F2}, without going through the Meijer $ G$-function and Legendre--Whipple duality. To show this, we start by appealing  to Kummer's $ {_3}F_2$ transformation~\cite[][Corollary 3.3.5]{AAR} \begin{align}\label{eq:Kummer3F2}
 {_3F_2}\left( \left.\begin{array}{c}
a_{1},a_{2},a_{3} \\[4pt]
b_{1} ,b_{2}\\
\end{array}\right| 1\right)=\frac{\Gamma(b_2)\Gamma(b_1+b_2-a_1-a_2-a_3)}{\Gamma(b_{2}-a_{1})\Gamma(b_{1}+b_{2}-a_{2}-a_{3})}{_3F_2}\left( \left.\begin{array}{c}
a_{1},b_{1}-a_{2},b_{1}-a_{3} \\[4pt]
b_{1} ,b_{1}+b_{2}-a_2-a_3\\
\end{array}\right| 1\right)
\end{align} with $ a_1=1,a_2=\nu+1,a_3=1,b_1=\nu+2,b_2=\nu+2$. Such a procedure furnishes us with an identity\begin{align}
{_3}F_2\left(\left.\begin{array}{c}1,1,\nu+1\\\nu +2 ,\nu +2\end{array}\right|1\right)=\frac{\Gamma (\nu +1) \Gamma (\nu +2)}{\Gamma (2 \nu +2)}{_3}F_2\left(\left.\begin{array}{c}\nu+1,\nu+1,\nu+1\\\nu +2 ,2\nu +2\end{array}\right|1\right).
\end{align}For a suitably chosen non-vanishing $ \varepsilon$, we may use Thomae's second fundamental relation (see~\cite[][Eq.~45]{Thomae1879} or~\cite[][Eq.~4.3.4]{Slater}) to verify the following identity:\begin{align}
&{_3}F_2\left(\left.\begin{array}{c}\nu+1,\nu+1,\nu+1+\varepsilon\\\nu +2 ,2\nu +2\end{array}\right|1\right)\notag\\={}&\frac{\Gamma (-\nu ) \Gamma (2 \nu +3) \Gamma (\varepsilon )}{2 \Gamma (\varepsilon+\nu +1)}+\frac{(\nu +1)^2 \Gamma (-\nu -1) \Gamma (2 \nu +2)}{\varepsilon ^2 \Gamma (\varepsilon ) \Gamma (-\varepsilon+\nu +1)}{_3}F_2\left(\left.\begin{array}{c}\varepsilon ,\varepsilon -\nu ,\nu+1 +\varepsilon\ \\1+\varepsilon ,1+\varepsilon\end{array}\right|1\right)\notag\\={}&\left[\frac{\Gamma (-\nu ) \Gamma (2 \nu +3) \Gamma (\varepsilon )}{2 \Gamma (\varepsilon+\nu +1)}+\frac{(\nu +1)^2 \Gamma (-\nu -1) \Gamma (2 \nu +2)}{\varepsilon ^2 \Gamma (\varepsilon ) \Gamma (-\varepsilon+\nu +1)}\right]\notag\\{}&+\frac{(\nu +1)^2 \Gamma (-\nu -1) \Gamma (2 \nu +2)}{\varepsilon ^2 \Gamma (\varepsilon ) \Gamma (-\varepsilon+\nu +1)}\sum_{n=1}^\infty\frac{(\varepsilon)_{n}(\varepsilon-\nu)_{n} (\nu+1 +\varepsilon)_{n}}{(1+\varepsilon)_{n}(1+\varepsilon)_{n}n!}.\label{eq:Thomae2nd_epsilon}
\end{align}In passage to the limit of  $ \varepsilon\to0$, we may use Eq.~\ref{eq:Thomae2nd_epsilon} to compute\begin{align}&{_3}F_2\left(\left.\begin{array}{c}\nu+1,\nu+1,\nu+1\\\nu +2 ,2\nu +2\end{array}\right|1\right)\notag\\={}&-\frac{4^{\nu +1} (\nu +1) \Gamma (-\nu ) \Gamma (\nu +\frac{3}{2}) [\psi ^{(0)}(\nu +1)+\gamma_{0} ]}{\sqrt{\pi }}+\frac{(\nu +1)^2 \Gamma (-\nu -1) \Gamma (2 \nu +2)}{\Gamma (\nu +1)}\sum_{n=1}^\infty\frac{(-\nu)_{n} (\nu+1 )_{n}}{(n!)^2n}\notag\\={}&-\frac{4^{\nu +1} (\nu +1) \Gamma (-\nu ) \Gamma (\nu +\frac{3}{2}) [\psi ^{(0)}(\nu +1)+\gamma_{0} ]}{\sqrt{\pi }}+\frac{(\nu +1)^2 \Gamma (-\nu -1) \Gamma (2 \nu +2)}{\Gamma (\nu +1)}\int_{-1}^1\frac{P_\nu(\xi)-1}{1-\xi}\D \xi.
\label{eq:Thomae2nd_app}
\end{align} Here, to deduce the integral representation in the last step of Eq.~\ref{eq:Thomae2nd_app}, we have invoked the hypergeometric form of Legendre functions   $ P_{\nu}(\xi)={_2}F_1\left( \left.\begin{smallmatrix}-\nu,\nu+1\\1\end{smallmatrix}\right| \smash{\frac{1-\xi}{2}}\right),$ $\xi\in\mathbb C\smallsetminus(-\infty,-1] $.  It can be shown that the infinite series in  Eq.~\ref{eq:Thomae2nd_app} has termwise agreement with the expansion  for $ -\nu(\nu+1){_4}F_3\left(\left.\begin{smallmatrix}1,1,1-\nu,\nu+2\\2 ,2,2\end{smallmatrix}\right|1\right)$, thus
   Eq.~\ref{eq:QnuQnu_xx_int_4F3} is recovered.

We shall evaluate the remaining integral in Eq.~\ref{eq:Thomae2nd_app} as\begin{align}
\int_{-1}^1\frac{P_\nu(\xi)-1}{1-\xi}\D \xi=\lim_{\mu\to0^{+}}\int_{-1}^{1-\mu}\frac{P_\nu^{\mu}(\xi)-1}{1-\xi}\D \xi,
\end{align}while adopting Hobson's definition for $ P^\mu_\nu(\xi),-1<\xi<1$ as follows~\cite[][p.~227, Eq.~55]{Hobson1931}:\begin{align}
P^{\mu}_\nu(\xi)=\frac{1}{\Gamma(1-\mu)}\left( \frac{1+\xi}{1-\xi} \right)^{\mu/2}{_2}F_1\left( \left.\begin{array}{c}
-\nu,\nu+1\ \\
1-\mu\ \\
\end{array}\right| \frac{1-\xi}{2}\right),\quad -1<\xi<1.
\end{align} Bearing in mind the associated Legendre differential equation~\cite[][item~8.700.1]{GradshteynRyzhik} \begin{align}\frac{\D}{\D \xi}\left[ (1-\xi^2) \frac{\D P^\mu_\nu(\xi)}{\D \xi}\right]+\left[ \nu(\nu+1)-\frac{\mu^2}{1-\xi^2} \right]P^\mu_\nu(\xi)=0,\label{eq:Pmunu_diff_eqn}\end{align}and the recursion relation $ (2\nu+1) \xi P^\mu_\nu(\xi)=(\nu-\mu+1)P^\mu_{\nu+1}(\xi)+(\nu+\mu)P^{\mu}_{\nu-1}(\xi)$~\cite[][item~8.733.2]{GradshteynRyzhik}, one may put down{\allowdisplaybreaks\begin{align}&\int_{-1}^{1-\mu}\frac{P_\nu^{\mu}(\xi)-1}{1-\xi}\D \xi=\log\frac{\mu}{2}+\int_{-1}^{1-\mu}\left[ P_\nu^{\mu}(\xi)+\frac{(\nu-\mu+1)P^\mu_{\nu+1}(\xi)+(\nu+\mu)P^{\mu}_{\nu-1}(\xi)}{2\nu+1} \right]\frac{\D \xi}{1-\xi^{2}}\notag\\={}& \log\frac{\mu}{2}+\left.\frac{(1-\xi^2)}{\mu^{2}} \frac{\D }{\D \xi}\left[ P_\nu^{\mu}(\xi)+\frac{(\nu-\mu+1)P^\mu_{\nu+1}(\xi)+(\nu+\mu)P^{\mu}_{\nu-1}(\xi)}{2\nu+1} \right]\right|_{x=1-\mu}\notag\\&+\frac{1}{\mu^{2}}\int_{-1}^{1-\mu}\left[ \nu(\nu+1)P_\nu^{\mu}(\xi)+\frac{(\nu+1)(\nu+2)(\nu-\mu+1)P^\mu_{\nu+1}(\xi)+(\nu-1)\nu(\nu+\mu)P^{\mu}_{\nu-1}(\xi)}{2\nu+1} \right]\D \xi\notag\\={}&\log\frac{\mu}{2}+\left.\frac{(1-\xi^2)}{\mu^{2}} \frac{\D }{\D \xi}\left[ (1+\xi)P_\nu^{\mu}(\xi) \right]\right|_{\xi=1-\mu}+\frac{2^{\mu }\sin(\nu\pi)}{\mu\sin\frac{\mu\pi}{2}}  \left[\frac{\Gamma (\frac{1-\nu}{2}) \Gamma (\frac{\nu }{2}+1)}{\Gamma (\frac{-\mu -\nu +1}{2} ) \Gamma (\frac{-\mu +\nu +2}{2})}-\frac{\Gamma (-\frac{\nu }{2}) \Gamma (\frac{\nu +1}{2})}{\Gamma (-\frac{\mu+\nu }{2}) \Gamma (\frac{-\mu +\nu +1}{2} )}\right]\notag\\&-\frac{1}{\mu^{2}}\int_{1-\mu}^{1}\left[ \nu(\nu+1)P_\nu^{\mu}(\xi)+\frac{(\nu+1)(\nu+2)(\nu-\mu+1)P^\mu_{\nu+1}(\xi)+(\nu-1)\nu(\nu+\mu)P^{\mu}_{\nu-1}(\xi)}{2\nu+1} \right]\D \xi\notag\\={}&\frac{2}{\mu}-2\gamma_{0}+\frac{2^{\mu }\sin(\nu\pi)}{\mu\sin\frac{\mu\pi}{2}}  \left[\frac{\Gamma (\frac{1-\nu}{2}) \Gamma (\frac{\nu }{2}+1)}{\Gamma (\frac{-\mu -\nu +1}{2} ) \Gamma (\frac{-\mu +\nu +2}{2})}-\frac{\Gamma (-\frac{\nu }{2}) \Gamma (\frac{\nu +1}{2})}{\Gamma (-\frac{\mu+\nu }{2}) \Gamma (\frac{-\mu +\nu +1}{2} )}\right]+O(\mu\log^2\mu).\label{eq:small_mu}\end{align}}Here, we have relied on the standard integral formula for $ \int_{-1}^1 P^\mu_\nu(\xi)\D \xi$ applicable to $ |\R\mu|<2$~\cite[][item~7.132.1]{GradshteynRyzhik} and have used the expansion\begin{align*}P_{\nu}^\mu(\xi)={}& \frac{2^{\mu /2}}{\Gamma (1-\mu )}(1-\xi)^{-\frac{\mu }{2}}-\frac{2^{\frac{\mu }{2}-2}  [\mu (1-\mu)+2 \nu  (\nu +1)]}{\Gamma (2-\mu )}(1-\xi)^{1-\frac{\mu }{2}}\notag\\&+\frac{2^{\frac{\mu }{2}-5}  \left[4 \mu  \nu  (\nu +1)-\frac{4 (\nu -1) \nu  (\nu +1) (\nu +2)}{\mu -2}-(\mu -2) (\mu -1) \mu \right]}{\Gamma (2-\mu )}(1-\xi)^{2-\frac{\mu }{2}}+O((1-\xi)^{3-\frac{\mu }{2}})\end{align*}for the computation of the derivative at $ \xi=1-\mu$ as well as the integration over $ \xi\in[1-\mu,1]$. As we take the limit $ \mu\to 0^+$ in Eq.~\ref{eq:small_mu}, we can evaluate Eq.~\ref{eq:Thomae2nd_app} in closed form, and hence Eq.~\ref{eq:QnuQnu_xx_int_closed_form}. \item Using the following Mellin inversion formula~\cite[][\S10.11, Eq. 57(1)]{Marichev1983}: \begin{align}&
P_\nu^{-\mu}(\sqrt{1+t})Q_\nu^{\mu}(\sqrt{1+t})\notag\\={}&\frac{e^{i\mu\pi}}{2\sqrt{\pi}}\frac{1}{2\pi i}\int_{c-i\infty}^{c+i\infty}\frac{\Gamma (s)\Gamma(s+\mu)\ \Gamma (\frac{1}{2}-s) \Gamma (\nu +1-s) }{ \Gamma (\nu +1+s)\Gamma (1+\mu-s)}\frac{\D s}{t^s},&&\max\{0,-\mu\}<c<\min\{\nu+1,1/2\}
\end{align}one can verify that \begin{align*}\int_1^\infty\frac{P_\nu(\xi)Q_{\nu}(\xi)}{\xi^2}\D \xi={}&\int_0^\infty\frac{P_\nu(\sqrt{1+t})Q_{\nu}(\sqrt{1+t})}{2(1+t)^{3/2}}\D t=\frac{1}{2\pi i}\int_{c-i\infty}^{c+i\infty}\frac{[\Gamma (s)]^{2}\Gamma(s+\frac{1}{2}) \Gamma (\frac{1}{2}-s) \Gamma (\nu +1-s) }{ \Gamma (\nu +1+s)}\frac{\D s}{2\pi}.\end{align*}Again, one can treat such  a contour integral with Barnes's second lemma~\cite[][Eq.~4.2.2.1]{Slater}, which results in the first equality in Eq.~\ref{eq:PnuQnu_xx_int_closed_form}.

Instead of directly manipulating the hypergeometric series in  Eq.~\ref{eq:PnuQnu_xx_int_closed_form}, we shall evaluate it by a comparison to Eq.~\ref{eq:PnuPnu_xx_int_closed_form}. We may employ the residue theorem to establish a vanishing integral:\begin{align}
0={}&\int_{-\infty+i0^+}^{+\infty+i0^+}\frac{[P_\nu(\xi)]^2-[P_\nu(0)]^2}{\xi^2}\D \xi\notag\\={}&\int_{-\infty+i0^+}^{-1+i0^+}\frac{[P_\nu(\xi)]^2-[P_\nu(0)]^2}{\xi^2}\D \xi+\mathcal P\int_{-1}^{1}\frac{[P_\nu(\xi)]^2-[P_\nu(0)]^2}{\xi^2}\D \xi\notag\\{}&-2i\sin(\nu\pi)+\int_1^\infty\frac{[P_\nu(\xi)]^2-[P_\nu(0)]^2}{\xi^2}\D \xi.\label{eq:PnuPnu_xx_int_3parts}
\end{align} For $ \xi<-1$, we  use the identity\begin{align}
P_\nu(\xi\pm i0^{+})=[\cos (  \nu\pi )\pm i \sin (\nu\pi  )] P_{\nu }(-\xi)-\frac{2 \sin (  \nu\pi ) Q_{\nu }(-\xi)}{\pi }\label{eq:Pnu_Qnu_rln}
\end{align} to rewrite the first integral in the last line of Eq.~\ref{eq:PnuPnu_xx_int_3parts} as \begin{align}
\int_{-\infty+i0^+}^{-1+i0^+}\frac{[P_\nu(\xi)]^2-[P_\nu(0)]^2}{\xi^2}\D \xi=\int_{1}^\infty\left\{ [\cos (  \nu\pi )+ i \sin (\nu\pi  )]P_{\nu }(\xi)-\frac{2 \sin (\nu\pi) Q_{\nu }(\xi)}{\pi } \right\}^2\frac{\D \xi}{\xi^2}-[P_\nu(0)]^2.
\end{align}As we read off the imaginary part of  Eq.~\ref{eq:PnuPnu_xx_int_3parts}, we obtain \begin{align}
2\sin(\nu\pi)\cos(\nu\pi)\int_1^\infty\frac{[P_\nu(\xi)]^2}{\xi^2}\D \xi-\frac{4\sin^2(\nu\pi)}{\pi}\int_1^\infty\frac{P_\nu(\xi)Q_{\nu}(\xi)}{\xi^2}\D \xi-2\sin(\nu\pi)=0,\label{eq:Pnu_Qnu_x_2_0sum}
\end{align}  which allows us to deduce the final evaluation in   Eq.~\ref{eq:PnuQnu_xx_int_closed_form} from Eq.~\ref{eq:PnuPnu_xx_int_closed_form}.

Using the residue theorem as well as  the relation between $ P_\nu$ and $Q_\nu$ (Eq.~\ref{eq:Pnu_Qnu_rln}), we can verify\begin{align}
\R\int_{-1}^1\frac{P_\nu(\xi)P_\nu(-\xi)}{\xi^2}\D \xi={}&-2\cos(\nu\pi)\int_1^\infty\frac{[P_\nu(\xi)]^2}{\xi^2}\D \xi+\frac{4\sin(\nu\pi)}{\pi}\int_{1}^\infty\frac{P_\nu(\xi)Q_\nu(\xi)}{\xi^2}\D \xi=-2\label{eq:minus2_id}
\end{align}for $ -1<\nu<0$, by virtue of the evaluations given in Eqs.~\ref{eq:PnuPnu_xx_int_closed_form} and \ref{eq:PnuQnu_xx_int_closed_form}. It is not hard to show that   Eq.~\ref{eq:minus2_id} entails  Eq.~\ref{eq:minus2_id'}.

As we take the real part of  Eq.~\ref{eq:PnuPnu_xx_int_3parts},  we arrive at the first equality in Eq.~\ref{eq:PnuPnu_xx_int_unit_interval_polygamma}. The remaining steps of Eq.~\ref{eq:PnuPnu_xx_int_unit_interval_polygamma} are straightforward computations. \item Adhering to the convention that $ \xi^s=e^{s\log \xi}$ with $ \log \xi=\log|\xi|+i\arg \xi,-\pi<\arg \xi\leq \pi$,  we may extract the imaginary part of a vanishing identity \begin{align*}0=\int_{-\infty+i0^+}^{+\infty+i0^+}\frac{[P_\nu(\xi)]^2-[P_\nu(-\xi)]^2-2i\sin(\nu\pi)P_\nu(\xi)P_\nu(-\xi)}{(-\xi)^s}\D \xi,\quad 1<s<2,-1<\nu<0
\end{align*}as follows:\begin{align}0={}&\frac{4}{\pi}\sin(\nu\pi)\sin(s\pi)\left[\cos(\nu\pi)\int_1^\infty\frac{P_\nu(\xi)Q_{\nu}(\xi)}{\xi^s}\D \xi-\frac{\sin(\nu\pi)}{\pi}\int_1^\infty\frac{[Q_\nu(\xi)]^2}{\xi^s}\D \xi\right]\notag\\&+\sin(s\pi)\int_0^1\left\{\frac{[P_\nu(\xi)]^2-[P_\nu(-\xi)]^2}{\xi^{s}}-\frac{4\sin(\nu\pi)}{\pi \xi^{s-1}}\right\}\D \xi+\frac{4\sin(s\pi)\sin(\nu\pi)}{\pi(2-s)}\notag\\&-2\sin(\nu\pi)[1+\cos(s\pi)]\left\{\int_{0}^1\frac{P_{\nu}(\xi)P_{\nu}(-\xi)-[P_\nu(0)]^{2}}{\xi^{s}}\D \xi-\frac{[P_\nu(0)]^{2}}{s-1}\right\}.\label{eq:Pnu_Qnu_int_x_s_0sum}\end{align} As we differentiate Eq.~\ref{eq:Pnu_Qnu_int_x_s_0sum} in $s$ and send it to the limit of $ s\to 2$, we obtain the first equality in Eq.~\ref{eq:Pnu_sqr_diff_int_log_over_sqr}.

We may use the Hobson coupling formula for $ P_\nu(\xi)P_{\nu}(-\xi)$  \cite[][Eq.~11]{Zhou2013Pnu} and the Mellin inversion formula for $ P_\nu(\xi)$~\cite[][\S10.11, Eq. 3(1)]{Marichev1983} to deduce\begin{align}&
P_{\nu}(\xi)P_{\nu}(-\xi)=\int_0^{2\pi}P_\nu(-\xi^2+(1-\xi^2)\cos\phi)\frac{\D\phi}{2\pi}\notag\\={}&\frac{\sin^2(\nu\pi)}{\pi^2}\int_0^{2\pi}\left\{\frac{1}{2\pi i}\int_{c-i\infty}^{c+i\infty}\frac{[\Gamma(s)]^{2}\Gamma(\nu+1-s)\Gamma(-\nu-s)2^s\D s}{(1-\xi^2)^s(1+\cos\phi)^{s}}\right\}\frac{\D\phi}{2\pi}\notag\\={}&\frac{\sin^2(\nu\pi)}{\pi^2}\frac{1}{2\pi i}\int_{c-i\infty}^{c+i\infty}\frac{[\Gamma(s)]^{2}\Gamma(\nu+1-s)\Gamma(-\nu-s)\Gamma(\frac{1}{2}-s)}{\sqrt{\pi}\Gamma(1-s)}\frac{\D s}{(1-\xi^2)^s},\quad 0<c<\min\{\nu+1,-\nu\},-1<\xi<1.
\end{align} Consequently, we have \begin{align}
\int _{0}^{1}\xi^{t}P_{\nu}(\xi)P_{\nu}(-\xi)\D \xi={}&\frac{\sin^2(\nu\pi)}{\pi^2}\frac{\Gamma(\frac{1+t}{2})}{2\pi i}\int_{-c-i\infty}^{-c+i\infty}\frac{[\Gamma(-s)]^{2}\Gamma(\nu+1+s)\Gamma(-\nu+s)\Gamma(\frac{1}{2}+s)}{2\sqrt{\pi}\Gamma(\frac{3+t}{2}+s)}\D s\notag\\={}&\frac{1}{1+t}{_3}F_2\left(\left.\begin{array}{c}\nu+1,-\nu,\frac{2+t}{2}\\1 ,\frac{3+t}{2}\end{array}\right|1\right),\quad t>-1,\label{eq:teeth_moment_3F2_original}
\end{align}according to  Barnes's second lemma~\cite[][Eq.~4.2.2.1]{Slater}. One can then show that  the identity\begin{align}
\int _{0}^{1}\xi^{t}\{P_{\nu}(\xi)P_{\nu}(-\xi)-[P_{\nu}(0)]^{2}\}\D \xi=\frac{1}{1+t}\left\{{_3}F_2\left(\left.\begin{array}{c}\nu+1,-\nu,\frac{2+t}{2}\\1 ,\frac{3+t}{2}\end{array}\right|1\right)-[P_{\nu}(0)]^{2}\right\}\label{eq:teeth_moment_3F2}
\end{align}admits an analytic continuation to $\R t>-3$, and incorporates Eq.~\ref{eq:minus2_id'} as a special case. Exploiting the fact that \begin{align}
\left.\frac{\partial}{\partial t}\right|_{t=-2}\frac{\left(\frac{2+t}{2}\right)_n}{\left(\frac{3+t}{2}\right)_n}=\frac{\sqrt{\pi}\Gamma(n)}{2\Gamma(n+\frac{1}{2})},\quad n\in\mathbb Z_{>0},
\end{align}we may differentiate Eq.~\ref{eq:teeth_moment_3F2} in $t$ at $ t=-2$ and obtain\begin{align}&
\int_{0}^1\frac{P_{\nu}(\xi)P_{\nu}(-\xi)-[P_\nu(0)]^{2}}{\xi^{2}}\log \xi\D \xi=\int_{0}^1\frac{P_{\nu}(\xi)P_{\nu}(-\xi)-[P_\nu(0)]^{2}}{\xi^{2}}\D\xi-\frac{1}{2}\sum_{n=1}^\infty\frac{(\nu+1)_{n}(-\nu)_n}{n!(\frac{1}{2})_{n}}\frac{1}{n}\notag\\={}&[P_\nu(0)]^{2}-1-\frac{1}{2}\int_{0}^1\left[ {_2}F_1\left(\left.\begin{array}{c}\nu+1,-\nu\\\frac{1}{2}\end{array}\right|\xi\right)-1 \right]\frac{\D \xi}{\xi}=[P_\nu(0)]^{2}-1+\frac{1}{2}\int_0^{\pi/2}\frac{\cos\theta-\cos((2\nu+1)\theta)}{\sin\theta}\D\theta\notag\\={}&[P_\nu(0)]^{2}-1+\frac{1-\cos (  \nu\pi )}{\nu }+\psi ^{(0)}(\nu )+\gamma_{0} +\log 2\notag\\{}&+\frac{\pi\cos (  \nu\pi ) \tan\frac{  \nu\pi }{2}}{4}   +\frac{\cos (  \nu\pi )\left[\psi ^{(0)}\left(\frac{\nu +1}{2}\right)+\psi ^{(0)}\left(\frac{1-\nu}{2}\right)-2 \psi ^{(0)}\left(\frac{\nu }{2}\right)\right]}{4}  ,\label{eq:Pnu_refl_over_sqr_log_int}
\end{align} as one can compute the integral over $ \theta$ following the procedures in Eq.~\ref{eq:angle_int_psi}. Thus, substituting Eq.~\ref{eq:Pnu_refl_over_sqr_log_int} back into the first equality of Eq.~\ref{eq:Pnu_sqr_diff_int_log_over_sqr}, we obtain
the evaluation declared in the last step of  Eq.~\ref{eq:Pnu_sqr_diff_int_log_over_sqr}.

We can modify Eq.~\ref{eq:Pnu_Qnu_int_x_s_0sum} into a form that analytically continues to a neighborhood of $ s=0$:\begin{align}0={}&\frac{4}{\pi}\sin(\nu\pi)\sin(s\pi)\left\{\cos(\nu\pi)\int_1^\infty\left[\frac{P_\nu(\xi)Q_{\nu}(\xi)}{\xi^s}-\frac{1}{(2\nu+1)\xi^{s+1}}\right]\D\xi-\frac{\sin(\nu\pi)}{\pi}\int_1^\infty\frac{[Q_\nu(\xi)]^2}{\xi^s}\D \xi\right\}\notag\\&+\frac{4\sin(\nu\pi)\cos(\nu\pi)}{(2\nu+1)\pi}\frac{\sin(s\pi)}{s}+\sin(s\pi)\int_0^1\frac{[P_\nu(\xi)]^2-[P_\nu(-\xi)]^2}{\xi^{s}}\D \xi\notag\\{}&-2\sin(\nu\pi)[1+\cos(s\pi)]\int_{0}^1\frac{P_{\nu}(\xi)P_{\nu}(-\xi)}{\xi^{s}}\D\xi.\label{eq:Pnu_Qnu_int_x_s_0sum'}\tag{\ref{eq:Pnu_Qnu_int_x_s_0sum}$'$}\end{align}We can then differentiate Eq.~\ref{eq:Pnu_Qnu_int_x_s_0sum'} in $s$ at $ s=0$ to verify the first equality of Eq.~\ref{eq:PQ-1_over_x_int}. For $ \nu>-1/2$, the following integral formula \cite[][item~7.114.3]{GradshteynRyzhik}\begin{align}
\int_1^\infty[Q_{\nu}(\xi)]^2\D \xi=\frac{\psi^{(1)}(\nu+1)}{2\nu+1}\label{eq:Qnu_sqr_1toinf_int}
\end{align}is a standard result, and can be derived from  simple applications of Legendre differential equations. The evaluation \begin{align}
\int_0^1[P_\nu(\xi)]^2\D \xi-\int_{0}^{1}[P_\nu(-\xi)]^2\D \xi={}&2\int_0^1[P_\nu(\xi)]^2\D \xi-\int_{-1}^1[P_\nu(\xi)]^2\D \xi\notag\\={}&\frac{2\sin(\nu\pi)}{(2\nu+1)\pi}\left[ \psi^{(0)}\left( \frac{\nu+2}{2} \right) - \psi^{(0)}\left( \frac{\nu+1}{2} \right) +\frac{2\sin(\nu\pi)}{\pi}\psi^{(1)}(\nu+1)\right].\label{eq:Pnu_sqr_antisymm_0to1_int}
\end{align}follows immediately from Eqs.~\ref{eq:Pnu_sqr_0to1} and \ref{eq:P_nu_sqr}. By Kummer's  $ _3F_2$ transformation (Eq.~\ref{eq:Kummer3F2}) with  $ a_1=-\nu,a_2=\nu+1,a_3=(2+t)/2,b_1=1,b_2=(3+t)/2$, we may turn  Eq.~\ref{eq:teeth_moment_3F2_original} into\begin{align}
\int _{0}^{1}\xi^{t}P_{\nu}(\xi)P_{\nu}(-\xi)\D \xi={}&\frac{1}{1+t}{_3}F_2\left(\left.\begin{array}{c}-\nu,\nu+1,\frac{2+t}{2}\\1 ,\frac{3+t}{2}\end{array}\right|1\right)\notag\\={}&\frac{\sqrt{\pi } \Gamma (\frac{1+t}{2})}{2 \Gamma ( \frac{1}{2}-\nu) \Gamma (\frac{3+t}{2} +\nu)}{_3}F_2\left(\left.\begin{array}{c}-\nu,-\nu,-\frac{t}{2}\\1 ,\frac{1}{2}-\nu\end{array}\right|1\right),\quad t>-1,\tag{\ref{eq:teeth_moment_3F2_original}$'$}
\end{align}whose derivative at $ t=0$ yields the evaluation\begin{align}&
\int _{0}^{1}P_{\nu}(\xi)P_{\nu}(-\xi)\log \xi\D \xi=-\frac{\cos (  \nu\pi ) \left[\psi ^{(0)}\left(\frac{3}{2}+\nu \right)+\gamma _{0}+2\log 2\right]}{2(2\nu+1)}-\frac{\cos(\nu\pi)}{2(2\nu+1)}\sum_{n=1}^\infty\frac{(-\nu)_n(-\nu)_n}{n!(\frac{1}{2}-\nu)_{n}}\frac{1}{n}\notag\\={}&-\frac{\cos (  \nu\pi ) }{2(2\nu+1)}\left[\psi ^{(0)}\left(\frac{3}{2}+\nu \right)+\gamma _{0}+2\log 2+\frac{2\nu^{2}}{1-2\nu}{_4}F_3\left(\left.\begin{array}{c}1,1,1-\nu ,1-\nu\ \\[4pt]2,2,\frac{3}{2}-\nu\end{array}\right|1\right)\right]\notag\\={}&\frac{\cos (\nu\pi )}{\nu  (2 \nu +1)}-\frac{\cos (  \nu\pi ) \left[\psi ^{(0)}\left(\frac{1}{2}-\nu \right)-2 \psi ^{(0)}(\nu )+\psi ^{(0)}\left(\nu +\frac{3}{2}\right)+2 \log 2\right]}{2 (2 \nu +1)}+\frac{\psi ^{(0)}\left(\frac{\nu+1 }{2}{}\right)-\psi ^{(0)}\left(\frac{\nu+2 }{2}\right)-\pi  \sin (  \nu\pi )}{2 (2 \nu +1)}.\label{eq:PnuPnu_logx_int}
\end{align} Here, in Eq.~\ref{eq:PnuPnu_logx_int},  the infinite series in $n$ matches the expansion of a special $_4F_3 $, which in turn, can be computed from a comparison of Eqs.~\ref{eq:PnuPnu_xx_int_closed_form} and \ref{eq:PnuPnu_xx_int_4F3}. Combining the results in Eqs.~\ref{eq:Qnu_sqr_1toinf_int}--\ref{eq:PnuPnu_logx_int}, one can confirm the last step of~Eq.~\ref{eq:PQ-1_over_x_int}. \qedhere       \end{enumerate}\end{proof}\subsection{Integral Representations  of Weight-4 Renormalized Automorphic Green's Functions\label{subsec:int_repn_G2_GZ_rn}}The   automorphic Green's function $ G_s^{\mathfrak H/\overline {\varGamma}}(z,z')$, $\R s>1$ tends to infinity, as $\hat\gamma z'\to z $ for a certain $ \hat\gamma\in\varGamma$. It is possible to appropriately subtract the logarithmic divergence of the automorphic Green's function on the diagonal of $(\varGamma\backslash\mathfrak H)\times(\varGamma\backslash\mathfrak H) $, and define the remaining finite part as an arithmetically meaningful ``self-energy'' $ G_s^{\mathfrak H/\overline{\varGamma}}(z)$. When $ \varGamma=\varGamma_0(N)$ is the Hecke congruence group, such a subtraction scheme is the Gross--Zagier renormalization~\cite{GrossZagier1985,GrossZagierI}.  In~\cite[][Chap.~II, \S5]{GrossZagierI}, the ``automorphic self-energy'' (or ``renormalized automorphic Green's function'') $ G_s^{\mathfrak H/\overline{\varGamma}_0(N)}(z)$ was prescribed as follows: \begin{align}
G_s^{\mathfrak H/\overline{\varGamma}_0(N)}(z)={}&-2\sum_{\hat  \gamma\in\overline {\varGamma}_0(N),\hat \gamma z\neq z}Q_{s-1}
\left( 1+\frac{\vert z -\hat  \gamma z\vert ^{2}}{2\I z\I(\hat\gamma z)} \right)\notag\\{}&-2\left\{ \log\left\vert 4\pi y[\eta(z)]^4 \right\vert+\psi^{(0)}(1)-\psi^{(0)}(s) \right\}m^{\overline {\varGamma}_0(N)}_{z},\label{eq:GZ_rn_defn}
\end{align}where $ \eta(\cdot)$ is the Dedekind eta function, $ \psi^{(0)}(\cdot)$ is the digamma function, and   $ m^{\overline {\varGamma}}_{z}:=\#\{\hat \gamma\in\overline{\varGamma}|\hat \gamma z=z\}$ is  the number of inequivalent transformations in the projective  congruence subgroup $ \overline{\varGamma}$ that fix the point $ z$. The Gross--Zagier algebraicity conjecture for $ G_s^{\mathfrak H/\overline{\varGamma}_0(N)}(z)$ can be stated as \cite[][p.~317, Conjecture~4.4]{GrossZagierI}\begin{align}
\exp \left[ (\I z)^{k-2}G^{\mathfrak H/\overline {\varGamma}_0(N)}_{k/2}(z)\right]\in\overline{\mathbb Q},\quad \text{if }\dim\mathcal S_{k}(\varGamma_0(N))=0\text{ and }[\mathbb Q(z):\mathbb Q]=2.\label{eq:GZ_conj_rn}
\end{align}

As we have $ m_z^{\overline{\varGamma}}=1$ (\textit{i.e.}~$z$ is a non-elliptic point) for almost every $ z\in\mathfrak H$, we may reformulate the Gross--Zagier renormalization (Eq.~\ref{eq:GZ_rn_defn}) for weight-4 automorphic Green's functions as follows:\begin{align}
G_2^{\mathfrak H/\overline{\varGamma}_0(N)}(z)={}&\lim_{z'\to z}\left[G_2^{\mathfrak H/\overline{\varGamma}_0(N)}(z,z')+2Q_{1}
\left( 1+\frac{\vert z -z'\vert ^{2}}{2yy'} \right)\right]-2\left\{ \log\left\vert 4\pi y[\eta(z)]^4 \right|-1 \right\}\notag\\={}&\lim_{z'\to z}\left[G_2^{\mathfrak H/\overline{\varGamma}_0(N)}(z,z')-2\log|z-z'|\right]-2\log\left\vert 2\pi[\eta(z)]^4 \right\vert,\quad \text{a.e. }z\in\mathfrak H.\label{eq:G2_rn_non_ell_defn}
\end{align}

In the next proposition, we derive integral representations of weight-4 Gross--Zagier renormalized Green's functions (Eq.~\ref{eq:G2_rn_non_ell_defn}) for  non-elliptic points.\begin{proposition}[Integral Representations for $ G_2^{\mathfrak H/\overline{\varGamma}_0(N)}(z)$, $ N\in\{1,2,3,4\}$ where $ m_z^{\overline{\varGamma}_0(N)}=1$]\label{prop:G2_rn_non_ell_pts}With the notations $ \rho_{2,-1/6}$, $ \varrho_{2,-1/4}$, $ \varrho_{2,-1/3}$ and $ \varrho_{2,-1/2}$ defined  in Proposition \ref{prop:G2_HeckeN_Pnu}, we have {\allowdisplaybreaks\begin{align}&
G_2^{\mathfrak H/PSL(2,\mathbb Z)}(z)\notag\\={}&{}\log\frac{2^{10}3^6}{|j(z)|^{5/3}\left\vert 1-\sqrt{\frac{j(z)-1728}{j(z)}} \right\vert^2}+\R\int_{\sqrt{(j(z)-1728)/{j(z)}}}^1\left\{ \vphantom{\left[\frac{iP_{-1/6}(-\xi)}{P_{-1/6}(\xi)}-\vphantom{\overline{\frac{\sqrt{j}}{}}}\frac{iP_{-1/6}(-\sqrt{(j(z)-1728)/j(z)})}{P_{-1/6}(\sqrt{(j(z)-1728)/j(z)})}\right]} \right.\frac{2}{\smash[b]{\xi-\sqrt{\frac{j(z)-1728}{j(z)}}}}+\frac{\pi[P_{-1/6}(\xi)]^2\rho_{2,-1/6}(\xi|z)}{\smash[b]\I \frac{iP_{-1/6}(-\sqrt{(j(z)-1728)/j(z)})}{P_{-1/6}(\sqrt{(j(z)-1728)/j(z)})}}\times\notag\\&\times\left.\left[\frac{iP_{-1/6}(-\xi)}{P_{-1/6}(\xi)}-\vphantom{\overline{\frac{\sqrt{j}}{}}}\frac{iP_{-1/6}(-\sqrt{(j(z)-1728)/j(z)})}{P_{-1/6}(\sqrt{(j(z)-1728)/j(z)})}\right]\left[\frac{iP_{-1/6}(-\xi)}{P_{-1/6}(\xi)}-\overline{\left(\frac{iP_{-1/6}(-\sqrt{(j(z)-1728)/j(z)})}{P_{-1/6}(\sqrt{(j(z)-1728)/j(z)})}\right)}\right]\right\}\D \xi\notag\\&+\frac{\pi}{\I \frac{iP_{-1/6}(-\sqrt{(j(z)-1728)/j(z)})}{P_{-1/6}(\sqrt{(j(z)-1728)/j(z)})}
}\R\int_{-1}^1[P_{-1/6}(-\xi)]^2\rho_{2,-1/6}(\xi|z)\D \xi+2,\quad \emph{a.e. }z\in\mathfrak H;\label{eq:G2_PSL2Z_z_rn}\\&G_2^{\mathfrak H/\overline{\varGamma}_0(2)}(z)\notag\\={}&\frac13\log\frac{2^{4}}{|1-2\lambda(2z+1)|^{6}|1-\lambda(2z+1)||\lambda(2z+1)|}+\R\int_{1-2\alpha_{2}(z)}^1\left\{ \vphantom{\left[\frac{iP_{-1/6}(-\xi)}{P_{-1/6}(\xi)}-\vphantom{\overline{\frac{iP_{-1/3}(2\alpha(z)-1)}{P_{-1/3}(1-2\alpha(z))}}}\right]} \right.\frac{2}{\smash[b]{\xi-1+2\alpha_2(z)}}+\frac{\pi[P_{-1/4}(\xi)]^2\varrho_{2,-1/4}(\xi|z)}{\sqrt{2}\I\frac{iP_{-1/4}(2\alpha_2(z)-1)}{P_{-1/4}(1-2\alpha_2(z))}}\times\notag\\&\times\left.\left[\frac{iP_{-1/4}(-\xi)}{P_{-1/4}(\xi)}-\vphantom{\overline{\frac{1}{}}}\frac{iP_{-1/4}(2\alpha_2(z)-1)}{P_{-1/4}(1-2\alpha_2(z))}\right]\left[\frac{iP_{-1/4}(-\xi)}{P_{-1/4}(\xi)}-\overline{\left(\frac{iP_{-1/4}(2\alpha_2(z)-1)}{P_{-1/4}(1-2\alpha_2(z))}\right)}\right]\right\}\D \xi\notag\\&+\frac{\pi}{\sqrt{2}\I \frac{iP_{-1/4}(2\alpha_2(z)-1)}{P_{-1/4}(1-2\alpha_2(z))}
}\R\int_{-1}^1[P_{-1/4}(-\xi)]^2\varrho_{2,-1/4}(\xi|z)\D \xi+2,\quad \emph{a.e. }z\in\mathfrak H;\label{eq:G2_Hecke2_z_rn}\\&G_2^{\mathfrak H/\overline{\varGamma}_0(3)}(z)\notag\\={}&\log\frac{3|1-\alpha_{3}(z)|}{|\alpha_{3}(z)|^{1/3}}+\R\int_{1-2\alpha_{3}(z)}^1\left\{ \vphantom{\left[\frac{iP_{-1/6}(-\xi)}{P_{-1/6}(\xi)}-\vphantom{\overline{\frac{iP_{-1/3}(2\alpha(z)-1)}{P_{-1/3}(1-2\alpha(z))}}}\right]} \right.\frac{2}{\smash[b]{\xi-1+2\alpha_{3}(z)}}+\frac{\pi[P_{-1/3}(\xi)]^2\varrho_{2,-1/3}(\xi|z)}{\sqrt{3}\I\frac{iP_{-1/3}(2\alpha_{3}(z)-1)}{P_{-1/3}(1-2\alpha_{3}(z))}}\times\notag\\&\times\left.\left[\frac{iP_{-1/3}(-\xi)}{P_{-1/3}(\xi)}-\vphantom{\overline{\frac{1}{}}}\frac{iP_{-1/3}(2\alpha_{3}(z)-1)}{P_{-1/3}(1-2\alpha_{3}(z))}\right]\left[\frac{iP_{-1/3}(-\xi)}{P_{-1/3}(\xi)}-\overline{\left(\frac{iP_{-1/3}(2\alpha_{3}(z)-1)}{P_{-1/3}(1-2\alpha_{3}(z))}\right)}\right]\right\}\D \xi\notag\\&+\frac{\pi}{\sqrt{3}\I \frac{iP_{-1/3}(2\alpha_{3}(z)-1)}{P_{-1/3}(1-2\alpha_{3}(z))}
}\R\int_{-1}^1[P_{-1/3}(-\xi)]^2\varrho_{2,-1/3}(\xi|z)\D \xi+2,\quad \emph{a.e. }z\in\mathfrak H;\label{eq:G2_Hecke3_z_rn}\\&G_2^{\mathfrak H/\overline{\varGamma}_0(4)}(z)\notag\\={}&\frac13\log\frac{2^{4}|1-\lambda(2z)|^2}{|\lambda(2z)|}+\R\int_{1-2\lambda(2z)}^1\left\{ \vphantom{\left[\frac{iP_{-1/6}(-\xi)}{P_{-1/6}(\xi)}-\vphantom{\overline{\frac{iP_{-1/3}(2\alpha(z)-1)}{P_{-1/3}(1-2\alpha(z))}}}\right]} \right.\frac{2}{\smash[b]{\xi-1+2\lambda(2z)}}+\frac{\pi[P_{-1/2}(\xi)]^2\varrho_{2,-1/2}(\xi|z)}{2\I\frac{iP_{-1/2}(2\lambda(2z)-1)}{P_{-1/2}(1-2\lambda(2z))}}\times\notag\\&\times\left.\left[\frac{iP_{-1/2}(-\xi)}{P_{-1/2}(\xi)}-\vphantom{\overline{\frac{1}{}}}\frac{iP_{-1/2}(2\lambda(2z)-1)}{P_{-1/2}(1-2\lambda(2z))}\right]\left[\frac{iP_{-1/2}(-\xi)}{P_{-1/2}(\xi)}-\overline{\left(\frac{iP_{-1/2}(2\lambda(2z)-1)}{P_{-1/2}(1-2\lambda(2z))}\right)}\right]\right\}\D \xi\notag\\&+\frac{\pi}{2\I \frac{iP_{-1/2}(2\lambda(2z)-1)}{P_{-1/2}(1-2\lambda(2z))}
}\R\int_{-1}^1[P_{-1/2}(-\xi)]^2\varrho_{2,-1/2}(\xi|z)\D \xi+2,\quad \emph{a.e. }z\in\mathfrak H.\label{eq:G2_Hecke4_z_rn}
\end{align}}\end{proposition}\begin{proof}We first give a detailed derivation for Eq.~\ref{eq:G2_PSL2Z_z_rn}. As $ \xi\to\sqrt{(j(z')-1728)/{j(z')}}$ and $ z'\to z$, one has{\allowdisplaybreaks\begin{align}&\frac{\pi[P_{-1/6}(\xi)]^2\rho_{2,-1/6}(\xi|z)}{\I \frac{iP_{-1/6}(-\sqrt{(j(z')-1728)/j(z')})}{P_{-1/6}(\sqrt{(j(z')-1728)/j(z')})}}\left[\frac{iP_{-1/6}(-\xi)}{P_{-1/6}(\xi)}-
\vphantom{\overline{\frac{\sqrt{j}}{}}}\frac{iP_{-1/6}(-\sqrt{(j(z')-1728)/j(z')})}{P_{-1/6}(\sqrt{(j(z')-1728)/j(z')})}\right]\times\notag\\&\times
\left[\frac{iP_{-1/6}(-\xi)}{P_{-1/6}(\xi)}-\overline{\left(\frac{iP_{-1/6}(-\sqrt{(j(z')-1728)/j(z')})}{P_{-1/6}(\sqrt{(j(z')-1728)/j(z')})}\right)}
\right]\notag\\\sim&-2\pi[P_{-1/6}(\xi)]^2\rho_{2,-1/6}(\xi|z)\left( \xi-\sqrt{\frac{j(z)-1728}{j(z)}}+ \sqrt{\frac{j(z)-1728}{j(z)}}-\sqrt{\frac{j(z')-1728}{j(z')}}\right)\frac{\D}{\D\xi}\frac{P_{-1/6}(-\xi)}{P_{-1/6}(\xi)}\notag\\\sim&-\frac{2}{\xi-\sqrt{\frac{j(z)-1728}{j(z)}}}-\left(\sqrt{\frac{j(z)-1728}{j(z)}}-\sqrt{\frac{j(z')-1728}{j(z')}}\right)\frac{2\rho_{2,-1/6}(\xi|z)}{1-\xi^{2}}+O(1),\label{eq:P_sixth_rn_asym1}\end{align}}and
\begin{align}&-\lim_{z'\to z}\left(\sqrt{\frac{j(z)-1728}{j(z)}}-\sqrt{\frac{j(z')-1728}{j(z')}}\right)\int_{\sqrt{(j(z')-1728)/{j(z')}}}^1\frac{2\rho_{2,-1/6}(\xi|z)}{1-\xi^{2}}\D \xi\notag\\={}&-\lim_{z'\to z}\left(\sqrt{\frac{j(z)-1728}{j(z)}}-\sqrt{\frac{j(z')-1728}{j(z')}}\right)\int_{\sqrt{(j(z')-1728)/{j(z')}}}^1\frac{2\D \xi}{\left[ \xi-\sqrt{\frac{j(z)-1728}{j(z)}} \right]^2}=2.\label{eq:P_sixth_rn_asym2}\end{align}With the asymptotic behavior displayed above in Eqs.~\ref{eq:P_sixth_rn_asym1} and \ref{eq:P_sixth_rn_asym2}, we may  insert the integral representation for $G_2^{\mathfrak H/PSL(2,\mathbb Z)}(z,z') $ (Eq.~\ref{eq:G2_z_z'_arb}) into the following expression\begin{align*}&
G_2^{\mathfrak H/PSL(2,\mathbb Z)}(z,z')-2\log\left\vert \frac{\sqrt{\frac{j(z')-1728}{j(z')}}-\sqrt{\frac{j(z)-1728}{j(z)}}}{1-\sqrt{\frac{j(z)-1728}{j(z)}}} \right\vert\notag\\={}&G_2^{\mathfrak H/PSL(2,\mathbb Z)}(z,z')+\R\int_{\sqrt{(j(z')-1728)/{j(z')}}}^1\frac{2\D\xi}{\xi-\sqrt{\frac{j(z)-1728}{j(z)}}}
,\end{align*}combine the integrals with respect to $\xi\in(\sqrt{(j(z')-1728)/j(z')},1)$, and take the limit $ z'\to z$.  So far, we have accounted for the two integrals on the right-hand side of Eq.~\ref{eq:G2_PSL2Z_z_rn}, and  the trailing constant 2. As we note that\begin{align}&6\lim_{z'\to z}\log\left\vert \frac{\sqrt{\frac{j(z')-1728}{j(z')}}-\sqrt{\frac{j(z)-1728}{j(z)}}}{z'-z} \right\vert\notag\\={}&36\log2+18\log3+6\log\pi+24\log|\eta(z)|-5\log|j(z)|,\quad \text{a.e. }z\in\mathfrak H,\label{eq:log_deriv_P_sixth_inv}\end{align}the  leading logarithmic term on the right-hand side of Eq.~\ref{eq:G2_PSL2Z_z_rn} is also explained.

The proofs of Eqs.~\ref{eq:G2_Hecke2_z_rn}--\ref{eq:G2_Hecke4_z_rn} are essentially similar to that of  Eq.~\ref{eq:G2_PSL2Z_z_rn}. It would suffice to establish analogs of Eq.~\ref{eq:log_deriv_P_sixth_inv} for  the modular lambda function $ \lambda(z)$, the Ramanujan cubic invariant  $ \alpha_{3}(z)$, and the expression $ \alpha_{2}(\frac{z-1}{2})=[1-2\lambda(z)]^{-2}$, applicable to $\text{a.e. }z\in\mathfrak H $:{\allowdisplaybreaks
\begin{align}\lim_{z'\to z}\log\left\vert \frac{\lambda(z')-\lambda(z)}{z'-z} \right\vert={}&\frac{4\log 2}{3}+\log\pi+4\log|\eta(z)|+\frac{2\log|1-\lambda(z)|}{3}+\frac{2\log|\lambda(z)|}{3}\notag\\={}&\frac{2\log 2}{3}+\log\pi+4\log\left|\eta\left(\frac{z}{2}\right)\right|+\frac{\log|1-\lambda(z)|}{3}+\frac{5\log|\lambda(z)|}{6};\\
\lim_{z'\to z}\log\left\vert \frac{\alpha_{3}(z')-\alpha_{3}(z)}{z'-z} \right\vert={}&\log 2+\frac{\log3}{2}+\log\pi+4\log|\eta(z)|+\frac{5\log|\alpha_{3}(z)|}{6}+\frac{\log|1-\alpha_{3}(z)|}{2};\\\lim_{z'\to z}\log\left\vert \frac{\alpha_{2}(\frac{z'-1}{2})-\alpha_{2}(\frac{z-1}{2})}{z'-z} \right\vert={}&\frac{10\log 2}{3}+\log\pi+4\log|\eta(z)|+\frac{2\log|1-\lambda(z)|}{3}\notag\\{}&+\frac{2\log|\lambda(z)|}{3}-3\log|1-2\lambda(z)|\notag\\={}&\frac{8\log 2}{3}+\log\pi+4\log\left|\eta\left(\frac{z-1}{2}\right)\right|+\frac{5\log|1-\lambda(z)|}{6}\notag\\{}&+\frac{5\log|\lambda(z)|}{6}-3\log|1-2\lambda(z)|.
\end{align}}These formulae, of course, can be verified using Ramanujan's differentiation formula (Eq.~\ref{eq:P_nu_ratio_deriv}) along with  the relations between the Weierstra{\ss} discriminant $ \Delta(z)=[\eta(z)]^{24}$  and fractional degree Legendre functions (Eqs.~\ref{eq:alpha_N_ratio_ids}, \ref{eq:P_sixth_eta} and \ref{eq:P_nu_sqr_E2_diff}).    \end{proof}
\begin{remark}\label{rmk:spec_G2_rn}In the proposition above, the constraint ``a.e.~$z\in\mathfrak H $'' has two connotations: (i)~the point $z$ has to be non-elliptic such that $m_z^{\overline \varGamma}=1 $; (ii)~the paths of integration (being straight lines in the complex $ \xi$-plane  connecting the ends of designated intervals) do not intersect the branch cut of $ P_\nu(\xi)$ or $ P_\nu(-\xi)$. Therefore, the integral representations in Eqs.~\ref{eq:G2_Hecke2_z_rn}--\ref{eq:G2_Hecke4_z_rn} are applicable to $ G_2^{\mathfrak H/\overline\varGamma_0(N)}(i/\sqrt{N})$ for $ N\in\{2,3,4\}$. One can also exploit Eq.~\ref{eq:Pnu_sqr_diff_int_log_over_sqr} to evaluate these three integrals exactly: {\allowdisplaybreaks\begin{align}&G_2^{\mathfrak H/\overline{\varGamma}_0(2)}\left( \frac{i}{\sqrt{2}} \right)\notag\\={}&\log2+\int_{0}^1\left\{\frac{2}{\smash[b]{\xi}}+\frac{\pi[P_{-1/4}(\xi)]^2}{\sqrt{2}\xi^{2}}-\frac{\pi[P_{-1/4}(-\xi)]^2}{\sqrt{2}\xi^{2}}\right\}\D \xi+\frac{\pi}{\sqrt{2}
}\R\int_{-1}^1\frac{[P_{-1/4}(\xi)]^2}{\xi^{2}}\D \xi+2\notag=-3\log2;\\&
G_2^{\mathfrak H/\overline{\varGamma}_0(3)}\left( \frac{i}{\sqrt{3}} \right)\notag\\={}&\log 3-\frac23\log2+\int_{0}^1\left\{\frac{2}{\smash[b]{\xi}}+\frac{\pi[P_{-1/3}(\xi)]^2}{\sqrt{3}\xi^{2}}-\frac{\pi[P_{-1/3}(-\xi)]^2}{\sqrt{3}\xi^{2}}\right\}\D \xi+\frac{\pi}{\sqrt{3}
}\R\int_{-1}^1\frac{[P_{-1/3}(\xi)]^2}{\xi^{2}}\D \xi+2\notag\\={}&\frac{4}{3}\log2-2\log3;\\&G_2^{\mathfrak H/\overline{\varGamma}_0(4)}\left( \frac{i}{\sqrt{4}} \right)\notag\\={}&\log2+\int_{0}^1\left\{\frac{2}{\smash[b]{\xi}}+\frac{\pi[P_{-1/2}(\xi)]^2}{\sqrt{4}\xi^{2}}-\frac{\pi[P_{-1/2}(-\xi)]^2}{\sqrt{4}\xi^{2}}\right\}\D \xi+\frac{\pi}{\sqrt{4}
}\R\int_{-1}^1\frac{[P_{-1/2}(\xi)]^2}{\xi^{2}}\D \xi+2=-\log2.\label{eq:G2Hecke4_rn_Fricke_pt}
\end{align}}One may wish to compare the last integral identity with the generic formula for $ G_2^{\mathfrak H/\overline \varGamma_0(4)}(z)$ (to be derived in Proposition~\ref{prop:G2_Hecke4_z_KZ_rn}). \eor\end{remark}

  The fundamental domains of the  congruence subgroups $ \varGamma_0(1)=SL(2,\mathbb Z)$, $ \varGamma_0(2)$ and $ \varGamma_0(3)$ contain elliptic points. We need to single out these cases for the evaluations of renormalized Green's functions $ G_2^{\mathfrak H/\overline{\varGamma}}(z)$.
(It is worth noting that the renormalized Green's function $ G_2^{\mathfrak H/\overline{\varGamma}}(z)$ is \textit{not} continuous at the elliptic points: it diverges logarithmically as $z$ approaches an elliptic point; it nonetheless assumes a well-defined finite value at any elliptic point.) In the following proposition, we shall derive integral representations for   $ G_2^{\mathfrak H/\overline{\varGamma}}(z)$ ($ \varGamma=\varGamma_0(1)$, $ \varGamma_0(2)$ and $ \varGamma_0(3)$) at elliptic points $z$, and evaluate these integrals in closed form. \begin{proposition}[Evaluations of  $ G_2^{\mathfrak H/\overline{\varGamma}_0(N)}(z)$, $ N\in\{1,2,3\}$ where $ m_z^{\overline{\varGamma}_0(N)}>1$]We have the following identities:{\allowdisplaybreaks\begin{align}
G_2^{\mathfrak H/PSL(2,\mathbb Z)}(i)={}&2\log3+\int_{0}^1\left\{\frac{4}{\xi}+\frac{2\pi[P_{-1/6}(\xi)]^2}{\xi^{2}}-\frac{2\pi[P_{-1/6}(-\xi)]^2}{\xi^{2}}\right\}\D \xi-\frac{40 G}{\pi}+4\notag\\={}&-4(\log 2+\log3);\label{eq:G2_PSL2Z_i_rn}\\
G_2^{\mathfrak H/PSL(2,\mathbb Z)}\left( \frac{1+i\sqrt{3}}{2} \right)={}&3\log\frac{2^2}{3}+\int_{1}^\infty\left\{\frac{4}{\xi}-\frac{8 P_{-1/6}(\xi)Q_{-1/6}(\xi)}{3}-\frac{8[ Q_{-1/6}(\xi)]^{2}}{3\sqrt{3}\pi}\right\}\D \xi-\frac{30\sqrt{3} L(2,\chi_{-3})}{\pi}+6\notag\\={}&-3(2\log2+\log3);\label{eq:G2_PSL2Z_ell3_rn}\\G_2^{\mathfrak H/\overline{\varGamma}_0(2)}\left( \frac{i-1}{2} \right)={}&-2\log2+\int_{1}^\infty\left\{\frac{2}{\xi}- P_{-1/4}(\xi)Q_{-1/4}(\xi)-\frac{[ Q_{-1/4}(\xi)]^{2}}{\pi}\right\}\D \xi-\frac{16G}{\pi}+4\notag\\={}&-4\log2;\label{eq:G2_Hecke2_ell2_rn}\\G_2^{\mathfrak H/\overline{\varGamma}_0(3)}\left( \frac{3+i\sqrt{3}}{6} \right)={}&\log\frac{1}{2^{2}3^3}+\int_{1}^\infty\left\{\frac{2}{\xi}- \frac{2P_{-1/3}(\xi)Q_{-1/3}(\xi)}{3}-\frac{2[ Q_{-1/3}(\xi)]^{2}}{\sqrt{3}\pi}\right\}\D \xi-\frac{9\sqrt{3} L(2,\chi_{-3})}{\pi}+6\notag\\={}&-3\log3,\label{eq:G2_Hecke3_ell3_rn}
\end{align}}where \[G\equiv L(2,\chi_{-4}):=\sum_{n=0}^\infty\frac{(-1)^n}{(2n+1)^2},\quad L(2,\chi_{-3}):=\sum_{n=0}^\infty\left[\frac{1}{(3n+1)^2}-\frac{1}{(3n+2)^2}\right].\]\end{proposition}\begin{proof}Recalling   that $ m_i^{PSL(2,\mathbb Z)}=2$, we may compute\begin{align*}G_2^{\mathfrak H/PSL(2,\mathbb Z)}(i)={}&\lim_{z\to i}\left[G_2^{\mathfrak H/PSL(2,\mathbb Z)}(z,i)-4\log|z-i|\right]-4\log\left\vert 2\pi[\eta(i)]^4 \right|\end{align*} using the Legendre--Ramanujan form (Eq.~\ref{eq:G2_z_z'_arb}) of the integral representation for $ G_2^{\mathfrak H/PSL(2,\mathbb Z)}(z,i)$ (see Eq.~\ref{eq:G2_PSL2Z_z_i}), which results in the integral formula in Eq.~\ref{eq:G2_PSL2Z_i_rn}. We can subsequently evaluate the integral by considering Eq.~\ref{eq:Pnu_sqr_diff_int_log_over_sqr} in the special case where $ \nu=-1/6$, in parallel to what we did in Remark~\ref{rmk:spec_G2_rn}.

We can modify Eq.~\ref{eq:G2_z_z'_arb}  as follows:\begin{align}&G_{2}^{\mathfrak H/PSL(2,\mathbb Z)}(z,e^{\pi i/3})\notag\\={}&2\log\left\vert \frac{j(z)}{j(z)-1728} \right\vert-4\R\int_{\sqrt{\frac{j(z)-1728}{j(z)}}}^1\left\{\frac{1}{\xi}+\frac{\pi [P_{-1/6}(\xi)]^2}{3y}\left[ \frac{iP_{-1/6}(-\xi)}{P_{-1/6}(\xi)} -z\right]\left[ \frac{iP_{-1/6}(-\xi)}{P_{-1/6}(\xi)} -\overline{z}\right]\right\}\D \xi\notag\\&-\frac{4\pi}{3y}\R\int_{-1}^1[P_{-1/6}(\xi)]^2\D \xi,\quad |z|\geq1,-\frac12<\R z\leq\frac{1}{2},z\neq \frac{1}{2}+i\frac{\sqrt{3}}{2}.\label{eq:G2PSL2Z_ell3_alt}\end{align}One may recall an integral formula for $ L(2,\chi_{-3})$ from Eq.~\ref{eq:L2chi3}, subtract $ 6\log|z-e^{\pi i/3}|+6\log|2\pi[\eta(e^{i \pi/3})]^4|$ from Eq.~\ref{eq:G2PSL2Z_ell3_alt}, and take the limit $ z\to e^{\pi i/3}$. This leads to the integral representation in Eq.~\ref{eq:G2_PSL2Z_ell3_rn}, where we have already used Eq.~\ref{eq:Pnu_Qnu_rln} to recast the integrand in terms of two real-valued functions $ P_\nu(t),Q_\nu(t),t>1$.  One can then evaluate the integral representation in   Eq.~\ref{eq:G2_PSL2Z_ell3_rn} with the help of the integral formulae in Eqs.~\ref{eq:PQ-1_over_x_int} and \ref{eq:Qnu_sqr_1toinf_int}, as well as   the polygamma identities in Eqs.~\ref{eq:psi1_id1} and \ref{eq:psi1_id2}.

By  Eq.~\ref{eq:G2Hecke234_Pnu}, we can verify  the first equality in Eq.~\ref{eq:G2_Hecke2_ell2_rn} after some routine computations similar to what was described in the last paragraph. The integral formulae in Eqs.~\ref{eq:PQ-1_over_x_int} and \ref{eq:Qnu_sqr_1toinf_int} will then result in the closed-form expression for $ G_2^{\mathfrak H/\overline{\varGamma}_0(2)}((i-1)/2)$. Alternatively,  we can exploit the evaluation in Eq.~\ref{eq:G2Hecke4_rn_Fricke_pt} as follows:\begin{align*}G_2^{\mathfrak H/\overline{\varGamma}_0(2)}\left(\frac{z-1}{2},\frac{i-1}{2} \right)={}&2G_2^{\mathfrak H/\overline{\varGamma}(2)}(z,i)=2G_2^{\mathfrak H/\overline{\varGamma}_0(4)}\left(\frac{z}{2},\frac{i}{2} \right)\notag\\={}&2G_2^{\mathfrak H/\overline{\varGamma}_0(4)}\left(\frac{i}{\sqrt{4}} \right)+4\log\left\vert \frac{z-i}{2} \right\vert+4\log\left\vert 2\pi\left[ \eta\left( \frac{i}{2} \right) \right]^4 \right\vert+o(1)\notag\\\Longrightarrow\qquad G_2^{\mathfrak H/\overline{\varGamma}_0(2)}\left(\frac{i-1}{2} \right)={}&2G_2^{\mathfrak H/\overline{\varGamma}_0(4)}\left(\frac{i}{\sqrt{4}} \right)+16\log\left\vert \frac{\eta(\frac{i}{2})}{\eta(\frac{i-1}{2})} \right\vert=-2\log2+16\log\frac{1}{\sqrt[8]{2}}=-4\log2.\end{align*}

The proof of Eq.~\ref{eq:G2_Hecke3_ell3_rn} is similar to that of Eq.~\ref{eq:G2_PSL2Z_ell3_rn}.\end{proof}

 \subsection{A Few Applications of Ramanujan Rotations and Jacobi Elliptic Functions\label{subsec:Ramanujan_Jacobi}}In the following lemma, we present a geometric transformation that will facilitate the evaluation of certain multiple elliptic integrals by ``separating the variables''.\begin{lemma}[Ramanujan Rotations]\label{lm:R_rot}Let $ \mathcal R(s,t),s\in[0,1],t\in[0,+\infty)$ be a bivariate function with suitable regularity, then we have the following integral identity for $ \lambda\in(0,1)$:\begin{align}&
\int_0^{\pi/2}\left[ \int_0^{\pi/2}\frac{\mathcal R(\cos\theta,\tan\phi)\D\phi}{\sqrt{\smash[b]{1-\lambda\cos^2\phi}}} \right]\frac{\D\theta}{\sqrt{\smash[b]{1-\lambda\cos^2\theta}}}\notag\\={}&\int_0^{\pi/2}\left[ \int_0^{\pi/2}\mathcal R\left( \frac{\cos\theta}{\sqrt{\smash[b]{1-\lambda\sin^2\theta\sin^2\phi}}}, \sqrt{\frac{1-\lambda(1-\sin^2\theta\cos^2\phi)}{\smash[b]{1-\lambda\sin^2\theta\sin^2\phi}}}\tan\phi\right)\frac{\D\phi}{\sqrt{\smash[b]{1-\lambda+\lambda^{2}\sin^2\theta\sin^2\phi\cos^2\phi}}} \right]\D\theta\notag\\={}&\int_{S^2_{+++}}\mathcal R\left(\frac{Z}{\sqrt{1-\lambda Y^{2}}},\sqrt{\frac{1-\lambda(1-X^{2})}{1-\lambda Y^{2}}} \frac{Y}{X}\right)\frac{\D\sigma}{\sqrt{(1-\lambda)(1-Z^2)+\lambda^{2}Y^2X^2}}.\label{eq:Ramanujan_rotation}\end{align}Here in the last step, the integral is carried out on   the $ (+,+,+)$-octant of the unit sphere $ S^2_{+++}:=\{(X,Y,Z)\in\mathbb R^3|X^2+Y^2+Z^2=1,X>0,Y>0,Z>0\}$,  equipped with the standard  spherical coordinates $ (X,Y,Z)=(\sin\theta\cos\phi,\sin\theta\sin\phi,\cos\theta)$ and surface element $ \D\sigma=\sin\theta\D\theta\D\phi$.   \end{lemma}\begin{proof}The following proof is a modest extension of our previous work  \cite[][Proposition~4.3]{Zhou2013Pnu}, which  handled a special case (where the function $ \mathcal R(s,t)\equiv 1$ was held a constant) in an effort to recover the geometric motivations behind Entry~7(x) in Chapter 17 of Ramanujan's second notebook \cite[][pp.~110--111]{RN3}.

With the spherical coordinates and a rotation of the coordinate axes, it is straightforward to verify that \begin{align}&
\int_0^{\pi/2}\left[ \int_0^{\pi/2}\frac{ \mathcal R(\cos\theta,\tan\phi)\D\phi}{\sqrt{\smash[b]{1-\lambda\cos^2\theta}}\sqrt{\smash[b]{1-\lambda\cos^2\phi}}} \right]\D\theta\notag\\={}&\int_{S^{2}_{+++}}\frac{ \mathcal R(Z,Y/X)}{\sqrt{1-\lambda Z^{2}}\sqrt{1-Z^2-\lambda X^2}}\D\sigma=\int_{S^{2}_{+++}}\frac{ \mathcal R(X,Y/Z)}{\sqrt{1-\lambda X^{2}}\sqrt{1-X^2-\lambda Z^2}}\D\sigma\notag\\={}&\int_0^{\pi/2}\left[ \int_0^{\pi/2}\frac{ \mathcal R(\sin\theta\cos\phi,\tan\theta\sin\phi)\D\phi}{\sqrt{\smash[b]{1-\lambda\sin^{2}\theta\cos^{2}\phi}}\sqrt{\smash[b]{1-\sin^{2}\theta\cos^{2}\phi-\lambda\cos^2\theta}}} \right]\sin\theta\D\theta.\label{eq:Ramanujan_rotation1}
\end{align}Now, substituting $ \phi=\arctan(\sqrt{1-\lambda\sin^2\theta}\tan\psi)$, we can verify that\begin{align}&
\int_0^{\pi/2}\frac{ \mathcal R(\sin\theta\cos\phi,\tan\theta\sin\phi)\D\phi}{\sqrt{\smash[b]{1-\lambda\sin^{2}\theta\cos^{2}\phi}}\sqrt{\smash[b]{1-\sin^{2}\theta\cos^{2}\phi-\lambda\cos^2\theta}}}\notag\\={}&\int_0^{\pi/2}\frac{ \mathcal R\left(\frac{\sin\theta\cos\psi}{\sqrt{\smash[b]{1-\lambda\sin^{2}\theta\sin^2\psi}}},\sqrt{\frac{1-\lambda\sin^2\theta}{\smash[b]{1-\lambda\sin^{2}\theta\sin^2\psi}}} \frac{\sin\theta\sin\psi}{\cos\theta} \right)\D\psi}{\sqrt{\smash[b]{(1-\lambda)(1-\sin^{2}\theta\cos^{2}\psi)+\lambda^2\cos^2\theta\sin^2\theta\sin^2\psi}}}.\label{eq:Ramanujan_rotation2}
\end{align} Therefore, we may use Eq.~\ref{eq:Ramanujan_rotation2} to convert Eq.~\ref{eq:Ramanujan_rotation1} into\begin{align*}\int_0^{\pi/2}\left[ \int_0^{\pi/2}\frac{ \mathcal R(\cos\theta,\tan\phi)\D\phi}{\sqrt{\smash[b]{1-\lambda\cos^2\theta}}\sqrt{\smash[b]{1-\lambda\cos^2\phi}}} \right]\D\theta={}&
\int_{S^2_{+++}}\frac{\mathcal R\left(\frac{X}{\sqrt{1-\lambda Y^{2}}}, \sqrt{\frac{1-\lambda(1-Z^{2})}{\smash[b]{1-\lambda Y^{2}}}} \frac{Y{}}{Z}\right)}{\sqrt{(1-\lambda)(1-X^2)+\lambda^{2}Y^2Z^2}}\D\sigma\notag\\={}&\int_{S^2_{+++}}\frac{\mathcal R \left(\frac{Z}{\sqrt{1-\lambda Y^{2}}}, \sqrt{\frac{1-\lambda(1-X^{2})}{\smash[b]{1-\lambda Y^{2}}}} \frac{Y}{X}\right)}{\sqrt{(1-\lambda)(1-Z^2)+\lambda^{2}Y^2X^2}}\D\sigma,
\end{align*}which confirms Eq.~\ref{eq:Ramanujan_rotation}.\end{proof}We recall the definition of the Jacobi $ \Theta$-function:\begin{align}&
\Theta(u|\lambda)\notag\\:={}&\prod_{n=1}^\infty\left\{\left[ 1-e^{-2n{\pi\mathbf K(\sqrt{1-\lambda})}/{\mathbf K(\sqrt{\lambda})}} \right]\left[1-2e^{-(2n-1){\pi\mathbf K(\sqrt{1-\lambda})}/{\mathbf K(\sqrt{\lambda})}}\cos\frac{\pi u}{{\mathbf K(\sqrt{\lambda})}}+e^{-2(2n-1){\pi\mathbf K(\sqrt{1-\lambda})}/{\mathbf K(\sqrt{\lambda})}}\right]\right\}\notag\\\equiv{}&\sum_{n\in\mathbb Z}(-1)^n e^{-n^{2}\pi\mathbf K(\sqrt{1-\lambda})/\mathbf K(\sqrt{\lambda})}\cos\frac{n\pi u}{\mathbf K(\sqrt{\lambda})},\quad u\in\mathbb C,\lambda\in(\mathbb C\smallsetminus\mathbb R)\cup(0,1).
\end{align} Its logarithmic derivative with respect to $u$ is the Jacobi $ \JZ$-function: \begin{align} \JZ(u|\lambda):=\frac{\partial}{\partial u}\log\Theta(u|\lambda).\end{align}
The Jacobi $ \Theta$-function can be used to introduce the Jacobi elliptic functions $ \sn(u|\lambda)$, $\cn(u|\lambda)$, $\dn(u|\lambda)$ as follows:\begin{align}
\sn(u|\lambda):={}&-\frac{i}{\sqrt[4]{\lambda}}\frac{\Theta(u+i\mathbf K(\sqrt{1-\lambda})|\lambda)}{\Theta(u|\lambda)}\exp\left\{{\frac{\pi i}{2\mathbf K(\sqrt{\lambda})}}\left[ u+\frac{i\mathbf K(\sqrt{1-\lambda})}{2} \right]\right\},\\\cn(u|\lambda):={}&-i\sqrt[^4\!\!]{\frac{1-\lambda}{\lambda}}\frac{\Theta(u+\mathbf K(\sqrt{\lambda})+i\mathbf K(\sqrt{1-\lambda})|\lambda)}{\Theta(u|\lambda)}\exp\left\{{\frac{\pi i}{2\mathbf K(\sqrt{\lambda})}}\left[ u+\mathbf K(\sqrt{\lambda})+\frac{i\mathbf K(\sqrt{1-\lambda})}{2} \right]\right\},\\\dn(u|\lambda):={}&\sqrt[4]{1-\lambda}\frac{\Theta(u+\mathbf K(\sqrt{\lambda})|\lambda)}{\Theta(u|\lambda)}.
\end{align}All the aforementioned functions named after Jacobi can be also defined for $ \lambda\in(-\infty,0)\cup(1,+\infty)$ by analytic continuation. The Jacobi elliptic functions satisfy the following properties:\begin{align}
u=\int_0^{\sn(u|\lambda)}\frac{\D t}{\sqrt{1-t^2}\sqrt{1-\lambda t ^2}},\quad \sn^2(u|\lambda)+\cn^2(u|\lambda)=1,\quad \lambda\sn^2(u|\lambda)+\dn^2(u|\lambda)=1
\end{align}for $ 0<\lambda<1,0\leq u\leq\mathbf K(\sqrt{\lambda})$. \begin{lemma}[Some Integral Formulae Related to Jacobi Elliptic Functions]We have the following integral identities for $ 0<\lambda<1$:\begin{align}
\int_0^{\pi/2}\frac{\log\cos\phi\D\phi}{\sqrt{\smash[b]{1-\lambda\cos^2\phi}}}={}&-\frac{\pi}{4}\mathbf K(\sqrt{1-\lambda})-\frac{\mathbf K(\sqrt{\lambda})}{4}\log\lambda\label{eq:log_sn_int};\\\int_0^{\pi/2}\frac{\log(1-\lambda\cos^{2}\phi)\D\phi}{\sqrt{\smash[b]{1-\lambda\cos^2\phi}}}={}&\frac{\mathbf K(\sqrt{\lambda})}{2}\log(1-\lambda);\label{eq:log_dn_int}\\\int_0^{\pi/2}\left[\int_0^{\pi/2}\frac{\log(1-\lambda\cos^2\theta\cos^{2}\phi)\D\phi}{\sqrt{\smash[b]{1-\lambda\cos^2\phi}}}\right]\frac{\D\theta}{\sqrt{\smash[b]{1-\lambda\cos^2\theta}}}={}&-\frac{\pi}{6}\mathbf K(\sqrt{1-\lambda})\mathbf K(\sqrt{\lambda})+\frac{[\mathbf K(\sqrt{\lambda})]^{2}}{3}\log\frac{4(1-\lambda)}{\sqrt{\lambda}}.\label{eq:log_sn_sn_int_int}
\end{align}\end{lemma}\begin{proof}With a special case of Jacobi's Fourier expansions  \cite[][\S22.5, Example~3]{WhittakerWatson1927}:\begin{align}
\log\frac{\sn(u|\lambda)}{\sin\frac{\pi u}{2\mathbf K(\sqrt{\lambda})}}={}&\frac{1}{4}\log\frac{16e^{-\pi\mathbf K(\sqrt{1-\lambda})/\mathbf K(\sqrt{\lambda})}}{\lambda}+\sum_{n=1}^\infty\frac{2}{n}\frac{e^{-n\pi\mathbf K(\sqrt{1-\lambda})/\mathbf K(\sqrt{\lambda})}}{1+e^{-n\pi\mathbf K(\sqrt{1-\lambda})/\mathbf K(\sqrt{\lambda})}} \cos\frac{n\pi u}{\mathbf K(\sqrt{\lambda})},
\end{align}we may compute \begin{align}
\int_0^{\pi/2}\frac{\log\cos\phi\D\phi}{\sqrt{\smash[b]{1-\lambda\cos^2\phi}}}={}&\int_0^{\mathbf K(\sqrt{\lambda})}\log\sn(u|\lambda)\D u\notag\\={}&\int_0^{\mathbf K(\sqrt{\lambda})}\log\sin\frac{\pi v}{2\mathbf K(\sqrt{\lambda})}\D v+\frac{\mathbf K(\sqrt{\lambda})}{4}\log\frac{16e^{-\pi\mathbf K(\sqrt{1-\lambda})/\mathbf K(\sqrt{\lambda})}}{\lambda}\notag\\={}&-\mathbf K(\sqrt{\lambda})\log2+\frac{\mathbf K(\sqrt{\lambda})}{4}\log\frac{16e^{-\pi\mathbf K(\sqrt{1-\lambda})/\mathbf K(\sqrt{\lambda})}}{\lambda},
\end{align}where the integration over $v$ is elementary. This verifies Eq.~\ref{eq:log_sn_int}, which was also formerly mentioned in    \cite[][\S22.5, Example~4]{WhittakerWatson1927}, \cite[][item~4.386.3]{GradshteynRyzhik} and \cite[][item~800.01]{ByrdFriedman}.

Likewise, one can verify Eq.~\ref{eq:log_dn_int} \cite[][item~800.03]{ByrdFriedman} with Jacobi's Fourier series for $ \log\dn(u|\lambda)$.

To demonstrate Eq.~\ref{eq:log_sn_sn_int_int}, we  need an identity\begin{align}
\int_0^{\pi/2}\frac{\log(1-\lambda\sn^2(v|\lambda)\sin^2\varphi)}{\sqrt{\smash[b]{1-\lambda\sin^2\varphi}}}\D \varphi=-2\mathbf K(\sqrt{\lambda})\int_0^v\JZ(u|\lambda)\D u.\label{eq:log_1_minus_sn_sqr_int}
\end{align}The proof of Eq.~\ref{eq:log_1_minus_sn_sqr_int} is straightforward: one simply differentiates both sides with respect to $v$, and compares the result with a standard integral representation for the Jacobi $ \JZ$-function $ \JZ(v|\lambda)$ \cite[][item~140.03]{ByrdFriedman}. Thus, we can compute\begin{align}
&\int_0^{\pi/2}\left[\int_0^{\pi/2}\frac{\log(1-\lambda\cos^2\theta\cos^{2}\phi)\D\phi}{\sqrt{\smash[b]{1-\lambda\cos^2\phi}}}\right]\frac{\D\theta}{\sqrt{\smash[b]{1-\lambda\cos^2\theta}}}=-2\mathbf K(\sqrt{\lambda})\int_0^{\mathbf K(\sqrt{\lambda})}\left[\int_0^v\JZ(u|\lambda)\D u\right]\D v\notag\\={}&-2\mathbf K(\sqrt{\lambda})\int_0^{\mathbf K(\sqrt{\lambda})}\log\frac{\Theta(v|\lambda)}{\Theta(0|\lambda)}\D v=-2[\mathbf K(\sqrt{\lambda})]^2\log\prod_{n=1}^\infty \frac{1}{[1-e^{-(2n-1){\pi\mathbf K(\sqrt{1-\lambda})}/{\mathbf K(\sqrt{\lambda})}}]^2}\notag\\={}& -4[\mathbf K(\sqrt{\lambda})]^2\log\prod_{n=1}^\infty \frac{1-e^{-2n{\pi\mathbf K(\sqrt{1-\lambda})}/{\mathbf K(\sqrt{\lambda})}}}{1-e^{-n{\pi\mathbf K(\sqrt{1-\lambda})}/{\mathbf K(\sqrt{\lambda})}}}.\label{eq:log_Theta_int}
\end{align}   Here in Eq.~\ref{eq:log_Theta_int}, we have made use of the fact that\begin{align}
\int_0^{\mathbf K(\sqrt{\lambda})}\log\left[1-2e^{-(2n-1){\pi\mathbf K(\sqrt{1-\lambda})}/{\mathbf K(\sqrt{\lambda})}}\cos\frac{\pi v}{{\mathbf K(\sqrt{\lambda})}}+e^{-2(2n-1){\pi\mathbf K(\sqrt{1-\lambda})}/{\mathbf K(\sqrt{\lambda})}}\right]\D v=0,\quad 0<\lambda<1,n\in\mathbb Z_{>0},
\end{align}which originates from an elementary application of Cauchy's integral formula:\begin{align}
\int_0^{ \pi}\log(1-2a\cos \theta+a^2)\D\theta=\R\counterint_{|w|=1}\log(1-aw)\frac{\D w}{iw}=0,\quad 0<a<1.
\end{align}Thus, as we compare  Eq.~\ref{eq:log_Theta_int} with the following consequence of Eqs.~\ref{eq:E4_Ell_Ramanujan}--\ref{eq:E6_Ell_Ramanujan} and Eqs.~\ref{eq:Landen_2}--\ref{eq:double_half_lambda}:\begin{subequations}\begin{align}
e^{-2{\pi\mathbf K(\sqrt{1-\lambda})}/{\mathbf K(\sqrt{\lambda})}}\prod_{n=1}^\infty\left[ 1-e^{-2n{\pi\mathbf K(\sqrt{1-\lambda})}/{\mathbf K(\sqrt{\lambda})}} \right]^{24}={}&\left[ \frac{2\mathbf K(\sqrt{\lambda})}{\pi} \right]^{12}\frac{\lambda^{2}(1-\lambda)^2}{256},\\e^{-{\pi\mathbf K(\sqrt{1-\lambda})}/{\mathbf K(\sqrt{\lambda})}}\prod_{n=1}^\infty\left[ 1-e^{-n{\pi\mathbf K(\sqrt{1-\lambda})}/{\mathbf K(\sqrt{\lambda})}} \right]^{24}={}&\left[ \frac{2(1+\sqrt{\lambda})\mathbf K(\sqrt{\lambda})}{\pi} \right]^{12}\frac{\frac{16\lambda}{(1+\sqrt{\lambda})^{4}}\left[1-\frac{4\sqrt{\lambda}}{(1+\sqrt{\lambda})^{2}}\right]^2}{256},
\end{align}\end{subequations}where $  0<\lambda<1$, we see that the statement in Eq.~\ref{eq:log_sn_sn_int_int} is true.
\end{proof}\begin{proposition}[Some Integral Identities in the Spirit of Jacobi and Ramanujan]For $0<\lambda<1 $, we have \begin{align}
\int_0^\pi\frac{\mathbf K(\sin\theta)\sin\theta\D \theta}{\sqrt{(2-\lambda)^2-\lambda^2\sin^2\theta}}={}&[\mathbf K(\sqrt{\lambda})]^2,\label{eq:R_rot1}\\\int_0^\pi\frac{\mathbf K(\sin\theta)\sin\theta}{\sqrt{(2-\lambda)^2-\lambda^2\sin^2\theta}}\log\frac{\sqrt{(2-\lambda)^2-\lambda^2\sin^2\theta}}{2\sqrt{\sin\theta}}\D \theta={}&\frac{\pi}{3}\mathbf K(\sqrt{\lambda})\mathbf K(\sqrt{1-\lambda})+\frac{[\mathbf K(\sqrt{\lambda})]^2}{3}\log\frac{\lambda(1-\lambda)}{2}\notag\\{}&-\frac{\pi}{2}\int_0^1\frac{\mathbf K(k)\D k}{\sqrt{4(1-\lambda)+\lambda^2k^2}},\label{eq:R_rot2}\\\int_0^1\frac{\mathbf K(k)\D k}{\sqrt{4(1-\lambda)+\lambda^2k^2}}+\int_0^1\frac{\mathbf K(k)\D k}{\sqrt{\lambda^2+4(1-\lambda)k^2}}={}&\mathbf K(\sqrt{\lambda})\mathbf K(\sqrt{1-\lambda}).\label{eq:R_rot3}
\end{align}

\end{proposition}\begin{proof}The identity in Eq.~\ref{eq:R_rot1} appeared as \cite[][Eq.~40$'$]{Zhou2013Pnu}, and was proved by an application of Eq.~\ref{eq:Ramanujan_rotation} in the simplest case $ \mathcal R(s,t)\equiv1$.

Now, we work out Eq.~\ref{eq:R_rot2} in detail, starting with the transformations{\allowdisplaybreaks\begin{align}
&\int_0^\pi\frac{\mathbf K(\sin\theta)\sin\theta}{\sqrt{(2-\lambda)^2-\lambda^2\sin^2\theta}}\log\frac{\sqrt{(2-\lambda)^2-\lambda^2\sin^2\theta}}{2}\D \theta=\int_0^{\pi/2}\frac{\mathbf K(\sin\theta)\sin\theta}{\sqrt{1-\lambda +\frac{\lambda^2}{4}\cos^2\theta}}\log\sqrt{1-\lambda+\frac{\lambda^2}{4}\cos^2\theta}\D\theta\notag\\={}&\int_0^{\pi/2}\sin\theta\D \theta\int_0^{\pi/2}\D\phi\frac{\log\sqrt{1-\lambda+\frac{\lambda^2}{4}\cos^2\theta}}{\sqrt{1-\sin^2\theta\sin^2\phi}\sqrt{1-\lambda+\frac{\lambda^2}{4}\cos^2\theta}}=\int_{S^2_{+++}}\frac{\log\sqrt{1-\lambda+\frac{\lambda^2}{4}Z^{2}}\D\sigma}{\sqrt{(1-Y^{2})(1-\lambda+\frac{\lambda^{2}}{4}Z^{2})}}\notag\\={}&\int_{S^2_{+++}}\frac{\log\sqrt{1-\lambda+\frac{\lambda^2}{4}Y^{2}}\D\sigma}{\sqrt{(1-Z^{2})(1-\lambda+\frac{\lambda^{2}}{4}Y^{2})}}=\frac{1}{2}\int_0^{\pi/2}\D\theta\int_0^\pi\D\phi\frac{\log\sqrt{1-\lambda+\frac{\lambda^2}{4}\sin^2\theta\sin^2\phi}}{\sqrt{1-\lambda+\frac{\lambda^2}{4}\sin^2\theta\sin^2\phi}}\notag\\={}&\int_0^{\pi/2}\D\theta\int_0^{\pi/2}\D\phi\frac{\log\sqrt{\smash[b]{1-\lambda+\lambda^{2}\sin^2\theta\sin^2\phi\cos^2\phi}}}{\sqrt{\smash[b]{1-\lambda+\lambda^{2}\sin^2\theta\sin^2\phi\cos^2\phi}}}=\int_{S^2_{+++}}\frac{\log\sqrt{\frac{(1-\lambda)(1-Z^{2})+\lambda^{2}Y^2X^2}{1-Z^{2}}}\D\sigma}{\sqrt{(1-\lambda)(1-Z^{2})+\lambda^{2}Y^2X^2}}\notag\\={}&\int_0^{\pi/2}\left[ \int_0^{\pi/2}\frac{1}{\sqrt{\smash[b]{1-\lambda\cos^2\phi}}}\log\sqrt{\frac{\cos^2\phi(1-\lambda\cos^2\phi)[1-\lambda\cos^2\theta+(1-\lambda)\tan^2\phi]}{1-\lambda\cos^2\theta\cos^2\phi}}\D\phi \right]\frac{\D\theta}{\sqrt{\smash[b]{1-\lambda\cos^2\theta}}},\label{eq:K_over_sqrt_log_sqrt_int}
\end{align}where we have employed  Eq.~\ref{eq:Ramanujan_rotation} in the last step. Meanwhile, we have the following computations:\begin{align}
&\int_0^\pi\frac{\mathbf K(\sin\theta)\sin\theta}{\sqrt{(2-\lambda)^2-\lambda^2\sin^2\theta}}\log\sin\theta\D \theta=\int_0^{\pi/2}\sin\theta\D \theta\int_0^{\pi/2}\D\phi\frac{\log\frac{\sin\theta\sin\phi}{\sin\phi}}{\sqrt{1-\sin^2\theta\sin^2\phi}\sqrt{1-\lambda+\frac{\lambda^2}{4}\cos^2\theta}}\notag\\={}&\int_{S^2_{+++}}\frac{\log Z\D\sigma}{\sqrt{(1-\lambda)(1-Z^{2})+\lambda^{2}Y^2X^2}}-\int_0^{\pi/2}\sin\theta\D \theta\int_0^{\pi/2}\D\phi\frac{\log\sin\phi}{\sqrt{1-\sin^2\theta\sin^2\phi}\sqrt{1-\lambda+\frac{\lambda^2}{4}\cos^2\theta}}\notag\\={}&\int_0^{\pi/2}\left[ \int_0^{\pi/2}\frac{1}{\sqrt{\smash[b]{1-\lambda\cos^2\phi}}}\log\sqrt{\frac{\cos^2\theta\cos^2\phi[1-\lambda\cos^2\theta+(1-\lambda)\tan^2\phi]}{1-\lambda\cos^2\theta}}\D\phi \right]\frac{\D\theta}{\sqrt{\smash[b]{1-\lambda\cos^2\theta}}}\notag\\{}&+\int_0^{\pi/2}\frac{\frac{\pi}{4}\mathbf K(\cos\theta)+\frac{\mathbf K(\sin\theta)}{2}\log\sin\theta}{\sqrt{1-\lambda+\frac{\lambda^2}{4}\cos^2\theta}}\sin\theta\D \theta,\label{eq:K_over_sqrt_log_sin_int}
\end{align}}where the last step involves an application of Ramanujan's rotations (Eq.~\ref{eq:Ramanujan_rotation}) to the integral over $S^2_{+++}$, and a reference to Eq.~\ref{eq:log_sn_int} for the integration over $\phi\in(0,\pi/2)$. It is then clear that Eqs.~\ref{eq:K_over_sqrt_log_sqrt_int} and \ref{eq:K_over_sqrt_log_sin_int} combine into\begin{align}
&\int_0^\pi\frac{\mathbf K(\sin\theta)\sin\theta}{\sqrt{(2-\lambda)^2-\lambda^2\sin^2\theta}}\log\frac{\sqrt{(2-\lambda)^2-\lambda^2\sin^2\theta}}{2\sqrt{\sin\theta}}\D \theta\notag\\={}&\int_0^{\pi/2}\left[ \int_0^{\pi/2}\frac{1}{\sqrt{\smash[b]{1-\lambda\cos^2\phi}}}\log\sqrt{\frac{(1-\lambda\cos^2\theta)(1-\lambda\cos^2\phi)}{\cos^2\theta(1-\lambda\cos^2\theta\cos^2\phi)}}\D\phi \right]\frac{\D\theta}{\sqrt{\smash[b]{1-\lambda\cos^2\theta}}}\notag\\{}&-\frac{\pi}{2}\int_0^{\pi/2}\frac{\mathbf K(\cos\theta)\sin\theta\D \theta}{\sqrt{(2-\lambda)^2-\lambda^2\sin^2\theta}}.\label{eq:K_over_sqr_log_quotient_int}
\end{align}As we simplify the last line of Eq.~\ref{eq:K_over_sqr_log_quotient_int} with the integral formulae in Eqs.~\ref{eq:log_sn_int}--\ref{eq:log_sn_sn_int_int}, we can confirm Eq.~\ref{eq:R_rot2}.

We shall prove  Eq.~\ref{eq:R_rot3} by verifying the following chain of identities:
\begin{align}
\int_0^1\frac{\mathbf K(k)\D k}{\sqrt{4(1-\lambda)+\lambda^2k^2}}={}&\iint_{1<|z|<{1/\sqrt[4]{1-\lambda}},\I z>0}\frac{\D \R z\D \I z}{|(1-z^{2})[1-(1-\lambda)\ z^2]|}\notag\\={}&\iint_{-\mathbf K(\sqrt{1-\lambda})<u<\mathbf K(\sqrt{1-\lambda}),0<v<\frac{1}{2}\mathbf K(\sqrt{\lambda}),\left\vert \sn(u+iv| 1-\lambda)\right|>1}\D u\D v\notag\\={}&\mathbf K(\sqrt{\lambda})\mathbf K(\sqrt{1-\lambda})-\iint_{|z|<1,\I z>0}\frac{\D \R z\D \I z}{|(1-z^{2})[1-(1-\lambda)z^2]|}\notag\\={}&\mathbf K(\sqrt{\lambda})\mathbf K(\sqrt{1-\lambda})-\int_0^1\frac{\mathbf K(k)\D k}{\sqrt{\lambda^2+4(1-\lambda)k^2}}.\label{eq:Jacobi_area_comp1}
\end{align}
For convenience, we write $ \lambda_*=1-\lambda$. To prove the first equality in  Eq.~\ref{eq:Jacobi_area_comp1}, we employ the polar coordinates $ z=re^{i\theta}$  to compute{\allowdisplaybreaks\begin{align}&
\iint_{1<|z|<{1/\sqrt[4]{\lambda_*}},\I z>0}\frac{\D \R z\D \I z}{|(1-z^{2})(1-\lambda_* z^2)|}=\int_1^{1/\sqrt[4]{\lambda_*}}r\left[ \int_0^\pi\frac{\D\theta}{|(1-r^2e^{2i\theta})(1-\lambda_* r^2e^{2i\theta})|}\right]\D r\notag\\={}&\int_1^{1/\sqrt[4]{\lambda_*}}r\left[ \int_0^\pi\frac{\D\phi}{\sqrt{\smash[b]{(1+r^{2})^{2}(1-\lambda_* r^{2})^{2}-4(1-\lambda_*) r^2(1-\lambda_* r^4)\sin^2\phi}}}\right]\D r\notag\\={}&2\int_1^{1/\sqrt[4]{\lambda_*}}\frac{r}{(1+r^{2})(1-\lambda_* r^{2})}\mathbf K\left( \sqrt{1-\left(\frac{r^2-1}{r^2+1}\right)^2 \left(\frac{1+\lambda_* r^2}{1-\lambda_* r^2}\right)^2} \right)\D r\notag\\={}&2\int_1^{1/\sqrt[4]{\lambda_*}}\frac{1}{(1-\lambda_*) r}\mathbf K\left( \frac{1-\lambda_* r^{4}}{(1-\lambda_*)r^{2}}\right)\D r=\int_0^1\frac{\mathbf K(k)\D k}{\sqrt{4\lambda_*+(1-\lambda_*)^2k^2}}=\int_0^1\frac{\mathbf K(k)\D k}{\sqrt{4(1-\lambda)+\lambda^2k^2}}.\label{eq:Jacobi_repn_a}
\end{align}}Here in Eq.~\ref{eq:Jacobi_repn_a}, we have used  an angular transformation (which is reminiscent of the derivation for Ramanujan rotations in Lemma~\ref{lm:R_rot})\begin{align*}\theta=\arctan\left( \frac{r^2-1}{r^2+1}\tan\phi \right)\end{align*}to complete the integral in $\phi$ with the standard integral representation of $ \mathbf K$, before invoking Landen's transformation (Eq.~\ref{eq:Landen_1}) and a variable substitution $ k=(1-\lambda_* r^4)/[(1-\lambda_*)r^2]$. To deduce the second equality in  Eq.~\ref{eq:Jacobi_area_comp1} from the first one, we recall that the conformal mapping $ u+iv\mapsto \sn(u+iv|\lambda_*)$ establishes a bijection between  the open rectangle $ -\mathbf K(\sqrt{\lambda_*})<u<\mathbf K(\sqrt{\lambda_*}),0<v<\mathbf K(\sqrt{1-\lambda_*})$ and the upper half-plane \cite[][p.~234]{SteinII}, and that \begin{align}\left\vert \sn\left(\left.u+i\frac{\mathbf K(\sqrt{1-\lambda_*})}{2}\right| \lambda_*\right)\right|=\frac{1}{\sqrt[4]{\lambda_*}},\quad \forall u\in\mathbb R,\lambda_*\in(0,1)\label{eq:sn_abs_fourth_root}\end{align} follows from the addition formula for $ \sn(u+w|\lambda_*)$ \cite[][item~123.01]{ByrdFriedman} and special values of  $ \sn(w|\lambda_*)$,  $ \cn(w|\lambda_*)$, $ \dn(w|\lambda_*)$ at $ w=i\mathbf K(\sqrt{1-\lambda_*})/2$ \cite[][item~122.11]{ByrdFriedman}. We may rewrite the second line in  Eq.~\ref{eq:Jacobi_area_comp1}  as\begin{align*}&\iint_{-\mathbf K( \sqrt{\lambda_*} )<u<\mathbf K( \sqrt{\lambda_*} ),0<v<\frac{1}{2}\mathbf K( \sqrt{1-\lambda_*})}\D u\D v-\iint_{-\mathbf K( \sqrt{\lambda_*} )<u<\mathbf K( \sqrt{\lambda_*} ),0<v<\frac{1}{2}\mathbf K( \sqrt{1-\lambda_*}),\left\vert \sn\left(u+iv| \lambda_*\right)\right|<1}\D u\D v\notag\\={}&\mathbf K(\sqrt{\lambda})\mathbf K(\sqrt{1-\lambda})-\iint_{-\mathbf K( \sqrt{\lambda_*} )<u<\mathbf K( \sqrt{\lambda_*} ),0<v<\frac{1}{2}\mathbf K( \sqrt{1-\lambda_*}),\left\vert \sn\left(u+iv| \lambda_*\right)\right|<1}\D u\D v,\end{align*} which is conformally equivalent to  the third line. To complete the verification of   Eq.~\ref{eq:Jacobi_area_comp1}, we compute \begin{align}&\iint_{|z|<1,\I z>0}\frac{\D \R z\D \I z}{|(1-z^{2})(1-\lambda_* z^2)|}=\int_0^{1}r\left[ \int_0^\pi\frac{\D\theta}{|(1-r^2e^{2i\theta})(1-\lambda_* r^2e^{2i\theta})|}\right]\D r\notag\\={}&\int_0^{1}r\left[ \int_0^\pi\frac{\D\phi}{\sqrt{\smash[b]{(1+r^{2})^{2}(1-\lambda_* r^{2})^{2}-4(1-\lambda_*) r^2(1-\lambda_* r^4)\sin^2\phi}}}\right]\D r\notag\\={}&2\int_0^{1}\frac{r}{(1+r^{2})(1-\lambda_* r^{2})}\mathbf K\left( \sqrt{1-\left(\frac{1-r^2}{1+r^2}\right)^2 \left(\frac{1+\lambda_* r^2}{1-\lambda_* r^2}\right)^2} \right)\D r=2\int_0^1\frac{r}{1-\lambda_* r^4}\mathbf K\left(\frac{(1-\lambda_*) r^{2}}{1-\lambda_* r^4} \right)\D r \notag\\={}&\int_{0}^{1}\frac{\mathbf K(k)\D k}{\sqrt{(1-\lambda_*)^2+4\lambda_* k^{2}}}=\int_0^1\frac{\mathbf K(k)\D k}{\sqrt{\lambda^2+4(1-\lambda)k^2}}.\label{eq:Jacobi_repn_d}\end{align} (We also note that  Eq.~\ref{eq:R_rot3} can be proved without invoking Jacobi elliptic functions  \cite[][Eq.~91]{Zhou2013Spheres}.)\end{proof}
\subsection{\label{subsec:G2_Hecke4_GZ_rn}Elementary Evaluation of $ G_2^{\mathfrak H/\overline{\varGamma}_0(4)}(z)$ with Gross--Zagier Renormalization}

 The Kontsevich--Zagier integral representation for $ G_2^{\mathfrak H/\overline{\varGamma}_0(4)}(z)$  has appeared in Eq.~\ref{eq:G2_Hecke4_z_rn}. To prepare for the main goal of this subsection (evaluation of   $ G_2^{\mathfrak H/\overline{\varGamma}_0(4)}(z),\forall z\in\mathfrak H$  in Proposition~\ref{prop:G2_Hecke4_z_KZ_rn}), we explore  certain integrals whose integrands involve complete elliptic integrals and logarithms, in the next lemma.\begin{lemma}[Some Special Integral Formulae]\label{lm:GZ_lim}\begin{enumerate}[label=\emph{(\alph*)}, ref=(\alph*), widest=a] \item For any $\lambda\in\mathbb C\smallsetminus\mathbb R$ and $ \mu\in\mathbb (1,+\infty)$, we have the following identity:\begin{align}&
\frac{2}{\pi}\int_0^1\frac{\mathbf K(\sqrt{t})\mathbf K(\sqrt{1-t})}{t-\lambda}\log(\mu -t)\D t\notag\\={}&\int_{0}^1\frac{[\mathbf K(\sqrt{t})]^2}{t-\lambda}\D t-\int_{1/\mu}^{1}\frac{[\mathbf K(\sqrt{t})]^2}{\lambda t-1}\D t- \frac{1}{\lambda}\left[ \mathbf K\left( \sqrt{\frac{1}{\lambda}} \right) \right]^2\log(\lambda-\mu)+\frac{\pi}{\lambda} \mathbf K\left( \sqrt{\frac{1}{\lambda}} \right)\mathbf K\left( \sqrt{1-\frac{1}{\lambda}} \right);\label{eq:int_K'K_log_etc}\end{align}while the following formula holds for  $\lambda\in\mathbb C\smallsetminus\mathbb R$ and $ \mu\in\mathbb (0,1)$:\begin{align}
\frac{2}{\pi}\int_0^1\frac{\mathbf K(\sqrt{t})\mathbf K(\sqrt{1-t})}{1-\lambda t}\log(1-\mu t)\D t=-\int_0^\mu\frac{[\mathbf K(\sqrt{t})]^2\D t}{t-\lambda}+[\mathbf K(\sqrt{\lambda})]^2\log\frac{\mu-\lambda}{\lambda}+\pi i\frac{\I \lambda}{|\I\lambda|}[\mathbf K(\sqrt{\lambda})]^2.\label{eq:int_K'K_log_etc_alt_form}
\end{align}Here, all the integrations run along subsets of the open unit interval $ t\in(0,1)$.
\item
For   $ \I\lambda\neq0$,   we have the following limit formulae:{\allowdisplaybreaks\begin{align}&
\lim_{\varepsilon\to0^{+}}\left\{ \int^{\lambda+\varepsilon}_0\frac{[\mathbf K(\sqrt{t})]^2}{ t-\lambda}\D t-[\mathbf K(\sqrt{\lambda})]^2\log\varepsilon \right\}\notag\\={}&\pi i\frac{\I \lambda}{|\I \lambda|}[ \mathbf K( \sqrt{\lambda} ) ]^2+\frac{\pi}{3}\mathbf K(\sqrt{\lambda})\mathbf K(\sqrt{1-\lambda})-\frac{2[\mathbf K(\sqrt{\lambda})]^2}{3}\log[4\lambda(1-\lambda)];\label{eq:GZ_lim1}\\
&\lim_{\varepsilon\to0^{+}}\left[ \int^{\lambda+\varepsilon}_0\frac{2\mathbf K(\sqrt{t})\mathbf K(\sqrt{1-t})}{ t-\lambda}\D t-2\mathbf K(\sqrt{\lambda})\mathbf K(\sqrt{1-\lambda})\log\varepsilon \right]\notag\\
={}&-\frac{\pi}{\lambda}\left[ \mathbf K\left( \sqrt{\frac{1}{\lambda}} \right) \right]^2+\frac{2\pi[\mathbf K(\sqrt{\lambda})]^2}{3}-\frac{\pi[\mathbf K(\sqrt{1-\lambda})]^2}{3}-\frac{4}{3}\mathbf K(\sqrt{\lambda})\mathbf K(\sqrt{1-\lambda})\log[4\lambda(1-\lambda)];\label{eq:GZ_lim2}\\&\lim_{\varepsilon\to0^{+}}\left\{ \int_{\lambda+\varepsilon}^1\frac{[\mathbf K(\sqrt{1-t})]^2}{ t-\lambda}\D t+[\mathbf K(\sqrt{1-\lambda})]^2\log\varepsilon \right\}\notag\\={}&-\frac{\pi}{3}\mathbf K(\sqrt{\lambda})\mathbf K(\sqrt{1-\lambda})+\frac{2[\mathbf K(\sqrt{1-\lambda})]^2}{3}\log[4\lambda(1-\lambda)],\label{eq:GZ_lim3}
\end{align}}where all the integrals are computed along straight line segments joining the end points.\end{enumerate}\end{lemma}\begin{proof}\begin{enumerate}[label=(\alph*),widest=a]\item
Using the  imaginary modulus   and  inverse modulus transformations (Eqs.~\ref{eq:im_mod}--\ref{eq:inv_mod}), one may verify that the expression \begin{align*}f_{\mu}(t):=[\mathbf K(\sqrt{t/\mu})]^2\log\left(\frac{\mu}{t}-\mu\right)-\pi \mathbf K(\sqrt{t/\mu})\mathbf K(\sqrt{1-(t/\mu)}),\quad t\in\mathbb C\smallsetminus((-\infty,0]\cup[1,+\infty))\end{align*}can be extended as an analytic function for $ t\in\mathbb C\smallsetminus[1,+\infty)$, so long as $ \mu\in(1,+\infty)$. Hence, Eq.~\ref{eq:f_slit_plane_Cauchy_int} is applicable, and we have\begin{align}&
[\mathbf K(\sqrt{s})]^2\log\left(\frac{1}{s}-\mu\right)-\pi \mathbf K(\sqrt{s})\mathbf K(\sqrt{1-s})=f_{\mu}(\mu s)=\frac{1}{2\pi i}\int_{0}^1\frac{f_{\mu}\big(\frac{1}{t}+i0^+\big)-f_{\mu}\big(\frac{1}{t}-i0^+\big)}{1-\mu st}\frac{\D t}{t}\notag\\={}& -\int_{0}^1\frac{[\mathbf K(\sqrt{t})]^2}{1-st}\D t-\int_{1/\mu}^{1}\frac{[\mathbf K(\sqrt{t})]^2}{t-s}\D t+\frac{2}{\pi}\int_{0}^1\frac{\mathbf K(\sqrt{t})\mathbf K(\sqrt{1-t})}{1-st}\log(\mu-t)\D t,\quad s\in\mathbb C\smallsetminus\mathbb R.\label{eq:K_sqr_log_Cauchy}
\end{align}Setting $ s=1/\lambda$ in Eq.~\ref{eq:K_sqr_log_Cauchy}, we can verify the identity claimed in Eq.~\ref{eq:int_K'K_log_etc}.

Before establishing Eq.~\ref{eq:int_K'K_log_etc_alt_form} in its stated form, we may first consider the scenarios where $ 0<\mu<\lambda<1$, and read off the imaginary part of the following vanishing integral:\begin{align}
0=\int_{-\infty+i0^{+}}^{+\infty+i0^+}\frac{[\mathbf K(\sqrt{1-t})]^{2}+2i\mathbf K(\sqrt{1-t})\mathbf K(\sqrt{t})-[\mathbf K(\sqrt{t})]^2}{1-\lambda t}\log(1-\mu t)\D t.\label{eq:int0_res_sum_1}
\end{align} Here, the right-hand side of Eq.~\ref{eq:int0_res_sum_1} can be decomposed into three parts. With the transformation $ t=s/(s-1)$ and Eqs.~\ref{eq:im_mod}--\ref{eq:inv_mod}, we see that  the portion $ \R t<0$ has vanishing contribution to the imaginary part of Eq.~\ref{eq:int0_res_sum_1}:\addtocounter{equation}{-1}\begin{subequations}\begin{align}
\I\int_{-\infty+i0^{+}}^{0+i0^{+}}(\cdots)\D t={}&\I\int_{0-i0^+}^{1-i0^+}\frac{[\mathbf K(\sqrt{1-s})]^2}{1-(1-\lambda)s}\log\left( 1-\frac{\mu s}{s-1} \right)\D s=0.\label{eq:int0_res_sum_1a}
\intertext{The second part is simply}
\I\int_{0+i0^{+}}^{1+i0^{+}}(\cdots)\D t={}&2\int_0^1\frac{\mathbf K(\sqrt{t})\mathbf K(\sqrt{1-t})}{1-\lambda t}\log(1-\mu t)\D t.\label{eq:int0_res_sum_1b}\intertext{With the transformation $ t=1/s$ and Eqs.~\ref{eq:im_mod}--\ref{eq:inv_mod}, we obtain}\I\int_{1+i0^{+}}^{+\infty+i0^{+}}(\cdots)\D t={}&\I\int_{0-i0^{+}}^{1-i0^{+}}\frac{[\mathbf K(\sqrt{s})]^2}{\lambda-s}\log\left( 1-\frac{\mu }{s} \right)\D s\notag\\={}&\pi\int_0^\mu\frac{[\mathbf K(\sqrt{t})]^2\D t}{t-\lambda}-\pi[\mathbf K(\sqrt{\lambda})]^2\log\frac{\lambda-\mu}{\lambda},\label{eq:int0_res_sum_1c}
\end{align}\end{subequations}upon application of the Plemelj jump formula.
  Adding through Eqs.~\ref{eq:int0_res_sum_1a}--\ref{eq:int0_res_sum_1c}, we see that\begin{align*}
{}&\frac{2}{\pi}\int_0^1\frac{\mathbf K(\sqrt{t})\mathbf K(\sqrt{1-t})}{1-\lambda t}\log(1-\mu t)\D t=-\int_0^\mu\frac{[\mathbf K(\sqrt{t})]^2\D t}{t-\lambda}+[\mathbf K(\sqrt{\lambda})]^2\log\frac{\lambda-\mu}{\lambda},\quad 0<\mu<\lambda<1
\end{align*} is true. Thus, after a judicious analytic continuation of the equation above in  the variable $\lambda$, one can verify Eq.~\ref{eq:int_K'K_log_etc_alt_form}.

\item
We may also examine the left-hand side of  Eq.~\ref{eq:int_K'K_log_etc} from a different  perspective. Without loss of generality, we momentarily assume that $ \lambda>1$ and $ \mu>1$, and use  \cite[][Eq.~43]{Zhou2013Pnu} to deduce\begin{align}
\frac{2}{\pi}\int_0^1\frac{\mathbf K(\sqrt{t})\mathbf K(\sqrt{1-t})}{t-\lambda}\log(\mu -t)\D t={}&-\frac{1}{2\pi}\int_{0}^{\pi}\mathbf K(\sin\theta)\sin\theta\left[ \int_0^{2\pi}\frac{\log\frac{2\mu-1-\sin\theta\cos\phi}{2}}{2\lambda-1-\sin\theta\cos\phi} \D\phi\right]\D\theta.\label{eq:int_K'K_log_etc_alt}
\end{align} For $ \lambda,\mu\in(1,+\infty)$ and $ \theta\in(0,\pi)$, the integration over $\phi$ can now be completed with residue calculus: \begin{align}&
\int_0^{2\pi}\frac{\log\frac{2\mu-1-\sin\theta\cos\phi}{2}}{2\lambda-1-\sin\theta\cos\phi} \D\phi\notag\\={}&\R\counterint_{|w|=1}\frac{\log\frac{\sin^2\theta}{4[2\mu-1+\sqrt{(2\mu-1)^2-\sin^2\theta}]}+2\log\left( \frac{2\mu-1+\sqrt{(2\mu-1)^2-\sin^2\theta}}{\sin\theta}-w\right)}{2\lambda-1-\frac{\sin\theta}{2}(w+\frac{1}{w})}\frac{\D w}{iw}\notag\\={}&\frac{2\pi}{\sqrt{(2\lambda-1)^{2}-\sin^2\theta}}\log\frac{\big[\sqrt{\smash[b]{(2\lambda-1)^2-\sin^2\theta}}+\sqrt{\smash[b]{(2\mu-1)^2-\sin^2\theta}}-2(\lambda-\mu)\big]^{2}}{4\big[2\mu-1+\sqrt{\smash[b]{(2\mu-1)^2-\sin^2\theta}}\big]},\label{eq:log_angle_int}
\end{align} where we have collected the residue at the simple pole $z=(2\lambda-1-\sqrt{(2\lambda-1)^2-\sin^2\theta})/\sin\theta\in(-1,1) $. Combining Eqs.~\ref{eq:int_K'K_log_etc_alt} with \ref{eq:log_angle_int}, we obtain an identity\begin{align}&
\frac{2}{\pi}\int_0^1\frac{\mathbf K(\sqrt{t})\mathbf K(\sqrt{1-t})}{t-\lambda}\log(\mu -t)\D t\notag\\={}&\int_{0}^{\pi}\frac{\mathbf K(\sin\theta)\sin\theta}{\sqrt{(2\lambda-1)^{2}-\sin^2\theta}}\log\frac{4\big[2\mu-1+\sqrt{\smash[b]{(2\mu-1)^2-\sin^2\theta}}\big]}{\big[\sqrt{\smash[b]{(2\lambda-1)^2-\sin^2\theta}}+\sqrt{\smash[b]{(2\mu-1)^2-\sin^2\theta}}-2(\lambda-\mu)\big]^{2}}\D\theta,\quad \lambda>1,\mu>1.\label{eq:K'K_log_to_K_log_id}
\end{align}

In Eq.~\ref{eq:K'K_log_to_K_log_id}, one may trade $ \lambda$ and $ \mu$ for their respective inverses $1/\lambda $ and $ 1/\mu$, to produce a formula \begin{align}&
\frac{2}{\pi}\int_0^1\frac{\mathbf K(\sqrt{t})\mathbf K(\sqrt{1-t})}{1-\lambda t}\log(1 -\mu t)\D t\notag\\={}&-\int_0^\pi\frac{\mathbf K(\sin\theta)\sin\theta}{\sqrt{(2-\lambda)^{2}-\lambda^{2}\sin^2\theta}}\log\frac{4\big[2-\mu+\sqrt{\smash[b]{(2-\mu)^2-\mu^{2}\sin^2\theta}}\big]}{\big[\frac{\mu}{\lambda}\sqrt{\smash[b]{(2-\lambda)^2-\lambda^{2}\sin^2\theta}}+\sqrt{\smash[b]{(2-\mu)^2-\mu^{2}\sin^2\theta}}-2(\frac{\mu}{\lambda}-1)\big]^{2}}\D\theta\label{eq:K'K_log_mu_lambda_new}
\end{align}for $ 0<\lambda<1$ and $0<\mu<1 $, upon exploiting   the following identity (see Eqs.~\ref{eq:K2} and \ref{eq:R_rot1}):\begin{align}\frac{2}{\pi}\int_0^1\frac{\mathbf K(\sqrt{\smash[b]{\vphantom{1}t}})\mathbf K(\sqrt{\smash[b]{1-t}})}{1-\lambda t}\D t=
\int_0^\pi\frac{\mathbf K(\sin\theta)\sin\theta\D \theta}{\sqrt{(2-\lambda)^2-\lambda^2\sin^2\theta}},\quad 0<\lambda<1.
\end{align}Especially, setting $ \mu=\lambda$ in Eq.~\ref{eq:K'K_log_mu_lambda_new}, we obtain\begin{align}&
\frac{2}{\pi}\int_0^1\frac{\mathbf K(\sqrt{t})\mathbf K(\sqrt{1-t})}{1-\lambda t}\log(1 -\lambda t)\D t=-\int_0^\pi\frac{\mathbf K(\sin\theta)\sin\theta}{\sqrt{(2-\lambda)^{2}-\lambda^{2}\sin^2\theta}}\log\frac{2-\lambda+\sqrt{\smash[b]{(2-\lambda)^2-\lambda^{2}\sin^2\theta}}}{(2-\lambda)^2-\lambda^{2}\sin^2\theta}\D\theta\notag\\={}&-\int_0^\pi\frac{\mathbf K(\sin\theta)\sin\theta}{\sqrt{(2-\lambda)^{2}-\lambda^{2}\sin^2\theta}}\log\frac{2-\lambda+\sqrt{\smash[b]{(2-\lambda)^2-\lambda^{2}\sin^2\theta}}}{4\sin\theta}\D\theta+\frac{2\pi}{3}\mathbf K(\sqrt{\lambda})\mathbf K(\sqrt{1-\lambda})\notag\\{}&+\frac{2[\mathbf K(\sqrt{\lambda})]^2}{3}\log\frac{\lambda(1-\lambda)}{2}-\pi\int_0^1\frac{\mathbf K(k)\D k}{\sqrt{4(1-\lambda)+\lambda^2k^2}},\quad 0<\lambda<1,\label{eq:K'K_log_lambda_lambda_new}
\end{align} after referring to   Eq.~\ref{eq:R_rot2}. By an elementary identity\begin{align*}
\log\frac{2-\lambda+\sqrt{\smash[b]{(2-\lambda)^2-\lambda^{2}\sin^2\theta}}}{4\sin\theta}=\log\frac{\lambda}{4}+\tanh^{-1}\frac{\sqrt{\smash[b]{(2-\lambda)^2-\lambda^{2}\sin^2\theta}}}{2-\lambda},
\end{align*}we can further simplify the last line of Eq.~\ref{eq:K'K_log_lambda_lambda_new}:\begin{align}
&-\int_0^\pi\frac{\mathbf K(\sin\theta)\sin\theta}{\sqrt{(2-\lambda)^{2}-\lambda^{2}\sin^2\theta}}\log\frac{2-\lambda+\sqrt{\smash[b]{(2-\lambda)^2-\lambda^{2}\sin^2\theta}}}{4\sin\theta}\D\theta\notag\\={}&-[\mathbf K(\sqrt{\lambda})]^2\log\frac{\lambda}{4}-2\int_0^1\frac{\mathbf K(\sqrt{1-\kappa^2})}{\sqrt{4(1-\lambda)+\lambda^{2}\kappa^2}}\tanh^{-1}\sqrt{\frac{4(1-\lambda)+\lambda^{2}\kappa^2}{4(1-\lambda)+\lambda^{2}}}\D \kappa\notag\\={}&-[\mathbf K(\sqrt{\lambda})]^2\log\frac{\lambda}{4}-2\int_{0}^{1}\left[\int_0^1\frac{\mathbf K(\sqrt{1-\kappa^2})}{1-k^2\kappa^2}\D \kappa\right]\frac{\D k}{\sqrt{\lambda^{2}+4(1-\lambda)k^2}}\notag\\={}&-[\mathbf K(\sqrt{\lambda})]^2\log\frac{\lambda}{4}-\pi\int_0^1\frac{\mathbf K(k)\D k}{\sqrt{\lambda^2+4(1-\lambda)k^2}},\label{eq:K'K_artanh_redn}
\end{align}   where we have resorted to Eq.~\ref{eq:K1} for the reduction of the integral in $ \kappa\in(0,1)$. Recalling the integral identity in Eq.~\ref{eq:R_rot3}, we can merge Eqs.~\ref{eq:K'K_log_lambda_lambda_new} and \ref{eq:K'K_artanh_redn} into \begin{align}&
\frac{2}{\pi}\int_0^1\frac{\mathbf K(\sqrt{t})\mathbf K(\sqrt{1-t})}{1-\lambda t}\log(1 -\lambda t)\D t\notag\\={}&-\frac{\pi}{3}\mathbf K(\sqrt{\lambda})\mathbf K(\sqrt{1-\lambda})+\frac{2[\mathbf K(\sqrt{\lambda})]^2}{3}\log\frac{4(1-\lambda)}{\sqrt{\lambda}},\quad \forall\lambda\in(\mathbb C\smallsetminus\mathbb R)\cup(0,1),\label{eq:K'K_xlogx_int_id}
\end{align} after analytic continuation.

By an analytic continuation of Eq.~\ref{eq:int_K'K_log_etc_alt_form}, the limit  expressed on the left-hand side of  Eq.~\ref{eq:GZ_lim1} amounts to \begin{align*}&\pi i\frac{\I \lambda}{|\I \lambda|}[ \mathbf K( \sqrt{\lambda} ) ]^2-[\mathbf K(\sqrt{\lambda})]^2\log\lambda-\frac{2}{\pi}\int_0^1\frac{\mathbf K(\sqrt{t})\mathbf K(\sqrt{1-t})}{1-\lambda t}\log(1 -\lambda t)\D t.
\end{align*} As we simplify the last integral with  Eq.~\ref{eq:K'K_xlogx_int_id}, we see that
both sides  of  Eq.~\ref{eq:GZ_lim1} are indeed equal.

For $\R \mu\in(0,1),\I\mu>0$, we may employ the inverse modulus transformation (Eq.~\ref{eq:inv_mod}) to verify that \begin{align}
\int_{1/\mu}^{1}\frac{[\mathbf K(\sqrt{t})]^2}{\lambda t-1}\D t=\int_\mu^1\frac{[\mathbf K(\sqrt{t})]^2-2i\mathbf K(\sqrt{t})\mathbf K(\sqrt{1-t})-[\mathbf K(\sqrt{1-t})]^2}{t-\lambda}\D t.\label{eq:K_sqr_1to3}
\end{align}Therefore, for $\R \lambda\in(0,1),\I\lambda>0$, we have the following  integral identity, by courtesy of  Eqs.~\ref{eq:int_K'K_log_etc}, \ref{eq:GZ_lim1}, \ref{eq:K'K_xlogx_int_id} and \ref{eq:K_sqr_1to3}:{\allowdisplaybreaks\begin{align}
&\frac{\pi}{3\lambda} \mathbf K\left( \sqrt{\frac{1}{\lambda}} \right)\mathbf K\left( \sqrt{1-\frac{1}{\lambda}} \right)-\frac{2}{3\lambda}\left[ \mathbf K\left( \sqrt{\frac{1}{\lambda}} \right) \right]^2\log[4\lambda(\lambda-1)]\notag\\={}&\int_{0}^1\frac{[\mathbf K(\sqrt{t})]^2}{t-\lambda}\D t+\frac{\pi}{\lambda} \mathbf K\left( \sqrt{\frac{1}{\lambda}} \right)\mathbf K\left( \sqrt{1-\frac{1}{\lambda}} \right)\notag\\{}&-\lim_{\mu\to\lambda+0^{+}}\left\{\int_\mu^1\frac{[\mathbf K(\sqrt{t})]^2-2i\mathbf K(\sqrt{t})\mathbf K(\sqrt{1-t})-[\mathbf K(\sqrt{1-t})]^2}{t-\lambda}\D t+ \frac{1}{\lambda}\left[ \mathbf K\left( \sqrt{\frac{1}{\lambda}} \right) \right]^2\log(\lambda-\mu)\right\}\notag\\={}&-\frac{2\pi}{3}\mathbf K(\sqrt{\lambda})\mathbf K(\sqrt{1-\lambda})-\frac{2[\mathbf K(\sqrt{\lambda})]^2}{3}\log[4\lambda(1-\lambda)]\notag\\&+\lim_{\mu\to\lambda+0^{+}}\left\{\int_\mu^1\frac{2i\mathbf K(\sqrt{t})\mathbf K(\sqrt{1-t})}{t-\lambda}\D t+2i\mathbf K(\sqrt{\lambda})\mathbf K(\sqrt{1-\lambda})\log(\mu-\lambda)\right\}\notag\\&+\lim_{\mu\to\lambda+0^{+}}\left\{\int_\mu^1\frac{[\mathbf K(\sqrt{1-t})]^2}{t-\lambda}\D t+[\mathbf K(\sqrt{1-\lambda})]^2\log(\mu-\lambda)\right\},\label{eq:GZ_lim1'}
\end{align}}where $ \log(-1)=\pi i$.

Now, we examine Eq.~\ref{eq:GZ_lim1'} in the limit scenario where $ \I\lambda\to0^+$, and $ \mu$ tends to $\lambda+0^{+}$ along the real axis. Reading off the  imaginary part, we arrive at the following identity:\begin{align}&
\lim_{\mu\to\lambda+0^+}\left[\int_\mu^1\frac{2\mathbf K(\sqrt{t})\mathbf K(\sqrt{1-t})}{t-\lambda}\D t+2\mathbf K(\sqrt{\lambda})\mathbf K(\sqrt{1-\lambda})\log(\mu-\lambda)\right]\notag\\={}&\frac{\pi [\mathbf K(\sqrt{1-\lambda}) ]^2}{3}-\frac{2\pi [\mathbf K(\sqrt{\lambda}) ]^2}{3}+\frac{4\mathbf K(\sqrt{\lambda})\mathbf K(\sqrt{1-\lambda})}{3}\log[4\lambda(1-\lambda)],\quad  0<\lambda<1.\label{eq:GZ_lim2'}
\end{align}Clearly, analytic continuations of Eq.~\ref{eq:GZ_lim2'}  lead to Eq.~\ref{eq:GZ_lim2}. Subtracting  Eq.~\ref{eq:GZ_lim2'} from Eq.~\ref{eq:GZ_lim1'}, we arrive at Eq.~\ref{eq:GZ_lim3}.
\qedhere\end{enumerate}  \end{proof}\begin{proposition}[Evaluation of $G_2^{\mathfrak H/\overline{\varGamma}_0(4)}(z) $]\label{prop:G2_Hecke4_z_KZ_rn}For all $ z\in\mathfrak H$, we have \begin{align} G_2^{\mathfrak H/\overline{\varGamma}_0(4)}(z)=-\frac{1}{3}\log\left\vert\frac{\Delta(z)}{\Delta(2z)}\right\vert=-8\log\left\vert \frac{\eta(z)}{\eta(2z)} \right\vert=-\frac{1}{3}\log\frac{2^4|1-\lambda(2z)|^{2}}{|\lambda(2z)|},\label{eq:G2_Hecke4_z_rn_eval}\end{align}which verifies the corresponding Gross--Zagier algebraicity conjecture (Eq.~\ref{eq:GZ_conj_rn}).\end{proposition}
\begin{proof}Without loss of generality, we first suppose that the point $ z\in\mathfrak H$ satisfies\begin{align}
0<|\R z|<\frac{1}{2},\left\vert z+\frac{1}{4} \right\vert>\frac{1}{4},\left\vert z-\frac{1}{4} \right\vert>\frac{1}{4},\label{eq:z_narrow_range}
\end{align}so that $ \I\lambda(2z)\R z>0$, and pick an infinitesimal $ \delta\to0$ such that $ \varepsilon:=\lambda(2(z+\delta))-\lambda(2z)\to0^+$. The infinitesimal parameter $\varepsilon $ is related to the Gross--Zagier renormalization procedure (Eq.~\ref{eq:G2_rn_non_ell_defn}) via the following limit:\begin{align}
\lim_{\delta\to0}\log\left\vert \frac{\lambda(2(z+\delta))-\lambda(2z)}{2\delta} \right\vert={}&\frac{4\log 2}{3}+\log\pi+4\log|\eta(2z)|+\frac{2\log|1-\lambda(2z)|}{3}+\frac{2\log|\lambda(2z)|}{3}\notag\\={}&\frac{2\log 2}{3}+\log\pi+4\log|\eta(z)|+\frac{\log|1-\lambda(2z)|}{3}+\frac{5\log|\lambda(2z)|}{6}.\label{eq:lambda_deriv_abs}
\end{align}

Naturally, the Kontsevich--Zagier integral representation in  Eq.~\ref{eq:G2Hecke234_Pnu} and the limit formulae in  Eqs.~\ref{eq:GZ_lim1}--\ref{eq:GZ_lim3} allow us to compute {\allowdisplaybreaks\begin{align}&
G_2^{\mathfrak H/\overline{\varGamma}_0(4)}(z,z+\delta)\notag\\={}&\frac{2}{\pi \I(z+\delta)}\R\int_0^{\lambda(2(z+\delta))}[\mathbf K(\sqrt{t})]^2\varrho_{2,-1/2}(1-2t|z)\left[\frac{i\mathbf K(\sqrt{1-t})}{\mathbf K(\sqrt{t})}-\vphantom{\overline{\frac{1}{}}}2(z+\delta)\right]\left[\frac{i\mathbf K(\sqrt{1-t})}{\mathbf K(\sqrt{t})}-2\overline{(z+\delta)}\right]\D t\notag\\&+\frac{2}{\pi  \I(z+\delta)}\R\int_{0}^1[\mathbf K(\sqrt{1-t})]^2\varrho_{2,-1/2}(1-2t|z)\D t\notag\\={}&-\frac{1}{4\I(z+\delta)}\R\int_0^{\lambda(2(z+\delta))}[\mathbf K(\sqrt{t})]^2\left[\frac{i\mathbf K(\sqrt{1-t})}{\mathbf K(\sqrt{t})}-\vphantom{\overline{\frac{1}{}}}2(z+\delta)\right]\left[\frac{i\mathbf K(\sqrt{1-t})}{\mathbf K(\sqrt{t})}-2\overline{(z+\delta)}\right]\times\notag\\&\times y\frac{\partial}{\partial y}\left\{\frac{1}{t-\lambda(2z)}\frac{1}{[\mathbf K(\sqrt{\lambda(2z)})]^2y}\right\}\D t\notag\\&-\frac{1}{4\I(z+\delta)}y\frac{\partial}{\partial y}\R\left\{\int_{0}^1\frac{[\mathbf K(\sqrt{1-t})]^2}{t-\lambda(2z)}\frac{\D t}{[\mathbf K(\sqrt{\lambda(2z)})]^2y}\right\}\notag\\={}&-\frac{|z+\delta|^{2}}{  \I(z+\delta)}y\frac{\partial}{\partial y}\R\left\{\int^{\lambda(2(z+\delta))}_0\frac{[\mathbf K(\sqrt{t})]^2}{t-\lambda(2z)}\frac{\D t}{[\mathbf K(\sqrt{\lambda(2z)})]^2y}\right\}\notag\\&+\frac{\R(z+\delta)}{ \I(z+\delta)}y\frac{\partial}{\partial y}\R\left\{\int_0^{\lambda(2(z+\delta))}\frac{i\mathbf K(\sqrt{1-t})\mathbf K(\sqrt{t})}{t-\lambda(2z)}\frac{\D t}{[\mathbf K(\sqrt{\lambda(2z)})]^2y}\right\}\notag\\&-\frac{1}{4  \I(z+\delta)}y\frac{\partial}{\partial y}\R\left\{\int_{\lambda(2(z+\delta))}^1\frac{[\mathbf K(\sqrt{1-t})]^2}{t-\lambda(2z)}\frac{\D t}{[\mathbf K(\sqrt{\lambda(2z)})]^2y}\right\}\notag\\={}&-\frac{|z+\delta|^{2}}{  \I(z+\delta)}y\frac{\partial}{\partial y}\R\left\{\frac{\log\varepsilon+o(1)}{y}+\pi i\frac{x}{|x|}\frac{1}{y}-\frac{2\pi i z}{3y}-\frac{2\log[4\lambda(2z)(1-\lambda(2z))]}{3y}\right\}\notag\\{}&+\frac{\R(z+\delta)}{ \I(z+\delta)}y\frac{\partial}{\partial y}\R\left\{\frac{2z\log\varepsilon+o(1)}{y}-\pi i\frac{(1-2z\frac{x}{|x|})^2}{2y}+\frac{\pi i(1+2z^{2})}{3y}-\frac{4z\log[4\lambda(2z)(1-\lambda(2z))]}{3y}\right\}\notag\\&-\frac{1}{4  \I(z+\delta)}y\frac{\partial}{\partial y}\R\left\{\frac{4z^{2}\log\varepsilon+o(1)}{y}+\frac{2\pi iz}{3y}-\frac{8z^{2}\log[4\lambda(2z)(1-\lambda(2z))]}{3y}\right\}\notag\\={}&2\log\varepsilon-\frac{4}{3}\log|4\lambda(2z)(1-\lambda(2z))|+o(1),\quad\text{as }\delta\to0.\label{eq:G2_Hecke4_close_pts}
\end{align}}Here, to arrive at  the last step of Eq.~\ref{eq:G2_Hecke4_close_pts}, we have replaced all the occurrences of $ z+\delta$ in the penultimate step with $ z$ (which introduces at most $ O(|\delta|\log|\delta|)$ error).

For points $ z\in\mathfrak H$ in the range specified by Eq.~\ref{eq:z_narrow_range}, one can directly check that Eq.~\ref{eq:G2_Hecke4_z_rn_eval} is a consequence of Eq.~\ref{eq:lambda_deriv_abs} and \ref{eq:G2_Hecke4_close_pts}:\begin{align*}&
G_2^{\mathfrak H/\overline{\varGamma}_0(4)}(z,z+\delta)-2\log|\delta|-2\log\left\vert 2\pi[\eta(z)]^4 \right|\notag\\={}&-\frac{4}{3}\log|4\lambda(2z)(1-\lambda(2z))|+\frac{4\log 2}{3}+\frac{2\log|1-\lambda(2z)|}{3}+\frac{5\log|\lambda(2z)|}{3}+o(1).
\end{align*} The rest of our claim follows from continuity and  $ \varGamma_{0}(4)$-invariance with respect to the variable $z$, which are manifested by the integral representation in Eq.~\ref{eq:G2_Hecke4_z_rn}.    \end{proof}
\begin{remark}
As a side note, we mention that many explicit formulae for the function $ \Delta(z)/\Delta(2z)$ at CM points $z$ have been produced  in Weber's treatise on elliptic functions \cite{WeberVol3,WeberErrata}, as well as Ramanujan's notebooks \cite[][Chap.~34]{RN5}. To put their work into the context of automorphic Green's functions, we give a couple of examples involving $ G_2^{\mathfrak H/\overline{\varGamma}_0(4)}(z)$:
\allowdisplaybreaks{\begin{align}&
G_2^{\mathfrak H/\overline{\varGamma}_0(4)}\left( \frac{1+i\sqrt{47}}{2} \right)\notag\\={}&-4\log2-8\log\xi,\quad\text{where }\xi>0,\xi ^5-\xi ^3-2 \xi ^2-2 \xi -1=0;\label{eq:ex_47}\\& G_2^{\mathfrak H/\overline{\varGamma}_0(4)}\left( \frac{i\sqrt{1848}}{2} \right)\notag\\={}&7 \log 2-4 \log (1+\sqrt{3})-4 \log (\sqrt{2}+\sqrt{3})-4 \log (3+\sqrt{7})-2 \log(3+\sqrt{11})-2 \log (\sqrt{3}+\sqrt{7})\notag\\{}&-2 \log (\sqrt{6}+\sqrt{7})-2 \log (\sqrt{7}+\sqrt{11})-2 \log (2 \sqrt{2}+\sqrt{7})-2 \log (7 \sqrt{2}+3 \sqrt{11})\notag\\{}&-2 \log (5 \sqrt{3}+\sqrt{77})-\log (\sqrt{21}+\sqrt{22})-\frac{2}{3} \log (22 \sqrt{22}+39 \sqrt{7}).\label{eq:ex_1848}
\end{align}}Here in Eq.~\ref{eq:ex_47}, the quintic polynomial equation satisfied by $\xi$ has  a unique positive root and a solvable Galois group $D_5$ (the pentagon symmetry group). An explicit radical form for $ \xi$ is known \cite{WatsonQuintic}:
\begin{align*}
\xi=\sqrt{\frac{10}{\sqrt[5]{A}-\sqrt[5]{B}-\sqrt[5]{C}-\sqrt[5]{D}}}
\end{align*}where\begin{align*}A={}&+39000+18200\sqrt{5}+(1720+920\sqrt{5})\sqrt{\smash[b]{\vphantom{\tfrac12}\smash[t]{235+94\sqrt{5}}}};\notag\\B={}&-39000-18200\sqrt{5}+(1720+920\sqrt{5})\sqrt{\smash[b]{\vphantom{\tfrac12}\smash[t]{235+94\sqrt{5}}}};\notag\\C={}&-39000+18200\sqrt{5}-(1720-920\sqrt{5})\sqrt{\smash[b]{\vphantom{\tfrac12}\smash[t]{235-94\sqrt{5}}}};\notag\\D={}&-39000+18200\sqrt{5}+(1720-920\sqrt{5})\sqrt{\smash[b]{\vphantom{\tfrac12}\smash[t]{235-94\sqrt{5}}}},\end{align*} and all the radicals represent positive roots.
The evaluation in Eq.~\ref{eq:ex_1848} paraphrases the last entry in Weber's list  \cite[][Tabelle~VI]{WeberVol3}.\eor\end{remark}\begin{remark}An interesting consequence of the explicit evaluation for $ G_2^{\mathfrak H/\overline{\varGamma}_0(4)}(z)$ (Eq.~\ref{eq:G2_Hecke4_z_rn_eval}) is the following logarithmic addition formula\begin{align}
G_2^{\mathfrak H/\overline{\varGamma}_0(4)}\left( \frac{z+1}{2} ,\frac{z}{2}\right)=G_2^{\mathfrak H/\overline{\varGamma}_0(4)}\left( -\frac{1}{2(2z+1)} ,z+\frac{1}{2}\right)-\frac{1}{3}\log\frac{|1-\lambda(2z)|^{2}}{4|\lambda(2z)|},\quad \alpha_2(z)\in\mathbb C\label{eq:G2_Hecke4_log_add_1}
\end{align} connecting two (un-renormalized) weight-4 automorphic Green's functions on $ \overline\varGamma_0(4)$. To prove Eq.~\ref{eq:G2_Hecke4_log_add_1}, we first reformulate the addition formulae in Eqs.~\ref{eq:Hecke2_Gamma2_add} and \ref{eq:G_s_Theta_add_form} as\begin{align}
G_2^{\mathfrak H/\overline{\varGamma}_0(2)}(z,z')={}&G_2^{\mathfrak H/\overline{\varGamma}_0(4)}\left( \frac{z}{2} ,\frac{z'}{2}\right)+G_2^{\mathfrak H/\overline{\varGamma}_0(4)}\left( \frac{z+1}{2} ,\frac{z'}{2}\right),\tag{\ref{eq:Hecke2_Gamma2_add}$'$}\\G_2^{\mathfrak H/\overline{\varGamma}_0(2)}(z,z')={}&G_2^{\mathfrak H/\overline{\varGamma}_0(4)}\left( z+\frac{1}{2} ,z'+\frac{1}{2}\right)+G_2^{\mathfrak H/\overline{\varGamma}_0(4)}\left( -\frac{1}{2(2z+1)} ,z'+\frac{1}{2}\right);\tag{\ref{eq:G_s_Theta_add_form}$'$}\label{eq:G_s_Theta_add_form'}
\end{align}  then, the $ z'\to z$ limit brings us to  \begin{align}
G_2^{\mathfrak H/\overline{\varGamma}_0(2)}(z)={}&G_2^{\mathfrak H/\overline{\varGamma}_0(4)}\left( \frac{z}{2} \right)+G_2^{\mathfrak H/\overline{\varGamma}_0(4)}\left( \frac{z+1}{2} ,\frac{z}{2}\right)-2\log2+\frac{1}{3}\log\frac{2^4|1-\lambda(z)|^{2}}{|\lambda(z)|}\notag\\={}&G_2^{\mathfrak H/\overline{\varGamma}_0(4)}\left( \frac{z+1}{2} ,\frac{z}{2}\right)-2\log2,\label{eq:G2_Hecke2_rn_add_1}\\G_2^{\mathfrak H/\overline{\varGamma}_0(2)}(z)={}&G_2^{\mathfrak H/\overline{\varGamma}_0(4)}\left( z+\frac{1}{2} \right)+G_2^{\mathfrak H/\overline{\varGamma}_0(4)}\left( -\frac{1}{2(2z+1)} ,z+\frac{1}{2}\right)-\log|1-\lambda(2z)|\notag\\={}&G_2^{\mathfrak H/\overline{\varGamma}_0(4)}\left( -\frac{1}{2(2z+1)} ,z+\frac{1}{2}\right)-\frac{1}{3}\log\frac{2^4|1-\lambda(2z)|^{2}}{|\lambda(2z)|},\label{eq:G2_Hecke2_rn_add_2}
\end{align}so long as $  \alpha_2(z)\in\mathbb C$. Eliminating  $ G_2^{\mathfrak H/\overline{\varGamma}_0(2)}(z)$ from Eqs.~\ref{eq:G2_Hecke2_rn_add_1} and  \ref{eq:G2_Hecke2_rn_add_2}, we arrive at Eq.~\ref{eq:G2_Hecke4_log_add_1}.\eor\end{remark}\appendix\section{\label{app:algebraicity}Algebraicity of Some Classical Modular Functions}\setcounter{equation}{0}\numberwithin{equation}{section}\numberwithin{theorem}{section}In Eq.~\ref{eq:fns_alg_val_at_CM_pts}, we supplied a brief list of special functions that assume solvable algebraic values at CM points. We also mentioned that these algebraic numbers could generate abelian extensions of imaginary quadratic fields. These facts are well known to experts. For the sake of completeness, we outline the arithmetic theory that supports all these classical  statements.

When $ z$ is a CM point,  the algebraicity of the $ j$-invariant (\textit{i.e.~}$ j(z)\in\overline{\mathbb Q}$)  follows  from isomorphism classes of an elliptic curve $ E_z(\mathbb C)\cong \mathbb C/(2\pi i\mathbb Z+2\pi i z\mathbb Z)$~\cite[][Theorem~4.14]{Shimura1994}. Two CM points $ z_1$ and $ z_2$ generate the same field extension $ \mathbb Q(z_1,j(z_1))=\mathbb Q(z_2,j(z_2))$ if the discriminants of their respective minimal polynomials coincide: $ z_1\in\mathfrak Z_{D_1}=\mathfrak Z_{D_2}\ni z_2$. (Here,  $ \mathfrak Z_D:=\{z\in\mathfrak H|\exists a,b,c \in\mathbb Z,a>0,\gcd(a,b,c)=1,b^2-4ac=D,az^2+bz+c=0\}$, as in Remark~\ref{rmk:GKZ_class_number_one}.) Via isomorphism to a (necessarily commutative) class group of the ring $ \mathbb Z[(D+i\sqrt{|D|})/2]$, one can show that the (finite) Galois group $ \Gal(\mathbb Q(z,j(z))/\mathbb Q(z)),z\in\mathfrak Z_D$ is abelian~\cite[][Theorem~5.7(i)]{Shimura1994}, and consequently solvable.

For any given  integers $ a,b,c\in\mathbb Z$ satisfying $ ac>0$, we have two claims: (i)~The expression\begin{align}
\frac{\Delta\left( \frac{az+b}{c} \right)}{\Delta(z)}\label{eq:Delta_abc_ratio}
\end{align}can be locally identified with an algebraic function  $A_{a,b,c} (j(z))$; (ii)~For all  CM points $z$, both the function  $A_{a,b,c} (j(z))$ and its logarithmic derivative \begin{align}
\frac{A_{a,b,c} '(j(z))}{A\vphantom{'}_{a,b,c} (j(z))},\quad\text{where }A'_{a,b,c}(j):=\frac{\partial A\vphantom{'}_{a,b,c}(j)}{\partial j}\label{eq:Aabc_deriv}
\end{align}assume algebraic values solvable in radical form.
Based on a factorization argument~\cite[][p.~167]{LangGTM112}, it would suffice to demonstrate the truthfulness of these two claims for the cases where  $ (a,b,c)=(p,0,1)$ and $(a,b,c)=(1,m,p),m\in\mathbb Z\cap[0,p-1] $, with $p$ being a prime number. Concretely speaking, the expression in Eq.~\ref{eq:Delta_abc_ratio} factorizes into a finite product of functions in the forms of \begin{align}
\frac{\Delta(p z')}{\Delta(z')}\in\mathbb Q(j(z'),j(pz')),\quad \text{and}\quad \frac{\Delta\left( \frac{z'+m}{p} \right)}{\Delta(z')}=\frac{\Delta\left( \frac{z'+m}{p} \right)}{\Delta\left(p \frac{z'+m}{p} \right)}\in\mathbb Q\left( j\left( \frac{z'+m}{p} \right) ,j(z')\right),\;m\in\mathbb Z\cap[0,p-1],\label{eq:appl_Shimura}
\end{align}
with $ z'=(a'z+b')/(c'z+d')$ for some integers $ a',b',c',d'\in\mathbb Z$ (dependent on $a$, $b$, $c$, $p$ and $m$, but not on $z$) such that $ a'd'-b'c'>0$. Here, all the members in the field extension $ \mathbb Q(X,Y)$ are rational functions of variables $X$ and $Y$ with rational coefficients (dependent on $p$ but not on $z'$ or $m$); the two $ j$-invariants $ j(z)$ and $ j(pz)$ are tied to each other via an algebraic equation $ \Phi_p(j(z),j(pz))=0$ --- the modular  equation of $ p$-th degree (which is a bivariate polynomial equation with integer coefficients that depend on $p$ but not on $z$). The field extension relations in Eq.~\ref{eq:appl_Shimura} are well-known results in the theory of automorphic functions (see~\cite[][Proposition~6.9(2)]{Shimura1994} and~\cite[][Chap.~11, \S2, Corollary~2]{LangGTM112}). Thus far, we have constructed the algebraic function $ A_{a,b,c}(j)$  along with its logarithmic derivative $ A'_{a,b,c}(j)/A\vphantom{'}_{a,b,c}(j)$, and have demonstrated the algebraicity of their CM values. This  confirms the two claims at the beginning of the current paragraph.

In the light of the foregoing arguments,  we know that for any integer $N$, \begin{align}
\frac{\Delta(Nz)}{\Delta(z)}\quad  \text{and}\quad\left[\frac{1}{\partial j(z)/\partial z}\frac{\partial}{\partial z}\log\frac{\Delta(Nz)}{\Delta(z)}\right]^2=\frac{[NE_{2}(Nz)-E_2(z)]^{2}}{j(z)[j(z)-1728]E_{4}(z)}
=\frac{\Delta(z)[NE_{2}(Nz)-E_2(z)]^{2}}{j(z)E_{4}(z)[E_{6}(z)]^2}\label{eq:Delta_ratios_deriv}\end{align} assume solvable algebraic values (possibly infinity, as one has $ j(e^{\pi i/3})=E_4(e^{\pi i/3})=E_6(i)=0$) at CM points.
These algebraic numbers (if finite) are expressible as multivariate rational functions (of rational coefficients) for certain CM values of the $j$-invariant: they belong to a field extension $ \mathbb Q(j(z_1),\dots,j(z_n))$ where all the $n$ numbers $z_1,\dots,z_n$ reside in the imaginary quadratic field $ \mathbb Q(z)$. In fact, any such algebraic number $ \alpha_z\in \mathbb Q(j(z_1),\dots,j(z_n))$ generates an abelian extension of the field $K=\mathbb Q(z)$, whose ring of algebraic integers is  $ \mathfrak o_{\raisebox{-0em}{$_K$}}=\mathbb Z[z_{\raisebox{-0em}{$_K$}}]$ for a certain $ z_{\raisebox{-0em}{$_K$}}\in K$. To show this, we may suppose that $ z_\ell\in\mathfrak Z_{D_\ell},\ell\in\mathbb{Z}\cap[1,n]$ and pick   positive integers $ f_\ell\in\mathbb Z_{>0}$ such that $\mathfrak o_\ell=\mathbb Z[ z_{\raisebox{-0em}{$_K$}}f_\ell]=\mathbb Z[(D_\ell+i\sqrt{\smash[b]{|D_\ell|}})/2] $,  with $ f_\ell$ being the conductor of $ \mathfrak o_\ell$, $ \ell\in\mathbb{Z}\cap[1,n]$ \cite[][p.~91]{LangGTM112}. Take a ring $\mathfrak o =\mathbb Z[z_{\raisebox{-0em}{$_K$}}f]$ where  $ f=\lcm(f_1,\dots,f_n)$, then we have $ \mathfrak o\subseteq\bigcap_{\ell=1}^n \mathfrak o_\ell$, which entails the field inclusion relations $ \mathbb Q(z_{\raisebox{-0em}{$_K$}},j(z_{\raisebox{-0em}{$_K$}}f))\supseteq\mathbb Q(z_\ell,j(z_\ell))$, $ \ell\in\mathbb{Z}\cap[1,n]$~\cite[][p.~134]{LangGTM112}. Thus, the tower of fields $K\subseteq K(\alpha_z)\subseteq K(j(z_{\raisebox{-0em}{$_K$}}f)) $ reveals the commutativity of the Galois group $ \Gal(K(\alpha_z)/K)\cong\Gal(K(j(z_{\raisebox{-0em}{$_K$}}f))/K)/\Gal(K(j(z_{\raisebox{-0em}{$_K$}}f))/K(\alpha_z))$.

Then, we show that the same algebraic property holds if one replaces $ NE_{2}(Nz)-E_2(z)$ by $ E_2(z)$ in the last expression of Eq.~\ref{eq:Delta_ratios_deriv}. Suppose that a CM  point $ z_*\in\mathfrak H$  solves the quadratic equation $ az_*^2+bz_*+c=0$ with integer coefficients $ (a,b,c)$. Then it is straightforward to compute that \begin{align}&
\left.\frac{\partial}{\partial z}\right|_{z=z_*}\log\frac{\Delta\left( \frac{az+b}{c} \right)}{\Delta(z)}+\left.\frac{1}{z_*^2}\frac{\partial}{\partial z}\right|_{z=-1/z_*}\log\frac{\Delta\left( \frac{cz-b}{a} \right)}{\Delta(z)}\notag\\={}&\frac{b^2-4ac}{ac}\left[\frac{\Delta'(z_{*})}{\Delta(z_{*})}-\frac{6i}{\I z_*}\right]=2\pi i\frac{b^2-4ac}{ac}E_2(z_*),
\end{align}so the assertion follows from the properties of the logarithmic derivative that appeared in  Eq.~\ref{eq:Aabc_deriv}.

Thus, we have confirmed that the following arithmetic functions\begin{align*}j(z),\quad &\frac{E_2(Nz)}{E_2(z)},\quad \left[ \frac{E_4(Nz)}{E_4(z)} \right]^3=\frac{\Delta(Nz)j(Nz)}{\Delta(z)j(z)},\notag\\ \left[ \frac{E_6(Nz)}{E_6(z)} \right]^2={}&\frac{\Delta(Nz)[j(Nz)-1728]}{\Delta(z)[j(z)-1728]},\quad \frac{[E_2(z)]^2}{E_4(z)},\quad \frac{[E_2(z)]^3}{E_6(z)}\end{align*}all possess the desired algebraic attributes: when $ [\mathbb Q(z):\mathbb Q]=2$, a number from the list above  is either  infinity or an algebraic number that generates an abelian extension of $ \mathbb Q(z)$.

\renewcommand{\thetheorem}{\Alph{section}.\arabic{theorem}}\renewcommand{\thelemma}{\Alph{section}.\arabic{lemma}}\numberwithin{equation}{section}\section{Solutions to  $ \dim\mathcal S_k(\varGamma)=0$ on Congruence Subgroups $ \varGamma$\label{app:dimSk_vanish}}

For completeness, we present in this appendix an expository note on the solutions to the cusp-form-free condition $ \dim\mathcal S_k(\varGamma)=0$, where $ \varGamma$ is a congruence subgroup of $ SL(2,\mathbb Z)$ that has arithmetic interest, and $ k$ is an even number greater than or equal to $4$.

For  $ \varGamma=SL(2,\mathbb Z)$, the dimension formulae for  even weights $ k\geq4$ are familiar (see~\cite[][Proposition 2.26]{Shimura1994} or~\cite[][Theorem 3.5.2]{DiamondShurman}):\begin{align}
\dim\mathcal S_k(SL(2,\mathbb Z))=\begin{cases}\left\lfloor\dfrac{k}{12}\right\rfloor-1, & k\equiv2\pmod {12}; \\[8pt]
\left\lfloor\dfrac{k}{12}\right\rfloor, & k\not\equiv2\pmod {12}. \\
\end{cases}\label{eq:dimSk_SL2Z}
\end{align} Here, $ \lfloor x\rfloor$ stands for the greatest integer less than or equal to $ x\in\mathbb R$.  It is  clear from Eq.~\ref{eq:dimSk_SL2Z} that for even weights $ k\geq4$, the only solutions to $ \dim\mathcal S_k(SL(2,\mathbb Z))=0$ are $ k=4,6,8,10,14$.

If we shift our attention to $\varGamma(N)$, the principal congruence group of  level $N\geq2$, which consists of transformations in $ SL(2,\mathbb Z)$ that are congruent to the identity matrix  $ \left(\begin{smallmatrix}1&0\\0&1\end{smallmatrix}\right)$ modulo an integer $ N\geq2$, then we have the following dimension formula for even weights $ k\geq4$ (as one may compute from~\cite[][Eqs.~1.6.2--1.6.3, Proposition 1.40, Theorem 2.24]{Shimura1994} or~\cite[][Theorems 3.1.1, 3.5.1, Figure~3.4]{DiamondShurman}):\begin{align}\dim\mathcal S_k(\varGamma(N))=\begin{cases}\dfrac{k-4}{2}, & N=2 \\[8pt]
\dfrac{[(k-1)N-6]N^{2}}{24}\nprod\limits_{p\mid N}\left( 1-\dfrac{1}{p^{2}} \right), & N>2 \\
\end{cases}\end{align} and $ \dim\mathcal S_k(\varGamma(N))=0$ admits just one solution: $ k=4,N=2$.\footnote{The product ``$ \prod_{p\mid N}$'' runs through distinct prime numbers $ p$ that divide $ N$, so the result never vanishes. The equation $ (k-1)N-6=0$ has no integer solutions such that $N>2,k\geq4$. }

The compact Riemann surfaces $ X_0(N)(\mathbb C)=\varGamma_0(N)\backslash\mathfrak H^*$  and $ X_1(N)(\mathbb C)=\varGamma_1(N)\backslash\mathfrak H^*$ are important to the arithmetic studies of elliptic curves. Here, the Hecke congruence groups of level $N$ are  defined by \begin{align*}\varGamma_0(N):=\left\{ \left.\begin{pmatrix}a & b \\
Nc & d \\
\end{pmatrix}\right|a,b,c,d\in\mathbb Z;ad-Nbc=1 \right\},\end{align*}and their projective counterparts are $\overline {\varGamma}_0(N)=\varGamma_0(N)/\{I,-I\} $; for the congruence groups \begin{align*}\varGamma_1(N):=\left\{ \left.\begin{pmatrix}1+Na & b \\
Nc & 1+Nd \\
\end{pmatrix}\right|a,b,c,d\in\mathbb Z;(1+Na)(1+Nd)-Nbc=1 \right\},\end{align*}the respective projective versions are $ \overline {\varGamma}_1(2)={\varGamma}_1(2)/\{I,-I\}$ and $ \overline {\varGamma}_1(N)={\varGamma}_1(N)$ for $ N>2$. For  $ N=2,$ $3$ or $4$, it is well known that the projective congruence group $ \overline{\varGamma}_1(N)$ coincides with the  projective  Hecke congruence group $\overline{ \varGamma}_0(N)$ of the same level $ N$~\cite[][Exercise 3.9.3]{DiamondShurman}.

In this work, we are concerned with weight-4 automorphic Green's functions $ G_2^{\mathfrak H/\overline{\varGamma}_0(N)}(z,z')=G_2^{\mathfrak H/\overline{\varGamma}_1(N)}(z,z')$ for $ N\in\{2,3,4\}$, as well as a weight-6 automorphic Green's function  $ G_3^{\mathfrak H/\overline{\varGamma}_0(2)}(z,z')=G_3^{\mathfrak H/\overline{\varGamma}_1(2)}(z,z')$. These cases account for all the ``cusp-form-free'' scenarios relevant to  the Gross--Kohnen--Zagier algebraicity conjecture on congruence subgroups of levels $N\geq2$, as explained in the next lemma.\begin{lemma}[Solutions to $ \dim\mathcal S_k(\varGamma_0(N))=0$ and $ \dim\mathcal S_k(\varGamma_1(N))=0$]\begin{enumerate}[label=\emph{(\alph*)}, ref=(\alph*), widest=a] \item For even weights $ k\geq4$, there are only four isolated cases free from  cusp forms on $ \varGamma_0(N)$ (see \cite[][Table~A]{Miyake1989J1}):\begin{align}
\dim\mathcal S_4(\varGamma_0(2))=0,\quad \dim\mathcal S_4(\varGamma_0(3))=0,\quad \dim\mathcal S_4(\varGamma_0(4))=0,\quad \dim\mathcal S_6(\varGamma_0(2))=0.\label{eq:Gamma_0_cusp_dim}
\end{align}\item The solutions to the equation $ \dim{\mathcal S_k(\varGamma_1(N))}=0$ for even weights $k\geq4 $ are exhausted by\begin{align}
\dim\mathcal S_4(\varGamma_1(2))=0,\quad \dim\mathcal S_4(\varGamma_1(3))=0,\quad \dim\mathcal S_4(\varGamma_1(4))=0, \quad\dim\mathcal S_6(\varGamma_1(2))=0.\label{eq:Gamma_1_cusp_dim}
\end{align}\end{enumerate}\end{lemma}\begin{proof}\begin{enumerate}[label=(\alph*),widest=a]\item For even weights $ k\geq4$, we may compute $ \dim\mathcal S_k(\varGamma_0(N))$ explicitly, drawing on some classical results.

Let $ \mathfrak H^*:=\mathfrak H\cup\mathbb Q\cup\{i\infty\}$ be the extended upper half-plane, then the genus of the compact Riemann surface $ X_{0}(N)(\mathbb C):=\varGamma_0(N)\backslash\mathfrak H^*$ is given by (see~\cite[][Proposition~1.40]{Shimura1994} or~\cite[][Theorem~3.1.1]{DiamondShurman})\begin{align*}g(\varGamma_0(N)\backslash\mathfrak H^*)=1+\frac{N}{12}\prod_{p\mid N}\left( 1+\frac{1}{p} \right)-\frac{\nu_2(\varGamma_0(N))}{4}-\frac{\nu_3(\varGamma_0(N))}{3}-\frac{\nu_\infty(\varGamma_0(N))}{2}.\end{align*}Here, $ \nu_2(\varGamma_0(N))$  (resp.~$ \nu_3(\varGamma_0(N))$) counts the number of inequivalent elliptic fixed points of period $2$ (resp.~$3$), explicitly quantified by the formulae (see~\cite[][Proposition 1.43]{Shimura1994} or~\cite[][Figure~3.3]{DiamondShurman}):\begin{align}\nu_{2}(\varGamma_0(N))=\begin{cases}\nprod\limits_{p\mid N}\left( 1+\left( \dfrac{-1}{p} \right) \right), & 4\nmid N, \\
0, & 4\mid N, \\
\end{cases}\qquad \qquad\nu_3(\varGamma_0(N))=\begin{cases}\nprod\limits_{p\mid N}\left( 1+\left( \dfrac{-3}{p} \right) \right), & 9\nmid N, \\
0, & 9\mid N, \\
\end{cases}\label{eq:nu2_nu3}\end{align} where the Kronecker--Legendre symbols are evaluated as~\cite[][Proposition 1.43]{Shimura1994}\begin{align*}\left( \frac{-1}{p} \right)=\begin{cases}0, & p=2, \\
1, & p\equiv1\bmod4, \\
-1, & p\equiv3\bmod4, \\
\end{cases}\qquad\qquad\left( \frac{-3}{p} \right)=\begin{cases}0, & p=3, \\
1, & p\equiv1\bmod3, \\
-1, & p\equiv2\bmod3. \\
\end{cases}\end{align*}In the meantime, \begin{align} \nu_\infty(\varGamma_0(N))=\sum_{d\mid N}\varphi(\gcd(d,N/d))\label{eq:nu_inf}\end{align} counts the number of inequivalent cusps in the fundamental domain of $ \varGamma_0(N)$, where the sum ``$ \sum_{d\mid N}$'' runs over every positive divisor $d$ of $N$, ``$ \gcd(n,m)$'' denotes the greatest common divisor for a pair of   integers $ (n,m)$, and $ \varphi(n):=n\prod_{p|n}(1-\frac1p)$ is Euler's totient function.

For an even weight $ k\geq4$, we may quote the dimension formula (see~\cite[][Theorem~2.24]{Shimura1994} or~\cite[][Theorem~3.5.1]{DiamondShurman})\begin{align*}\dim\mathcal S_k(\varGamma_0(N))=(k-1)(g(\varGamma_0(N)\backslash\mathfrak H^*)-1)+\left\lfloor\frac{k}{4}\right\rfloor\nu_2(\varGamma_0(N))+\left\lfloor\frac{k}{3}\right\rfloor\nu_3(\varGamma_0(N))+\left( \frac{k}{2} -1\right)\nu_\infty(\varGamma_0(N)).\end{align*}For $ N\in\{2,3,4\}$, this specializes to the formulae\begin{align*}\dim\mathcal S_k(\varGamma_0(2))=\left\lfloor\frac{k}{4}\right\rfloor-1,\quad\dim\mathcal S_k(\varGamma_0(3))=\left\lfloor\frac{k}{3}\right\rfloor-1, \quad \dim\mathcal S_k(\varGamma_0(4))=\frac{k-4}{2},\end{align*} which confirm the equalities displayed in Eq.~\ref{eq:Gamma_0_cusp_dim}.

To rule out solutions to the equation $\dim\mathcal S_k(\varGamma_0(N))=0 $ with an integer $N>4$ for even weights $ k\geq4$, we   need to assess the lower bound of\begin{align*}\dim\mathcal S_k(\varGamma_0(N))={}&\frac{(k-1)N}{12}\prod_{p\mid N}\left( 1+\frac{1}{p} \right)+\left(\frac{1}{4}+\left\lfloor\frac{k}{4}\right\rfloor-\frac{k}{4}\right)\nu_2(\varGamma_0(N))\notag\\{}&+\left(\frac{1}{3}+\left\lfloor\frac{k}{3}\right\rfloor-\frac{k}{3}\right)\nu_3(\varGamma_0(N))-\frac{\nu_\infty(\varGamma_0(N))}{2},\end{align*}starting from the following estimates of G. Martin~\cite{Martin2005}:\begin{align*}\max\{\nu_2(\varGamma_0(N)),\nu_3(\varGamma_0(N))\}\leq2^{4-\frac{\log16}{\log11}}N^{\frac{\log2}{\log11}},\quad \nu_\infty(\varGamma_0(N))\leq\sqrt{N}\prod_{p\mid N}\left( 1+\frac{1}{p} \right).\end{align*}

Now that we have the lower bound\begin{align*}\left(\frac{1}{4}+\left\lfloor\frac{k}{4}\right\rfloor-\frac{k}{4}\right)+\left(\frac{1}{3}+\left\lfloor\frac{k}{3}\right\rfloor-\frac{k}{3}\right)\geq-\frac{7}{12}\end{align*}valid for even integers $k$, we obtain the inequality \begin{align*}\dim\mathcal S_k(\varGamma_0(N))\geq\frac{(k-1)N-6\sqrt{N}}{12}\prod_{p\mid N}\left( 1+\frac{1}{p} \right)-\frac{7}{12}2^{4-\frac{\log16}{\log11}}N^{\frac{\log2}{\log11}}>\frac{3N-6\sqrt{N}}{12}-\frac{7}{12}2^{4-\frac{\log16}{\log11}}N^{\frac{\log2}{\log11}}>0\end{align*}for $ N\geq77$.

Since $ 2\times3\times 5\times7=210$, any integer in the range $ 5\leq N\leq 76$ contains at most 3 distinct prime factors. Thus, within the range $ 30\leq N\leq 76$, a more stringent upper bound\begin{align*}\max\{\nu_2(\varGamma_0(N)),\nu_3(\varGamma_0(N))\}\leq2^3=8\end{align*}allows us to refine the former inequality on dimensions as\begin{align*}\dim\mathcal S_k(\varGamma_0(N))>\frac{3N-6\sqrt{N}}{12}-\frac{7}{12}\times8>0.\end{align*}

 Then, for the regime $ 12\leq N\leq 29$, the improvement\begin{align*}\max\{\nu_2(\varGamma_0(N)),\nu_3(\varGamma_0(N))\}\leq2^{\sum_{p\mid N,p\geq5} 1}=\prod_{p\mid N,p\geq5}2\leq2\end{align*}   leads to\begin{align*}\dim\mathcal S_k(\varGamma_0(N))>\frac{3N-6\sqrt{N}}{12}-\frac{7}{12}\times2>0.\end{align*}

The remaining scenarios $ N\in\{5,6,7,8,9,10,11\}$ can be studied case by case. If $N=6$, $8$ or $9 $, then we have $\nu_2(\varGamma_0(N))=\nu_3(\varGamma_0(N))=0 $ and $ \nu_\infty(\varGamma_0(N))=4$, thus\begin{align*}\dim\mathcal S_k(\varGamma_0(N))=\frac{(k-1)N}{12}\prod_{p\mid N}\left( 1+\frac{1}{p} \right)-2\geq\dim\mathcal S_4(\varGamma_0(N))=\frac{N}{4}\prod_{p\mid N}\left( 1+\frac{1}{p} \right)-2=1.\end{align*} The  explicit formulae    \begin{align*}\dim\mathcal S_k(\varGamma_0(5))={}&2\left\lfloor \frac{k}{4} \right\rfloor-1,&\dim\mathcal S_k(\varGamma_0(7))={}&2\left\lfloor \frac{k}{3} \right\rfloor-1, \notag\\\dim\mathcal S_k(\varGamma_0(10))={}&2\left\lfloor\frac{k}{4}\right\rfloor+k-3,&\dim\mathcal S_k(\varGamma_0(11))={}&k-2\end{align*}rule out the solutions to $ \dim\mathcal S _k(\varGamma_0(N))=0$ for even weights $ k\geq4$ and $ N\in\{5,7,10,11\}$.\item When it comes to the numbers of inequivalent elliptic fixed points $ \nu_2,\nu_3$, the number of inequivalent cusps $ \nu_\infty$, and the genus $g$, the data for $ \varGamma_1(N)$ coincide with those of $ \varGamma_0(N)=\{I,-I\}\varGamma_1(N)$ for $N\in\{2,3,4\}$ (see~\cite[][Figure~3.3]{DiamondShurman}), so   the relations in Eq.~\ref{eq:Gamma_1_cusp_dim} follow from Eq.~\ref{eq:Gamma_0_cusp_dim}.

It is sensible to restrict the rest of our  discussions to $ \varGamma_1(N),N>4$, where we have \cite[][Figure~3.4]{DiamondShurman}\begin{align*}\nu_{2}(\varGamma_1(N))=0,\quad \nu_3(\varGamma_1(N))=0,\quad \nu_\infty(\varGamma_1(N))=\frac{1}{2}\sum_{d\mid N}\varphi(d)\varphi(N/d)\end{align*}and\begin{align*}g(\varGamma_1(N)\backslash\mathfrak H^*)=1+\frac{N^{2}}{24}\prod_{p\mid N}\left( 1-\frac{1}{p^{2}} \right)-\frac{1}{4}\sum_{d\mid N}\varphi(d)\varphi(N/d).\end{align*}Consequently, for even weights $ k\geq4$, we obtain~\cite[][Figure~3.4]{DiamondShurman}\begin{align*}\dim\mathcal S_k(\varGamma_1(N))={}&(k-1)(g(\varGamma_1(N)\backslash\mathfrak H^*)-1)+\left( \frac{k}{2} -1\right)\nu_\infty(\varGamma_1(N))\notag\\={}&\frac{(k-1)N^{2}}{24}\prod_{p\mid N}\left( 1-\frac{1}{p^{2}} \right)-\frac{1}{4}\sum_{d\mid N}\varphi(d)\varphi(N/d).\end{align*}If $ N=p$ turns out to be a prime number greater than or equal to $5$, we immediately have \begin{align*}\dim\mathcal S_k(\varGamma_1(p))\geq\dim\mathcal S_4(\varGamma_1(p))=\frac{3(p^{2}-1)}{24}-\frac{p-1}{2}=\frac{(p-1)(p-3)}{8}>0.\end{align*}For composite numbers $N$, we may use the inequalities\begin{align*}\prod_{p\mid N}\left( 1-\frac{1}{p^{2}} \right)>\prod_p\left( 1-\frac{1}{p^{2}} \right)=\frac{6}{\pi^2}\end{align*}and \begin{align*}\dim\mathcal S_k(\varGamma_1(N))\geq\dim\mathcal S_4(\varGamma_1(N))>\frac{3N^{2}}{4\pi^{2}}-\frac{1}{4}\sum_{d\mid N}\varphi(d)\varphi(N/d).\end{align*}Here, by virtue of the relations $ \varphi(n)\leq n$ and $ n=\sum_{d\mid n}\varphi(d)$, we may estimate\begin{align*}\sum_{d\mid N}\varphi(d)\varphi(N/d)={}&\sum_{\substack{d\mid N\\0<d\leq\sqrt{N}}}\varphi(d)\varphi(N/d)+\sum_{\substack{d\mid N\\\sqrt{N}< d\leq N}}\varphi(d)\varphi(N/d)\notag\\\leq{}&\left(\max_{n\in\mathbb Z\cap(0,\sqrt{N}]}\varphi(n)\right)\sum_{\substack{d\mid N\\0<d\leq\sqrt{N}}}\varphi(N/d)+\left(\max_{n\in\mathbb Z\cap(0,\sqrt{N}]}\varphi(n)\right)\sum_{\substack{d\mid N\\\sqrt{N}< d\leq N}}\varphi(d)\leq 2N^{3/2}.\end{align*} For $ N\geq44$, we thus have\begin{align*}\dim\mathcal S_k(\varGamma_1(N))\geq\dim\mathcal S_4(\varGamma_1(N))>\frac{3N^{2}}{4\pi^{2}}-\frac{1}{4}\sum_{d\mid N}\varphi(d)\varphi(N/d)\geq\frac{3N^{2}}{4\pi^{2}}- \frac{N^{3/2}}{2}>0.\end{align*}For integers in the range $ 27\leq N<44$, we may exploit the relation $ \max_{n\in\mathbb Z\cap(0,\sqrt{44}]}\varphi(n)=\varphi(5)=4$ to improve the estimate:\begin{align*}\sum_{d\mid N}\varphi(d)\varphi(N/d)\leq 8N,\quad \dim\mathcal S_k(\varGamma_1(N))\geq\dim\mathcal S_4(\varGamma_1(N))>\frac{3N^{2}}{4\pi^{2}}-2N>0.\end{align*}By direct computation, we can verify that the inequality\begin{align*}\dim\mathcal S_k(\varGamma_1(N))\geq\dim\mathcal S_4(\varGamma_1(N))>\frac{3N^{2}}{4\pi^{2}}-\frac{1}{4}\sum_{d\mid N}\varphi(d)\varphi(N/d)>0\end{align*}remains valid for $ N=25,\sum_{d\mid 25}\varphi(d)\varphi(25/d)=56$ and  $ N=26,\sum_{d\mid 26}\varphi(d)\varphi(26/d)=48$. Next, within the regime $ 14\leq N\leq 24$, we have $ \max_{n\in\mathbb Z\cap(0,\sqrt{24}]}\varphi(n)=\varphi(3)=\varphi(4)=2$, and accordingly,   \begin{align*}\sum_{d\mid N}\varphi(d)\varphi(N/d)\leq 4N,\quad \dim\mathcal S_k(\varGamma_1(N))\geq\dim\mathcal S_4(\varGamma_1(N))>\frac{3N^{2}}{4\pi^{2}}-N>0.\end{align*}There are only a few composite numbers within the range $ 5\leq N\leq 13$,  for which we may directly compute as follows:\begin{align*}\dim\mathcal S_4&(\varGamma_1(6))=1,\quad \dim\mathcal S_4(\varGamma_1(8))=3,\quad \dim\mathcal S_4(\varGamma_1(9))=5,\notag\\& \dim\mathcal S_4(\varGamma_1(10))=5,\quad \dim\mathcal S_4(\varGamma_1(12))=7.\end{align*}Therefore, for all integers $ N\geq5$ and even weights $ k\geq4$, we have $ \dim\mathcal S_k(\varGamma_1(N))\geq\dim\mathcal S_4(\varGamma_1(N))>0$.  \qedhere\end{enumerate}\end{proof}

\bibliography{AGF}

\bibliographystyle{plain}

\end{document}